\documentclass[11pt,letterpaper]{article}
\pdfoutput=1
\newif\iffocs
\focsfalse
\usepackage[margin=1in]{geometry}
\usepackage[ascii]{inputenc}
\usepackage{etoolbox}
\usepackage{silence}
\usepackage[dvipsnames]{xcolor}
\usepackage[bookmarksnumbered=true,linktocpage=true,hypertexnames=false,colorlinks=true,linkcolor=blue,urlcolor=BrickRed,citecolor=ForestGreen,anchorcolor=green,breaklinks=true,pdfusetitle]{hyperref}
\usepackage{bbm,braket,mathrsfs,amsmath,amssymb,amsthm,amsfonts,latexsym,mathtools,graphicx,enumitem,caption,booktabs,bm,mathdots,ellipsis,mleftright,framed}
\usepackage{thmtools} % used for thm-restate
\usepackage{microtype}
\usepackage[T1]{fontenc}
\usepackage{textcomp}
\usepackage[english]{babel}
\usepackage[capitalize,nameinlink]{cleveref}
\numberwithin{equation}{section}
\allowdisplaybreaks[4]
\setenumerate{label=(\roman*)}
\newtheorem{thm}{Theorem}[section]\crefname{thm}{Theorem}{Theorems}
\newtheorem{lem}[thm]{Lemma}\crefname{lem}{Lemma}{Lemmas}
\newtheorem{cor}[thm]{Corollary}\crefname{cor}{Corollary}{Corollaries}
\newtheorem{rem}[thm]{Remark}\crefname{rem}{Remark}{Remarks}
\newtheorem{prop}[thm]{Proposition}\crefname{prop}{Proposition}{Propositions}
\crefname{conj}{Conjecture}{Conjectures}
\newtheorem{defn}[thm]{Definition}\crefname{defn}{Definition}{Definitions}
\let\Re\undefined
\DeclareMathOperator{\Re}{Re}

\DeclareMathOperator{\Tr}{Tr}

\DeclareMathOperator{\diag}{diag}

\DeclareMathOperator{\var}{Var}
\DeclareMathOperator{\cov}{Cov}
\DeclareMathOperator{\Hess}{Hess}
\DeclareMathOperator{\grad}{grad}
\DeclareMathOperator*{\argmin}{argmin}

\DeclarePairedDelimiter\abs{\lvert}{\rvert}
\DeclarePairedDelimiter\norm{\lVert}{\rVert}
\DeclarePairedDelimiter\normHS{\lVert}{\rVert_{\mathrm{HS}}}

\DeclarePairedDelimiter\parens{\lparen}{\rparen}
\DeclarePairedDelimiter\braces{\lbrace}{\rbrace}
\DeclarePairedDelimiter\bracks{\lbrack}{\rbrack}
\DeclarePairedDelimiter\angles{\langle}{\rangle}
\newcommand{\Lie}{\mathrm{Lie}}
\newcommand{\End}{\mathrm{End}}
\newcommand{\SL}{\mathrm{SL}}
\newcommand{\GL}{\mathrm{GL}}

\newcommand{\U}{\mathrm{U}}

\newcommand{\PD}{\mathrm{PD}}
\newcommand{\SPD}{\mathrm{SPD}}
\newcommand{\Herm}{\mathrm{Herm}}
\newcommand{\CAT}{\mathrm{CAT}}

\newcommand{\Exp}{\mathrm{Exp}}

\newcommand{\CC}{\mathbbm C}\newcommand{\C}{\mathbbm C}
\newcommand{\RR}{\mathbbm R}\newcommand{\R}{\mathbbm R}

\newcommand{\N}{\mathbbm N}
\newcommand{\HH}{\mathbbm H}

\newcommand{\NP}{\mathsf{NP}}

\newcommand{\ot}{\otimes}
\newcommand{\op}{\oplus}
\newcommand{\eps}{\varepsilon}

\newcommand{\bigO}{\mathcal O}

\renewcommand{\vec}[1]{\pmb{#1}}

\newcommand{\WeirdConstant}{\zeta}

\newcommand{\medianobj}{s}
\newcommand{\sff}{\mathrm{I\!I}} %second fundamental form
\newcommand{\covder}[1]{\nabla_{\!{#1}}}
\newcommand{\sections}[1]{\Gamma\!\left({#1}\right)}
\newcommand{\transport}[2]{\tau_{#1\to#2}}

\newcommand{\lvl}{\eta}
\newcommand{\openlvl}{\mathcal{L}^\circ}
\makeatletter
\newcommand\footnoteref[1]{\protected@xdef\@thefnmark{\ref{#1}}\@footnotemark}
\makeatother
\newcommand{\newtonquadraticconvergencecontent}{%
  Let~$f\colon D \to \R$ be a strongly $\alpha$-self-concordant function defined on an open convex set~$D\subseteq M$, with positive definite Hessian.
  Let~$p\in D$ be a point such that~$\lambda_{f,\alpha}(p) < 1$.
  Then the Newton iterate remains in the domain, i.e., $p_{f,+} \in D$, and moreover
  \begin{align*}
    \lambda_{f,\alpha}(p_{f,+}) \leq \parens*{ \frac {\lambda_{f,\alpha}(p)} {1 - \lambda_{f,\alpha}(p)} }^2.
  \end{align*}
}
\newcommand{\TITLE}{Interior-point methods on manifolds: theory and applications}
\begin{document}
%-----------------------------------------------------------------------------
\iffocs
\pagestyle{empty}
\begin{center}
\LARGE\TITLE
\end{center}
%-----------------------------------------------------------------------------

%-----------------------------------------------------------------------------
\else
%-----------------------------------------------------------------------------
\title{\TITLE}
\author{
  Hiroshi Hirai\texorpdfstring{%
  \footnote{Department of Mathematical Informatics, Graduate School of Information Science and Technology, The University of Tokyo, Tokyo, 113-8656, Japan, \href{mailto:hirai@mist.i.u-tokyo.ac.jp}{hirai@mist.i.u-tokyo.ac.jp}}}{}
\and Harold Nieuwboer\texorpdfstring{%
\footnote{Korteweg-de Vries Institute for Mathematics and QuSoft, University of Amsterdam, The Netherlands and Faculty of Computer Science, Ruhr University Bochum, Germany, \href{mailto:h.a.nieuwboer@uva.nl}{h.a.nieuwboer@uva.nl}}}{}
\and Michael Walter\texorpdfstring{%
\footnote{Faculty of Computer Science, Ruhr University Bochum, Germany, \href{mailto:michael.walter@rub.de}{michael.walter@rub.de}}}{}}
\date{}
\maketitle
\fi
%-----------------------------------------------------------------------------

\begin{abstract}
  Interior-point methods offer a highly versatile framework for convex optimization that is effective in theory and practice.
  A key notion in their theory is that of a self-concordant barrier.
  We give a suitable generalization of self-concordance to Riemannian manifolds and show that it gives the same structural results and guarantees as in the Euclidean setting, in particular local quadratic convergence of Newton's method.
  We analyze a path-following method for optimizing compatible objectives over a convex domain for which one has a self-concordant barrier, and obtain the standard complexity guarantees as in the Euclidean setting.
  We provide general constructions of barriers, and show that on the space of positive-definite matrices and other symmetric spaces, the squared distance to a point is self-concordant.
  To demonstrate the versatility of our framework, we give algorithms with state-of-the-art complexity guarantees for the general class of scaling and non-commutative optimization problems, which have been of much recent interest, and we provide the first algorithms for efficiently finding high-precision solutions for computing minimal enclosing balls and geometric medians in nonpositive curvature.
\end{abstract}

\tableofcontents

%=============================================================================
\iffocs
\setcounter{page}0
\clearpage
\pagestyle{plain}
\fi
\section{Introduction and summary of results}\label{sec:introduction}
%=============================================================================

The development of interior-point methods is one of the greatest successes in convex optimization, and by now has a long history dating back to the works of Frisch~\cite{frisch1955logarithmic}, Karmarkar~\cite{karmarkarNewPolynomialtimeAlgorithm1984stoc,karmarkarNewPolynomialtimeAlgorithm1984combinatorica}, Gill et al.~\cite{gillProjectedNewtonBarrier1986} and many others.
It led to one of the first polynomial-time algorithms for linear programming (in contrast with the simplex algorithm due to Dantzig~\cite{dantzigLinearProgrammingExtensions1963}), the other being the ellipsoid method due to Khachiyan~\cite{khachiyanPolynomialAlgorithmsLinear1980}.
In the seminal work of Nesterov and Nemirovskii~\cite{nesterov-nemirovskii-ipm}, it was shown that the key property to the analysis of interior-point methods is the notion of \emph{self-concordance}.
Essentially every convex programming problem is in principle amenable to interior-point methods, which follows from constructions of self-concordant barriers for arbitrary (bounded) convex domains, cf.~\cite{nesterov-nemirovskii-ipm,hildebrand2014canonical,fox2015schwarz,bubeck2019entropic,chewi2021entropic}.
Furthermore, interior-point methods are eminently practical, and currently give the best algorithms for linear programming~\cite{leeSolvingLinearPrograms2020,vandenbrandDeterministicLinearProgram2019}.

So far, these successes have been restricted to convex optimization on Euclidean space.
While there is a strong connection between self-concordance-based interior-point methods and Riemannian geometry~\cite{duistermaatBoundaryBehaviourRiemannian,nesterovriemannian2002,nesterovPrimalCentralPaths2008}, the framework of interior-point methods has not yet been generalized to objectives which are \emph{geodesically convex}, i.e., convex on Riemannian manifolds.
Indeed, while there have been previous attempts at extending interior-point methods to this setting~\cite{udristeOptimizationMethodsRiemannian1997,ji2007optimization,jiang2007self}, a satisfactory generalization of the Euclidean theory had still been elusive -- in particular, the natural quadratic convergence analysis of Newton's method for self-concordant functions, which in turn enables efficient path-following methods with global guarantees.

Instead, research on Riemannian optimization has so far largely focused on different approaches.
There is extensive literature on first- and second-order methods for convex and non-convex optimization, see e.g.~\cite{udristeConvexFunctionsOptimization1994,absilOptimizationAlgorithmsMatrix2009,satoRiemannianOptimizationIts2021,boumal2023intromanifolds} for comprehensive overviews and~\cite{ferreiraKantorovichTheoremNewton2002,dedieuNewtonMethodRiemannian2003a,alvarezUnifyingLocalConvergence2008,sraConicGeometricOptimisation2015,zhangFirstorderMethodsGeodesically2016,ahnNesterovEstimateSequence2020a,weberRiemannianOptimizationFrankWolfe2022,srinivasanSufficientConditionsNonasymptotic2022}.
Recently,~\cite{laiRiemannianInteriorPoint2022} gave a path-following method for non-convex constrained manifold optimization % via perturbed Karush--Kuhn--Tucker conditions, but based on local Newton method guarantees
which does not use self-concordance.
In another direction, geodesic updates can also be useful for Euclidean convex optimization problems~\cite{permenterGeodesicInteriorpointMethod2020,permenterLogdomainInteriorpointMethods2022}.

\begin{framed}
\noindent
\textbf{Structural results:}
In this work, we extend the interior-point method framework to Riemannian manifolds.
We generalize the key notion of self-concordance, and show that (unlike prior definitions) it gives the same structural results and guarantees as in the Euclidean setting, in particular local quadratic convergence of Newton's method.
This allows us to give a path-following method for optimizing suitable objective functions over domains for which a self-concordant barrier is available, and we give complexity guarantees that match the Euclidean ones.
\end{framed}

These results are already very interesting from a theoretical perspective.
To put the framework to use, however, one still has to find explicit self-concordant barriers.
To this end, we give general constructions of barriers, as well as several explicit examples, and we show that our framework is indeed applicable to a wide variety of problems, including (but not limited to):
\begin{enumerate}[label={(\Alph*)}]
\item\label{item:intro geometric problem} \emph{Geometry:} Given points~$p_1, \dotsc, p_m$ on a Riemannian manifold, what is the minimum radius ball that contains all these points? What is their geometric median, i.e., the point that minimizes the sum of distances to each~$p_i$?
\item\label{item:intro quantum marginal} \emph{Quantum marginals:} Given density matrices~$\rho_1,\dots,\rho_k$, each describing the quantum state of one party, does there exist a $k$-party pure quantum state with marginals equal to the~$\rho_k$?
\item\label{item:intro peps scaling} \emph{Tensor networks:} Given a $(2d+1)$-leg tensor, does it ever define a non-zero tensor network state of PEPS type? And how can one efficiently compute a canonical form?
\item\label{item:intro brascamp lieb} \emph{Brascamp--Lieb inequalities:} Given linear maps~$L_k\colon\R^m \to \R^{m_k}$ and numbers~$q_k>0$ for~$k\in[n]$, what is the optimal constant~$C>0$ such that
$\int_{\R^m} \prod_{k=1}^n f_k(L_k x) \, dx \leq C \prod_{k=1}^n \norm{f_k}_{1/q_k}$
for all non-negative functions~$f_k$ on~$\R^{m_k}$? % (if such a constant exists)?
Many classical integral inequalities fall into this setting, such as the H\"older, Young, Loomis--Whitney, and certain hyper-contractivity inequalities.
\end{enumerate}
The first question in Problem~\ref{item:intro geometric problem} on finding a minimum enclosing ball has been studied before in the Riemannian setting~\cite{arnaudonApproximatingRiemannian1center2013,NH15}, and~\cite{NH15} gave an algorithm for the specific case of hyperbolic space, yielding a ball with radius at most a factor~$1 + \delta$ larger than the optimal radius in~$\bigO(1/\delta^2)$ iterations.
The geometric median problem has been studied in~\cite{fletcherGeometricMedianRiemannian2009,yangRiemannianMedianIts2010}, and~\cite{yangRiemannianMedianIts2010} gave an explicit subgradient algorithm on general manifolds, finding a point whose squared distance to the point achieving minimal sum of distances to the~$p_i$ is at most~$\eps$ in~$\bigO(1/\eps)$ iterations.
Interestingly, the other problems are not even obviously related to geodesically convex optimization in the first place.
Problem~\ref{item:intro quantum marginal} is not currently known to be solvable in polynomial time in all parameters, although partial results are known~\cite{burgisserEfficientAlgorithmsTensor2018focs}, and there is complexity-theoretic evidence that polynomial-time algorithms might exist, as it is in~$\NP\cap\mathsf{coNP}$~\cite{burgisserMembershipMomentPolytopes2017a}.
Problem~\ref{item:intro peps scaling} arose very recently in quantum information~\cite{acuavivaMinimalCanonicalForm2022}, and again no algorithms are known that run in polynomial time in all parameters, except for~$d=1$.
Problem~\ref{item:intro brascamp lieb} was studied in~\cite{gargAlgorithmicOptimizationAspects2017}, but current methods have an exponential dependence on the bit complexity of the coefficients~$q_j$.

What connects problems~\ref{item:intro quantum marginal}--\ref{item:intro brascamp lieb} to each other, and to geodesically convex optimization, is that they all belong to the broad class of \emph{scaling problems}.
In particular, problem~\ref{item:intro quantum marginal} can be reduced to \emph{tensor scaling}~\cite{burgisserEfficientAlgorithmsTensor2018focs}, and problems~\ref{item:intro peps scaling} and~\ref{item:intro brascamp lieb} generalize respectively reduce to \emph{operator scaling} (but not efficiently so)~\cite{acuavivaMinimalCanonicalForm2022,gargAlgorithmicOptimizationAspects2017}.
There has been much recent progress on this class of problems~\cite{linialDeterministicStronglyPolynomial2000,gurvitsClassicalComplexityQuantum2004,walterEntanglementPolytopesMultiparticle2013,walter2014multipartite,gargOperatorScalingTheory2020,gargDeterministicPolynomialTime2016,gargAlgorithmicOptimizationAspects2017,cohenMatrixScalingBalancing2017a,allen-zhuMuchFasterAlgorithms2017a,allen-zhuOperatorScalingGeodesically2018a,burgisserEfficientAlgorithmsTensor2018focs,burgisserAlternatingMinimizationScaling2018,burgisserTheoryNoncommutativeOptimization2021}, and the quest of finding better algorithms is one key motivation for our work.
Scaling problems have strong connections to many different areas in mathematics and theoretical computer science beyond those mentioned above:
they are related to approximating permanents~\cite{linialDeterministicStronglyPolynomial2000},
non-commutative rational identity testing~\cite{gargDeterministicPolynomialTime2016},
Horn's problem on spectra of sums of Hermitian matrices~\cite{franksOperatorScalingSpecified2018a},
the Paulsen problem~\cite{kwokPaulsenProblemContinuous2018,hamiltonPaulsenProblemMade2021},
canonical forms and zero-testing of tensor networks~\cite{acuavivaMinimalCanonicalForm2022},
strengthening the Sylvester--Gallai theorem~\cite{BarakDWY12,dvirIMPROVEDRANKBOUNDS2014,dvirRankBoundsDesign2018},
approximating optimal transport plans in machine learning~\cite{cuturiSinkhornDistancesLightspeed2013},
maximum-likelihood estimation in statistics~\cite{amendolaToricInvariantTheory2021,amendolaInvariantTheoryScaling2021,franksOptimalSampleComplexity2021},
the asymptotic non-vanishing of Kronecker coefficients in representation theory~\cite{ikenmeyerVanishingKroneckerCoefficients2017,burgisserEfficientAlgorithmsTensor2018focs},
and geometric invariant theory~\cite{kempfLengthVectorsRepresentation1979,nessStratificationNullCone1984,mumfordGeometricInvariantTheory1994}.
As elucidated in a long sequence of works, see~\cite{burgisserTheoryNoncommutativeOptimization2021}, these are all related to a \emph{norm minimization} problem: given a linear action of a nice (complex reductive) Lie group~$G$ on a vector space~$V$, and a vector~$v\in V$, the goal is to minimize the norm over the orbit~$G \cdot v$ (see \cref{subsubsec:application kempf ness}).
When~$G$ is commutative, such as in the case of matrix scaling, these problems reduce to geometric programming (a well-known generalization of linear programming) and hence they can be solved efficiently~\cite{cohenMatrixScalingBalancing2017a,allen-zhuMuchFasterAlgorithms2017a,singhEntropyOptimizationCounting2014,burgisserInteriorpointMethodsUnconstrained2020}.
In the most difficult situations, however, including in most of the mentioned applications, the group~$G$ is non-commutative; hence this class of problems has also been called \emph{non-commutative (group) optimization} problems.
In this case, efficient algorithms are known only in special cases, which have recently been understood to all satisfy a certain total unimodularity~\cite{gargOperatorScalingTheory2020,allen-zhuOperatorScalingGeodesically2018a,burgisserTheoryNoncommutativeOptimization2021}.
For general non-commutative optimization, and in particular for problems~\ref{item:intro quantum marginal}--\ref{item:intro brascamp lieb}, there are currently no algorithms that run in time polynomial in all parameters.

\begin{framed}
\noindent
\textbf{Algorithmic applications:}
For problem~\ref{item:intro geometric problem}, our framework gives (to the best of our knowledge) the first algorithms for efficiently finding high-precision solutions in nonpositive curvature.
For the entire class of \emph{scaling} or \emph{non-commutative optimization problems}, and in particular for problems~\ref{item:intro quantum marginal}--\ref{item:intro brascamp lieb}, our framework yields new algorithms that match the complexity guarantees of the state-of-the-art algorithms~\cite{burgisserTheoryNoncommutativeOptimization2021}, while not obviously suffering from the same obstructions as those methods, opening up a new avenue for future research.
\end{framed}

Indeed, the current state-of-the-art methods are fundamentally incapable of providing algorithms that run in polynomial time in all parameters for the general scaling problem, and in particular for problems~\ref{item:intro quantum marginal}--\ref{item:intro brascamp lieb}.
The main reason that we lack the kind of sophisticated optimization methods that are known in the Euclidean setting, as reviewed earlier, is due to the geometry of the spaces that one has to optimize over, which poses fundamental new challenges and obstructions.

To make this more concrete, we consider the quantum marginal problem~\ref{item:intro quantum marginal}.
For simplicity, we take $k=3$ parties of the same dimension~$n\geq2$.
Let~$G = \SL(n, \CC) \times \SL(n, \CC) \times \SL(n, \CC)$ act on 3-tensors in~$V = \CC^n \otimes \CC^n \otimes \CC^n$ by simultaneous base change or \emph{tensor scaling}, i.e., $(g_1, g_2, g_3) \cdot v = (g_1 \otimes g_2 \otimes g_3) v$ for~$g_j \in \SL(n,\CC)$ and $v\in V$.
Then the relevant optimization problem amounts to minimizing the (for convenience squared) $\ell^2$-norm over all such scalings:%
\footnote{To see that this is related to quantum marginals, consider the pure quantum state~$\rho = w w^*$ with~$w := g \cdot v / \norm{g \cdot v}$.
Then the (logarithmic) gradient of the objective is given by $(\rho_1 - I/n, \rho_2 - I/n, \rho_3 - I/n)$, where~$\rho_1, \rho_2, \rho_3$ are the one-body reduced density matrices or marginals of~$\rho$.
Therefore, minimizers of \cref{eq:tensor scaling} correspond to quantum state with maximally mixed marginals.
The general quantum marginal problem amounts to characterizing the set of possible \emph{gradients} for generic~$v$.
See~\cite{burgisserEfficientAlgorithmsTensor2018focs} for more detail.}
\begin{align}\label{eq:tensor scaling}
  \inf_{g_1,g_2,g_3\in\SL(n)} \norm*{(g_1 \ot g_2 \ot g_3) v}_2^2.
\end{align}
We can reduce the optimization to~$M = \SPD(n) \times \SPD(n) \times \SPD(n)$, where~$\SPD(n)$ denotes the complex positive-definite matrices of unit determinant.
Indeed, since~$P_j := g_j^* g_j$ is an arbitrary matrix in~$\SPD(n)$, we see that \cref{eq:tensor scaling} is equivalent to:
\begin{align}\label{eq:tensor scaling after change of variables}
  \inf_{P_1,P_2,P_3 \in \SPD(n)} \angles{v | P_1 \ot P_2 \ot P_3 | v}
\end{align}
Unfortunately, the domain is non-convex as a subset of the Euclidean space of triples of Hermitian matrices, and in any case the objective is not a convex function of the variables.

However, a key observation is that the objective becomes convex when $\SPD(n)$ and hence~$M$ is given a natural non-Euclidean geometry, namely the so-called \emph{affine-invariant} metric, which also appears as the \emph{Fisher-Rao metric} for Gaussian covariance matrices in statistics (see \cref{subsubsec:examples of sc functions} for a precise definition).
Then the straight lines of Euclidean space get replaced by the geodesics of the new metric, which take the form~$P_j(t) = \sqrt{P_j} e^{H_j t} \sqrt{P_j}$ for traceless Hermitian matrices~$H_j$ and clearly remain in~$\SPD(n)$.
It is easy to verify that the objective in \cref{eq:tensor scaling after change of variables} is \emph{convex along such geodesics} (in fact, log-convex).
% Geodesic convexity of~$\phi_v$ is equivalent to: for all~$P_1, P_2, P_3, Q_1, Q_2, Q_3 \in \SPD(n)$, one has
% \begin{equation*}
%     \phi_v(P_1 \# Q_1, P_2 \# Q_2, P_3 \# Q_3) \leq \frac12 \parens*{\phi_v(P_1, P_2, P_3) + \phi_v(Q_1, Q_2, Q_3)},
% \end{equation*}
% where~$P \# Q = P^{1/2} (P^{-1/2} Q P^{-1/2})^{1/2} P^{1/2}$ is the \emph{operator geometric mean}.
% This is exactly the ``geodesic midpoint'' between any~$P, Q \in \PD(n)$, so the inequality states that~$\phi_v$ is ``geodesically midpoint convex''.
The same phenomenon occurs for any scaling or non-commutative optimization problem; while non-convex in the Euclidean sense, these problems become convex when formulated appropriately~\cite{burgisserTheoryNoncommutativeOptimization2021}.
In most applications, the domain is given by the positive-definite matrices~$\PD(n)$, by~$\SPD(n)$, or by products of these spaces.

A key property of these domains is that they have \emph{non-positive curvature}, in contrast with Euclidean space, which has zero curvature.
This gives rise to significant geometric challenges for optimization algorithms.
For example, Rusciano~\cite{ruscianoRiemannianCorollaryHelly2019} gave a (non-constructive) cutting-plane method in non-positive curvature, with a logarithmic dependence on the volume of the domain.
Unfortunately, the volume of balls in manifolds of non-positive curvature grows \emph{exponentially} with the radius (in constant dimension).
In a black-box setting, where one can make queries to a function- and gradient oracle, the same geometric fact implies that any algorithm that wants to find an approximate minimizer must make a number of queries that is \emph{linear} in the distance to the approximate minimizer~\cite{hamiltonNogoTheoremAcceleration2021,criscitielloNegativeCurvatureObstructs2022}.
This again suggests that efficient algorithms for geodesic convex optimization in non-positive curvature in general, and for non-commutative optimization problems in particular, must make use of additional structure beyond diameter bounds, as the distance to an approximate minimizer is in general exponential in the input size~\cite{franksBarriersRecentMethods2021a}.
The best current algorithms for non-commutative optimization~\cite{burgisserTheoryNoncommutativeOptimization2021} also only have a linear dependence on these diameter bounds.
The reason is that they are \emph{box-constrained Newton methods}, i.e., a Newton-type method where the steps are constrained to a subdomain of essentially fixed size.
To traverse the in general exponentially large distance to the approximate minimizer, such algorithms must perform exponentially many iterations.

\begin{framed}
\noindent
To overcome these challenges and obstructions, it is natural to resort to methods which are capable of better exploiting the structure of the optimization problem at hand.
Interior-point methods offer a powerful such framework in the Euclidean case, and they have already proved successful for commutative scaling problems~\cite{cohenMatrixScalingBalancing2017a,burgisserInteriorpointMethodsUnconstrained2020}.
With this work, we hope to contribute a first clear step towards generalizing this powerful framework to the manifold setting.
\end{framed}

Indeed, we believe that our results suggest several interesting directions for follow-up research.
For instance, does every convex domain admit a self-concordant barrier, as is the case in the Euclidean setting?
Do there exist self-concordant barriers with better barrier parameters which can be used for these applications, leading to better algorithms?
Alternatively, can it be shown that our constructions are essentially optimal?
Can interior-point methods on manifold always be initialized efficiently, and is there a suitable notion of duality?%
\footnote{The lack of nontrivial linear functions in the presence of curvature poses significant challenges.}
We discuss these questions in more detail in~\cref{sec:outlook}.

\subsection{Self-concordance and Newton's method on manifolds}
\label{subsubsec:self-concordance and newton on manifolds}
In the remainder of this introduction, we give a more detailed overview of our results, starting with our proposed notion of self-concordance.
Throughout, $f \colon D \to \RR$ is a smooth function defined on a convex subset~$D \subseteq M$ of a connected, geodesically complete Riemannian manifold~$M$.
Then~$f$ is called \emph{convex} if it is convex along geodesics.
Let~$\nabla$ denote the \emph{covariant derivative} (or Levi--Civita connection), which allows taking derivatives of vector and tensor fields, and in particular to define Hessians~$\nabla^2f$ and higher derivatives (we review the required Riemannian geometry in \cref{sec:preliminaries}).
Then our proposed generalization of self-concordance to possibly curved manifolds is as follows.

\begin{defn}[Self-concordance]
For~$\alpha > 0$, a convex function $f$ is called~\emph{$\alpha$-self-concordant} if, for all~$p \in D$ and for all tangent vectors~$u,v,w \in T_p M$, we have
\begin{equation}\label{eq:sc defn intro}
  \abs{(\nabla^3 f)_p(u,v,w)} \leq \frac{2}{\sqrt{\alpha}} \sqrt{(\nabla^2 f)_p(u, u)} \sqrt{(\nabla^2 f)_p(v, v)} \sqrt{(\nabla^2 f)_p(w, w)}.
\end{equation}
If~$f$ is closed convex, meaning its epigraph is closed, then $f$ is called \emph{strongly~$\alpha$-self-concordant}.
\end{defn}

Self-concordance can be interpreted as giving a bound on the norm of the third derivative~$(\nabla^3 f)_p$, that is, on the change of the Hessian~$(\nabla^2 f)_p$, with respect to the (possibly degenerate) inner product defined by the Hessian itself.
% The third covariant derivative~$\nabla^3 f$ is a~$(0,3)$-tensor field on~$M$.
% If~$M = \RR^n$ endowed with the Euclidean metric, then~$\nabla^3 f$ agrees with the usual notion of the third derivative.
We say that $f$ is~\emph{$\alpha$-self-concordant along geodesics} if one requires the above bound only for~$u = v = w$, that is, if for all~$p \in D$ and for all~$u \in T_p M$, we have
\begin{equation}\label{eq:scag defn intro}
  \abs{(\nabla^3 f)_p(u,u,u)} \leq \frac{2}{\sqrt{\alpha}} ((\nabla^2 f)_p(u, u))^{3/2}.
\end{equation}
When~$M = \RR^n$, the third derivative is a symmetric tensor and hence the two notions coincide.
However, in general, the third derivative is \emph{not} symmetric in all its arguments, and indeed its asymmetry is precisely related to the manifold's \emph{curvature} via the Ricci identity~\cite[Thm.~7.14]{lee-riemannian-manifolds}, as we discuss in \cref{sec:sc}.
Prior work only considered self-concordance along geodesics~\cite{ji2007optimization} (which suffices for a damped Newton method) and did not take the asymmetry into account~\cite{udristeOptimizationMethodsRiemannian1997,jiang2007self}.

Here we show explicitly that self-concordance is in general strictly stronger than self-concordance along geodesics (cf.~\cref{subsubsec:examples of sc functions}), and it is the stronger notion that allows for the desired quadratic convergence of Newton's method -- a cornerstone of the interior point theory.
Assume for simplicity that the Hessian~$(\nabla^2 f)_p$ is positive definite for all~$p \in D$.
Then the \emph{Newton iterate} of~$f$ at~$p \in D$ is defined by minimizing the local quadratic approximation:
\begin{equation*}
  p_{f,+} := \Exp_p(u^*), \quad u^* = \argmin_{u \in T_p M} \parens*{f(p) + df_p(u) + \frac12 (\nabla^2 f)_p(u, u)}.
\end{equation*}
The progress is quantified in terms of the \emph{Newton decrement}, which is directly related to the gap between the original function value and the minimum of the local quadratic approximation.
It is defined for any $\alpha > 0$ and~$p \in D$ as
\begin{equation}
  \label{eq:newton decrement defn intro}
  \lambda_{f,\alpha}(p) = \sup_{0 \neq u \in T_p M} \frac{\abs{df_p(u)}}{\sqrt{\alpha (\nabla^2 f)_p(u, u)}}.
\end{equation}
Then we prove following result on general Riemannian manifolds in \cref{thm:newton decrement after newton step}:
\begin{thm}[Quadratic convergence]\label{thm:newton decrement after newton step intro}
  \newtonquadraticconvergencecontent
\end{thm}
To relate the Newton decrements at~$p$ and~$p_{f,+}$, we control the change in the Hessian of~$f$ along the geodesic from~$p$ to~$p_{f,+}$.
This crucially uses the notion of self-concordance of \cref{eq:sc defn intro}, rather than the weaker definition along geodesics as in \cref{eq:scag defn intro}.
This is because there are two directions involved: the one of the geodesic, and the one corresponding to the subsequent Newton decrement.

\subsection{Barriers and a path-following method on manifolds}
\label{subsubsec:barriers and path following on manifolds}
Interior-point methods provide a natural and modular approach for minimizing an objective~$f$ constrained to a bounded convex domain~$D \subseteq M$.
The key idea is to rather minimize, for $t>0$,
\begin{equation*}%\label{eq:F_t intro}
    F_t \colon D \to \R, \quad F_t := t f + F,
\end{equation*}
where~$F$ is a self-concordant ``barrier'' that is finite on~$D$ and diverges to~$\infty$ on its boundary.%
\footnote{In the Euclidean setting, the barrier $F(x) = -\log x$ models the constraint that~$x>0$, and $F(X) = -\log\det X$ defines the constraint that~$X$ is a positive-definite matrix~\cite{nesterov-nemirovskii-ipm,renegar-ipm}. Constraints are combined simply by adding the respective barriers. In the manifold setting, barriers are much harder to come by, but we give general constructions and concrete examples in \cref{sec:barriers compatibility path-following,sec:distsq,sec:applications}.}
This automatically ensures the constraint, as~$F_t$ is finite only on~$D$, and for large~$t$ the objective dominates.
One then starts with an approximate minimizer of~$F$ and~$t\approx0$, and follows the~\emph{central path}~$z(t) := \argmin_{p \in D} F_t(p)$ by iteratively performing two steps:
increase~$t$ to some~$t'$ such that the current point is still not too far from~$z(t')$
, and then take a Newton step for~$F_{t'}$ to move closer to it.
For large enough~$t > 0$, we arrive at an approximate minimizer of~$f$ on~$D \subseteq M$.

More precisely, the function~$F \colon D \to \RR$ is required to be a \emph{(non-degenerate strongly self-concordant) barrier} for~$D$, with barrier parameter~$\theta \geq 0$, which means that~$F$ is strongly $1$-self-concordant, has positive definite Hessian, and~$\lambda_F(p)^2 \leq \theta$ for all~$p \in D$.
The barrier parameter~$\theta$ controls how rapidly~$t$ can be increased in every iteration.

In order to guarantee that Newton's method indeed moves closer to the central path, we are interested in conditions on~$f$ that ensure that the functions~$F_t$ are self-concordant for every~$t>0$, with a constant independent of~$t$.
One way to guarantee this is to assume that the objective~$f\colon D \to \RR$ is \emph{compatible} with the barrier~$F$ in the following sense:
there are constants~$\beta_1,\beta_2\geq0$ such that, for all~$p \in D$ and~$u, v \in T_p M$,
\begin{align*}
    \abs{(\nabla^3 f)_p(u, v, v)}
&\leq 2 \beta_1 \sqrt{(\nabla^2 F)_p(u, u)} (\nabla^2 f)_p(v, v) \\
& \, + 2 \beta_2 \sqrt{(\nabla^2 F)_p(v, v)} \sqrt{(\nabla^2 f)_p(u, u)} \sqrt{(\nabla^2 f)_p(v, v)}.
\end{align*}
In particular, linear and quadratic functions are compatible with arbitrary self-concordant barriers, but these are not the only examples, and we crucially use this level of generality to give algorithms for the general scaling or non-commutative optimization problem.
We expand on compatibility in~\cref{subsec:compatibility}, and show that it is also useful for constructing new self-concordant barriers, for instance for the epigraph of a function compatible with a self-concordant barrier (\cref{thm:compatible function epigraph barrier construction}).%
\footnote{While optimizing a function~$f$ on a domain~$D$ can always be reduced to optimizing a linear function over its epigraph~$\{(p, t) \in D \times \RR : f(p) < t \}$, this requires a barrier for the epigraph. We construct such a barrier precisely when~$f$ is compatible with~$F$. However, it may be more difficult to initialize the path-following method on the epigraph rather than directly on~$D$, so it can be advantageous to optimize~$f$ directly. See \cref{subsec:kempf-ness}.}
Our notion of compatibility is inspired by a similar notion in the Euclidean setting, as is our analysis of the path-following method~\cite{nesterov-nemirovskii-ipm}.
Its precise guarantees match those from the Euclidean setting, and are given in the following theorem, which we prove in~\cref{thm:main stage}:

\begin{thm}[Path-following method]
  Let~$D \subseteq M$ be an open, bounded, and convex domain, and let $f, F\colon D \to \RR$ be smooth convex functions, such that~$F$ is a self-concordant barrier with barrier parameter~$\theta \geq 0$ and $f$ has a closed convex extension.
  Let~$\alpha > 0$ be such that~$F_t := tf + F$ is~$\alpha$-self-concordant for all~$t \geq 0$.
  Let~$p \in D$ be such that~$\lambda_{F}(p) \leq \frac{\sqrt{\alpha}}{8}$, and let~$\eps > 0$.
  Then, using
  \begin{equation*}
    \bigO\parens*{\parens*{1 + \sqrt{\frac{\theta}{\alpha}}} \log\parens*{\frac{(\theta + \alpha) \norm{df_p}_{F,p}^*}{\eps \sqrt{\alpha}}}}
  \end{equation*}
  Newton iterations, one can find a point~$p_\eps \in D$ such that
  \begin{equation*}
    f(p_\eps) - \inf_{q \in D} f(q) \leq \eps.
  \end{equation*}
\end{thm}

The quantity~$\norm{df_p}_{F,p}^*$ is a lower bound on the variation~$\sup_{q \in D} f(q) - \inf_{q \in D} f(q)$ of~$f$ over~$D$ (\cref{lem:local norm of derivative bound}), and hence imposes a natural notion of scale in the complexity bound.
% It is defined as the dual norm of the differential with respect to the inner product induced by the Hessian~$(\nabla^2 F)_p$.

\subsection{Examples of self-concordance: Squared distance in non-positive curvature}
\label{subsubsec:examples of sc functions}

Self-concordance on manifolds is much more difficult to verify than for Euclidean space, and this begs the question whether nontrivial examples even exist.
A natural candidate is~$f(p) = d(p,p_0)^2$, the \emph{squared distance} function to some point~$p_0 \in M$.
On Euclidean space, $f$ is trivially self-concordant, as its third derivative vanishes~identically.
In the presence of curvature the third derivative can be nonzero.
Nevertheless, we prove that the squared distance is self-concordant on~$\PD(n)$ and, as a corollary, also on a broad class of manifolds with non-positive curvature.

We now discuss this in more detail.
As in the introduction, we denote by~$\PD(n) = \PD(n,\C)$ the complex positive-definite matrices, endowed with the well-known \emph{affine-invariant} Riemannian metric, which is given as follows.
Since~$\PD(n)$ is an open subset of~$\Herm(n)$, the Hermitian~$n \times n$-matrices, we can identify the tangent space~$T_P \PD(n)$ at every~$P \in \PD(n)$ with~$\Herm(n)$.
Then the Riemannian metric is defined as follows: for any two tangent vectors~$U, V \in T_P \PD(n)$, their inner product is
\begin{equation*}
  \braket{U, V}_P = \Tr \left[ P^{-1} U P^{-1} V \right].
\end{equation*}
With this metric,~$\PD(n)$ is a Hadamard manifold, i.e., a simply connected geodesically complete Riemannian manifold with non-positive curvature.
% In this case the sectional curvatures take values in~$[-\frac12, 0]$.
Its geodesics, parallel transport, covariant derivatives, and so forth all have well-known closed-form expressions, which are amenable to tools from matrix analysis.
For example, the geodesics through~$P \in \PD(n)$ are of the form~$t \mapsto \sqrt P e^{tH} \sqrt P$ for~$H \in \Herm(n)$, and geodesic midpoints are the same as operator geometric means.
The distance between two matrices~$P, Q\in\PD(n)$, defined as the minimum length of any path connecting them, is
\[ d(P, Q) = \normHS{\log(P^{-1/2} Q P^{-1/2})}, \] % = \normHS{\log(Q^{-1/2} P Q^{-1/2})}, \]
where~$\normHS{\cdot}$ denotes the Hilbert--Schmidt (i.e., Frobenius) norm.
In \cref{thm:distsq pdn self-concordant expanded} we show:

\begin{thm}[Self-concordance of squared distance]\label{thm:distsq pdn self-concordant}
  For any~$P_0 \in \PD(n)$, the squared distance~$f\colon \PD(n) \to \RR$ to~$P_0$, defined by~$f(P) = d(P, P_0)^2$, is~$2$-self-concordant.
\end{thm}

% It is important to note that while the squared distance manifestly has vanishing third derivative on Euclidean space, this is \emph{not} the case on~$\PD(n)$.
% In fact, the curvature tensor precisely characterizes the asymmetry in the third derivative.
We conjecture that the squared distance is actually~$8$-self-concordant, see \cref{rem:conjectured inequality improvement}.
Self-concordance on~$\PD(n, \CC)$ implies the same result for the squared distance on any convex subset of it.
Therefore, the self-concordance holds on any Hadamard manifold that is also a so-called symmetric space;%
\footnote{Any such space is the product of a symmetric space of non-compact type and a Euclidean space~\cite[Prop.~V.4.2]{helgason1979differential}, and embeds, possibly after rescaling the metric on each of its de Rham factors, as a complete convex submanifold of~$\PD(n, \RR)$ for some~$n \geq 1$, and hence also in~$\PD(n, \CC)$~\cite[Thm.~2.6.5]{eberlein1997geometry}. See~\cite{helgason1979differential} for more background.}
we will call this a \emph{Hadamard symmetric space}.
In particular, using \cite[Prop.~10.58]{bridson-haefliger-nonpositive-curvature} we obtain the following result, which covers most non-positively curved spaces of import in applications, including the general scaling or non-commutative optimization problem (cf.~\cref{subsec:kempf-ness}):

\begin{cor}
  \label{cor:distsq sym space self-concordant}
  Let~$G \subseteq \GL(n, \RR)$ be an algebraic subgroup%
\footnote{This means that $G$ is a subset of~$\GL(n, \RR)$ determined by polynomial equations in the matrix entries.}
  such that $g^T \in G$ for every $g \in G$.
  Set~$M := \{ g^T g : g \in G \} \subseteq \PD(n, \RR)$.
  Then~$M \subseteq \PD(n, \RR)$ is a convex subset, and for every~$p_0 \in M$, the function $f \colon M \to \RR, f(p) = d(p, p_0)^2$ is~$2$-self-concordant.
\end{cor}

% We note that~$M$ admits an alternative description as $M = G \cap \PD(n, \RR)$, the positive-definite matrices in~$G$.
% If~$K = G \cap \mathrm{O}(n,\R)$, then~$M$ is also isometric to $K \backslash G$ (after rescaling the metric).

Hyperbolic space~$\HH^n$ is a paradigmatic example of a manifold with non-positive curvature in this class.
\Cref{cor:distsq sym space self-concordant} implies that the squared distance function to a point in~$\HH^n$ is $1$-self-concordant, as one has to rescale the curvature by a factor~$2$ to obtain an isometric embedding into~$\PD(n, \CC)$.
Similarly, the conjectured~$8$-self-concordance on~$\PD(n, \CC)$ would imply~$4$-self-concordance on~$\HH^n$.

We are able to prove the stronger result that the squared distance on~$\HH^n$ is in fact $8$-self-concordant, and that this is optimal, see \cref{thm:SC_hyperbolic}.
In contrast, the squared distance on hyperbolic space is $\frac{27}2$-self-concordant along geodesics, as was shown previously in~\cite[Lem.~11]{ji2007optimization}.%
\footnote{They prove that $M_f = \sqrt{16/27}$, where the constant~$M_f$ is related to the constant~$\alpha$ in our definition of self-concordance along geodesics by~$M_f = 2/{\sqrt \alpha}$.}
% $\alpha = \frac 4 {M_f^2} = \frac{27}{4}$
% These results easily extend to the model spaces~$M_{-\kappa}^n$ with~$\kappa > 0$, i.e., to the unique simply connected geodesically complete Riemannian manifold with constant sectional curvature~$-\kappa$.
It is an interesting open question whether there exists a universal constant~$C > 0$ such that if~$M$ is a Hadamard manifold with all sectional curvatures in~$[-\kappa, 0]$, then for every~$p_0 \in M$, $f(p) = d(p, p_0)^2$ is~$C/\kappa$-self-concordant.

Using the self-concordance of the squared distance, it is easy to construct a self-concordant barrier for its epigraph (cf.~\cref{thm:compatible function epigraph barrier construction}).
To this end we provide the following result, which applies in particular to~$\PD(n)$, hyperbolic space, and all other Hadamard symmetric spaces.

\begin{thm}[Epigraph barrier]\label{prop:hadamard distsq epigraph barrier}
  Let~$M$ be a Hadamard manifold, and let~$p_0 \in M$.
  Assume that the function~$f\colon M \to \R$, $f(p) = d(p, p_0)^2$ is $\alpha$-self-concordant.
  Let~$D = \{ (p, S) \in M \times \RR : f(p) < S \}$. % be the open epigraph of~$f$.
  Then, the function~$F\colon D\to\R$ defined by
  \begin{equation}\label{eq:epi barr intro}
    F(p, S) = -\log \parens*{ S - d(p, p_0)^2 } + \frac{1}{\alpha} d(p, p_0)^2
  \end{equation}
  is strongly $1$-self-concordant, and $\lambda_{F}(p, S)^2 \leq 1 + \frac{2}{\alpha} \, d(p, p_0)^2$.
  % \lambda_{f/C} = \lambda_{f,C} = lambda_{f} / sqrt C
\end{thm}

The reason that the proposition does \emph{not} state that~$F$ is a barrier is that the Newton decrement~$\lambda_{F}(p, S)$ is not bounded by a constant, but rather depends on the distance to the point~$p_0$.
To obtain a barrier, on needs to impose an additional constraint on the domain to force it to be bounded, for instance by requiring that~$S < S_0$, which can be implemented by adding a logarithmic barrier term~$-\log(S_0 - S)$ to~$F$.
The dependence of the Newton decrement on the distance to~$p_0$ is caused by the term~$\frac{1}{\alpha} d(p, p_0)^2$ in \cref{eq:epi barr intro}, but without this term the function would not be self-concordant.
See also~\cref{thm:F_lvl}, where we construct a barrier for the sublevel set of a self-concordant function, with barrier parameter depending on the gap in function value.

We also provide a strengthening of the above theorem for hyperbolic space (see \cref{thm:hypn dist epigraph barrier}):

\begin{thm}\label{thm:hypn dist epigraph barrier intro}%
  Let~$M = \HH^n$, $p_0 \in M$, and define~$f\colon M \to \RR$ by~$f(p) = d(p, p_0)^2$.
  Let $D = \{ (p, R, S) \in M \times \RR_{>0} \times \RR_{>0} \;:\; RS - f(p) > 0 \}$.
  Then the function $F\colon D \to \RR$ by
  \begin{equation*}
    F(p, R, S) = - \log(RS - f(p)) + f(p)
  \end{equation*}
  is strongly $\frac12$-self-concordant.
  Furthermore, $\lambda_{F,\frac12}(p, R, S)^2 \leq 4 + 4 f(p)$.
\end{thm}

The significance of this result is that it can be used to construct a barrier for the epigraph of the \emph{distance} to a point, rather than the squared distance, by restricting to the subspace defined by the equation $S = R$.
This is essential for applying the framework to the geometric median problem, see~\cref{subsubsec:application geometric problems}.
In the Euclidean setting, the additional~$f$-term is unnecessary; see for instance the proof of~\cite[Prop.~5.4.3]{nesterov-nemirovskii-ipm}.
In our setting the proof is more complicated, as it involves a strengthening of the self-concordance estimate on the third derivative of the squared distance.
The key estimates which enable our proof of the above theorem are given in \cref{thm:SC_hyperbolic}.

\subsection{Application I: Non-commutative optimization and scaling problems}
\label{subsubsec:application kempf ness}
Our first application is the one which motivated us to extend the framework in the first place.
To state our result in full generality requires a bit of setup~\cite{burgisserTheoryNoncommutativeOptimization2021}.
Let~$G \subseteq \GL(n, \CC)$ be a connected algebraic subgroup such that~$g^* \in G$ for all~$g \in G$.
Let~$\pi\colon G \to \GL(V)$ be a rational representation on a finite-dimensional complex vector space~$V$.
Assume~$V$ is endowed with an inner product such that the unitary matrices in~$G$ act unitarily.
The general \emph{norm minimization} problem asks to minimize the norm over the orbit of a given vector~$v\in V$.
That is, we wish to minimize~$\norm{\pi(g) v}$ over~$g\in G$.
As described earlier, this problem subsumes the class of \emph{non-commutative optimization} or \emph{scaling problems} that have been of much interest in the recent literature.
Note that~$\norm{\pi(g) v}^2 = \braket{v|\pi(g^* g)|v}$.
Accordingly, it suffices to minimize the so-called \emph{log-norm} or \emph{Kempf--Ness function} defined by
\begin{equation*}
  \phi_v\colon M \to \RR, \quad
  \phi_v(p) = \log \braket{v|\pi(p)|v}
\end{equation*}
over~$M = \{ g^* g : g \in G \} = G \cap \PD(n)$.
This function is convex along the geodesics of~$M$.
It is also $N(\pi)^2$-smooth in the convexity sense, where~$N(\pi)$ is the so-called \emph{weight norm} of the action, which is easy to compute and at most polynomially large (see \cref{subsec:kempf-ness} and \cite{burgisserTheoryNoncommutativeOptimization2021} for details).
Therefore, if~$\phi_v$ is bounded from below, a simple gradient descent algorithm can be used to find a point~$p \in M$ such that~$\norm{\grad(\phi_v)_p} \leq \delta$ within~$\bigO(N(\pi)^2 [ \phi_v(I) - \inf_{q \in M} \phi_v(q) ]/ \delta^2)$ iterations~\cite[Thm.~4.2]{burgisserTheoryNoncommutativeOptimization2021}.
A more sophisticated box-constrained Newton method is able to find an~$\eps$-approximate minimizer~$p_\eps$ of~$\phi_v$ within~$\bigO((1 + R_0) N(\pi) \log[(\phi_v(I) - \inf_{q \in M} \phi_v(q))/\eps])$ iterations, where~$R_0 > 0$ is an upper bound on the distance to such a minimizer~\cite[Thms.~5.1~\&~5.7]{burgisserTheoryNoncommutativeOptimization2021}.
Using our interior-point path-following method we prove the following result in \cref{thm:kempf ness complexity}:

\begin{thm}[Non-commutative optimization]
  Let~$0\neq v\in V$ and $R_0$, $\eps > 0$.
  Let~$M = \{ g^* g : g \in G \} \subseteq \PD(n)$ and~$D = \{ p \in M : d(p, p_0) \leq R_0 \}$, and define~$\phi_v\colon M \to \RR$ by $\phi_v(p) = \log \braket{v|\pi(p)|v}$.
  Then there is an algorithm that within~$\bigO\parens*{(1 + R_0) N(\pi) \log(N(\pi) R_0 / \eps)}$ iterations of the path-following method finds~$p_\eps \in D$ such that
  \begin{equation*}
    \phi_v(p_\eps) - \inf_{p \in D} \phi_v(p) \leq \eps.
  \end{equation*}
\end{thm}

This essentially matches the complexity of the box-constrained Newton method mentioned above, which is currently the state-of-the-art.
There is a small difference, in that our complexity has~$N(\pi) R_0$ in the logarithm, rather than the potential gap~$\phi_v(I) - \inf_{q \in M} \phi_v(q)$; these are related since~$\phi_v$ is~$N(\pi)$-Lipschitz.
The approach we take to obtain this result is to use the barrier on~$M$ which arises from \cref{cor:distsq sym space self-concordant,prop:hadamard distsq epigraph barrier}, and to show that the Kempf--Ness function is compatible with the squared distance function, which is enough to implement the path-following method, as explained earlier.
It would be very interesting to find a suitable barrier for this problem with a smaller barrier parameter (or prove that no such barrier exists).

\subsection{Application II: Minimum-enclosing ball problem on~\texorpdfstring{$\PD(n)$}{PD(n)}}
\label{subsec:MEB intro}
Next we consider the \emph{minimum enclosing ball (MEB)} problem: given distinct points~$p_1, \dotsc, p_m \in M$, find~$p$ such that~$R(p) := \max_i d(p, p_i)$ is minimal.
When~$M = \RR^n$ is Euclidean space, this is a well-studied problem in computational geometry.
There, it can be formulated as a second-order cone problem, to which interior-point methods are applicable (see, e.g.,~\cite{KMY03}).

When~$M$ is a Hadamard manifold, the distance to a point is convex, and hence the MEB problem is a convex optimization problem.
In particular, for hyperbolic space~$M = \HH^n$, there has been previous work on the MEB problem~\cite{arnaudonApproximatingRiemannian1center2013,NH15}.
The only algorithm with explicit complexity bounds that we are aware of is due to Nielsen and Hadjeres~\cite{NH15}.
If~$R_*$ is the minimal radius of an MEB and~$\delta > 0$, then they can find a point~$p \in \HH^n$ such that~$\max_i d(p, p_i) \leq (1 + \delta) R_*$ within~$\bigO(1 / \delta^2)$ iterations of an algorithm, each of which is simple to implement.

To find MEBs using interior-point methods, it is sufficient to have a barrier for the epigraph of the squared distance.
In particular, the barrier constructed using~\cref{thm:distsq pdn self-concordant,prop:hadamard distsq epigraph barrier} can be used to solve this problem on~$\PD(n)$, and we prove the following result in~\cref{thm:meb complexity}
\begin{thm}[Minimum enclosing ball]
  \label{thm:intro minimum enclosing ball complexity}
  Let~$p_1, \dotsc, p_m \in \PD(n)$ be~$m \geq 3$ points, and set $R_0 = \max_{i \neq j} d(p_i, p_j)$.
  Let~$R(p) = \max_i d(p, p_i)$, set~$R_* = \inf_{p \in M} R(p)$, and let~$\eps > 0$.
  Then with~$\bigO((m + 1) R_0^2)$ iterations of a damped Newton method and
  \begin{equation*}
    \bigO\parens*{\sqrt{1 + m (R_0^2 + 1)} \log \left( \frac{m (R_0^2 + 1)}{\eps} \right)}
  \end{equation*}
  iterations of the path following method, one can find~$p_\eps \in \PD(n)$ such that
  \begin{equation*}
    R(p_\eps) - R_* \leq \eps.
  \end{equation*}
\end{thm}
A similar result can be obtained on arbitrary Hadamard symmetric spaces.
We also note that the optimal radius~$R_*$ satisfies~$R_0 \leq 2 R_*$ (\cref{lem:MEB optimal value}), so that the above also yields a multiplicative error guarantee.
Compared to the results of~\cite{NH15}, we have a logarithmic dependence on the precision~$\eps$, but a linear dependence on~$R_0$ (as opposed to no dependence).
% Instead of taking the ``worst-case'' curvature rescaling factor, one could also use the self-concordance on each individual de Rham factor. The MEB problem would then become a problem involving \emph{weighted} sums of distance squareds, which may yield a slightly better complexity.

\subsection{Application III: Geometric median on hyperbolic space}
\label{subsubsec:application geometric problems}
Our last application is the \emph{geometric median problems}.
In the Euclidean setting this is also known as the \emph{Fermat--Weber problem}~\cite{cohenGeometricMedianNearly2016}.
It is formally defined as follows: given points~$p_1, \dotsc, p_m \in M$, not all contained in a single geodesic, find~$p_0 \in M$ such that
\begin{equation*}
  p_0 \in \argmin_{p \in \HH^n} \medianobj(p) := \sum_{j=1}^m d(p, p_j).
\end{equation*}
The objective function~$\medianobj$ is convex on Hadamard manifolds~$M$.
In contrast with the \emph{geometric mean (or barycenter) problem}, which is to find the minimizer of~$\sum_{j=1}^m d(p, p_j)^2$, finding the geometric median is non-trivial even on~$M = \RR^n$.
The first and one of the best-known algorithms for this problem on Euclidean space is Weiszfeld's algorithm~\cite{weiszfeldPointPourLequel1937}, which is a simple iterative procedure based on solving the first-order optimality condition~$\grad(\medianobj)_p = \sum_{j=1}^m (p - p_j) / d(p, p_j) = 0$ for~$p$, while treating the~$d(p, p_j)$ as constants.
Unfortunately, the update rule is not well-defined when~$p$ is one of the~$p_j$'s (which can be fixed, see e.g.~\cite{ostreshConvergenceClassIterative1978}), and it may converge very slowly in general.
In~\cite{xueEfficientAlgorithmMinimizing1997} it was observed that one can also apply interior-point methods, by viewing the geometric median problem as a second-order cone program.
%Recall that the second-order cone is given by
%\begin{equation*}
% \label{eq:second order cone}
%  \{ (x, t) \in \RR^n \times \RR : \norm{x}_2 \leq t \},
%\end{equation*}
%which admits the self-concordant barrier~$-\log(t^2 - \norm{x}_2^2)$, with barrier parameter~$2$.
More recent work~\cite{cohenGeometricMedianNearly2016} has shown that a specialized long-step interior-point method is capable of solving the geometric median problem on~$\RR^n$ in nearly-linear time, and we refer the reader to their paper for a broader literature review.
Weiszfeld's approach has been generalized to the Riemannian setting~\cite{fletcherGeometricMedianRiemannian2009}.
A sub-gradient approach~\cite{yangRiemannianMedianIts2010} can find a point with squared distance to the minimizer of~$s$ at most~$\eps$ in~$\bigO(1/\eps)$ iterations; however, in the negatively curved setting, it suffers from an exponential dependence on the quantity~$R_0 = \max_{i \neq j} d(p_i, p_j)$.

We can solve the geometric median problem on hyperbolic space~$\HH^n$ by using our interior-point framework and our barrier for the epigraph of the \emph{distance} constructed using \cref{thm:hypn dist epigraph barrier intro}, which serve as analogs of the second-order cone and the associated barrier.
In \cref{thm:geometric median complexity} we prove:

\begin{thm}[Geometric median]
  \label{thm:geometric median complexity intro}
  Let~$p_1, \dotsc, p_m \in \HH^n$ be~$m \geq 3$ points, not all on one geodesic, and set $R_0 = \max_{i \neq j} d(p_i, p_j)$.
  Define~$\medianobj\colon \HH^n \to \RR$ by~$\medianobj(p) = \sum_{j=1}^m d(p, p_j)$, and let~$\eps > 0$.
  Then with~$\bigO((m + 1) R_0^2)$ iterations of a damped Newton method and
  \begin{equation*}
    \bigO\parens*{\sqrt{m (R_0^2 + 1)} \log \left( \frac{m R_0 (R_0^2 + 1)}{\eps} \right)}
  \end{equation*}
  iterations of the path following method, one can find~$p_\eps \in \HH^n$ such that
  \begin{equation*}
    \medianobj(p_\eps) - \inf_{q \in \HH^n} \medianobj(q) \leq \eps.
  \end{equation*}
\end{thm}

For not too small~$\eps$, the cost is dominated by the damped Newton method, which we use to find a good starting point for the path-following method.
We leave it as an open problem as to whether this can be avoided.
Furthermore, the above applies only to~$\HH^n$ rather than to~$\PD(n)$: it relies on the barrier constructed using~\cref{thm:hypn dist epigraph barrier intro}, which uses a non-trivial strengthening of the self-concordance estimates for the squared distance.
We expect that such a strengthening can also be obtained more generally, and this would immediately generalize the algorithmic result from~\cref{thm:geometric median complexity intro} to these spaces; we also leave this as a problem for future work.

%-----------------------------------------------------------------------------
\subsection{Organization of the paper}
%-----------------------------------------------------------------------------
In \cref{sec:preliminaries}, we review standard concepts from Riemannian geometry and convexity that we use later.
In \cref{sec:sc} we define self-concordance and analyze Newton's method, showing its quadratic convergence for self-concordant functions.
In \cref{sec:barriers compatibility path-following}, we define self-concordant barriers and the notion of compatibity, discuss how to construct new self-concordant functions out of old ones, and analyze a path-following method.
In \cref{sec:distsq} we describe general properties of the distance function on Hadamard manifolds, show that the squared distance is self-concordant on~$\PD(n)$, and give refinements of the self-concordance estimate on the model spaces for constant negative sectional curvature, which are used to construct a barrier for the epigraph of the distance function.
In \cref{sec:applications}, we discuss our applications: the first is on norm minimization and noncommutative optimization, the second is on computing the minimum enclosing ball on Hadamard symmetric spaces, and the third is on computing geometric medians on model spaces. We also briefly discuss an application to the Riemannian barycenter problem.
We conclude in \cref{sec:outlook}, where we mention interesting open problems and future research directions.

%=============================================================================
\section{Preliminaries in Riemannian geometry}\label{sec:preliminaries}
%=============================================================================
In this section, we recall and fix our notation for some basic concepts in Riemannian geometry that we will need in the remainder.
We follow the conventions of~\cite{lee-riemannian-manifolds}.
See~\cite{lee-riemannian-manifolds,bridson-haefliger-nonpositive-curvature} for comprehensive introductions to Riemannian geometry and non-positive curvature, respectively.

%-----------------------------------------------------------------------------
\subsection{Metric, lengths, distances}\label{subsec:metric and distance}
%----------------------------------------------------------------------------
Throughout this paper, we let~$M$ denote a connected Riemannian manifold.
Unless specified otherwise, all differential geometric objects (manifolds, functions, sections, etc.) are assumed to be $C^\infty$-smooth.
We write~$T_p M$ and~$T_p^* M$ for the tangent and cotangent space at a point~$p \in M$, and write~$TM$ and~$T^*M$ the tangent and cotangent bundle of~$M$, respectively.
The space of sections of a vector bundle~$E$ on~$M$ is denoted by~$\sections{E}$.
Sections of the (co)tangent bundle are called (co)vector fields.
Given a function~$f$, we write $df$ for its differential, which is a covector field.
Then~$Xf = df(X)$ is the directional derivative of~$f$ in direction~$X$ for any vector field~$X$.
The Lie bracket of two vector fields~$X$ and~$Y$ is the vector field~$[X,Y]$ that acts as~$[X,Y]f = X(Yf) - Y(Xf)$ on any function~$f$.
More generally, for $k, l \geq 0$, a \emph{$(k,l)$-tensor field} is by definition a section of the bundle~$T^{(k,l)}M := (TM)^{\otimes k} \otimes (T^*M)^{\otimes l}$ or, equivalently, a $C^\infty(M)$-multilinear map $\Gamma(T^* M)^k \times \Gamma(T M)^l \to C^\infty(M)$; when $k=1$ we can also think of it as a $C^\infty(M)$-multilinear map $\Gamma(T M)^l \to \Gamma(T M)$.

The Riemannian metric on~$M$ is a smoothly varying family of inner products on the tangent spaces, i.e., for every~$p \in M$ we have an inner product $\braket{\cdot, \cdot}_p$ on~$T_pM$ such that the map~$p \mapsto \braket{\cdot, \cdot}_p$ is a section of the bundle $T^{(0,2)} M$.
The induced norm on $T_pM$ is denoted by~$\norm{\cdot}_p$. % = \sqrt{\braket{\cdot,\cdot}}_p$.
We write~$\braket{X,Y}$ and $\norm X$ for the functions computing the pointwise inner product and norm, respectively, of vector fields~$X$,~$Y$.

Using the Riemannian metric, we can define the \emph{length} of a piecewise regular (meaning smooth and non-zero derivative) curve by $L(\gamma) = \int_a^b \norm{\dot\gamma(t)}_{\gamma(t)} dt$.
This is independent of the parameterization.
In particular, we may always reparameterize such that the curve has unit speed, i.e., $\norm{\dot\gamma(t)}=1$, except for finitely many points; in this case the length is $L(\gamma) = b - a$.
Given a notion of length, we define the \emph{Riemannian distance}~$d(p,q)$ between any two points $p,q\in M$ as the infimum of the lengths of all piecewise regular curves from~$p$ to~$q$.
In this way, $M$ becomes a metric space.
Its topology is the same as the original topology of the manifold~$M$.

%-----------------------------------------------------------------------------
\subsection{Covariant derivative and curvature}\label{subsec:covariant derivative and curvature}
%----------------------------------------------------------------------------
The Riemannian metric determines the \emph{Levi-Civita connection}~$\nabla$.
It assigns to any two vector fields~$X$ and~$Y$ the \emph{covariant derivative}~$\nabla_X Y$ of~$Y$ along~$X$, which is again a vector field, and is determined uniquely by being a connection on the tangent bundle (meaning it is $C^\infty$-linear in~$X$, $\RR$-linear in~$Y$, and satisfies the product rule $\covder{X} (fY) = f \covder{X} Y + (Xf) Y$ for all functions~$f$) which is \emph{compatible with the metric} in the sense that
$X\!\braket{Y, Z} = \braket{\nabla_X Y, Z} + \braket{Y, \nabla_X Z}$
and \emph{symmetric}, meaning~$\nabla_X Y - \nabla_Y X = [X,Y]$, where $[X,Y]$ denotes the Lie bracket.
The $C^\infty(M)$-linearity in~$X$ implies that $\nabla_X Y\bigr|_p$ depends only on the tangent vector~$v:=X_p$ at the point~$p\in M$ and the values of~$Y$ in an arbitrarily small neighbourhood of~$p$; accordingly we will also write~$\nabla_v Y$.
% It also means that we can think of the map $(\omega,X) \mapsto \omega(\nabla_X Y)$ as a $C^\infty(M)$-multilinear map $\Gamma(T^*M) \times \Gamma(TM) \to C^\infty(M)$, i.e., as defining a (1,1)-tensor field, called the \emph{total covariant derivative}~$\nabla Y$ of the vector field~$Y$.
Moreover,~$X \mapsto \nabla_X Y$ defines a (1,1)-tensor field, called the \emph{total covariant derivative}~$\nabla Y$ of~$Y$.

One can uniquely extend the above to define connections and covariant derivatives for all tensor bundles~$T^{(k,l)}M$ by demanding that for functions it agrees with the differential, that it satisfies a product rule with respect to tensor products, $\nabla_X (T \ot S) = (\nabla_X T) \ot S + T \ot (\nabla_X S)$ for all vector fields $X$ and tensor fields $T$, $S$, and that it commutes with all contractions.
As a consequence,
\begin{equation}\label{eq:tensor contraction deriv}
\begin{aligned}
&X(T(\omega_1,\dots,\omega_k,Z_1,\dots,Z_l)) = (\nabla_X T)(\omega_1,\dots,\omega_k,Z_1,\dots,Z_l) \\
&\qquad\qquad + T(\nabla_X \omega_1,\omega_2,\dots,\omega_k,Z_1,\dots,Z_l) + \ldots + T(\omega_1,\dots,\omega_k,Z_1,\dots,Z_{l-1},\nabla_X Z_l)
\end{aligned}
\end{equation}
for any $(k,l)$-tensor field $T$, vector fields $X$, $Z_1$, \dots, $Z_l$, and covector fields $\omega_1,\dots,\omega_k$.
Again, we write $\nabla_v T := (\nabla_X T)_p$ as this only depends on the tangent vector $v := X_p$ at the point $p\in M$.
For any $(k,l)$-tensor field~$T$, the map $(\omega_1,\dots,\omega_k,X,Z_1,\dots,Z_l) \mapsto (\nabla_X T)(\omega_1,\dots,\omega_k,Z_1,\dots,Z_l)$ defines a $(k,1+l)$-tensor field, called the \emph{total covariant derivative} and denoted by~$\nabla T$.
We note that~\cite{lee-riemannian-manifolds} uses a different convention. % for the position of~$X$.
In particular, we can define the \emph{Hessian} of a function~$f$ as $\nabla^2 f = \nabla(\nabla f)$, which is
a $(0,2)$-tensor field that turns out to be symmetric for the Levi-Civita connection; see \cref{subsec:hessian}.

Let~$\tilde M \subseteq M$ be an embedded submanifold, equipped with the induced metric, and let~$\tilde \nabla$ denote its Levi-Civita connection.
If $X, Y$ are vector fields on~$\tilde M$ that are extended arbitrarily to a neighborhood of~$\tilde M$ in~$M$, then the \emph{Gauss formula} holds on~$\tilde M$:
\begin{equation}\label{eq:gauss formula}
  \nabla_X Y = \tilde \nabla_X Y + \sff(X, Y),
\end{equation}
where $\sff(X,Y) := \pi^\perp(\nabla_X Y)$ is the \emph{shape tensor} or \emph{second fundamental form}~$\sff$ of~$\tilde M$, with $\pi^\perp\colon TM|_{\tilde M} \to (T\tilde M)^\perp$ the orthogonal projection~\cite[Thm.~8.2]{lee-riemannian-manifolds}.

While the covariant derivative itself is not a tensor field, it can be used to define the so-called \emph{Riemann curvature tensor} which is a fundamental local invariant of Riemannian manifolds.
Given vector fields $X$, $Y$, $Z$, we can define the vector field
\begin{align*}
  R(X,Y)Z := \nabla_X (\nabla_Y Z) - \nabla_Y (\nabla_X Z) - \nabla_{[X,Y]} Z.
\end{align*}
We may think of $R(X,Y)$ as a $C^\infty$-linear operator on the tangent bundle; hence~$R$ is a $(1,3)$-tensor field.
The operator $R(X,Y)$ is skew-symmetric, and it is a skew-symmetric function of~$X$ and~$Y$.
It further satisfies the algebraic Bianchi identity $R(X,Y)Z + R(Y,Z)X + R(Z,X)Y = 0$.
% The former imply that it remains invariant if one exchanges the first two with the last two arguments, i.e., $R(X,Y,Z,W) = R(Z,W,X,Y)$.
It can also be useful to define $R(X,Y,Z,W) := \braket{R(X,Y)Z,W}$, which is a $(0,4)$-tensor field.

A closely related object is the \emph{sectional curvature}, which given two linearly independent tangent vectors~$v,w\in T_pM$ at the same point~$p\in M$ is defined by
\begin{align*}
  K(v,w) = \frac {\braket{R(v,w)w,v}_p} {\braket{v,v}_p\braket{w,w}_p - \braket{v,w}_p^2}.
\end{align*}
It only depends on the two-dimensional tangent plane spanned by~$v$ and~$w$.
The sectional curvature determines the Riemann curvature tensor uniquely.
Its sign is an important characteristic of a Riemannian manifold.
We say that $M$ has \emph{non-positive (sectional) curvature} if $K(v,w) \leq 0$ for all~$v,w\in T_pM$ and~$p\in M$.
The next lemma records how these notions behave under rescaling of the Riemannian metric.
\begin{lem}
  \label{lem:rescaling curvature rescales distsq self-concordance}
  Let~$M$ be a Riemannian manifold with Riemannian metric~$\braket{\cdot, \cdot}$, and let~$c > 0$.
  Let~$M'$ be the same manifold but with Riemannian metric given by~$\braket{\cdot, \cdot}' = c \braket{\cdot, \cdot}$.
  Then~$M'$ has the same Levi--Civita connection as~$M$, and hence the same~$(1,3)$-curvature tensor.
  For every~$p, q \in M$, one has~$d_{M'}(p, q) = \sqrt{c} \, d_M(p, q)$.
  Furthermore, for all~$p \in M$ and linearly independent~$v,w \in T_p M = T_p M'$, the sectional curvature satisfies $K_{M'}(v,w) = K_{M}(v,w) / c$.
\end{lem}

%-----------------------------------------------------------------------------
\subsection{Parallel transport, geodesics, completeness}\label{subsec:geodesics and hadamard}
%-----------------------------------------------------------------------------
All definitions given so far restrict naturally to open subsets.
However, it is often useful to restrict to curves in a manifold and differentiate a vector or tensor field along it.
If~$\gamma$ is a curve defined on an interval $I\subseteq\RR$, then a \emph{$(k,l)$-tensor field along~$\gamma$} is a function $Y\colon I \to T^{(k,l)}M$ such that $Y(t) \in T_{\gamma(t)}^{(k,l)} M$ for every $t \in I$, i.e., a section of the pullback bundle~$\gamma^* T^{(k,l)}$. % of~$T^{(k,l)}M$ along~$\gamma$.
Then there is a unique $\R$-linear operator~$D_t$, called the \emph{covariant derivative along~$\gamma$}, that satisfies the product rule~$D_t (fY) = \dot f Y + f D_t Y$ for $f \in C^\infty(I)$ for $f\in C^\infty(I)$, and which agrees with $\nabla_{\dot\gamma(t)}$ for every tensor field that extends to a neighborhood of~$\gamma$.

A vector or tensor field~$Y$ along a curve~$\gamma$ is called \emph{parallel} if its covariant derivative along $\gamma$ vanishes identically, i.e., $D_t Y \equiv 0$.
For any curve~$\gamma \colon I \to M$, $0 \in I$, and any tensor~$y_0 \in T^{(k,l)}_{\gamma(0)} M$, standard results in ordinary differential equations imply that there always exists a unique parallel tensor field~$Y$ along~$\gamma$ such that $\gamma(0) = y_0$, called the \emph{parallel transport} of~$y_0$ along~$\gamma$.
For any~$t\in I$, we get a linear isomorphism~$\tau_{\gamma,t} \colon T^{(k,l)}_{\gamma(0)} M \to T^{(k,l)}_{\gamma(t)} M$ by setting $\tau_{\gamma,t}(y_0) = Y(t)$ called a \emph{parallel transport map}.
This is useful to compute covariant derivatives: if $T$ is a $(k,l)$-tensor field then for all $p \in M$, $v \in T_p M$, $\eta_1,\dots,\eta_k \in T_p^* M$, and $w_1,\dots,w_l \in T_p M$ we have
\begin{align}\label{eq:deriv via transport}
  \nabla_v T(\eta_1,\dots,\eta_k,w_1,\dots,w_l)
= \partial_{t=0} T_{\gamma(t)}(\tau_{\gamma,t}\eta_1,\dots,\tau_{\gamma,t}\eta_k,\tau_{\gamma,t}w_1,\dots,\tau_{\gamma,t}w_l),
\end{align}
where $\gamma$ is an arbitrary curve such that $\gamma(0)=p$ and $\dot\gamma(0)=v$.
We are often interested in parallel transport along the manifold's geodesics, which we introduce next.

A curve~$\gamma$ is called a \emph{geodesic} if it is parallel to its own tangent vector field, i.e., $D_t \dot\gamma \equiv 0$.
For every $p\in M$ and $v \in T_pM$, there is a unique geodesic~$\gamma\colon I\to M$ with~$\gamma(0)=p$ and~$\dot\gamma(0)=v$, defined on some maximal open interval~$I$ containing 0.
Note that $\dot\gamma(t) = \tau_{\gamma,t}(\dot\gamma(0))$ for all~$t\in I$.
If $1 \in I$, we define $\Exp_p(v) := \gamma_v(1)$.
We call~$M$ \emph{geodesically complete} if~$I=\RR$,
% for every $p\in M$ and $v \in T_pM$ there exists a (necessarily unique) geodesic $\gamma\colon\RR\to M$ such that $\gamma(0) = p$ and $\dot\gamma(0) = v$,
i.e., if geodesics with arbitrary initial data exist for arbitrary times.
Then the exponential map is defined on the whole tangent space, $\Exp_p \colon T_pM \to M$.
 % by $\Exp_p(v) = \gamma_v(1)$, where $\gamma_v$ denotes the unique geodesic with the above properties.
The Hopf--Rinow theorem states that if $M$ is connected, geodesic completeness is equivalent to completeness with respect to the Riemannian distance function, as well as to the Heine--Borel property (bounded closed subsets are compact). % distance function of \cref{subsec:metric and distance}.

Any length-minimizing curve is a geodesic when parameterized with unit speed.
In general, geodesics are only locally length-minimizing, but when~$M$ is connected and complete then any two points~$p,q\in M$ are connected by a length-minimizing geodesic, although there may be many other geodesics.
% \cite[Lem.~6.18]{lee-riemannian-manifolds}.
However, if~$M$ is not only complete but also has non-positive sectional curvature, then by the Cartan--Hadamard theorem the exponential map at each point is a covering map.
In particular, if~$M$ also is simply connected, then the exponential map is a diffeomorphism, so there is a \emph{unique} (up to reparameterization) geodesic connecting any two points~$p$ and~$q$.
We will denote % this geodesic by~$\gamma_{p\to q}$ and
the corresponding parallel transport by~$\transport p q$.
Manifolds that are simply connected, geodesically complete, and have non-positive sectional curvature are called \emph{Hadamard manifolds}.
This includes a great variety of spaces of import in applications, such as Euclidean and hyperbolic spaces, the positive definite matrices, and other symmetric spaces with non-positive curvature (see \cref{sec:distsq,sec:applications}).

%-----------------------------------------------------------------------------
\subsection{Gradient and Hessian}\label{subsec:hessian}
%-----------------------------------------------------------------------------
Given a function $f\colon D\to \RR$ defined on an open subset $D\subseteq M$, we define its \emph{gradient} as the vector field~$\grad(f)$ that is dual to its differential.
That is, for all vector fields $X$ we have
\begin{align*}
  \braket{\grad(f), X} = df(X) = Xf.
\end{align*}

The \emph{Hessian} of $f$ is defined as the second covariant derivative~$\nabla^2 f = \nabla (\nabla f) = \nabla df$, which is a $(0,2)$-tensor field, that is, a smoothly varying family of bilinear forms.
By definition and using \cref{eq:tensor contraction deriv}, we have for any two vector fields~$X$ and~$Y$ that
\begin{align}\label{eq:defn hessian}
  (\nabla^2 f)(X,Y)
% = (\nabla df)(X,Y)
= (\nabla_X df)(Y)
= X(df(Y)) - df(\nabla_X Y)
= X(Yf) - (\nabla_X Y)f,
\end{align}
which implies that Hessian is a \emph{symmetric} tensor, by the symmetry of the Levi-Civita connection.
Since the Hessian is a symmetric tensor, it is determined by the associated quadratic form.
The latter can be conveniently calculated in terms of geodesics:
for any $p \in M$ and $v \in T_pM$,
\begin{align}\label{eq:hessian via diag}
  (\nabla^2 f)_p(v,v)
= \partial_{t=0}^2 f(\Exp_p(tv)).
\end{align}
% where we note that $\Exp_p(tv)$ is well-defined for sufficiently small~$t$. To see this, define $\gamma(t) = \Exp_p(tv)$. Define $X$ and $Y$ to be any vector fields extending $\gamma(t)$ in a neighborhood of $t=0$. Then $\nabla_X Y = D_t \dot\gamma(0) = 0$, since $\gamma$ is a geodesic, and hence $(\nabla^2 f)(v,v) = X(Y(f))$.
Using metric compatibility, one can write
$(\nabla^2 f)(X,Y)
% = X(Yf) - (\nabla_X Y)f
% = \nabla_X \braket{\grad(f), Y} - \braket{\grad(f), \nabla_X Y}
= \braket{\nabla_X \grad(f), Y}$,
which shows that the $(1,1)$-tensor field $\Hess(f) := \nabla \grad(f)$ is the natural operator definition of the Hessian.

One can similarly consider higher covariant derivatives, but these need no longer be symmetric as a consequence of the non-vanishing of the curvature tensor.
In particular, the third covariant derivative is no longer captured by its diagonal $(\nabla^3 f)_p(v,v,v) = \partial_{t=0}^3 f(\Exp_p(tv))$.
This complicates the theory of self-concordance, as we will discuss in \cref{sec:sc}.

%-----------------------------------------------------------------------------
\subsection{Convexity}\label{subsec:convexity}
%-----------------------------------------------------------------------------
Finally we recall here some basic notions of convexity on Riemannian manifolds.
We first discuss convexity of subsets and then turn to convexity of functions.
We assume that $M$ is connected and geodesically complete, so that any two points are connected by a (length-minimizing) geodesic.

A subset $D \subseteq M$ is called \emph{(totally) convex} if for every geodesic $\gamma\colon[0,1]\to M$ with $\gamma(0)\in D$ and $\gamma(1)\in D$, it holds that $\gamma(t) \in D$ for all $t\in[0,1]$.
We remark that, in general, two points can be connected by more than one geodesic; accordingly there is more than one natural definition of convexity.
We are primarily interested in applications to Hadamard spaces, where any two points are connected by a unique geodesic, just like in Euclidean space.

A (not necessarily continuous) function $f\colon D \to \RR$ defined on a convex subset~$D \subseteq M$ is called \emph{convex} if for every geodesic~$\gamma\colon[0,1]\to M$ with $\gamma(0)\in D$ and $\gamma(1)\in D$, it holds that $f\circ\gamma \colon [0,1]\to\RR$ is convex.
That is, $f$ is convex along all geodesics in its domain.
Equivalently, $f$ is convex if and only if its epigraph
\begin{align}\label{eq:epi}
  E_f = \braces*{ (p,t) \in D \times \R : f(p) \leq t }
\end{align}
is a convex subset of~$M \times \R$.
% \begin{proof}
% The geodesics of $M \times D$ from $(p,t)$ to $(p',t')$ take the form $s \mapsto (\gamma(s),(1-s) t + s t')$, where~$\gamma$ of~$M$ is a geodesic with $\gamma(0) = p$, $\gamma(1) = p'$.
% Thus, $E_f \subseteq M \times D$ is convex if and only if for any geodesic $\gamma$ with $p := \gamma(0) \in D$ and $p' := \gamma(1) \in D$ and any $t,t'$ such that $f(p) \leq t$ and $f(p') \leq t'$, we have that $f(\gamma(s)) \leq (1-s)t + st'$ for all $s\in[0,1]$.
% The latter condition is strongest when $t=f(p)$ and $t'=f(p')$, in which case it reads $f(\gamma(s)) \leq (1-s)f(p) + sf(p')$ for all $s\in[0,1]$.
% Thus the preceding is equivalent to convexity of~$f$.
% \end{proof}
If the epigraph is also closed as a subset of~$M \times \R$, then~$f$ is called \emph{closed convex}.
This useful condition controls the behavior of a convex function at its boundary, as in the following lemma, which thanks to the Hopf-Rinow theorem can be proved just like in the Euclidean case~\cite[Thm.~3.1.4]{nesterov2018lectures}.
In particular, any \emph{continuous} convex function on a \emph{closed} domain is closed convex.
Parts~\ref{it:lsc 1} and~\ref{it:lsc 2} state that any closed convex function~$f\colon D\to\R$ is lower semicontinuous, also if we extend it to~$M$ by setting~$f(p)=\infty$ for~$p\not\in D$ (in fact, this characterizes when a convex function is closed, but we will~not~need~this).

\begin{lem}\label{lem:closed convex}
Let $f\colon D \to \RR$ be a (not necessarily continuous) closed convex function defined on a convex subset~$D \subseteq M$.
Then:
\begin{enumerate}
\item\label{it:lsc 1} If~$(p_k) \subseteq D$ is a sequence such that $p_\infty := \lim_{k\to\infty} p_k \in D$, then $\liminf_{k\to\infty} f(p_k) \geq f(p_\infty)$.
\item\label{it:lsc 2} If~$(p_k) \subseteq D$ is a sequence such that $\lim_{k\to\infty} p_k \not\in D$, then $\lim_{k\to\infty} f(p_k) = \infty$.
\item\label{it:mini} If for some~$L\in\R$ the level set~$\mathcal L = \{ p \in D : f(p) \leq L \}$ is non-empty and bounded, then~$f$ attains its minimum.
\end{enumerate}
\end{lem}
\begin{proof}
\begin{enumerate}
\item
We need to show: for any subsequence~$(p_{k_j})$ such that~$\lim_{j\to\infty} f(p_{k_j}) = f_\infty$ for some $f_\infty \in \R\cup\{\pm\infty\}$, we have that $f_\infty \geq f(p_\infty)$.
If~$f_\infty=\infty$ there is nothing to show.
If~$f_\infty\in\R$ then we have $\lim_{j\to\infty} (p_{k_j},f(p_{k_j})) = (p_\infty,f_\infty) \in E_f$, since the epigraph is closed, and hence~$f_\infty \geq f(p_\infty)$.
Finally, we note $f_\infty=-\infty$ cannot occur.
Indeed, if $f_\infty=-\infty$ then~$f(p_{k_j}) \leq f(p_\infty) - 1$ for~$j$ large enough, hence~$(p_{k_j},f(p_\infty) - 1) \in E_f$ for~$j$ large enough and hence~$\lim_{j\to\infty} (p_{k_j},f(p_\infty) - 1) = (p_\infty,f(p_\infty) - 1) \in E_f$, which is a contradiction.

\item Assume this is not so.
Then there are a subsequence~$(p_{k_j})$ and~$L\in\R$ such that~$f(p_{k_j}) \leq L$ for all~$j$.
Now, $\lim_{j\to\infty} (p_{k_j},L) = (p_\infty,L)$, where $p_\infty := \lim_{k\to\infty} p_k$, but each~$(p_{k_j},L)$ is contained in the epigraph, and hence the same must be true for the limit.
It follows that~$p \in D$, which is a contradiction.

\item
Since the level set~$\mathcal L$ is non-empty, it contains a sequence $(p_k)$ such that $\lim_{k\to\infty} f(p_k) = f_* := \inf_{p \in D} f(p)$.
Because the epigraph is a closed subset of~$M \times \R$, the same is true for $\mathcal L \times \{L\} = E_f \cap (M \times \{L\})$, and hence~$\mathcal L$ is a closed subset of~$M$.
It is also bounded by assumption.
By the Hopf--Rinow theorem, which is applicable because we assume that~$M$ is geodesically complete, it follows that~$\mathcal L$ is compact.
After passing to a subsequence, we may therefore assume that $p_\infty := \lim_{k\to\infty} p_k$ exists and is in~$\mathcal L\subseteq D$.
For continuous~$f$, we then have~$f(p_\infty) = f_*$ and this concludes the proof.
If~$f$ is not continuous then we can proceed as follows.
First suppose that $f_* = -\infty$.
Fix any~$p_0 \in \mathcal L$.
Because~$M$ is geodesically complete and~$\mathcal L$ is bounded, there exists a constant~$C>0$ such that we can write~$p_k = \Exp_{p_0}(u_k)$ for some~$u_k\in T_{p_0}M$ such that~$\norm{u_k}_{p_0} = d(p_0, p_k) \leq C$ for all~$k$.
Then we can choose~$\alpha_k \in (0,1)$ such that $\alpha_k\to0$ and $\alpha_k f(p_k) \to -\infty$.
Then the points~$q_k := \Exp_{p_0}(\alpha_k u_k)$ satisfy
\begin{align*}
  f(q_k) \leq (1-\alpha_k) f(p_0) + \alpha_k f(p_k) = f(p_0) + \alpha_k \parens*{ f(p_k) - f(p_0) } \to -\infty,
\end{align*}
where the first inequality holds by geodesic convexity.
In particular, there is some constant~$K\in\R$ such that~$f(q_k) \leq K < f(p_0)$ for large enough~$k$.
Now, $(q_k,K)$ is in the epigraph and converges to $(p_0,K)$, because $\alpha_k \to 0$ and $\norm{u_k}_{p_0} \leq C$ for all~$k$.
But $f(p_0) > K$, so $(p_0,K)$ is not in the epigraph.
This contradicts the assumption that the epigraph is closed.
Thus we must have that $f_* > -\infty$.
Then, $\lim_{k\to\infty} (p_k,f(p_k)) = (p_\infty,f_*)$ and since the epigraph is closed, it must contain the latter, meaning that $f(p_\infty) \leq f_*$ and hence~$f(p_\infty) = f_*$.
\qedhere
\end{enumerate}
\end{proof}

We will later (in \cref{sec:barriers compatibility path-following}) be in the situation that~$D \subseteq M$ is open and we are interested in smooth objective functions~$f\colon D\to\R$ that have a \emph{closed convex extension}, meaning that~$f$ extends to a closed convex function on some convex superset of~$D$.
This is the case in particular if~$f$ extends to a continuous convex function on the closure~$\overline D$.

Just like in the Euclidean setting~\cite[Thm.~3.1.5]{nesterov2018lectures}, one can see that the sum of two closed convex functions is again closed convex.

\begin{lem}\label{lem:sum of lsc is lsc}
  Let~$f_1\colon D_1\to\R$, $f_2\colon D_2\to\R$ be closed convex functions defined on convex subsets~$D_1,D_2\subseteq M$.
  Then the function $f_1 + f_2$ is a closed convex function on~$D_1 \cap D_2$.
\end{lem}
\begin{proof}
  It is clear that~$f_1 + f_2$ is a convex function on~$D := D_1 \cap D_2$.
  To see that it is closed, consider an arbitrary convergent sequence~$(p_k,t_k)$ in~$E_{f_1+f_2}$, with limit point~$(p_\infty,t_\infty) \in M \times \R$.
  By \cref{lem:closed convex}, since~$f_1$ and~$f_2$ are closed convex, we have
  \begin{align*}
    \liminf_{k\to\infty} f_1(p_k) \geq f_1(p_\infty)
  \quad\text{and}\quad
    \liminf_{k\to\infty} f_2(p_k) \geq f_2(p_\infty),
  \end{align*}
  and hence
  \begin{align*}
    t_\infty
  = \lim_{k\to\infty} t_k
  % \geq \liminf_{k\to\infty} f_1(p_k) + f_2(p_k)
  \geq \liminf_{k\to\infty} f_1(p_k) + \liminf_{k\to\infty} f_2(p_k)
  \geq f_1(p_\infty) + f_2(p_\infty),
  \end{align*}
  which means that $(p_\infty,t_\infty) \in E_{f_1+f_2}$.
  Hence~$f_1+f_2$ is closed.
\end{proof}

As in the Euclidean setting, one can also characterize convexity differentially.
% A $C^1$-smooth function~$f\colon D \to \RR$ defined on an \emph{open} convex subset~$D \subseteq M$ is convex if and only if
% \begin{align*}
%   f(\Exp_p(v)) \geq f(p) + df(v)
% \end{align*}
% for every $p \in D$ and every tangent vector $v \in T_pM$ in the domain of the exponential map at~$p$.
In particular, a $C^2$-smooth function~$f\colon D \to \RR$ defined on an open convex subset~$D \subseteq M$ is convex if and only if the quadratic forms defined by the Hessian are positive semidefinite, i.e.,
\begin{align}\label{eq:convex via hessian}
  (\nabla^2 f)_p(v,v) \geq 0
\end{align}
for all $v \in T_p M$ and $p\in D$.
We discuss two refinements of the notion of convexity (for simplicity only in the $C^2$-smooth setting):
If~$f$ is strictly convex along any geodesic in the domain, then~$f$ is called strictly convex.
A sufficient condition for strict convexity is the following: for every~$p \in D$, the Hessian $(\nabla^2 f)_p$ is positive definite, i.e., \cref{eq:convex via hessian} holds with equality only for~$v = 0 \in T_p M$.
Similarly, we say that~$f$ is \emph{$\mu$-strongly convex} for some $\mu>0$ if it is so along any unit-speed geodesic in the domain.
This is the case if and only if, for all $v\in T_pM$ and $p \in D$,
\[ (\nabla^2 f)_p(v,v) \geq \mu \norm v_p^2. \]

In convex optimization, upper bounds on the Hessian of a convex function are often also useful.
We say that~$f$ is \emph{$\nu$-smooth} (not to be confused with smoothness in the sense of $C^\infty$) if it is so along any unit-speed geodesic in the domain, that is, if and only if
\[ (\nabla^2 f)_p(v,v) \leq \nu \norm v_p^2 \]
for all $v\in T_pM$ and $p \in D$.
When~$M$ is a Hadamard space then it is well-known that the distance~$d(\cdot,p_0)$ to any fixed point~$p_0\in M$ is convex, and that~$\frac12d^2(\cdot,p_0)$ is 1-strongly convex, just like in Euclidean space.
However, the latter will in general no longer be smooth.
We discuss these important functions in \cref{sec:distsq}.

Let $D \subseteq M$ be a convex subset (not necessarily open) that is also an embedded submanifold.
Equip~$D$ with the induced metric and let~$\tilde \nabla$ denote its Levi-Civita connection.
Then $D$ is a totally geodesic submanifold, so its shape tensor~$\sff$ vanishes~\cite[Prop.~8.12]{lee-riemannian-manifolds}.
% To see that~$D$ is totally geodesic, let~$p \in D$, $v \in T_p D$, and let~$\eps > 0$ be smaller enough such that the $\eps$-ball around~$p$ is uniquely geodesic, both with respect to~$M$ and~$D$, and such that~$\gamma(t) = \Exp_p^D(t\, v)$ is well-defined on~$(-2\eps, 2\eps)$.
% Then~$\gamma(t)$ is the unique length-minimizing curve in~$D$ between~$\gamma(\eps) \in D$ and~$\gamma(-\eps) \in D$.
% There exists a geodesic~$\sigma$ in~$M$ between these two points, which is of minimal length in~$M$; its length is at most that of~$\gamma$, because restricting to a submanifold cannot bring two points closer.
% By total convexity of~$D$, $\sigma$ is fully contained in~$D$, hence its length is also at least that of~$\gamma$.
% But~$\gamma$ was the unique length-minimizing curve in~$D$ between these two points, hence we must have~$\gamma = \sigma$.
% The~$M$-geodesic $\sigma$ therefore passes through~$p$ with initial velocity~$v$, and remains in~$D$ for a short time.
% Since~$p, v$ were arbitrary,~$D$ is totally geodesic.
Now let~$T$ be a~$(0,l)$-tensor field on~$D$ that is extended arbitrarily to a neighborhood of~$D$ in~$M$.
Then by \cref{eq:tensor contraction deriv,eq:gauss formula} we find that $\tilde\nabla T = \nabla T|_{(TD)^{\otimes(1+l)}}$, where the right-hand side notation means that we restrict $\nabla T$ to a $(0,1+l)$-tensor field on~$D$.
In particular, we inductively see that for every function~$f\colon M\to\RR$ and every~$l\geq0$, the following holds on~$D$:
\begin{equation}\label{eq:convex submanifold implies derivatives restrict}
  \tilde\nabla^l \tilde f = \nabla^l f|_{(TD)^{\ot l}}.
\end{equation}

%=============================================================================
\section{Self-concordance and Newton's method on manifolds}\label{sec:sc}
%=============================================================================
In this section we generalize the notion of self-concordance and the corresponding analysis of Newton's method from the Euclidean setting to the Riemannian setting, and we comment on complications incurred by curvature.
For expositions of the Euclidean theory of self-concordance and interior-point methods we refer to~\cite{nesterov-nemirovskii-ipm,nesterov2018lectures,renegar-ipm}.
Throughout this section we assume that $M$ is a connected and geodesically complete Riemannian manifold.

%-----------------------------------------------------------------------------
\subsection{Self-concordance}
%-----------------------------------------------------------------------------
Let $f\colon D \to \R$ be a convex function defined on an open convex subset $D \subseteq M$.
Then the Hessian is positive semidefinite, by \cref{eq:convex via hessian}, hence induces a (semi-)norm at each point.
The rate of change of the Hessian is captured by the third covariant derivative, $\nabla^3 f = \nabla (\nabla (\nabla f)) = \nabla (\nabla^2 f)$.
A function is called self-concordant if the latter can be bounded in terms of the former, as follows:

\begin{defn}[Self-concordance]\label{defn:sc}
  Let $f\colon D \to \R$ be a convex function defined on an open convex subset $D \subseteq M$, and let $a>0$.
  We say that $f$ is \emph{$\alpha$-self-concordant} if, for all~$p \in D$ and for all~$u,v,w\in T_pM$, we have
  \begin{align}\label{eq:defn sc}
    \abs{(\nabla^3 f)_p(u,v,w)} \leq \frac2{\sqrt \alpha} \sqrt{(\nabla^2 f)_p(u,u)} \sqrt{(\nabla^2 f)_p(v,v)} \sqrt{(\nabla^2 f)_p(w,w)},
  \end{align}
  It is called \emph{strongly $\alpha$-self-concordant} if is not just convex but closed convex, that is, if its epigraph~\eqref{eq:epi} is a closed subset of~$M \times \R$.
  % Finally, $f$ is \emph{(strongly) self-concordant} if it is (strongly) 1-self-concordant.
\end{defn}

Here we follow the conventions of~\cite{nesterov-nemirovskii-ipm}.
To interpret the definition, let us for a convex function~$f$, a point~$p$ in its domain, and~$\alpha>0$ define the positive semidefinite bilinear form and seminorm
\begin{align}\label{eq:norm}
  \braket{v,w}_{f,p,\alpha} = \frac{(\nabla^2 f)_p(v,w)}\alpha
\quad\text{and}\quad
  \norm{u}_{f,p,\alpha} = \sqrt{\frac {(\nabla^2 f)_p(u,u)} \alpha}.
\end{align}
When the Hessian is positive definite (as is the case, e.g., when $f$ is strongly convex), these endow~$M$ with a new Riemannian metric.
In convex optimization, $\braket{\cdot,\cdot}_{f,p,\alpha}$ is called the ``local inner product'' and $\norm{\cdot}_{f,p,\alpha}$ the ``local norm'', but we will refrain from using this terminology as it is ambiguous in the Riemannian setting.
For~$\alpha=1$, we will usually abbreviate $\braket{\cdot,\cdot}_{f,p} := \braket{\cdot,\cdot}_{f,p,1}$ and $\norm{\cdot}_{f,p} := \norm{\cdot}_{f,p,1}$.
We can now rewrite \cref{eq:defn sc} as follows:
\begin{align}\label{eq:defn sc via norms}
  \abs{(\nabla^3 f)_p(u,v,w)} \leq 2 \alpha \norm{u}_{f,p,\alpha} \norm{v}_{f,p,\alpha} \norm{w}_{f,p,\alpha}.
\end{align}
Thus self-concordance can be interpreted as a boundedness of the third covariant derivatives at each point with respect to the seminorms defined by the Hessian.

We record some basic properties.
Recall that self-concordant functions are defined on an open and convex domain, by definition.

\begin{lem}\label{lem:basic}
\begin{enumerate}
  \item\label{item:basic sc scaling} Let $f$ be a (strongly) $\alpha$-self-concordant function and let $c>0$. Then $c f$ is (strongly) $c \alpha$-self-concordant.
  \item\label{item:basic sc sum} Let $f_k \colon D_k \to \R$ be $\alpha_k$-self-concordant functions for $k=1,2$, and suppose~$D := D_1 \cap D_2$ is non-empty.
    Then $f := f_1 + f_2 \colon D \to \R$ is $\alpha$-self-concordant, with $\alpha := \min(\alpha_1,\alpha_2)$.
    If the functions $f_k$ are strongly $\alpha_k$-self-concordant, then $f$ is strongly $\alpha$-self-concordant.
  \item\label{item:basic sc product} Let $f_k \colon D_k \to \R$ be $\alpha$-self-concordant functions for $k=1,2$.
    Then the function $f\colon D_1\times D_2\to\R$ defined by $f(p_1,p_2) := f_1(p_1) + f_2(p_2)$ is $\alpha$-self-concordant.
    If both functions~$f_k$ are strongly $\alpha$-self-concordant, then so is $f$.
\end{enumerate}
\end{lem}
Property~\ref{item:basic sc scaling} follows from the definition, and~\ref{item:basic sc product} follows from~\ref{item:basic sc sum}. % , along with \cref{lem:sum of lsc is lsc}.
Before we prove~\ref{item:basic sc sum}, we give a simpler characterization of self-concordance.
As the Hessian is symmetric, third covariant derivatives are symmetric in the last two arguments.
This can also be seen explicity from the following formula for the third covariant derivative~$\nabla^3 f$, which follows from \cref{eq:tensor contraction deriv} and holds for any three vector fields $X$, $Y$, $Z$:
\begin{align}
  \label{eq:third cov}
  &(\nabla^3 f)(X,Y,Z)
  %&= (\nabla_X (\nabla^2 f))(Y,Z)
  = X\mleft( (\nabla^2 f)(Y,Z) \mright) - (\nabla^2 f)(\nabla_X Y, Z) - (\nabla^2 f)(Y, \nabla_X Z).
\end{align}
This leads to the following simplification:

\begin{lem}
  A convex function $f\colon D \to \R$ defined on an open convex subset $D \subseteq M$ is $\alpha$-self-con\-cor\-dant if, and only if, for all $p\in M$ and $u,v\in T_pM$, we have
  \begin{align}\label{eq:defn sc sym}
    \abs{(\nabla^3 f)_p(u,v,v)} \leq \frac2{\sqrt \alpha} \sqrt{(\nabla^2 f)_p(u,u)} \ (\nabla^2 f)_p(v,v)
  \end{align}
  or, equivalently,
  \begin{align}\label{eq:defn sc via norms sym}
    \abs{(\nabla^3 f)_p(u,v,v)} \leq 2 \alpha \norm{u}_{f,p,\alpha} \norm{v}_{f,p,\alpha}^2.
  \end{align}
\end{lem}

However, third covariant derivatives are \emph{not} symmetric when~$M$ is a curved manifold, as follows from the Ricci identity~\cite[Thm.~7.14]{lee-riemannian-manifolds}.
To see this, we combine \cref{eq:third cov,eq:defn hessian} to see that for any three vector fields $X$, $Y$, $Z$:
\begin{align*}
(\nabla^3 f)(X,Y,Z)
%&= (\nabla_X (\nabla^2 f))(Y,Z)
% = X\mleft( (\nabla^2 f)(Y,Z) \mright) - (\nabla^2 f)(\nabla_X Y, Z) - (\nabla^2 f)(Y, \nabla_X Z) \\
&= X\mleft( Y(Zf) \mright) - X\mleft( (\nabla_Y Z)f \mright) - (\nabla_X Y)(Zf) + \parens*{ \nabla_{\nabla_X Y} Z }f \\
&- Y\mleft( (\nabla_X Z) f \mright) + \parens*{ \nabla_Y(\nabla_X Z) }f.
\end{align*}
Using symmetry of the Levi-Civita connection, one finds that
\begin{align}\label{eq:asym}
  (\nabla^3 f)(X,Y,Z) - (\nabla^3 f)(Y,X,Z)
% = X\mleft( Y(Zf) \mright) - X\mleft( (\nabla_Y Z)f \mright) - (\nabla_X Y)(Zf) + \parens*{ \nabla_{\nabla_X Y} Z }f - Y\mleft( \nabla_X Z f \mright) + \parens*{ \nabla_Y(\nabla_X Z) }f
% - X\mleft( Y(Zf) \mright) + X\mleft( (\nabla_Y Z)f \mright) + (\nabla_X Y)(Zf) - \parens*{ \nabla_{\nabla_X Y} Z }f + Y\mleft( \nabla_X Z f \mright) - \parens*{ \nabla_Y(\nabla_X Z) }f
= -\parens*{ R(X,Y)Z } f
= -\braket{R(X,Y)Z, \grad(f)}
\end{align}
Accordingly, the third covariant derivative is in general \emph{not} symmetric.
Indeed, the asymmetry is precisely related to the nonvanishing of the Riemann curvature tensor!

Due to this asymmetry, to establish self-concordance, we have to show~\cref{eq:defn sc sym} for possibly different $u, v \in T_p M$, whereas we could assume $u=v$ in the Euclidean case; see~\cref{subsec:scag} for more details.
The following proof of~\cref{lem:basic}\ref{item:basic sc sum} is a generalization of \cite[Thm.~5.1.1]{nesterov2018lectures} to our setting. %i.e., for~$u = v$.
\begin{proof}[Proof of~\cref{lem:basic}\ref{item:basic sc sum}]
	For $p \in D = D_1 \cap D_2$ and $u,v \in T_p M$, we have
	\begin{align}
		\frac{\abs{(\nabla^3f)_p(u,v,v)}}{2\sqrt{(\nabla^2 f)_p(u,u)}(\nabla^2 f)_p (v,v)} &\leq \frac{\abs{(\nabla^3 f_1)_p(u,v,v)} + \abs{(\nabla^3f_2)_p(u,v,v)}}{2 \sqrt{ (\nabla^2 f_1)_p(u,u)+ (\nabla^2 f_2)_p(u,u)}((\nabla^2 f_1)_p(v,v)+ (\nabla f_2)_p(v,v))} \nonumber \\
		& \leq \frac{x_1 \omega_1/\sqrt{\alpha_1} + x_2 \omega_2/\sqrt{\alpha_2}}
		{\sqrt{x_1^2+ x_2^2}(\omega_1+ \omega_2)}, \label{eqn:quantity}
  \end{align}
  where we let  $x_i := \sqrt{(\nabla^2f_i)_p(u,u)}$ and $\omega_i := (\nabla^2 f_i)_p(v,v)$ for $i=1,2$, and for the last estimate we used~$\alpha_i$-self-concordance of~$f_i$.
  We now upper bound the quantity in~\cref{eqn:quantity}.
  Observing invariance under the change $(x_1,x_2, \omega_1,\omega_2) \to (sx_1,sx_2,t \omega_1,t\omega_2)$ for $s,t > 0$, we may consider the following optimization problem:
  \begin{eqnarray*}
    \mbox{maximize} && \omega_1 x_1/\sqrt{\alpha_1} + \omega_2 x_2/\sqrt{\alpha_2} \\
    \mbox{s.t.} && x_1^2 + x_2^2 = 1, \, \omega_1 + \omega_2 = 1, \\
                &&  x_1,x_2, \omega_1,\omega_2 \geq 0.
  \end{eqnarray*}
  First we fix $\omega_i$, and maximize over the choice of $x_i$.
  This is a linear maximization over the intersection of the unit circle with the positive orthant, with objective given by~$(\omega_1 / \sqrt{\alpha_1}, \omega_2 / \sqrt{\alpha_2})$, which is itself in the positive orthant.
  Therefore the maximum is attained at
  \begin{equation*}
    (x_1,x_2) = \frac{(\omega_1/\sqrt{\alpha_1}, \omega_2/\sqrt{\alpha_2})}{\sqrt{\omega_1^2/\alpha_1 + \omega_2^2/\alpha_2}},
  \end{equation*}
  where the value of the objective is~$\sqrt{\omega_1^2/\alpha_1 + \omega^2_2/\alpha_2}$.
  This reduces the problem to
  \[
    \mbox{maximize} \ \sqrt{\omega_1^2/\alpha_1 + \omega^2_2/\alpha_2} \ \
    \mbox{s.t.} \ \ \omega_1 + \omega_2 = 1, \, \omega_1,\omega_2 \geq 0.
  \]
  By convexity of the objective, the maximum is attained at~$(\omega_1, \omega_2) = (1,0)$ or~$(\omega_1, \omega_2) = (0, 1)$.
  Therefore \cref{eqn:quantity} is at most~$\max (1/\sqrt{\alpha_1}, 1/\sqrt{\alpha_2})$, and $f$ is $\alpha$-self-concordant for $\alpha = \min (\alpha_1,\alpha_2)$.
  The claim that~$f$ is strongly~$\alpha$-self-concordant whenever the~$f_i$ are strongly~$\alpha_i$-self-concordant then follows from~\cref{lem:sum of lsc is lsc}.
\end{proof}

We now state a key property that is required for the analysis of Newton's method of self-concordant functions.
It quantifies the change of the Hessian or local norm as a function of the distance, measured with respect to the norm~\eqref{eq:norm}, providing a finitary version of \cref{defn:sc}.
Then the following result is a direct translation of the Euclidean argument in~\cite[Thm.~2.1.1]{nesterov-nemirovskii-ipm} along with the notion of self-concordance from \cref{defn:sc}.

\begin{thm}[Stability of Hessians]\label{thm:sc hessian ratio bounds}
Let $f\colon D \to \R$ be an $\alpha$-self-concordant function defined on an open convex subset $D \subseteq M$, and let~$p\in D$.
Let $u \in T_pM$ be such that $r := \norm{u}_{f,p,\alpha} < 1$.
If $q := \Exp_p(u) \in D$, then we have the following estimate: for all~$v \in T_pM$,
\begin{align}
% \label{eq:sc norm ratio}
%   \parens*{1 - r} \norm{v}_{f,p,\alpha} &\leq\norm{\tau_{\gamma,1} v}_{f,q,\alpha} \leq \frac 1 {1 - r} \norm{v}_{f,p,\alpha}, \\
\label{eq:sc hess ratio}
  \parens*{1 - r}^2 \, (\nabla^2 f)_p(v,v) &\leq(\nabla^2 f)_q(\tau_{\gamma,1} v,\tau_{\gamma,1} v) \leq \frac 1 {\parens*{ 1 - r }^2} \, (\nabla^2 f)_p(v,v),
\end{align}
or, equivalently,
\begin{align*}
  \parens*{1 - r}^2 \, (\nabla^2 f)_p &\preceq\tau_{\gamma,1}^*(\nabla^2 f)_q \preceq \frac 1 {\parens*{ 1 - r }^2} \, (\nabla^2 f)_p,
\end{align*}
where $\tau_{\gamma,1}$ denotes the parallel transport along the geodesic $\gamma(t) := \Exp_p(tu)$ from~$p$ to~$q$.
\end{thm}
\begin{proof}
Since the domain is convex, we know that $\gamma(t) = \Exp_p(tu) \in D$ for all~$t \in [0,1]$.
Consider the following two functions:
\begin{align*}
  \phi\colon [0,1] \to \R, \quad &\phi(t) = (\nabla^2 f)_{\gamma(t)}(\tau_{\gamma,t} v, \tau_{\gamma,t} v), \\ % = \alpha \norm{\tau_{\gamma,t} v}_{f,\gamma(t),\alpha}^2, \\
  \psi\colon [0,1] \to \R, \quad &\psi(t) = (\nabla^2 f)_{\gamma(t)}(\tau_{\gamma,t} u, \tau_{\gamma,t} u).    % = \alpha \norm{\tau_{\gamma,t} u}_{f,\gamma(t),\alpha}^2.
\end{align*}
Using \cref{eq:deriv via transport}, with $T = \nabla^2 f$ and using that $\dot\gamma(t) = \tau_{\gamma,t} u$, we have
\begin{align*}
  \dot\phi(t)
= \parens*{ \nabla_{\dot\gamma(t)} (\nabla^2 f) }(\tau_{\gamma,t} v, \tau_{\gamma,t} v)
= (\nabla^3 f)(\tau_{\gamma,t} u, \tau_{\gamma,t} v, \tau_{\gamma,t} v).
\end{align*}
Hence, using $\alpha$-self-concordance as in \cref{eq:defn sc}, %,eq:defn sc via norms,eq:defn sc sym,eq:defn sc via norms sym},
\begin{align}\label{eq:phi ode}
  \abs{\dot\phi(t)} \leq
% 2 \alpha \norm*{\tau_{\gamma,t} u}_{f,\gamma(t),\alpha} \norm*{\tau_{\gamma,t} v}_{f,\gamma(t),\alpha}^2 =
\frac 2 {\sqrt \alpha} \sqrt{\psi(t)} \, \phi(t).
\end{align}
Similarly,
\begin{align*}
  \dot\psi(t)
= \parens*{ \nabla_{\dot\gamma(t)} (\nabla^2 f) }(\tau_{\gamma,t} u, \tau_{\gamma,t} u)
= (\nabla^3 f)(\tau_{\gamma,t} u, \tau_{\gamma,t} u, \tau_{\gamma,t} u).
\end{align*}
and hence using only $\alpha$-self-concordance along the geodesic~$\gamma$, as in \cref{eq:scag}, we find that
\begin{align}\label{eq:psi ode}
  \abs{\dot\psi(t)} \leq \frac 2 {\sqrt \alpha} \psi(t)^{3/2}.
\end{align}
With these estimates in place we can proceed as in the proof of~\cite[Thm.~2.1.1]{nesterov-nemirovskii-ipm}.
By Gr\"onwall's inequality, there are two cases: either $\psi$ vanishes identically on the interval~$[0,1]$, or it is everwhere positive.
In the former case, \cref{eq:phi ode} implies that~$\phi$ is constant and hence~$\phi(1)=\phi(0)$, which in turn implies the claim.
In the latter case, we can write \cref{eq:psi ode} as
\begin{align}\label{eq:psi bound}
  \abs*{ \partial_t \psi(t)^{-1/2} }
= \frac12 \frac {\abs{\dot\psi(t)}} {\psi(t)^{3/2}}
\leq \frac 1 {\sqrt \alpha},
\end{align}
from which it follows that
\begin{align*}
  \psi(t)^{-1/2}
\geq \psi(0)^{-1/2} - \frac t {\sqrt \alpha}
% = \frac 1 {\sqrt{\psi(0)}} - \frac t {\sqrt \alpha}
= \frac 1 {\sqrt \alpha \norm{u}_{f,p,\alpha}} - \frac t {\sqrt \alpha}
= \frac {1 - rt} {r \sqrt \alpha}
\end{align*}
and hence, since $r < 1$,
\begin{align*}
  \sqrt{\psi(t)} \leq \frac {r \sqrt \alpha} {1 - rt}.
\end{align*}
Thus \cref{eq:phi ode} implies
\begin{align*}
  \abs{\dot\phi(t)}
% \leq \frac 2 {\sqrt \alpha} \sqrt{\psi(t)} \, \phi(t)
\leq \frac {2r} {1 - rt} \, \phi(t).
\end{align*}
Similarly to the above, either~$\phi$ vanishes identically on~$[0,1]$, in which case there is nothing to prove, or it is everywhere positive, in which case we have
\begin{align*}
  \abs*{\partial_t \log \phi(t)} \leq \frac {2r} {1 - rt}
\end{align*}
and hence
\begin{align*}
  \abs*{ \log \frac {\phi(t)} {\phi(0)} } \leq
% \int_0^t \abs{\partial_s \log \phi(s)} ds \leq 2 \int_0^t \frac {r} {1 - rs} ds = 2 \log \frac 1 {1 - rs} \Bigr|_0^t =
     2 \log \frac 1 {1-rt}.
\end{align*}
For $t=1$ this yields the desired inequality.
\end{proof}

%-----------------------------------------------------------------------------
\subsection{Self-concordance along geodesics}\label{subsec:scag}
%-----------------------------------------------------------------------------

When $M=\R^n$ is a Euclidean space, then the third derivative is symmetric in all three arguments, and standard results on trilinear forms~\cite{Banach1938} imply that the above is equivalent to
$\abs*{\partial_{t=0}^3 f(p + tv)} = \abs{(\nabla^3 f)_p(v,v,v)]} \leq 2 \alpha \norm{v}_{f,p,\alpha}^3$ for all $p,v\in \R^n$,
which shows that self-concordance is equivalent to \emph{self-concordance along the geodesics} of Euclidean space.
This characterization is highly useful for showing that functions are self-concordant.
The richness of the family of self-concordant functions is a key reason for the wide applicability of interior-point methods~\cite{nesterov-nemirovskii-ipm,hildebrand2014canonical,fox2015schwarz,bubeck2019entropic,chewi2021entropic}.

This notion can also be generalized naturally to the Riemannian setting:

\begin{defn}[Self-concordance along geodesics]\label{defn:scag}
Let $f\colon D \to \R$ be a convex function defined on an open convex subset $D \subseteq M$, and let $\alpha>0$.
We say that $f$ is \emph{$\alpha$-self-concordant along geodesics} if, for all~$p \in D$ and for all~$v\in T_pM$, we have
\begin{align}\label{eq:scag}
  \abs*{\partial_{t=0}^3 f(\Exp_p(tv))} = \abs{(\nabla^3 f)_p(v,v,v)]} \leq \frac2{\sqrt \alpha} \parens*{ (\nabla^2 f)_p(v,v) }^{3/2}
\end{align}
or, equivalently,
\begin{align}\label{eq:scag via norms}
  \abs*{\partial_{t=0}^3 f(\Exp_p(tv))} = \abs{(\nabla^3 f)_p(v,v,v)]} \leq 2 \alpha \norm{v}_{f,p,\alpha}^3.
\end{align}
It is called \emph{strongly $\alpha$-self-concordant along geodesics} if is not just convex but closed convex, that is, if its epigraph~\eqref{eq:epi} is a closed subset of~$M \times \R$.
\end{defn}

In other words, $f$ is (strongly) $\alpha$-self-concordant along geodesics if and only if for every geodesic~$\gamma\colon \R \to M$, the function~$f\circ \gamma\colon I \to \R$ is (strongly) $\alpha$-self-concordant on $I := \gamma^{-1}(D)$.
% Indeed, we have $\dot\gamma(t) = \tau_{\gamma,t}(\dot\gamma(0)) = \tau_{\gamma,t}u$.
% Hence $\dot g(t) = df_{\gamma(t)}(\dot\gamma(t)) =  df_{\gamma(t)}(\tau_{\gamma,t}u)$ and, using \cref{eq:deriv via transport}, we get
% \begin{align*}
%   \ddot g(t) = \partial_t df_{\gamma(t)}(\tau_{\gamma,t}u) = \nabla_{\dot\gamma(t)} df_{\gamma(t)}(\tau_{\gamma,t}u) = (\nabla^2 f)_{\gamma(t)}(\tau_{\gamma,t}u, \tau_{\gamma,t}u)
% \end{align*}
% and similarly
% \begin{align*}
%   \dddot g(t) = \partial_t (\nabla^2 f)_{\gamma(t)}(\tau_{\gamma,t}u, \tau_{\gamma,t}u) = (\nabla^3 f)_{\gamma(t)}(\tau_{\gamma,t}u, \tau_{\gamma,t}u, \tau_{\gamma,t}u).
% \end{align*}
% It follows that $g$ is $\alpha$-self-concordant
There is also a version of \cref{lem:basic} as a direct consequence of the Euclidean result.

\Cref{defn:scag} had been proposed in~\cite{ji2007optimization,jiang2007self} as a suitable notion of self-concordance in the Riemannian setting.
Clearly, any (strongly) self-concordant function is also (strongly) self-concordant along geo\-de\-sics.
However, since third covariant derivatives are \emph{not} symmetric in all arguments when~$M$ is a curved manifold, as we saw in \cref{eq:asym}, self-concordance along geodesics need \emph{not} imply self-concordance in the stronger sense of \cref{defn:sc}, in contrast to what was suggested in~\cite[Eq.~(3) and Prop.~1]{jiang2007self}.
While self-concordance along geodesics already allows lifting several useful results from the Euclidean theory, it is the stronger notion of \cref{defn:sc} that is required to prove the fundamental \cref{thm:sc hessian ratio bounds}, which underpins the analysis of the Newton method in the quadratic convergence regime in \cref{thm:newton decrement after newton step}.
We give non-trivial examples of self-concordant functions on curved spaces in \cref{sec:distsq,sec:applications}.

In the remainder of this section we discuss a number of useful results for functions that are self-concordant along geodesics.
These follow directly from the Euclidean theory.
While some of these were already proved in~\cite{ji2007optimization,jiang2007self}, we give all proofs to keep the exposition self-contained.
We start with a version of~\cite[Thm.~5.1.5]{nesterov2018lectures}.

\begin{prop}[Stability of second derivative along geodesic]\label{prop:stab 2nd}
Let $f\colon D \to \R$ be $\alpha$-self-concordant along geodesics, with~$D \subseteq M$ open and convex, and let $p\in D$.
Consider any geodesic~$\gamma(t) = \Exp_p(tu)$ such that $\gamma(1) \in D$, and set~$r := \norm{u}_{f,p,\alpha}$.
Then the $\alpha$-self-concordant function $g(t) := f(\gamma(t))$ for $t\in[0,1]$ satisfies the lower bound
\begin{align}\label{eq:sc geod ratio lower}
\ddot g(t)
\geq \frac {\ddot g(0)} {\parens*{1 + tr}^2}
= \frac {\alpha r^2} {\parens*{1 + tr}^2},
\end{align}
and if $rt<1$ also the upper bound
\begin{align}\label{eq:sc geod ratio upper}
  \ddot g(t)
\leq \frac {\ddot g(0)} {\parens*{ 1 - tr }^2}
= \frac {\alpha r^2} {\parens*{ 1 - tr }^2}.
\end{align}
\end{prop}
\begin{proof}
As in the proof of \cref{thm:sc hessian ratio bounds}, we consider the function
\begin{align*}
  \psi\colon [0,1] \to \R, \quad \psi(t) = \ddot g(t),
\end{align*}
and find from \cref{eq:psi ode} that it either vanishes identically on $[0,1]$, in which case the claim holds trivially, or it is everywhere positive, in which case \cref{eq:psi bound} holds, namely for all $t\in[0,1]$,
\begin{align*}
  \abs*{ \partial_t \psi(t)^{-1/2} }
\leq \frac 1 {\sqrt \alpha}.
\end{align*}
Accordingly,
\begin{align*}
\psi(0)^{-1/2} \parens*{ 1 - tr }
= \psi(0)^{-1/2} - \frac t {\sqrt \alpha}
\leq \psi(t)^{-1/2}
\leq \psi(0)^{-1/2} + \frac t {\sqrt \alpha}
= \psi(0)^{-1/2} \parens*{ 1 + tr },
\end{align*}
which implies both bounds.
\end{proof}

The lower bound strengthens the one in \cref{eq:sc hess ratio} in the special case that~$v=u$.
The upper bound implies that any function that is strongly self-concordant along geodesics must contain a certain region in its domain.
We first define the region and then state the result.
\begin{defn}[Dikin ellipsoid]\label{defn:dikin}
Let $f\colon D \to \R$ be a convex function defined on an open convex subset $D \subseteq M$, and let~$\alpha>0$.
Then the \emph{(open) Dikin ellipsoid} of radius~$r>0$ at~$p\in M$ is
\begin{align*}
  B^\circ_{f,p,\alpha}(r) = \braces*{ \Exp_p(u) : u \in T_pM, \ \norm{u}_{f,p,\alpha} < r }.
\end{align*}
For $\alpha=1$, we abbreviate $B^\circ_{f,p} := B^\circ_{f,p,1}$.
\end{defn}

The following result is easily generalized from the Euclidean setting.
The proof is essentially the same as in~\cite[Thm.~2.1.1]{nesterov-nemirovskii-ipm}.
\begin{cor}[Dikin inclusion]\label{cor:dikin}
Let $f\colon D \to \R$ be strongly $\alpha$-self-concordant along geodesics, defined on an open convex subset $D \subseteq M$.
Then $B^\circ_{f,p,\alpha}(1) \subseteq D$ for every~$p\in D$.
\end{cor}
\begin{proof}
Take any $v \in T_pM$ such that $r := \norm{v}_{f,p,\alpha} < 1$.
Let~$\sigma$ be the supremum of those~$s\geq0$ such that~$\gamma(s) := \Exp_p(sv) \in D$.
Since~$p\in D$ and~$D$ is open, we know that~$\sigma>0$, and since~$D$ is convex, we know that~$\gamma(s) \in D$ for all $s\in[0,\sigma)$.

We need to show that $\gamma(1) \in D$ and claim that in fact~$\sigma > 1/r > 1$ (with $1/0 = \infty$).
For sake of finding a contradiction, assume that this is not so, i.e., that~$\sigma \leq 1/r$.
For every~$s \in [0,\sigma)$ we can apply \cref{prop:stab 2nd} with $u := sv$, which satisfies $\norm{u}_{f,p,\alpha} = sr < \sigma r \leq 1$.
Then the upper bound in \cref{eq:sc geod ratio upper} gives
\begin{align*}
  \ddot g(s) \leq \frac 1{\parens*{ 1 - sr }^2} \ddot g(0),
\end{align*}
where $g(s) = f(\gamma(s))$.
Accordingly, the function $g$ has bounded derivative on~$[0,\sigma)$, thus it is itself bounded on this interval, say~$g(s) \leq L$ for some~$L\in\R$.
As $f$ is strongly self-concordant, the level set $\{ q \in D : f(q) \leq L \}$ is closed in~$M$, and hence it must contain~$\gamma(\sigma) = \lim_{s\uparrow\sigma} \gamma(s)$.
But $D$ is open, so this in turn implies there must also exist some $t>\sigma$ such that $\gamma(t) \in D$, contradicting the definition of $\sigma$.
\end{proof}

In other words, for any $p\in D$ and $u \in T_pM$ such that $\norm{u}_{f,p,\alpha} < 1$ it is automatically true that $\Exp_p(u) \in D$, so we do not have to assume this in \cref{thm:sc hessian ratio bounds,prop:stab 2nd}.

The above also implies that a strongly-self-concordant function can only have a degenerate Hessian if its domain contains a geodesic.

\begin{cor}[Domain]\label{cor:domain}
If a strongly $\alpha$-self-concordant function $f\colon D \to \R$ contains no (infinite) geodesic in its domain, then $(\nabla^2 f)_p$ is positive definite for all~$p\in D$.
In particular, this is the case if $M$ is a Hadamard manifold and the domain is bounded.
\end{cor}
\begin{proof}
If $(\nabla^2 f)_p(u,u)=0$ for some $p\in D$ and $u\in T_pM$, then $\Exp_p(\R u) \subseteq B^\circ_{f,p,\alpha}(1)$.
Thus \cref{cor:dikin} shows that $D$ contains the geodesic~$\gamma(t) = \Exp_p(tu)$ for $t\in\R$.
\end{proof}

The following results bound a self-concordant function in terms of its linear approximation at some arbitrary point, in terms of the quantity
\begin{align}\label{eq:defn rho}
  \rho\colon (-\infty,1) \to \R, \quad \rho(r) = -r - \log(1-r),
\end{align}
which is $\rho(r) = \frac12r^2 + \bigO(r^3)$ for small~$r$.
The first result lifts~\cite[Thm.~5.1.8]{nesterov2018lectures} to the geodesic setting and follows directly by integrating the lower bound in \cref{prop:stab 2nd}.

\begin{cor}[Lower bound]\label{cor:lower}
Let $f\colon D \to \R$ be $\alpha$-self-concordant along geodesics, defined on an open convex subset $D \subseteq M$, and let $p \in D$.
Then, for every $u\in T_pM$ such that~$q := \Exp_p(u) \in D$, we have
\begin{align}\label{eq:lower diff}
  df_q(\tau_{\gamma, t} u) - df_p(u) \geq \frac{\alpha t r^2}{1 + t r}
\end{align}
where $r := \norm{u}_{f,p,\alpha}$ and $\tau_{\gamma,t}$ denotes the parallel transport along the geodesic $\gamma(t) := \Exp_p(tu)$ from~$p$ to~$q$, and
\begin{align*}
  f(q) \geq f(p) + df_p(u) + \alpha \rho(-r).
\end{align*}
\end{cor}
\begin{proof}
By \cref{prop:stab 2nd}, we see that $g(t) := f(\Exp_p(tu))$ satisfies
\begin{align*}
  \ddot g(t) \geq \frac {\alpha r^2} {\parens*{1 + t r}^2}
\end{align*}
for all $t\in[0,1]$.
By integrating,
\begin{align*}
  \dot g(t) - \dot g(0)
\geq \int_0^t \frac {\alpha r^2} {\parens*{1 + s r}^2} ds
% = \left[ -\frac {\alpha r} {1 + s r} \right]_{s=0}^t
% = \alpha r - \frac {\alpha r} {1 + tr}
= \frac {\alpha t r^2} {1 + tr}.
\end{align*}
Since $\dot g(0) = df_p(u)$ and $\dot g(1) = df_q(\tau_{\gamma,1}u)$, this proves the first bound.
One more integral yields
\begin{equation*}
  g(1) - g(0) - \dot g(0)
\geq \int_0^1  \frac {\alpha s r^2} {1 + sr} ds
% = \alpha r \int_0^1  \frac {sr} {1 + sr} ds
% = \alpha r \int_0^1 \parens*{ 1 - \frac 1 {1 + sr} } ds
% = \alpha r \parens*{ 1 - \int_0^1 \frac 1 {1 + sr} ds }
% = \alpha r \parens*{ 1 - \frac1r \left[ \log(1 + sr) \right]_{s=0}^1 }
% = \alpha r \parens*{ 1 - \frac1r \log(1+r) }
= \alpha \parens*{ r - \log(1+r) }
= \alpha \rho(-r).
\qedhere
\end{equation*}
\end{proof}

The second result generalizes~\cite[Thm.~5.1.9]{nesterov2018lectures} to the geodesic setting and follows by similarly integrating the upper bound in \cref{prop:stab 2nd}.

\begin{cor}[Upper bound]\label{cor:upper}
Let $f\colon D \to \R$ be $\alpha$-self-concordant along geodesics, defined on an open convex subset $D \subseteq M$, and let $p \in D$.
Then, for every $u\in T_pM$ such that~$q := \Exp_p(u) \in D$ and~$r := \norm{u}_{f,p,\alpha} < 1$, we have
\begin{align*}
  df_q(\tau_{\gamma,t} u) - df_p(u) \leq \frac{\alpha t r^2}{1 - r t},
\end{align*}
where $\tau_{\gamma,t}$ denotes the parallel transport along the geodesic $\gamma(t) = \Exp_p(tu)$ from $p$ to $q$, and
\begin{align*}
  f(q) \leq f(p) + df_p(u) + \alpha \rho(r).
\end{align*}
If $f$ is strongly $\alpha$-self-concordant along geodesics, then the requirement that~$q \in D$ is automatic (by \cref{cor:dikin}).
\end{cor}
\begin{proof}
Similarly to the proof of \cref{cor:lower}, we can apply \cref{prop:stab 2nd} to see that the function $g(t) := f(\Exp_p(tu))$ satisfies
\begin{align*}
  \ddot g(t) \leq \frac {\alpha r^2} {\parens*{1 - t r}^2}
\end{align*}
for all $t\in[0,1]$.
By integration,
\begin{align*}
  \dot g(t) - \dot g(0)
\leq \int_0^t \frac {\alpha r^2} {\parens*{1 - s r}^2} ds
% = \left[ \frac {\alpha r} {1 - s r} \right]_{s=0}^t
% = \frac {\alpha r} {1 - tr} - \alpha r
= \frac {\alpha t r^2} {1 - tr}
\end{align*}
and
\begin{equation*}
  g(1) - g(0) - \dot g(0)
\leq \int_0^1  \frac {\alpha s r^2} {1 - sr} ds
% = \alpha r \int_0^1  \frac {sr} {1 - sr} ds
% = \alpha r \int_0^1 \parens*{ \frac 1 {1 - sr} - 1 } ds
% = \alpha r \parens*{ \int_0^1 \frac 1 {1 - sr} - 1 ds }
% = \alpha r \parens*{ - \frac1r \left[ \log(1 - sr) \right]_{s=0}^1 - 1 }
% = \alpha r \parens*{ - 1 - \frac1r \log(1-r) }
= \alpha \parens*{ - r - \log(1-r) }
= \alpha \rho(r).
\qedhere
\end{equation*}
\end{proof}

%-----------------------------------------------------------------------------
\subsection{Newton's method}\label{subsec:newton}
%-----------------------------------------------------------------------------
We are now ready to give an analysis of Newton's method for self-concordant functions.
In particular, as in the Euclidean case, we are able to provide quadratic guarantees on the changes in the so-called Newton decrement (\cref{thm:newton decrement after newton step}).
This key result requires self-concordance.
Afterwards we also recall some useful results due to~\cite{ji2007optimization,jiang2007self} which only rely on self-concordance along geodesics.

Recall Newton's method (cf.~\cite[\S7.5]{udristeConvexFunctionsOptimization1994}): given a convex function~$f$ and a point $p$ in its domain, consider its local quadratic approximation
\begin{align*}
  % q_{f,p}\colon T_pM \to \RR, \quad q_{f,p}(v) = f(q) + df(v) + \frac12(\nabla^2f)_p(v,v).
  f(\Exp_p(v)) \approx f(p) + df_p(v) + \frac12(\nabla^2f)_p(v,v)
\end{align*}
and minimize the right-hand side over all $v\in T_pM$.
If $(\nabla^2f)_p$ is non-degenerate and hence positive definite, as we will assume for convenience, there is a unique minimizer called the Newton step.

\begin{defn}[Newton step and Newton iterate]
Let~$f\colon D \to \R$ be a convex function defined on an open convex set $D\subseteq M$, and let $p\in D$ be a point such that $(\nabla^2 f)_p$ is positive definite.
Then we define the \emph{Newton step} of~$f$ at~$p$ as the unique vector~$n_{f,p} \in T_pM$ such that
\begin{align}\label{eq:defn newton step}
  (\nabla^2f)_p(n_{f,p},\cdot) = -df_p
\end{align}
% i.e., $(\nabla^2f)_p(n_{f,p},v) = -df_p(v)$ for all $v\in T_pM$,
and the \emph{Newton iterate} of~$f$ at~$p$ is defined as
\begin{align*}
  p_{f,+} := \Exp_p(n_{f,p}) \in M,
\end{align*}
which need not be in~$D$.
We can also write
\begin{align*}
  n_{f,p} = - \Hess(f)_p^{-1} \grad(f)_p
\quad\text{and}\quad
  p_{f,+} = \Exp_p(- \Hess(f)_p^{-1} \grad(f)_p).
\end{align*}
in terms of the gradient vector and Hessian operator (cf.\ \cref{subsec:hessian}).
\end{defn}

The gap between the function value and the minimum of the quadratic approximation is
\begin{align*}
  \frac12(\nabla^2f)_p(n_{f,p},n_{f,p}) = \frac \alpha2 \norm{n_{f,p}}_{f,p,\alpha}^2 = \frac \alpha2 \lambda_{f,\alpha}(p)^2,
\end{align*}
where $\lambda_{f,\alpha}$ is the so-called Newton decrement, which we define next.
\begin{defn}[Newton decrement]\label{defn:newton decrement}
  Let~$f\colon D \to \R$ be a convex function defined on an open convex set $D\subseteq M$, let $p\in D$ be a point such that $(\nabla^2 f)_p$ is positive definite, and let~$\alpha>0$.
  Then we define the \emph{Newton decrement} of $f$ at $p$ by
  \begin{align*}
    \lambda_{f,\alpha}(p)
  := \norm{n_{f,p}}_{f,p,\alpha}
  = \tfrac 1 \alpha \norm{df_p}_{f,p,\alpha}^*
  = \max_{0\neq v \in T_pM} \frac {\abs{df_p(v)}} {\alpha \norm{v}_{f,p,\alpha}}
  = \max_{0\neq v \in T_pM} \frac {\abs{df_p(v)}} {\sqrt{\alpha (\nabla^2f)_p(v,v)}},
  \end{align*}
  where $\norm{\omega}_{f,p,\alpha}^* := \max_{0\neq v \in T_pM} \frac {\abs{\omega(v)}} {\,\norm{v}_{f,p,\alpha}}$ is the dual norm on~$T_p^*M$ induced by~$\norm{\cdot}_{f,p,\alpha}$.
  That~is,\footnote{To see the second equality, replace $u$ by $t u$ for $t \in \RR$, and maximize over $t$.}
  \begin{align}\label{eq:newton decrement variational}
    \lambda_{f,\alpha}(p)
    & = \min \braces*{ \lambda \geq 0 : df_p \ot df_p \preceq \lambda^2 \alpha \, (\nabla^2f)_p } \\
    & = \min \{\lambda \geq 0:  -df_p(u) - \frac12 (\nabla^2 f)_p(u,u) \leq \frac{\lambda^2 \alpha}{2} \,\ \forall u \in T_p M\}.
    % & = \min \{\lambda \geq 0: \max_{t \in \RR} - df_p(t u) - \frac{t^2}2 (\nabla^2 f)_p(u,u) \leq \frac{\lambda^2 \alpha}{2} \,\ \forall u \in T_p M\}. \\
    % & = \min \{\lambda \geq 0: \frac{df_p(u)^2}{2 (\nabla^2 f)_p(u,u)} \leq \frac{\lambda^2 \alpha}{2} \,\ \forall u \in T_p M\}.
  \end{align}
  For $\alpha=1$, we abbreviate $\lambda_f := \lambda_{f,1}$ and $\norm{\cdot}_{f,p}^* := \norm{\cdot}_{f,p,1}^*$.
\end{defn}

The Newton decrement is invariant under rescaling~$f$ in the sense that~$\lambda_{f,\alpha} = \lambda_{cf,c\alpha}$ for any constant~$c>0$ (cf.~\cref{lem:basic}).
When $(\nabla^2f)_p$ is degenerate, the Newton decrement can still be defined as
$\lambda_{f,\alpha}(p)
% &:= \inf \braces[\Big]{ c \geq 0 : \abs{df_p(v)} \leq \sqrt{\alpha} \, c \, \sqrt{(\nabla^2 f)(v,v)} \ \forall v \in T_p M } \\
= \inf \braces{ c \geq 0 : \abs{df_p(v)} \leq \alpha c \norm{v}_{f,p,\alpha} \, \forall v \in T_p M }$,
which has the same interpretation as explained above; but we will mostly not need this.

Just like in the Euclidean case the Newton decrement provides a certificate for the existence of minimizers and the function gap.
This essentially follows from the Euclidean argument~\cite[Thm.~5.1.13]{nesterov2018lectures}.

\begin{prop}[Existence of minimizers]\label{prop:minigap}
  Let $f\colon D \to \R$ be $\alpha$-self-concordant along geodesics, defined on an open convex subset $D \subseteq M$.
  If $p\in D$ is such that $\lambda_{f,\alpha}(p) < 1$, then $f$ is bounded from below: we have
  \begin{align}\label{eq:minigap}
    f_* := \inf_{q \in D} f(q) \geq f(p) - \alpha \rho\mleft( \lambda_{f,\alpha}(p) \mright),
  \end{align}
  where $\rho$ is the quantity defined in \cref{eq:defn rho}.
  If in addition $f$ is strongly $\alpha$-self-concordant along geodesics and $(\nabla^2f)_p$ is positive definite, then the function attains its minimum at some~$p_* \in D$.
\end{prop}
\begin{proof}
We abbreviate $\lambda := \lambda_{f,\alpha}(p)$ and $r := \norm{u}_{f,p,\alpha}$.
For every $q = \Exp_p(u) \in D$, we have using \cref{cor:lower} and the definition of the Newton decrement the lower bound
\begin{align}\label{eq:lower bound via delta}
  f(q) - f(p)
\geq df_p(u) + \alpha \rho(-r)
\geq - \alpha r\lambda + \alpha \rho(-r)
% = - \alpha r\lambda + \alpha \parens[\big]{ r - \log(1+r) } \\
% = \alpha r(1 - \lambda) - \alpha\log(1-r) \\
= \alpha \delta(r),
\end{align}
where
\begin{align*}
  \delta(r) = r(1-\lambda) - \log(1+r).
\end{align*}
If $\lambda<1$, $\delta(r)$ is minimized at~$r=\lambda/(1-\lambda)$,
% \begin{align*}
%   \partial_r \parens*{ r(1 - \lambda) - \log(1+r) } = (1-\lambda) - \frac1{1+r} = 0 \\
%   (1-\lambda) = \frac1{1+r} \\
%   1+r = \frac1{1-\lambda}
%   r = \frac1{1-\lambda} - 1 = \frac \lambda{1-\lambda}
% \end{align*}
and we obtain
\begin{align*}
  f(q) - f(p) \geq \alpha \parens*{ \lambda + \log(1 - \lambda) } = - \alpha \rho(\lambda).
\end{align*}
This implies \cref{eq:minigap}.

On the other hand, $\delta(r)\to\infty$ as $r\to\infty$, so \cref{eq:lower bound via delta} shows that the level set $\{ q \in D : f(q) \leq f(p) \}$ is contained in a Dikin ellipsoid of some suitable radius.
If we assume that~$(\nabla^2f)_p$ is positive definite then Dikin ellipsoids are bounded.
Thus if~$f$ is also $\alpha$-strongly self-concordant along geodesics then \cref{lem:closed convex}~\ref{it:mini} shows that~$f$ attains its minimum at some~$p_* \in D$.
\end{proof}

The minimizer in \cref{prop:minigap} is unique assuming strict convexity, as follows, e.g., if $\nabla^2 f$ is positive definite throughout the domain.
The Newton decrement also certifies closeness to minimizers if they exist:

\begin{lem}\label{lem:local norm to minimizer bound}
  Let $f\colon D \to \R$ be $\alpha$-self-concordant along geodesics, defined on an open convex subset $D \subseteq M$, and let $p \in D$ be such that $\lambda_{f,\alpha}(p) < 1$.
  If $f$ attains a minimum at~$p_* = \Exp_p(u)$ for~$u \in T_p M$, then
  \begin{align*}
    \norm{u}_{f,p,\alpha} \leq \frac{\lambda_{f,\alpha}(p)}{1 - \lambda_{f,\alpha}(p)}.
  \end{align*}
\end{lem}
\begin{proof}
  Consider the geodesic $\gamma(t) = \Exp_p(tu)$ from~$p$ to~$p_*$.
  Then by \cref{cor:lower}, we have
  \begin{align*}
    \frac{\alpha r^2}{1 + r} \leq df_{p_*}(\tau_{\gamma, 1} u) - df_p(u) = -df_p(u) \leq \abs{df_p(u)} \leq \alpha r \lambda_{f,\alpha}(p),
  \end{align*}
  where $r := \norm{u}_{f,p,\alpha}$; the equality follows because~$df_{p_*} = 0$ because~$p_*$ is a minimizer of~$f$.
  Thus we have
  \begin{align*}
    \frac{r}{1 + r} \leq \lambda_{f,\alpha}(p)
  \end{align*}
  and for $\lambda_{f,\alpha}(p) < 1$ this implies the desired bound.
  % \begin{align*}
  %   \frac{r}{1 + r} \leq \lambda \\
  %   1 - \lambda \leq \frac 1 {1 + r} \\ % now use \lambda < 1
  %   1 + r \leq \frac 1{1 - \lambda_{f,\alpha}(p)} \\ % multiply with the original
  %   r = \frac r {1 + r} (1 + r) \leq \frac {\lambda_{f,\alpha}(p)} {1 - \lambda_{f,\alpha}(p)}
  % \end{align*}
\end{proof}

The following theorem is key to the analysis of Newton's method for self-concordant functions.
It bounds the Newton decrement after one Newton step quadratically in terms of the original Newton decrement.
This requires self-concordance in the sense of \cref{defn:sc}, rather than the weaker notion along geodesics, as its proof involves comparing the length of the new Newton step transported along the geodesics given by the previous Newton step, i.e., there are two natural directions.
The proof adapts the Euclidean argument in~\cite[Thm.~2.2.4]{renegar-ipm}.
\begin{thm}
  \label{thm:newton decrement after newton step}
  \newtonquadraticconvergencecontent
\end{thm}
\begin{proof}
We abbreviate the Newton step, iterate, and increment by $n_p := n_{f,p}$, $p_+ := p_{f,+}$, and~$\lambda := \lambda_{f,\alpha}(p)$, respectively.
\Cref{cor:dikin} along with the definitions shows that~$p_+ \in D$.
Then the entire geodesic segment~$\gamma(t) := \Exp_p(t n_p)$ for~$t\in[0,1]$ is contained in the domain~$D$.
We now prove the desired estimate, starting with \cref{thm:sc hessian ratio bounds}, which gives the upper bound
\begin{align}\label{eq:newton decrement start}
  \lambda_{f,\alpha}(p_+)
= \max_{w \in T_{p_+} M} \frac {\abs{df_{p_+}(w)}} {\alpha \norm{w}_{f,p_+,\alpha}}
= \max_{v \in T_p M} \frac {\abs{df_{p_+}(\tau_{\gamma,1}v)}} {\alpha \norm{\tau_{\gamma,1}v}_{f,p_+,\alpha}}
\leq \frac 1 {1 - \lambda} \max_{v \in T_p M} \frac {\abs{df_{p_+}(\tau_{\gamma,1}v)}} {\alpha \norm{v}_{f,p,\alpha}},
\end{align}
where $\tau_{\gamma,1}$ denotes parallel transport along the geodesic~$\gamma$ from~$p$ to~$p_+$.
Next, we observe that by the fundamental theorem of calculus, \cref{eq:deriv via transport}, and \cref{eq:defn newton step}, for all~$v\in T_pM$,
\begin{align}
\nonumber
  df_{p_+}(\tau_{\gamma,1}v)
&= df_{p_+}(\tau_{\gamma,1}v) - df_p(v) + df_p(v) \\
\nonumber
&= \int_0^1 \partial_t df_{\gamma(t)}(\tau_{\gamma,t}v) \, dt + df_p(v) \\
\nonumber
&= \int_0^1 (\nabla_{\dot\gamma(t)} df)_{\gamma(t)}(\tau_{\gamma,t}v) \, dt + df_p(v) \\
\nonumber
&= \int_0^1 (\nabla^2 f)_{\gamma(t)}(\tau_{\gamma,t}n_p, \tau_{\gamma,t}v) \, dt + df_p(v) \\
\nonumber
&= \int_0^1 \bracks{ (\nabla^2 f)_{\gamma(t)}(\tau_{\gamma,t}n_p, \tau_{\gamma,t}v) - (\nabla^2f)_p(n_p,v) }  \, dt \\
\label{eq:grad vs beta}
&= \beta(n_p, v),
\end{align}
where we have introduced the symmetric bilinear form
\begin{align*}
  \beta\colon T_pM \times T_pM \to \R, \quad \beta(u,v) = \int_0^1 \bracks*{ (\nabla^2 f)_{\gamma(t)}(\tau_{\gamma,t}u, \tau_{\gamma,t}v) - (\nabla^2f)_p(u,v) }  \, dt.
\end{align*}
By \cref{thm:sc hessian ratio bounds} and using $\norm{t n_p}_{f,p,\alpha} = t \lambda$, we have, for all $v\in T_pM$,
\begin{align*}
    \bracks*{ \parens*{1 - t \lambda}^2 - 1} (\nabla^2 f)_p(v,v)
\leq (\nabla^2 f)_{\gamma(t)}(\tau_{\gamma,t} v,\tau_{\gamma,t} v)  - (\nabla^2f)_p(v,v)
\leq \bracks*{ \frac 1 {\parens*{ 1 - t \lambda }^2} - 1 } (\nabla^2 f)_p(v,v).
\end{align*}
By integrating the lower and upper bounds from $t=0$ to $t=1$,
\begin{align*}
  -\parens*{\lambda - \frac{\lambda^2}3} \, (\nabla^2 f)_p(v,v)
% = \int_0^1 \bracks*{ \parens*{1 - t \lambda}^2 - 1} \, dt \, (\nabla^2 f)_p(v,v)
\leq \beta(v,v) \leq
% \int_0^1 \bracks*{ \frac 1 {\parens*{ 1 - t \lambda }^2} - 1 } \, \, dt (\nabla^2 f)_p(v,v) =
  \parens*{ \frac \lambda {1-\lambda} } \, (\nabla^2 f)_p(v,v).
\end{align*}
One may verify that $\max \braces{ \lambda - \lambda^2/3, \lambda/(1-\lambda) } = \lambda/(1-\lambda)$ as~$\lambda<1$.
Together with the Cauchy-Schwarz inequality, this implies that for all $u,v\in T_pM$,
\begin{align*}
  \abs{\beta(u,v)}
\leq \frac \lambda {1-\lambda} \sqrt{ (\nabla^2 f)_p(u,u) } \sqrt{ (\nabla^2 f)_p(v,v) }
= \frac {\alpha \lambda } {1-\lambda} \norm{u}_{f,p,\alpha} \norm{v}_{f,p,\alpha}.
\end{align*}
% (i.e., the symmetric form $\beta$ is bounded by $\alpha \lambda/(1-\lambda)$ with respect to the Hessian norm).
Together with \cref{eq:newton decrement start,eq:grad vs beta}, we obtain the upper bound
\begin{equation*}
   \lambda_{f,\alpha}(p_+)
% \leq \frac 1 {1 - \lambda} \max_{v \in T_p M} \frac {\abs{df_{p_+}(\tau_{\gamma,1}v)}} {\alpha \norm{v}_{f,p,\alpha}}
\leq \frac 1 {1 - \lambda} \max_{v \in T_p M} \frac {\abs{\beta(n_p,v)}} {\alpha \norm{v}_{f,p,\alpha}}
% \leq \frac {\alpha \lambda} {(1 - \lambda)^2} \max_{v \in T_p M} \frac {\norm{n_p}_{f,p,\alpha} \norm{v}_{f,p,\alpha}} {\norm{v}_{f,p,\alpha}}
\leq \frac {\lambda} {(1 - \lambda)^2} \norm{n_p}_{f,p,\alpha}
= \frac {\lambda^2} {(1 - \lambda)^2}. \qedhere
\end{equation*}
\end{proof}

\Cref{thm:newton decrement after newton step} implies that the Newton method converges quadratically for sufficiently small~$\lambda$.
For example, suppose that $\lambda \leq \lambda_* := 1 - \frac1{\sqrt2}$.
Then we have
\begin{align}\label{eq:iterate}
  \parens*{ \frac {\lambda} {1 - \lambda} }^2
\leq \parens*{ \frac {\lambda} {1 - \lambda_*} }^2
= 2 \lambda^2
\leq \lambda_*,
\end{align}
meaning the Newton decrement decreases quadratically and stays below~$\lambda_*$, so we can iterate.
This implies the following result (cf.~\cite[Thm.~2.2.3]{nesterov-nemirovskii-ipm}):

\begin{thm}[Quadratic convergence of the Newton method]\label{thm:quadratic}
Let~$f\colon D \to \R$ be a strongly $\alpha$-self-concordant function defined on an open convex set $D\subseteq M$, with positive definite Hessian.
Let $p_0\in D$ be a point such that~$\lambda_{f,\alpha}(p_0) \leq \lambda_* := 1 - 1/{\sqrt2} \approx 0.293$.
Then the Newton iterations
\[
  p_{t+1} := \Exp_{p_t}(n_{f,p_t})
\]
are well-defined for all~$t\in\N$ (i.e., each $p_t \in D$) and we have
\begin{align*}
  \lambda_{f,\alpha}(p_t) \leq \frac12 \parens*{ 2 \lambda_{f,\alpha}(p_0) }^{2^t} \leq \frac12 \parens*{ 2 \lambda_* }^{2^t}.
\end{align*}
In particular, $\bigO(\log\log\frac\alpha\eps)$ Newton iterations suffice to find a point~$p_t$ such that~$f(p_t) \leq f_* + \eps$, for $\eps < \alpha/e$.
\end{thm}
\begin{proof}
We abbreviate $\lambda_t := \lambda_{f,\alpha}(p_t)$.
By \cref{thm:newton decrement after newton step,eq:iterate}, one can see inductively that $p_t \in D$ is well-defined for all~$t\in\N$ and that we have~$\lambda_t \leq \lambda_*$ and
\begin{align*}
  2 \lambda_t \leq \parens*{ 2 \lambda_{t-1} }^2 \leq \ldots \leq \parens*{ 2 \lambda_0 }^{2^t} \leq (2 \lambda_*)^{2^t},
\end{align*}
as claimed.
This also implies the last statement, since to achieve $f(p_t) \leq f_* + \eps$ it suffices to have $\rho(\lambda_t) \leq \eps/\alpha$, by \cref{prop:minigap}, and we have~$\rho(\lambda_t) \leq \lambda_t^2$ for $\lambda_t \leq \lambda_*$.
\end{proof}

What if we have a starting point such that the Newton decrement does \emph{not} guarantee quadratic convergence?
In this case it is well-known that one can employ a \emph{damped} Newton method, with a step size that ensures that one stays inside the Dikin ellipsoid (and hence in the domain) at each step.
This works just the same in the Riemannian setting and only requires self-concordance along geodesics (cf.~\cite[Thm.~5.1.15]{nesterov2018lectures}):

\begin{thm}[Damped Newton method]\label{thm:damped}
Let~$f\colon D \to \R$ be strongly $\alpha$-self-concordant along geodesics, defined on an open convex set $D\subseteq M$, with positive definite Hessian.
Let $p_0\in D$ be an arbitrary starting point. Then the damped Newton iterations
\[
  p_{t+1} := \Exp_{p_t}\mleft( u_t \mright) \quad\text{where}\quad u_t := \frac 1 {1 + \lambda_{f,\alpha}(p_t)} n_{f,p_t}
\]
are well-defined for all~$t\in\N$ (i.e., each $p_t \in D$) and we have
\begin{align*}
  f(p_{t+1}) \leq f(p_t) - \alpha \rho(-\lambda_t),
\end{align*}
where $\rho$ is the quantity defined in \cref{eq:defn rho}.
In particular, if~$f$ is bounded from below and we set~$f_* := \inf_{p \in D} f(p)$, then $\bigO((f(p_0) - f_*)/\alpha)$ damped Newton iterations suffice to find a point~$p_t$ such that~$\lambda_{f,\alpha}(p_t) \leq \lambda_*$ (or any other constant).
\end{thm}
\begin{proof}
We abbreviate $\lambda_t := \lambda_{f,\alpha}(p_t)$.
Using \cref{cor:upper} one can see inductively that $r := \norm{u_t}_{f,p,\alpha} = \lambda_t / (1 + \lambda_t) < 1$ and $p_t \in D$ is well-defined for all~$t\in\N$.
Moreover,
\begin{align*}
  f(p_{t+1})
&\leq f(p_t) + df_{p_t}(u_t) + \alpha \rho(r) \\
&= f(p_t) - (\nabla^2f)_p(n_{f,p_t},u_t) + \alpha \rho(r) \\
% &= f(p_t) - \alpha \parens*{ 1 + \lambda_t } \frac {(\nabla^2f)_p(u_t,u_t)} \alpha + \alpha \rho(r) \\
&= f(p_t) - \alpha \parens*{ \frac {\lambda_t^2} {1 + \lambda_t} - \rho(r) } \\
% &= f(p_t) - \alpha \bracks*{ \frac {\lambda_t^2} {1 + \lambda_t} + \frac {\lambda_t} {1 + \lambda_t} + \log\mleft(1 - \frac {\lambda_t} {1 + \lambda_t} \mright) } \\
&= f(p_t) - \alpha \parens*{ \lambda_t - \log \mleft( 1 + \lambda_t \mright) } \\
&= f(p_t) - \alpha \rho(-\lambda_t).
\qedhere
\end{align*}
\end{proof}

In particular, \cref{thm:damped,cor:dikin} have the following structural consequence.

\begin{cor}\label{cor:bounded below iff minimum}
Let~$f\colon D \to \R$ be strongly $\alpha$-self-concordant along geodesics, defined on an open convex set $D\subseteq M$, with positive definite Hessian.
Then $f$ is bounded from below if and only if it attains its minimum (necessarily at a unique minimizer, by strict convexity).
\end{cor}

By combining \cref{thm:quadratic,thm:damped}, we see that we can approximately minimize any strongly $\alpha$-self-concordant function with positive definite Hessian by first using damped Newton steps from an arbitrary starting point~$p_0$ until we arrive at point with Newton decrement~$\leq \lambda_*$; then we are in the quadratic convergence regime and we can take ordinary Newton steps until we arrive at a point~$p_t$ with~$\rho(\lambda_{f,\alpha}(p_t)) \leq \eps/\alpha$, so that $p_t$ is an $\eps$-approximate minimizer.
This requires $\bigO((f(p_0) - f_*)/\alpha + \log \log (\alpha/\eps))$ Newton iterations.

%=============================================================================
\section{Barriers, compatibility, path-following method on manifolds}\label{sec:barriers compatibility path-following}
%=============================================================================
The methods developed in \cref{sec:sc} are sufficient to optimize strongly self-concordant functions.
However, it is difficult to guarantee that one starts in the quadratic convergence regime for Newton's method, and the damped Newton method has a worst-case complexity which depends on the gap in function value.
Moreover, most convex optimization problems do \emph{not} take the form of a minimization of a strongly self-concordant function over its natural domain.
Rather, one is given a convex objective~$f$ and a domain~$D$ and wants to minimize the former over the latter.

In this section, we show how to circumvent these two issues, assuming one has a \emph{self-concordant barrier} for the domain over which one optimizes.
To this end, we generalize the analysis of so-called \emph{path-following (interior point) methods}~\cite{nesterov-nemirovskii-ipm} from the Euclidean to the Riemannian setting.
We treat not only the case of geodesically linear objectives, but the more general class of objectives that are \emph{compatible} with the given self-concordant barrier.
This will be useful for the applications discussed in \cref{sec:applications}.
Throughout this section we assume that~$M$ is a connected and geodesically complete Riemannian manifold.

%-----------------------------------------------------------------------------
\subsection{Self-concordant barriers}
%-----------------------------------------------------------------------------
We first define the notion of a self-concordant barrier.
The estimates in this section only require the self-concordance to be along geodesics, and we make explicit whenever this is the case.
However, the path-following method presented in \cref{subsec:path-following} requires the stronger notion.

\begin{defn}[Barrier]\label{defn:sc barrier}
  Let $D\subseteq M$ be an open and convex subset, and let $\theta\geq0$.
  We say that a function~$F\colon D\to\R$ is a \emph{non-degenerate strongly self-concordant barrier with parameter~$\theta$}, or in short a~\emph{$\theta$-barrier}, if~$F$ is a strongly~$1$-self-concordant function with positive definite Hessian such that~$\lambda_F(p) \leq \sqrt{\theta}$ for all~$p \in D$, with $\lambda_F = \lambda_{F,1}$
  % = \max_{0\neq v \in T_pM} \frac {\abs{dF_p(v)}} {\sqrt{(\nabla^2 F)_p(v,v)}}
  the Newton decrement (\cref{defn:newton decrement}).
  We say that $F$ is a~\emph{$\theta$-barrier along geodesics} if it is only strongly 1-self-concordant along geodesics.
\end{defn}

The parameter of a barrier plays an important role in the complexity analysis of the path-following method that we discuss in \cref{subsec:path-following}.
The following lemma follows readily from the definition:

\begin{lem}
  \label{lem:sum of barriers is barrier}
  Let $F_1\colon D_1 \to \R$ be a $\theta_1$-barrier and let $F_2\colon D_2 \to \R$ be a $\theta_2$-barrier.
  Then $F_1 + F_2$ is a $(\theta_1 + \theta_2)$-barrier for~$D := D_1 \cap D_2$, assuming~$D$ is non-empty.
\end{lem}

Next, we prove an important inequality which involves the barrier parameter.
To state the result, we define a Riemannian version of the so-called Minkowski function(al) or gauge function.
It measures the inverse distance from a point to the boundary of the domain.

\begin{defn}[Minkowski functional]\label{defn:minkowski functional}
  Let~$D \subseteq M$ be an open convex subset.
  For~$p \in D$, we define the \emph{Minkowski functional} by
  \begin{equation*}
    \pi_{D,p} \colon T_p M \to \R_{\geq0}, \quad \pi_{D,p}(u) = \inf \braces*{ s \geq 0 : \Exp_p\mleft( \tfrac1s u \mright) \in D }.
  \end{equation*}
\end{defn}

This is well-defined since~$D$ is open and hence~$\pi_{D,p}(u) < \infty$ for every~$u \in T_pM$.
Note that if~$s := \pi_{D,p}(u) = 0$, then the entire infinite geodesic ray~$\gamma(t) = \Exp_p(tu)$ is contained in the domain, while if $s > 0$ then~$\Exp_p(\tfrac1s u)$ is a point in its boundary~$\partial D = \overline{D} \setminus D$.
Moreover, if~$u \in T_pM$ is such that $\Exp_p(u) \in \overline D$, then $\pi_p(u) \leq 1$.
% \begin{rem}
%   We recall that (total) convexity of~$D$ is defined as follows: if~$\gamma\colon \RR \to M$ is any geodesic and~$\gamma(t_1), \gamma(t_2) \in D$ for~$t_1 < t_2$, then~$\gamma(t) \in D$ for all~$t \in [t_1, t_2]$.
%   Since~$D$ is also open, the set of~$s \geq 0$ such that~$\Exp_p(s^{-1} u) \in D$ is an interval of the form~$(s_0, \infty)$, unless the infimum is zero.
% \end{rem}

Then we have the following result, which can be deduced directly from its Euclidean version~\cite[\S2.3.2]{nesterov-nemirovskii-ipm}.
We provide a self-contained proof for convenience.

\begin{prop}\label{prop:sc barrier gradient bound}
  Let $D \subseteq M$ be open and convex, and let~$F\colon D \to \RR$ be a $\theta$-barrier along geodesics.
  Then one has, for all~$p \in D$ and~$u \in T_p M$,
  \begin{align*}
    dF_p(u) &\leq \theta \, \pi_{D,p}(u).
  \end{align*}
  In particular, if $q = \Exp_p(u) \in \overline D$ then
  \begin{align*}
    dF_p(u) &\leq \theta.
  \end{align*}
\end{prop}
\begin{proof}
  The second statement follows from the first by the preceding discussion.
  To prove the first, let $p\in D$ and $u\in T_pM$.
  If $dF_p(u) \leq 0$ then there is nothing to prove, so we assume that~$dF_p(u) > 0$.
  Define
  \begin{align*}
    g(t) := F(\Exp_p(t u)).
  \end{align*}
  Then~$g$ is well-defined on the interval~$I = [0, \pi_{D,p}(u)^{-1})$, where we interpret~$0^{-1} = \infty$.
  % Note: because~$dF_p(u) > 0$, if the barrier were not assumed to have positive definite Hessian, we can also deduce that~$\ddot g(0) > 0$, as otherwise the Newton decrement~$\lambda_{F}(p)$ would not be finite, contradicting $\lambda_F(p) \leq \sqrt{\theta}$.
  By definition of the Newton decrement and recalling that~$\ddot g(t) > 0$ as~$F$ has positive definite Hessian, we have
  \begin{equation*}
    \frac{(\dot g(t))^2}{\ddot g(t)}
  % = \frac {dF_{p(t)}(u(t))^2} {(\nabla^2 F)_{p(t)}(u(t),u(t))}
  \leq \lambda^2_F(p) = \theta.
  \end{equation*}
  Since we assumed that~$\dot g(0) = dF_p(u) > 0$, we find that~$\theta>0$, as well as $\dot g(t) > 0$ for all $t \in I$, by convexity.
  Accordingly, we can write the above as
  \begin{align*}
    \partial_t \parens*{ \frac 1 {\dot g(t)} } = -\frac {\ddot g(t)} {(\dot g(t))^2} \leq -\frac 1 \theta,
  \end{align*}
  which implies that
  \begin{align*}
    \frac 1 {\dot g(t)}
  = \frac 1 {\dot g(0)} + \int_0^t \partial_t \parens*{ \frac 1 {\dot g(t)} }
  \leq \frac 1 {\dot g(0)} - \frac t \theta,
  \end{align*}
  and hence
  \begin{align*}
    \dot g(t) \geq
    \frac 1 {\frac 1 {\dot g(0)} - \frac t \theta} =
    \frac {\theta \dot g(0)} {\theta - t \dot g(0)}.
  \end{align*}
  As the right-hand side diverges as~$t$ approaches $\theta / \dot g(0)$, we must have $t < \theta / \dot g(0)$ for all~$t \in I$.
  Hence
  \begin{align*}
    \pi_{D,p}(u)^{-1} \leq \frac {\theta} {\dot g(0)},
  \end{align*}
  which is the desired bound.
\end{proof}

As a consequence, non-trivial barriers must have positive parameter:

\begin{cor}\label{cor:zero param barrier is constant}
  Let $D \subseteq M$ be open and convex, and let~$F\colon D \to \R$ be a $\theta$-barrier along geodesics with~$\theta=0$.
  Then~$F$ is constant and $D = M$.
\end{cor}
\begin{proof}
  \Cref{prop:sc barrier gradient bound} shows that $dF = 0$, hence~$F$ is locally constant and $\nabla^2 F = 0$.
  Because~$F$ is strongly self-concordant, we may apply \cref{cor:dikin} to conclude that~$\Exp_p(T_pM) \subseteq D$ and hence $D = M$, since~$M$ is connected and geodesically complete.
\end{proof}

The minimizer of a barrier, which if it exists is necessarily unique (recall that barriers have positive definite Hessians by definition), plays a special role in the theory.

\begin{defn}[Analytic center]
  Let $D \subseteq M$ be open and convex, and let~$F\colon D \to \R$ be a $\theta$-barrier along geodesics.
  If $F$ attains its minimum, then the unique minimizer is called the \emph{analytic center} of~$D$.
\end{defn}

Recall that a barrier attains its minimum if and only if it is bounded from below (\cref{cor:bounded below iff minimum}).
The following result shows that the domain is necessarily enclosed in a Dikin ellipsoid about the analytic center, with radius given by the barrier's parameter.
It adapts the Euclidean argument (cf.~\cite[Thm.~5.3.9]{nesterov2018lectures}, \cite[Prop.~2.3.2~(iii)]{nesterov-nemirovskii-ipm}) to the Riemannian setting.

\begin{prop}[Enclosing Dikin ellipsoid]
  Let $D \subseteq M$ be open and convex, and let~$F\colon D \to \R$ be a $\theta$-barrier along geodesics.
  If $\theta>0$ and $F$ is bounded from below, with analytic center~$p_* \in D$, then
  \begin{align*}
    D \subseteq B_{F,p_*}^\circ(2\theta + 1),
  \end{align*}
  where $B^\circ_{F,p_*} = B^\circ_{F,p_*,1}$ denotes the Dikin ellipsoid (\cref{defn:dikin}).
  That is, the domain is contained in the Dikin ellipsoid with radius~$2\theta+1$ about~$p_*$.
\end{prop}
\begin{proof}
  Let~$u \in T_{p_*} M$ be such that~$\norm{u}_{F,p_*} = 1$, and let~$\gamma(t) := \Exp_{p_*}(tu)$.
  By \cref{cor:dikin}, we know that $B_{F,p_*}^\circ(1) \subseteq D$, hence $g(t) := F(\gamma(t))$ is well-defined for $t\in[0,1)$.

  To show that~$D \subseteq B_{F,p_*}^\circ(2\theta + 1)$, by convexity of~$D$ it suffices to show that $\gamma(1+2\theta) \not\in D$.
  From \cref{eq:lower diff} in \cref{cor:lower} and~$p_*$ being a minimizer of~$F$, it follows that, for~$t\in[0,1)$,
  \begin{align*}
    dF_{\gamma(t)}(\tau_{\gamma, t} u) =
    dF_{\gamma(t)}(\tau_{\gamma, t} u) - dF_{p_*}(u) =
    \dot g(t) - \dot g(0) \geq
    % \frac{a t r^2}{1 + t r} =
    \frac{t}{1 + t}.
  \end{align*}
  \Cref{prop:sc barrier gradient bound} on the other hand implies that for
  \begin{align*}
    dF_{\gamma(t)}(\tau_{\gamma, t} u) \leq \theta \, \pi_{D,{\gamma(t)}}(\tau_{\gamma, t} u).
  \end{align*}
  Together, we obtain that, for every~$t \in [0, 1)$,
  \begin{equation*}
    \theta \, \pi_{D,{\gamma(t)}}(\tau_{\gamma, t} u) \geq \frac{t}{1 + t}.
  \end{equation*}
  Now, by the definition of the Minkowski functional, for every $s \in [0, \pi_{D,\gamma(t)}(\tau_{\gamma,t}(u)))$, we have
  \begin{align*}
    \gamma\mleft( t + \tfrac1s \mright)
  = \Exp_{p_*}\mleft( \parens*{ t + \tfrac1s } u \mright)
  = \Exp_{\gamma(t)}\mleft( \tfrac1s \tau_{\gamma,t} u \mright)
  \not\in D.
  \end{align*}
  Therefore, for every~$t \in [0,1)$ and $s \in [0, \frac{t}{\theta (1+t)})$, we have
  \begin{align*}
    \gamma\mleft( t + \tfrac1s \mright)
  \not\in D.
  \end{align*}
  Letting~$t \to 1$ and~$s \to 1/(2\theta)$ gives that~$\gamma(1 + 2 \theta) \not\in D$, since~$M \setminus D$ is closed.
\end{proof}

%-----------------------------------------------------------------------------
\subsection{Compatibility}\label{subsec:compatibility}
%-----------------------------------------------------------------------------
Given a barrier~$F$, for which convex functions~$f$ is it the case that $tf + F$ is self-concordant for all~$t\geq0$, with parameter independent of~$t$?
This is clearly the case if $f$ is (affine) linear or quadratic in the sense that the third covariant derivative~$\nabla^3 f$ vanishes.
We now define the more general notion of \emph{compatibility}, which suffices for this, as shown in \cref{prop:compatible functions self concordance} below.

\begin{defn}[Compatibility]\label{defn:compatibility}
  Let $D \subseteq M$ be open and convex, let~$f,F\colon D \to \R$ be convex functions.
  For~$\beta_1,\beta_2\geq0$, we say that~$f$ is \emph{$(\beta_1,\beta_2)$-compatible with~$F$} if for all~$p \in D$ and~$u, v \in T_p M$, one has
  \begin{equation}\label{eq:defn compatibility}
  \begin{aligned}
    \abs{(\nabla^3 f)_p(u, v, v)} & \leq 2 \beta_1 \sqrt{(\nabla^2 F)_p(u, u)} (\nabla^2 f)_p(v, v) \\
                                  & \, + 2 \beta_2 \sqrt{(\nabla^2 F)_p(v, v)} \sqrt{(\nabla^2 f)_p(u, u)} \sqrt{(\nabla^2 f)_p(v, v)}.
  \end{aligned}
  \end{equation}
  For $\beta\geq0$, we say that $f$ is \emph{$\beta$-compatible with~$F$ along geodesics} if for all~$p \in D$ and~$v \in T_p M$,
  \begin{align}\label{eq:defn compatibility along geodesics}
    \abs{(\nabla^3 f)_p(v, v, v)} & \leq 2 \beta \sqrt{(\nabla^2 F)_p(v, v)} (\nabla^2 f)_p(v, v).
  \end{align}
\end{defn}

Clearly, if~$f$ is a linear or a convex quadratic function, in the sense that its second or third covariant derivative vanishes, then it is clearly automatically compatible with any convex~$F$.
Moreover, any $\alpha$-self-concordant function is $(\beta_1,\beta_2)$-compatible with itself, for~$\beta_1+\beta_2 = 1/\sqrt\alpha$.
As we show in \cref{prop:compatible functions self concordance}, given a barrier~$F$ for a domain~$D$ and a convex objective function~$f$, compatibility guarantees that $tf + F$ is self-concordant for all $t\geq0$, with a parameter independent of~$t$, and hence one can use the path-following method presented in \cref{subsec:path-following} below to optimize~$f$ over~$D$.
We apply this theory in \cref{sec:applications}.

Compatibility along geodesics reduces to the well-known Euclidean notion, see~\cite[Def.~3.2.1]{nesterov-nemirovskii-ipm} or~\cite[Def.~5.4.2]{nesterov2018lectures}.
In these works it is also explained how to generalize the notion of compatibility to vector-valued functions~$f$, which is useful for constructing new barriers out of old ones; see~\cite[\S5.1.2]{nesterov-nemirovskii-ipm} or~\cite[\S5.4.6]{nesterov2018lectures} for details.
We do not provide such a generalization here.
Clearly, if $f$ is $(\beta_1,\beta_2)$-compatible with~$F$ then it is also $\beta$-compatible with~$F$ along geodesics for $\beta := \beta_1 + \beta_2$.
Yet the latter does not imply the former, even in the Euclidean setting.

We may equivalently write \cref{eq:defn compatibility,eq:defn compatibility along geodesics} as follows in terms of the seminorms~$\norm{\cdot}_{g,p} = \norm{\cdot}_{g,p,1}$ induced by the inner products~$\braket{\cdot,\cdot}_{g,p} = \braket{\cdot,\cdot}_{g,p,1}$ defined in \cref{eq:norm}:
\begin{align}\label{eq:defn compatibility via norms}
  \abs{(\nabla^3 f)_p(u, v, v)} \leq 2 \beta_1 \norm{u}_{F,p} \norm{v}_{f,p}^2 + 2 \beta_2 \norm{v}_{F,p} \norm{u}_{f,p} \norm{v}_{f,p}
\end{align}
and
\begin{align}\label{eq:defn compatibility along geodesics via norms}
  \abs{(\nabla^3 f)_p(v, v, v)} \leq 2 \beta \norm{v}_{F,p} \norm{v}_{f,p}^2.
\end{align}

We now state some basic properties of compatibility.
The following result holds analogously for compatibility along geodesics.

\begin{lem}
  \label{lem:compatibility conic}
Let $D \subseteq M$ be open and convex, $F\colon D \to \R$ a convex function, and~$\beta \in \RR_{\geq 0}^2$.
\begin{enumerate}
% \item\label{item:comp basic prop mono} If $\beta' \in \RR_{\geq 0}^2$ is such that $\beta_1' \geq \beta_1$ and $\beta_2' \geq \beta_2$, then $f$ is $\beta'$-compatible with $F$.
\item\label{item:comp basic prop scaling} Let $f\colon D\to\R$ be a convex function that is $\beta$-compatible with~$F$ and let $c\geq0$.
Then $cf$ is $\beta$-compatible with~$F$.
\item\label{item:comp basic prop addition}
Let $f_1,f_2\colon D\to\R$ be two convex functions that are each $\beta$-compatible with~$F$.
Then their sum~$f_1 + f_2$ is $\beta$-compatible with~$F$.
\end{enumerate}
\end{lem}
\begin{proof}
  % Property~\ref{item:comp basic prop mono} is clear from the definition.
  Property~\ref{item:comp basic prop scaling} is clear from the definition, as both sides of \cref{eq:defn compatibility} are positively homogeneous in~$f$.
  To prove property~\ref{item:comp basic prop addition}, we note that for every $p \in D$ and $u, v \in T_p M$,
  \begin{align*}
  &\quad\, \abs{(\nabla^3 (f_1+f_2))_p(u, v, v)}
  \leq \abs{(\nabla^3 f_1)_p(u, v, v)} + \abs{(\nabla^3 f_2)_p(u, v, v)} \\
  &\leq 2 \beta_1 \sqrt{(\nabla^2 F)_p(u, u)} (\nabla^2 f_1)_p(v, v) + 2 \beta_1 \sqrt{(\nabla^2 F)_p(u, u)} (\nabla^2 f_2)_p(v, v) \\
  &\, + 2 \beta_2 \sqrt{(\nabla^2 F)_p(v, v)} \parens*{ \sqrt{(\nabla^2 f_1)_p(u, u)} \sqrt{(\nabla^2 f_1)_p(v, v)} + \sqrt{(\nabla^2 f_2)_p(u, u)} \sqrt{(\nabla^2 f_2)_p(v, v)} } \\
  &\leq 2 \beta_1 \sqrt{(\nabla^2 F)_p(u, u)} (\nabla^2 f_1)_p(v, v) + 2 \beta_1 \sqrt{(\nabla^2 F)_p(u, u)} (\nabla^2 f_2)_p(v, v) \\
  &\, + 2 \beta_2 \sqrt{(\nabla^2 F)_p(v, v)} \sqrt{ (\nabla^2 f_1)_p(u, u) + (\nabla^2 f_2)_p(u, u) } \sqrt{ (\nabla^2 f_1)_p(v, v) + (\nabla^2 f_2)_p(v, v) } \\
  &= 2 \beta_1 \sqrt{(\nabla^2 F)_p(u, u)} (\nabla^2 (f_1+f_2))_p(v, v) \\
  &\, + 2 \beta_2 \sqrt{(\nabla^2 F)_p(v, v)} \sqrt{ (\nabla^2 (f_1 + f_2))_p(u, u) } \sqrt{ (\nabla^2 (f_1 + f_2))_p(v, v) }.
  \end{align*}
  The first inequality holds by compatibility of~$f_1$ and of~$f_2$ with~$F$, and the second inequality uses the Cauchy-Schwarz inequality.
\end{proof}

We now show if a convex function~$f$ is compatible with a self-concordant function~$F$ (e.g., a barrier), then $tf + F$ is self-concordant for every $t \geq 0$, with a self-concordance constant that is independent of~$t$.
We emphasize that it is not necessary for~$f$ itself to be self-concordant.
The proof is inspired by~\cite[Prop.~3.2.2]{nesterov-nemirovskii-ipm} in the Euclidean setting.
The result holds analogously if we use compatibility and self-concordance along geodesics in the hypothesis and conclusion.

\begin{prop}\label{prop:compatible functions self concordance}
Let $D \subseteq M$ be open and convex and let $f,F\colon D\to\R$ be convex functions.
Suppose that~$f$ is~$(\beta_1, \beta_2)$-compatible with~$F$ and $F$ is 1-self-concordant.
Then~$tf + F \colon D\to\R$ is $\alpha$-self-concordant for every~$t \geq 0$, with
\begin{equation*}
  \alpha
:= \begin{cases}
    \frac{4 \parens*{ \beta_2^2 - (\beta_1 - 1)^2 }}{\beta_2^2 (\beta_2^2 + 4 \beta_1)}
  % & \text{if } (\beta_1 - 1)^2 < \beta_2^2 \text{ and } \beta_2^2 > 2 \max(\beta_1(\beta_1 - 1), 1 - \beta_1) \\
  & \text{if } \beta_2^2 > 2 \max \braces*{ \beta_1(\beta_1 - 1), 1 - \beta_1 }, \\
    \frac{1}{\max\{\beta_1^2, 1\}}
  & \text{otherwise.}
  \end{cases}
\end{equation*}
If in addition~$F$ is strongly 1-self-concordant and $f$ has a closed convex extension, then $tf + F \colon D \to \R$ is strongly $\alpha$-self-concordant for every~$t\geq0$.
\end{prop}
\begin{proof}
We abbreviate $F_t := tf + F$.
Clearly,~$F_t$ is convex for every $t \geq 0$, so it remains to prove the self-concordance estimate.
For any $p \in D$ and $u, v \in T_p M$, using \cref{eq:defn sc via norms sym,eq:defn compatibility via norms},
\begin{align*}
  \abs{(\nabla^3 F_t)_p(u, v, v)}
&\leq t \abs{(\nabla^3 f)_p(u, v, v)} + \abs{(\nabla^3 F)_p(u, v, v)} \\
&\leq 2 t \beta_1 \norm{u}_{F,p} \norm{v}_{f,p}^2 + 2 t \beta_2 \norm{v}_{F,p} \norm{u}_{f,p} \norm{v}_{f,p}
    + 2 \norm{u}_{F,p} \norm{v}_{F,p}^2 \\
&= 2 \parens*{
  \sqrt t \norm{u}_{f,p} \parens*{ \sqrt t \beta_2 \norm{v}_{F,p} \norm{v}_{f,p} }
  + \norm{u}_{F,p} \parens*{ t \beta_1 \norm{v}_{f,p}^2 + \norm{v}_{F,p}^2 }
} \\
&\leq 2 \sqrt{ t \norm{u}_{f,p}^2 + \norm{u}_{F,p}^2 }
  \sqrt{
    \parens*{ \sqrt t \beta_2 \norm{v}_{F,p} \norm{v}_{f,p} }^2
  + \parens*{ t \beta_1 \norm{v}_{f,p}^2 + \norm{v}_{F,p}^2 }^2
  } \\
&= 2 \norm{u}_{F_t,p}
  \sqrt{
    t \beta_2^2 \norm{v}_{F,p}^2 \norm{v}_{f,p}^2
  + \parens*{ t \beta_1 \norm{v}_{f,p}^2 + \norm{v}_{F,p}^2 }^2
  }
\end{align*}
using the Cauchy-Schwarz inequality in the second-to-last step.
To show that $F_t$ is $\alpha$-self-concordant, by \cref{eq:defn sc sym} it therefore suffices to show that
(note we use $\norm{\cdot}_{g,p,1}$ rather than~$\norm{\cdot}_{g,p,\alpha}$!)
\begin{align}\label{eq:compat key sc inequality C}
  \sqrt{
    t \beta_2^2 \norm{v}_{F,p}^2 \norm{v}_{f,p}^2
  + \parens*{ t \beta_1 \norm{v}_{f,p}^2 + \norm{v}_{F,p}^2 }^2
  } \leq \frac1{\sqrt\alpha} \norm{v}_{F_t,p}^2.
\end{align}
Without loss of generality, we can assume that $\norm{v}_{F_t,p}^2 = 1$.
Writing $x := \norm{v}_{f,p}^2$ and $y := \norm{v}_{F,p}^2$, we see that \cref{eq:compat key sc inequality C} holds provided we can prove that
\begin{align}\label{eq:compat goal}
  \beta_2^2 t x y + \parens*{ \beta_1 t x + y }^2
\leq \frac1\alpha
\end{align}
for all~$x,y\geq0$ subject to the constraint~$tx + y = 1$.
Eliminating~$t$ and $x$ using this constraint, the left-hand side can be written as
\begin{align*}
  q(y)
&:= \beta_2^2 (1-y) y + \parens*{ \beta_1 (1-y) + y }^2 \\
% &= \beta_2^2 (1-y) y + \beta_1^2 (1-y)^2 + 2 \beta_1 (1-y)y + y^2 \\
% &= \parens*{ 1  - 2 \beta_1 + \beta_1^2 - \beta_2^2 } y^2 + \parens*{ 2 \beta_1 - 2 \beta_1^2 + \beta_2^2 } y + \beta_1^2 \\
&= \parens*{ (1 - \beta_1)^2 - \beta_2^2 } y^2 + \parens*{ 2 \beta_1 - 2 \beta_1^2 + \beta_2^2 } y + \beta_1^2,
\end{align*}
so we wish to show that $q(y) \leq 1/\alpha$ for all $y\in[0,1]$.
Note that $q(y)$ is a quadratic polynomial.
We distinguish two cases:

If $(1 - \beta_1)^2 < \beta_2^2$, then $q$ is strictly concave and attains its maximum on~$\R$ at
\begin{align*}
  y_* = \frac { 2 \beta_1 - 2 \beta_1^2 + \beta_2^2 } { 2 \parens*{ \beta_2^2 - (1 - \beta_1)^2 } }.
\end{align*}
Note that $y_* \in (0,1)$ if and only if
\begin{align*}
  0 < 2 \beta_1 - 2 \beta_1^2 + \beta_2^2 < 2 \parens*{ \beta_2^2 - (1 - \beta_1)^2 },
\end{align*}
which is equivalent to
\begin{align*}
  \beta_2^2 > 2 \max \braces*{ \beta_1(\beta_1 - 1), 1 - \beta_1 }.
\end{align*}
If $y_* \in (0,1)$, then the maximum of $q(y)$ on $[0,1]$ is given by
\begin{align*}
  q(y_*)
% = \parens*{ (1 - \beta_1)^2 - \beta_2^2 } \parens*{ \frac { 2 \beta_1 - 2 \beta_1^2 + \beta_2^2 } { 2 \parens*{ \beta_2^2 - (1 - \beta_1)^2 } } }^2 + \parens*{ 2 \beta_1 - 2 \beta_1^2 + \beta_2^2 } \frac { 2 \beta_1 - 2 \beta_1^2 + \beta_2^2 } { 2 \parens*{ \beta_2^2 - (1 - \beta_1)^2 } } + \beta_1^2 \\
= \frac { \parens*{ 2 \beta_1 - 2 \beta_1^2 + \beta_2^2 }^2 } { 4 \parens*{ \beta_2^2 - (1 - \beta_1)^2 } } + \beta_1^2
% = \frac { \parens*{ 2 \beta_1 - 2 \beta_1^2 + \beta_2^2 }^2 + 4  \beta_1^2 \parens*{ \beta_2^2 - (1 - \beta_1)^2 } } { 4 \parens*{ \beta_2^2 - (1 - \beta_1)^2 } } \\
= \frac { \beta_2^4 + 4 \beta_1 \beta_2^2 } { 4 \parens*{ \beta_2^2 - (1 - \beta_1)^2 } },
\end{align*}
while otherwise it is attained at the boundary, where $q(0) = \beta_1^2$ and $q(1) = 1$.

If $(1 - \beta_1)^2 \geq \beta_2^2$, then~$q(y)$ is convex and hence attains its maximum always at the boundary.
Summarizing both cases, we find that
\begin{align*}
  \max_{y\in[0,1]} q(y)
= \begin{cases}
    \frac { \beta_2^4 + 4 \beta_1 \beta_2^2 } { 4 \parens*{ \beta_2^2 - (1 - \beta_1)^2 } }
  & \text{if } (\beta_1 - 1)^2 < \beta_2^2 \text{ and } \beta_2^2 > 2 \max \braces*{ \beta_1(\beta_1 - 1), 1 - \beta_1 }, \\
    \max\{\beta_1^2, 1\}
  & \text{otherwise.}
  \end{cases}
\end{align*}
The condition of the first case is equivalent to
% \begin{align*}
%   \beta_2^2 > \max \braces*{2 \beta_1(\beta_1 - 1), 2 (1 - \beta_1), (\beta_1 - 1)^2 }.
% \end{align*}
% If $0 \leq \beta_1 \leq 1$, then the right-hand side is the same as
% \begin{align*}
%   \max \braces* { 2 (1 - \beta_1), (1 - \beta_1)^2 } = 2 (1 - \beta_1),
% \end{align*}
% while if $\beta_1 \geq 1$, the right-hand side is the same as
% \begin{align*}
%   \max \braces*{2 \beta_1(\beta_1 - 1), (\beta_1 - 1)^2 } = 2 \beta_1 (\beta_1 - 1)
% \end{align*}
% Thus the condition is equivalent to
\begin{align*}
  \beta_2^2 > 2 \max \braces*{\beta_1(\beta_1 - 1), 1 - \beta_1 },
\end{align*}
and hence we have confirmed \cref{eq:compat goal}.
Thus we have proved that~$F_t = tf + F$ is an $\alpha$-self-concordant function on~$D$.
Finally, the last claim follows from \cref{lem:sum of lsc is lsc}
\end{proof}

Finally, we construct a self-concordant barrier for the epigraph of any function compatible with a barrier for its domain.
This result generalizes the Euclidean result~\cite[Thm.~5.3.5]{nesterov2018lectures}, which constructs a self-concordant barrier for the open epigraph
\begin{align}\label{eq:open epi}
  E^\circ_f := \braces*{(p,t) \in D \times \R : f(p) < t }
\end{align}
of a self-concordant barrier.
As before, it holds analogously if we use the notions along geodesics in the hypothesis and conclusion.

\begin{thm}[Barriers for epigraphs]\label{thm:compatible function epigraph barrier construction}
Let $D \subseteq M$ be open and convex and let $f,F\colon D\to\R$ be convex functions.
Suppose that~$f$ is~$(\beta_1, \beta_2)$-compatible with~$F$ and $F$ is 1-self-concordant.
Then, the function
\begin{align*}
  G\colon E^\circ_f \to \R, \quad G(p,t) = -\log\mleft( t - f(p) \mright) + F(p)
\end{align*}
defined on the open epigraph~$E^\circ_f$, see \cref{eq:open epi}, is convex and $\alpha$-self-concordant, with
\begin{align}\label{eq:epi alpha}
  \alpha := \frac 1 {\max \braces*{ 1 + \beta_1^2, \beta_1 + \tfrac12 \beta_2^2, \tfrac 2 3 \beta_2^2 }}.
\end{align}
Furthermore, for every $(p,t) \in E^\circ_f$ one has
\begin{align}\label{eq:newton inc bound}
  \lambda_{G,\alpha}(p,t)^2 = \frac{\lambda_G(p,t)^2}{\alpha} \leq \frac{1 + \lambda_F(p)^2}{\alpha}.
\end{align}
% i.e., for $v \in T_{(p,t)}E^\circ_f$ one has $\abs{dG_{(p,t)}(v)} \leq \sqrt{ 1 + \lambda_F(p)^2 } \sqrt{(\nabla^2 G)_{(p,t)}(v,v)}$
If in addition~$F$ is strongly 1-self-concordant and $f$ has a closed convex extension, then $G$ is strongly $\alpha$-self-concordant.
In particular, if~$F$ is a $\theta$-barrier for~$D$ and $f$ has a closed convex extension, then $G/\alpha$ is a $(1+\theta)/\alpha$-barrier for~$E^\circ_f$.
\end{thm}
\begin{proof}
  We identify $v \in T_{(p,t)} E^\circ_f \cong T_pD \op \R$ and write $v = (v_p,v_t)$, with~$v_p \in T_pD$ and~$v_t \in \R$.
  Then the differential of~$G$ is given by
  \begin{align}\label{eq:G diff}
    dG_{(p,t)}(v)
  = - \frac1{t - f(p)} \parens*{ v_t - df_p(v_p) } + dF_p(v_p)
  \end{align}
  and the Hessian of~$G$ by
  \begin{align}\label{eq:G hess}
    (\nabla^2 G)_{(p,t)}(v,v)
  = \underbrace{ \frac1{\parens*{ t - f(p) }^2} \parens*{ v_t - df_p(v_p) }^2 }_{=: A_v^2 }
  + \underbrace{ \frac1{t - f(p)} (\nabla^2f)_p(v_p,v_p) }_{=: B_v^2}
  + \underbrace{ (\nabla^2 F)_p(v_p, v_p) }_{=: C_v^2}.
  \end{align}
  The underbraced terms are all non-negative as~$t > f(p)$ and both~$f$ and~$F$ are convex, hence we can write them as squares of real numbers~$A_v,B_v,C_v$.
  This also shows that~$G$ is convex.
  % Alternatively, because~$(p, t) \mapsto t - f(p)$ is concave, $(p, t) \mapsto -\log(t - f(p))$ is convex, and so $G$ is convex.
  We now prove that~$G$ is self-concordant.
  The third covariant derivative can be computed as follows: for all~$u, v\in T_{(p,t)} E^\circ_f$, we have
  \begin{align*}
    (\nabla^3 G)_{(p,t)}(u, v, v)
  &= - \frac2{\parens*{ t - f(p) }^3} \parens*{ u_t - df_p(u_p) } \parens*{ v_t - df_p(v_p) }^2 \\
  &\quad - \frac2{\parens*{ t - f(p) }^2} \parens*{ v_t - df_p(v_p) } (\nabla^2 f)_p(u_p, v_p) \\
  &\quad - \frac1{\parens*{ t - f(p)}^2 } \parens*{ u_t - df_p(u_p) } (\nabla^2f)_p(v_p, v_p) \\
  &\quad + \frac1{t - f(p)} (\nabla^3f)_p(u_p, v_p, v_p)
  + (\nabla^3 F)_p(u_p, v_p, v_p) \\
  &= - 2 A_u A_v^2
  - 2 A_v \frac1{t - f(p)} (\nabla^2 f)_p(u_p, v_p)
  - A_u B_v^2 \\
  &\quad + \frac1{t - f(p)} (\nabla^3f)_p(u_p, v_p, v_p)
  + (\nabla^3 F)_p(u_p, v_p, v_p).
  \end{align*}
  Now, we have
  \begin{align*}
    \frac1{t - f(p)} (\nabla^2 f)_p(u_p, v_p) \leq B_u B_v
  \end{align*}
  by the Cauchy-Schwarz inequality,
  \begin{align*}
    \frac1{t - f(p)} \abs*{ (\nabla^3f)_p(u_p, v_p, v_p) }
  &\leq \frac2{t - f(p)} \parens*{ \beta_1 \norm{u}_{F,p} \norm{v}_{f,p}^2 + \beta_2 \norm{v}_{F,p} \norm{u}_{f,p} \norm{v}_{f,p} } \\
  &= 2 \parens*{ \beta_1 B_v^2 C_u + \beta_2 B_u B_v C_v}
  \end{align*}
  by compatibility of~$f$ with~$F$ as in \cref{eq:defn compatibility via norms}, and finally
  \begin{align*}
    \abs*{ (\nabla^3F)_p(u_p, v_p, v_p) }
  \leq 2 C_u C_v^2
  \end{align*}
  by 1-self-concordance of~$F$ (\cref{eq:defn sc sym}).
  Combining these estimates, we can upper bound the third covariant derivative of~$G$ in absolute value as
  \begin{align*}
    \abs*{ (\nabla^3G)_{(p,t)}(u, v, v) }
  &\leq 2 A_u A_v^2 + 2 A_v B_u B_v + A_u B_v^2 + 2 \parens*{ \beta_1 B_v^2 C_u + \beta_2 B_u B_v C_v } + 2 C_u C_v^2 \\
  &= A_u (2 A_v^2 + B_v^2) + B_u (2 A_v B_v + 2 \beta_2 B_v C_v) + C_u (2 \beta_1 B_v^2 + 2 C_v^2) \\
  &\leq \sqrt{A_u^2 + B_u^2 + C_u^2} \sqrt{(2 A_v^2 + B_v^2)^2 + (2 A_v B_v + 2 \beta_2 B_v C_v)^2 + (2 \beta_1 B_v^2 + 2 C_v^2)^2} \\
  % &= \sqrt{(\nabla^2G)_{(p,t)}(u, u)} \sqrt{(2 A_v^2 + B_v^2)^2 + (2 A_v B_v + 2 \beta_2 B_v C_v)^2 + (2 \beta_1 B_v^2 + 2 C_v^2)^2} \\
  &\leq 2 \sqrt{(\nabla^2G)_{(p,t)}(u, u)} \sqrt{ \max \braces*{ 1 + \beta_1^2, \beta_1 + \tfrac12 \beta_2^2, \tfrac 2 3 \beta_2^2 } } (\nabla^2G)_{(p,t)}(v, v) \\
  &= \frac 2 {\sqrt\alpha} \sqrt{(\nabla^2G)_{(p,t)}(u, u)} \ (\nabla^2G)_{(p,t)}(v, v),
  \end{align*}
  where the last inequality holds due to $2 xy \leq x^2 + y^2$, as in
  \begin{align*}
  &\quad \frac14 \bracks*{ (2 A_v^2 + B_v^2)^2 + (2 A_v B_v + 2 \beta_2 B_v C_v)^2 + (2 \beta_1 B_v^2 + 2 C_v^2)^2 } \\
  % &= 4 A_v^4 + 4 A_v^2 B_v^2 + B_v^4 + 4 A_v^2 B_v^2 + 8 \beta_2 A_v B_v^2 C_v + 4 \beta_2^2 B_v^2 C_v^2 \\
  % &= 4 A_v^4 + 4 \parens*{ \tfrac14 + \beta_1^2 } B_v^4 + 4 C_v^4 + 8 A_v^2 B_v^2 + 8 \parens*{ \beta_1 + \tfrac12 \beta_2^2 } B_v^2 C_v^2 + 8 \beta_2 A_v B_v^2 C_v \\
  &= A_v^4 + \parens*{ \tfrac14 + \beta_1^2 } B_v^4 + C_v^4 + 2 A_v^2 B_v^2 + 2 \parens*{ \beta_1 + \tfrac12 \beta_2^2 } B_v^2 C_v^2 + 2 \beta_2 A_v B_v^2 C_v \\
  &= A_v^4 + \parens*{ \tfrac14 + \beta_1^2 } B_v^4 + C_v^4 + 2 A_v^2 B_v^2 + 2 \parens*{ \beta_1 + \tfrac12 \beta_2^2 } B_v^2 C_v^2 + 2 \parens*{ \tfrac {\sqrt3} 2 B_v^2 } \parens*{ \tfrac 2 {\sqrt3} \beta_2 A_v C_v } \\
  &\leq
  A_v^4 + \parens*{ \tfrac14 + \beta_1^2 } B_v^4 + C_v^4 + 2 A_v^2 B_v^2 + 2 \parens*{ \beta_1 + \tfrac12 \beta_2^2 } B_v^2 C_v^2 + \tfrac34 B_v^4 + \tfrac 4 3 \beta_2^2 A_v^2 C_v^2 \\ &=
    A_v^4 + \parens*{ 1 + \beta_1^2 } B_v^4 + C_v^4 + 2 A_v^2 B_v^2 + 2 \parens*{ \beta_1 + \tfrac12 \beta_2^2 } B_v^2 C_v^2 + 2 \tfrac 2 3 \beta_2^2 A_v^2 C_v^2 \\
  &\leq
  % \max \braces*{ 1, 1 + \beta_1^2, \beta_1 + \tfrac12 \beta_2^2, \tfrac 2 3 \beta_2^2 } \parens*{ A_v^4 + B_v^4 + C_v^4 + 2 A_v^2 B_v^2 + 2 B_v^2 C_v^2 + 2 A_v^2 C_v^2 } \\ &=
    \max \braces*{ 1 + \beta_1^2, \beta_1 + \tfrac12 \beta_2^2, \tfrac 2 3 \beta_2^2 } \parens*{ A_v^2 + B_v^2 + C_v^2 }^2.
  \end{align*}
  We conclude that $G$ is indeed $\alpha$-self-concordant with~$\alpha$ as in \cref{eq:epi alpha}.

  Next, we prove the bound on the differential.
  Using \cref{eq:G diff} and with $A_v,B_v$ as in \cref{eq:G hess}, we have
  \begin{align*}
    \abs{dG_{(p,t)}(v)}
  &\leq A_v + \abs{dF_p(v_p)}
  \leq A_v + \lambda_F(p) C_v \\
  &\leq \sqrt{1 + \lambda_F(p)^2} \sqrt{A_v^2 + C_v^2}
  \leq \sqrt{1 + \lambda_F(p)^2} \sqrt{(\nabla^2 G)_{(p,t)}(v,v)},
  % = \sqrt{\frac{1 + \lambda_F(p)}{\alpha}} \norm{v}_{G,(p,t),\alpha},
  \end{align*}
  by definition of the Newton decrement and the Cauchy-Schwarz inequality.
  Thus we find that
  \begin{equation*}
    \lambda_{G,\alpha}(p,t) \leq \sqrt{\frac {1 + \lambda_F(p)^2} \alpha}.
  \end{equation*}
  which establishes \cref{eq:newton inc bound}.

  Finally, if $F$ is strongly 1-self-concordant, hence closed convex on~$D$, and if $f$ has a closed convex extension then it is easy to see that $G$ is closed convex on~$E_f^\circ$, using that~$(s,t) \mapsto -\log(t-s)$ is closed convex on~$\{ (s,t) \in \R^2 : s < t \}$.
  % More detailed proof:
  % Note that the function $(s,t) \mapsto -\log(t-s)$ is closed convex on $\{ (s,t) \in \R^2 : s < t \}$.
  % It follows readily that if $f$ has a closed convex extension to some~$D' \supseteq D$ then~$(p,t) \mapsto -\log(t-f(p))$ is closed convex on some superset of~$E_f^\circ$.
  % If $F$ is strongly 1-self-concordant, hence closed convex on~$D$, then it follows that~$G$ is closed convex on~$E_F^\circ$.
\end{proof}

In particular, we can apply this construction to any self-concordant function:

\begin{cor}\label{cor:barrier for epigraph}
Let $D \subseteq M$ be open and convex and let $f\colon D\to\R$ be 1-self-concordant.
Then $g(p,t) = -\log\mleft( t - f(p) \mright) + f(p)$ is a convex and 1-self-concordant function on the open epigraph~$E^\circ_f$ of~$f$, see \cref{eq:open epi}.
It satisfies $\lambda_{g}(p,t) \leq \sqrt{1 + \lambda_f(p)^2}$ for all $(p,t)\in E^\circ_f$.
If~$f$ is strongly self-concordant, so is~$g$.
In particular, if~$f$ is a $\theta$-barrier, $g$ is a $(1+\theta)$-barrier for~$E^\circ_f$.
\end{cor}
To end this section, we provide a variant of the above barrier for level sets of a convex function which does not use the notion of compatibility, but has a parameter that depends on the variation of the function.
For a convex function $f\colon M \to \RR$ and $\lvl \in \RR$ for which there is $p \in M$ with $f(p) < \lvl$, the open level set $\openlvl_{f, \lvl} \subseteq M$ is defined by
\begin{equation}
	\openlvl_{f,\lvl} = \{ p \in M \mid f(p) < \lvl \}.
\end{equation}
Define the logarithmic barrier $F_{\lvl}\colon \openlvl_{f,\lvl} \to \RR$ by
\begin{equation}
  \label{eq:Flvl definition}
	F_{\lvl}(p) = - \log (\lvl - f(p)) \quad (p \in \openlvl_{f,\lvl}).
\end{equation}
The logarithmic barrier is convex and has bounded Newton decrements as follows.
\begin{lem}
  The function $F= F_{\lvl}$ defined in~\cref{eq:Flvl definition} is smooth, closed convex, and satisfies
	\begin{equation}\label{eqn:SCB1}
		dF_p(u)^2 \leq (\nabla^2 F)_p (u,u) \quad (u \in T_p M, \, p \in \openlvl_{f,\lvl}).
	\end{equation}
\end{lem}
\begin{proof}
  Let $\omega(p) := \lvl - f(p) > 0$. Then we have
  \begin{equation}\label{eqn:DF_HF}
    dF_p(u) = \frac{df_p(u)}{\omega(p)},\quad 	(\nabla^2 F)_p(u,u) = \frac{(\nabla^2 f)_p(u,u)}{\omega(p)} + \frac{df_p(u)^2}{\omega(p)^2}.
  \end{equation}
  Then by convexity of~$f$, $(\nabla^2 F)_p(u,u) \geq 0$ and hence $F$ is convex, and satisfies $(\nabla^2 F)_p(u,u) \geq dF_p(u)^2$.

  The closedness of~$F$ is seen as follows: Consider a sequence~$(p_k,z_k)$ in the epigraph of~$F$, that converges to $(p_\infty,z_\infty) \in M \times \RR$.
  Note that~$f$ is smooth on $M$, and hence so is~$F$ on $\openlvl_{f,\lvl}$.
  By continuity of~$f$,~$\openlvl_{f,\lvl}$ is open, hence disjoint from its boundary in~$M$.
  Therefore any boundary point~$q$ of~$\openlvl_{f,\lvl}$ satisfies~$f(q) \geq \lvl$.
  Therefore, it is impossible for~$p_\infty$ to belong to the boundary of~$\openlvl_{f,\lvl}$: that would imply~$f(p_\infty) \geq \lvl$, which would imply~$z_\infty \geq \infty$.
  Hence~$p_\infty \in \openlvl_{f,\lvl}$, and $F(z_\infty) = \lim_{k\to \infty} F(p_k) \leq \lim_{k \to \infty} z_k = z_{\infty}$.
\end{proof}
If an $\alpha$-self-concordant function
$F$ satisfies \cref{eqn:SCB1}, then $F/\alpha$ is an $\alpha$-barrier.
The following is an extension of \cite[Thm.~5.1.4]{nesterov2018lectures} to our setting:
\begin{thm}[Barriers for level sets]\label{thm:F_lvl}
	Suppose that $f\colon M \to \RR$ is $\alpha$-self-concordant.
	Then $F_{\lvl}\colon \openlvl_{f,\lvl} \to \RR$ is $\alpha'$-self-concordant for
  \begin{equation}
    \label{eq:F_lvl alpha'}
    \alpha' = \frac{4 (\eta - f^*)/\alpha + 1}{(2 (\eta - f^*)/\alpha + 1)^2}
  \end{equation}
	where $f^* := \inf_{x \in M} f (x)$.
  In particular, $F_{\lvl}/\alpha'$ is an $O((\lvl -f^*)/\alpha)$-barrier for $\openlvl_{f,\lvl}$.
\end{thm}
When only considering self-concordance along geodesics, the constant $\alpha'$ can be taken as $\alpha/((\lvl-f^*)+\alpha)$, which is exactly what is proven in \cite[Thm.~5.1.4]{nesterov2018lectures}.
For self-concordance, however, a little modification is required, which leads to a weaker constant.
\begin{proof}
Our starting point is~\cref{eqn:DF_HF}, where we recall that~$\omega(p) = \eta - f(p)$.
Since $df_p(u)^2 = (df_p \otimes df_p) (u,u)$, and $(\nabla_v  (df \otimes df))_p(u,u) = ((\nabla_v df)_p \otimes df_p + df_p \otimes (\nabla_v df)_p)(u,u) = 2 df_p(u) (\nabla^2 f)_p(u,v)$,
%(see \cite[II. Proposition 1.3 (1)]{Sakai96}), 
the covariant derivative of $\nabla^2 F$ is given by (suppressing~$p$'s for convenience)
\begin{equation}\label{eqn:nablaHF}
  \nabla^3 F(v, u,u) = \frac{\nabla^3f(v, u,u)}{\omega} + \frac{df(v)\nabla^2f(u,u)}{\omega^2} + \frac{2df(u) \nabla^2 f(u,v) }{\omega^2} + \frac{2df(v)df(u)^2}{\omega^3}.
\end{equation}
Hence we have
\begin{align*}
  \abs{\nabla^3 F(v, u,u)} & \leq \frac{2 \sqrt{\nabla^2f(v,v)} \nabla^2f(u,u)}{\sqrt{\alpha}\omega} + \frac{\abs{df(v)} \nabla^2 f(u,u) }{\omega^2} \\ 
                          & + \frac{2\abs{df(u)} \sqrt{\nabla^2f(v,v)}\sqrt{ \nabla^2f(u,u)} }{\omega^2} +  \frac{2\abs{df(v)} df(u)^2 }{\omega^3}.
\end{align*}
Define $\tau_1,\tau, \xi_1, \xi$ by
\begin{equation*}
	\tau_1 := \sqrt{\nabla^2f(v,v)/\omega},\ \tau := \sqrt{\nabla^2 f(u,u)/\omega},\ \xi_1 := |df(v)|/\omega,\ \xi := |df(u)|/\omega.
\end{equation*}
Then we have
\begin{equation}\label{eqn:upperbound}
  \frac{\abs{\nabla^3F(v, u,u)}}{2 \sqrt{\nabla^2 F(v,v)} \nabla^2 F(u,u)}
  \leq \frac{(1/\sqrt{\alpha}) \omega^{1/2} \tau_1 \tau^2 + (1/2)\xi_1 \tau^2 + \xi \tau_1 \tau + \xi_1 \xi^2}{(\tau_1^2 + \xi_1^2)^{1/2}(\tau^2 + \xi^2)}.
\end{equation}
We bound the right-hand side as follows.
By homogeneity, we may consider the optimization problem:
\[
  \text{maximize } (1/\sqrt{\alpha}) \omega^{1/2} \tau_1 \tau^2 + (1/2) \xi_1 \tau^2+ \xi \tau_1 \tau + \xi_1 \xi^2	\, \text{ s.t. } \, \tau_1^2 + \xi_1^2 = 1,\ \tau^2 + \xi^2 = 1.
\]
For fixed $(\tau, \xi)$, optimizing with respect to $(\tau_1,\xi_1)$ is a linear optimization over the unit circle.
The optimum is $((1/\sqrt{\alpha}) \omega^{1/2} \tau^2 + \xi \tau, (1/2)\tau^2+ \xi^2)/\sqrt{((1/\sqrt{\alpha}) \omega^{1/2} \tau^2 + \xi \tau)^2+ ((1/2)\tau^2 + \xi^2)^2}$.
Then the problem reduces to
\[
	\text{maximize } \sqrt{((1/\sqrt{\alpha}) \omega^{1/2} \tau^2 + \xi \tau)^2 + ((1/2)\tau^2+\xi^2)^2}
	\, \text{ s.t. } \, \tau^2 + \xi^2 = 1.
\]
This optimization problem can be solved using the method of Lagrange multipliers.
For convenience set~$c = \sqrt{\omega/\alpha}$, and define~$q(\tau, \xi) = (\sqrt{\omega/\alpha} \tau^2 + \xi \tau)^2 + (\tau^2 / 2 + \xi^2)^2$.
The system of equations
\begin{equation*}
  \partial_\tau q(\tau, \xi) = \mu \tau, \quad \partial_\xi q(\tau, \xi) = \mu \xi, \quad \tau^2 + \xi^2 = 1, \quad \mu \in \RR
\end{equation*}
has six solutions~$(\tau, \xi, \mu)$, given by
\begin{align*}
  & (0, \pm 1, 4), \, \frac{1}{\sqrt{4 c^2 + 1}} (2 c, 1, 16c^4 + 16 c^2 + 4), \,  \frac{1}{\sqrt{4 c^2 + 1}} (-2 c, -1, 16c^4 + 16 c^2 + 4), \\
  & \, \frac{1}{\sqrt{4 c^2 + 9}} (3, -2c, 16c^2 + 9), \, \frac{1}{\sqrt{4 c^2 + 9}} (-3, 2c, 16c^2 + 9)
\end{align*}
and the largest value attained of~$q(\tau, \xi)$ attained at any of these points is~$(2 c^2 + 1)^2/(4 c^2 + 1)$.
Therefore, the right-hand side of \cref{eqn:upperbound} is at most
\begin{equation*}
  \sqrt{\frac{(2 (\omega/\alpha) + 1)^2}{4 (\omega/\alpha) + 1}}.
\end{equation*} 
In other words, this gives that~$\alpha' = (4(\omega/\alpha) + 1)/(2 (\omega/\alpha) + 1)^2$ is a suitable self-concordance constant at~$p$. 
Taking the maximum over~$p \in \openlvl_{f,\lvl}$ yields the choice of~$\alpha'$ in~\cref{eq:F_lvl alpha'}.
\end{proof}

%-----------------------------------------------------------------------------
\subsection{Path-following method}\label{subsec:path-following}
%-----------------------------------------------------------------------------
We now discuss a path-following method for objectives which are compatible with a barrier.
To this end, we consider the approach of~\cite[Ch.~3]{nesterov-nemirovskii-ipm}.
Their Euclidean framework is rather general, and deals with \emph{self-concordant families}.
We specialize to self-concordant families generated by a barrier, and generalize the corresponding path-following method to the Riemannian setting.
The goal is to minimize a convex objective function~$f$ over an open convex domain~$D$, that is, to find~$p \in D$ such that~$f(p) \approx \inf_{q \in D} f(q)$.
The running assumption we shall make is that we have a barrier~$F$ for the domain~$D$ such that the function
\begin{align*}
  F_t := tf + F \colon D \to \R
\end{align*}
is $\alpha$-self-concordant for all~$t\geq0$, with a parameter~$\alpha$ that is independent of~$t$.
One way to guarantee this is to assume that~$f$ is compatible with~$F$, as shown before in \cref{prop:compatible functions self concordance}.

The basic idea of the path-following method is as follows.
The algorithm keeps track of two pieces of data, a point~$p$ in the domain~$D$ and a time parameter~$t$.
The initial data to the algorithm is specified by a point~$p_{-1} \in D$ such that~$\lambda_{F,\alpha}(p_{-1})$ is small.
We then choose a time parameter~$t_0 > 0$ such that we are in the quadratic convergence regime for Newton's method for~$F_{t_0}$ as determined by \cref{thm:quadratic}, say~$\lambda_{F_{t_0}, \alpha}(p_{-1}) < \lambda_* = 1 - 1/\sqrt{2}$.
Such initial data can be obtained for instance by using the damped Newton method of \cref{thm:damped}, or in the Euclidean setting by a similar (reverse) path-following method.
We then iterate the following procedure for $k=0,1,2,\dots$:
\begin{enumerate}
  \item Update~$p_{k-1}$ to~$p_{k} \in D$ by taking one Newton step with respect to~$F_{t_{k}}$, so that~$\lambda_{F_{t_{k}}, \alpha}(p_{k+1})$ becomes smaller.
  \item Increase~$t_k$ to some $t_{k+1}$ by a constant factor such that one still has~$\lambda_{F_{t_{k+1}}, \alpha}(p_k) < \lambda_*$.
\end{enumerate}
Throughout the algorithm, $p_k$ will be an approximate minimizer of~$F_{t_k}$.
One can also show that if~$t_k$ is large enough, approximate minimizers of~$F_{t_k}$ are approximate minimizers of~$f$.

We first determine by what factor one can increase~$t$ while keeping the Newton decrement below some threshold.
The following result is a translation of~\cite[Thm.~3.1.1]{nesterov-nemirovskii-ipm} to our setting.
Note that here, we do not assume that~$tf + F$ is self-concordant.
% In fact,~$F$ need not even be self-concordant; one just needs the uniform bound on the Newton decrement.

\begin{lem}\label{prop:sc family time switch}
  Let $D \subseteq M$ be open and convex, let~$F\colon D \to \R$ be a $\theta$-barrier along geodesics, and let~$f\colon D\to\R$ be a convex function.
  Furthermore, let $t,t',\alpha,c>0$ and $p\in D$ be such that
  \begin{align*}
    \parens*{ 1 + \frac {\sqrt\theta} {c \sqrt\alpha} } \abs*{ \log \frac {t'} t }
  \leq 1 - \frac {\lambda_{F_t,\alpha}(p)} c.
  \end{align*}
  Then $\lambda_{F_t,\alpha}(p) \leq c$ implies that $\lambda_{F_{t'},\alpha}(p) \leq c$.
\end{lem}
\begin{proof}
  Let~$p \in D$.
  Throughout the proof, all derivatives of functions defined on~$M$ will be taken at the point~$p$, hence we shall omit the subscript.
  We will assume that~$t' \geq t$, but the proof for~$t' \leq t$ is analogous.
  For every~$0\neq u \in T_p M$, define a function~$\phi_u\colon [t, t'] \to \RR$ by
  \begin{equation*}
    \phi_u(s) = \frac{dF_s(u)}{\sqrt{\nabla^2 F_s(u,u)}}.
  \end{equation*}
  To prove the lemma, it suffices to show that $\abs{\phi_u(t')} \leq c \sqrt\alpha$ for all~$u\neq0$.
  Since~$\phi_{-u} = -\phi_u$, we may assume without loss of generality that~$\phi_u(t') \geq 0$.
  We first compute the derivative of~$\phi_u$:
  \begin{align*}
    \partial_s \phi_u(s) & = \frac{df(u)}{\sqrt{\nabla^2 F_s(u,u)}} - \frac{1}{2} \frac{dF_s(u) \cdot \nabla^2 f(u,u)}{(\nabla^2 F_s(u,u))^{3/2}} \\
                         & = \frac{1}{s} \phi_u(s) - \frac{1}{s} \frac{dF(u)}{\sqrt{\nabla^2 F_s(u,u)}} - \frac{1}{2} \frac{dF_s(u) \cdot \nabla^2 f(u,u)}{(\nabla^2 F_s(u,u))^{3/2}} \\
                         % & = \frac{1}{s} \phi_u(s) - \frac{1}{2} \frac{dF_s(u) \cdot (\frac{1}{s} \nabla^2 F_s(u,u) - \frac{1}{s} \nabla^2 F(u,u))}{(\nabla^2 F_s(u,u))^{3/2}} - \frac{1}{s} \frac{dF(u)}{\sqrt{\nabla^2 F_s(u,u)}} \\
                         & = \frac{1}{2s} \phi_u(s) - \frac{1}{s} \frac{dF(u)}{\sqrt{\nabla^2 F_s(u,u)}} + \frac{1}{2s} \frac{dF_s(u) \cdot \nabla^2 F(u,u)}{(\nabla^2 F_s(u,u))^{3/2}} \\
                         & = \frac{1}{2s} \phi_u(s) \left( 1 + \frac{\nabla^2 F(u,u)}{\nabla^2 F_s(u,u)} \right) - \frac{1}{s} \frac{dF(u)}{\sqrt{\nabla^2 F_s(u,u)}}.
  \end{align*}
  Let~$t_0$ be the largest~$s \in [t, t']$ such that~$\phi_u(t_0) = 0$; if such an~$s$ does not exist, then set~$t_0 = t$.
  Let~$t^* \in [t_0, t']$ be such that~$\phi_u(t^*)$ is maximal over this interval, and set~$\phi_u^* = \phi_u(t^*)$.
  Then,
  \begin{align*}
    \phi_u^* % = \phi_u(t^*)
              & = \phi_u(t_0) + \int_{t_0}^{t^*} \partial_{s} \phi(s) \, ds \\
              & \leq \phi_u(t_0) + \int_{t_0}^{t^*} \bracks*{ \frac{1}{2s} \phi_u(s) \left( 1 + \frac{\nabla^2 F(u,u)}{\nabla^2 F_{s}(u,u)} \right) + \frac{1}{s} \frac{\abs{dF(u)}}{\sqrt{\nabla^2 F_{s}(u,u)}} } \, ds \\
              & \leq \abs{\phi_u(t_0)} + \int_{t_0}^{t^*} \bracks*{ \frac{1}{s} \phi_u(s) + \frac{1}{s} \sqrt{\theta} } \, ds \\
              & \leq \abs{\phi_u(t)} + (\phi_u^* + \sqrt{\theta}) \log \frac {t^*} {t_0};
  \end{align*}
  the second inequality follows since~$\nabla^2 F_s \geq \nabla^2 F$ as $f$ is convex and using that~$F$ is a $\theta$-barrier;
  the last inequality is ensured by our choice of~$t_0$.
  % The choice of~$t_0$ guarantees that $\phi_u(t_0) = 0$ or $\phi_u(t_0) = \phi_u(t) \geq 0$, so $\abs{\phi_u(t_0)} \leq \abs{\phi_u(t)}$.
  Using $\abs{\phi_u(t)} \leq \sqrt{\alpha} \lambda_{F_t, \alpha}(p)$, we obtain
  \begin{equation}\label{eq:local norm over time bound}
    \phi_u^* \parens*{ 1 - \log \frac {t^*} {t_0} } \leq \sqrt{\alpha} \, \lambda_{F_t, \alpha}(p) + \sqrt{\theta} \log \frac {t^*} {t_0},
  \end{equation}
  On the other hand, since $t \leq t_0 \leq t^* \leq t'$, our assumption implies that
  \begin{equation*}
    \left(1 + \frac{\sqrt{\theta}}{c \sqrt{\alpha}} \right) \log \frac {t^*} {t_0}
  \leq \parens*{ 1 + \frac {\sqrt\theta} {c \sqrt\alpha} } \abs*{ \log \frac {t'} t }
  \leq 1 - \frac{\lambda_{F_t, \alpha}(p)}{c},
  \end{equation*}
  or equivalently
  \begin{align}\label{eq:time metric constraint implication}
    \sqrt\alpha \, \lambda_{F_t, \alpha}(p) + \sqrt{\theta} \log \frac {t^*} {t_0}
  \leq c \sqrt{\alpha} \parens*{ 1 - \log \frac {t^*} {t_0} }.
  \end{align}
  % Note that in particular $\log(t^* / t_0) \leq 1$.
  Combining \cref{eq:local norm over time bound,eq:time metric constraint implication} gives~$\phi_u^* \leq c \sqrt\alpha$, implying that $\abs{\phi_u(t')} \leq c \sqrt\alpha$ as desired.
\end{proof}

We now show that for large~$t > 0$, approximate minimizers of~$F_t$ correspond to approximate minimizers of~$f$.
The proposition and proof we give below are adapted from~\cite[Prop.~3.2.4]{nesterov-nemirovskii-ipm}.

\begin{prop}\label{prop:function gap at time t}
  Let~$D \subseteq M$ be open and convex, let~$F \colon D \to \RR$ be a $\theta$-barrier along geodesics for~$D$, and let~$f\colon D \to \RR$ be a smooth convex function which has a closed convex extension.
  For some fixed~$t>0$, suppose that~$F_t := tf + F$ is $\alpha$-self-concordant along geodesics for some~$\alpha>0$ and that it is bounded from below.
  Then for every~$p \in D$ such that~$\lambda_{F_t, \alpha}(p) < \frac{1}{3}$, we have
  \begin{equation*}
    f(p) - \inf_{q \in D} f(q) \leq \frac{2 \theta + \alpha \rho(\lambda_{F_t, \alpha}(p))}{t},
  \end{equation*}
  where we recall from \cref{eq:defn rho} that~$\rho(r) = - r - \log(1-r)$.
\end{prop}
\begin{proof}
  By \cref{lem:sum of lsc is lsc}, $F_t$ is closed convex and hence strongly $\alpha$-self-concordant along geodesics.
  Because it also has positive definite Hessians and we have~$\lambda_{F_t, \alpha}(p) < 1$, \cref{prop:minigap} implies that $F_t$ attains its minimum at a unique minimizer~$p_{t,*} \in D$ and moreover
  \begin{equation}
    \label{eq:Ft p qstar difference at most a rho lambda}
    F_t(p) - F_t(p_{t,*}) \leq \alpha \rho(\lambda_{F_t, \alpha}(p)).
  \end{equation}
  Furthermore, \cref{lem:local norm to minimizer bound} shows that if~$u \in T_p M$ is such that~$\Exp_p(u) = p_{t,*}$, then
  \begin{equation*}
    \norm{u}_{F_t,p,\alpha} \leq \frac{\lambda_{F_t,\alpha}(p)}{1 - \lambda_{F_t,\alpha}(p)} < \frac{1}{2}
  \end{equation*}
  where the last inequality follows from~$\lambda_{F_t,\alpha}(p) < \frac{1}{3}$.
  Using \cref{cor:upper}, we obtain that
  \[ \Exp_{p_{t,*}}(v) = \Exp_p(2u) \in D, \]
  where $v = \tau_{\gamma,1} u$ is the parallel transport of $u$ from $p$ to $p_{t,*}$ along the geodesic $\gamma(t) := \Exp_p(tu)$.
  By \cref{prop:sc barrier gradient bound}, it follows that
  \begin{align*}
    dF_{p_{t,*}}(v) \leq \theta
  \end{align*}
  and hence, using convexity of~$F$ and $\Exp_{p_{t,*}}(v) = p$,
  \begin{align}\label{eq:F qstar p difference at most theta}
    F(p_{t,*}) - F(p) \leq -dF_{p_{t,*}}(-v) = dF_{p_{t,*}}(v) \leq \theta.
  \end{align}
  Together, \cref{eq:Ft p qstar difference at most a rho lambda,eq:F qstar p difference at most theta} then show that
  \begin{align}
    f(p) & = \frac {F_t(p) - F(p)} t \nonumber \\
         & \leq \frac { F_t(p_{t,*}) + \alpha \rho(\lambda_{F_t, \alpha}(p)) - F(p) } t \nonumber \\
         & = f(p_{t,*}) + \frac { F(p_{t,*}) - F(p) + \alpha \rho(\lambda_{F_t, \alpha}(p)) } t \nonumber \\
         & \leq f(p_{t,*}) + \frac { \theta + \alpha \rho(\lambda_{F_t, \alpha}(p)) } t. \label{eq:f p upper bound}
  \end{align}

  We will now give an upper bound on~$f(p_{t,*}) - f(q)$ for every~$q \in D$.
  Let~$v \in T_{p_{t,*}} M$ be such that~$\Exp_{p_{t,*}}(v) = q$.
  Using the convexity of~$f$, the fact that $p_{t,*}$ is a minimizer of~$F_t$, and \cref{prop:sc barrier gradient bound} (in this order) gives
  \begin{equation*}
    f(p_{t,*})- f(q)
  \leq - df_{p_{t,*}}(v)
  = \frac {dF_{p_{t,*}}(v)} t
  \leq \frac \theta t.
  \end{equation*}
  Combining this with \cref{eq:f p upper bound} and optimizing over~$q \in D$ gives the desired bound.
\end{proof}

We now come to the main result of this section, giving a path-following method which converges to a minimizer of the objective, generalizing~\cite[Prop.~3.2.4]{nesterov-nemirovskii-ipm} to our setting.
\begin{thm}
  \label{thm:main stage}
  Let~$D \subseteq M$ be an open, convex, and bounded domain.
  Let~$F\colon D \to \RR$ be a $\theta$-barrier for~$D$, and let~$f\colon D \to \RR$ be a smooth convex function with a closed convex extension.
  Let~$\alpha > 0$ be such that~$F_t := tf + F$ is~$\alpha$-self-concordant for all~$t \geq 0$.
  Choose~$1 > \lambda^{(1)} > \lambda^{(2)} > 0$ such that
  $\parens*{ \frac{\lambda^{(1)}}{1 - \lambda^{(1)}} }^2 \leq \lambda^{(2)} < \frac13$;
  a suitable choice is given by $\lambda^{(1)} = \frac{1}{4}$, $\lambda^{(2)} = \frac{1}{9}$.
  Finally, let~$p \in D$ be given such that~$\lambda_{F,\alpha}(p) < \lambda^{(1)}$, and assume that~$p$ is not a minimizer of~$f$.
  Define a sequence of time parameters
  \begin{align*}
    t_0 = \frac{\sqrt{\alpha} \lambda^{(1)} - \lambda_{F}(p)}{\norm{df_{p}}_{F,p}^*}, \quad
    t_\ell = t_0 \cdot \exp \left( \ell \frac{\lambda^{(1)} - \lambda^{(2)}}{\lambda^{(1)} + \sqrt{\theta/\alpha}} \right)
     \text{ for } \ell = 0, 1, 2, \dotsc,
  \end{align*}
  and a sequence of points
  \begin{equation*}
    p_{-1} = p, \quad p_{\ell} = (p_{\ell - 1})_{F_{t_\ell},{+}} \text{ for } \ell = 0, 1, 2, \dotsc.
  \end{equation*}
  i.e.,~$p_{\ell}$ is the Newton iterate of~$p_{\ell - 1}$ with respect to~$F_{t_\ell}$.
  Then this sequence is well-defined, in the sense that~$p_\ell \in D$ for all~$\ell \geq 0$, and it satisfies
  \begin{align*}
    f(p_\ell) - \inf_{q \in D} f(q)
  \leq \frac {2 (\theta + \alpha)} {t_\ell}
  = \frac{2 (\theta + \alpha) \norm{df_p}_{F,p}^*}{\sqrt{\alpha} \lambda^{(1)} - \lambda_{F}(p)} \exp \left( - \ell \cdot \frac{\lambda^{(1)} - \lambda^{(2)}}{\lambda^{(1)} + \sqrt{\theta / \alpha}} \right).
  \end{align*}
\end{thm}
\begin{proof}
  By the assumptions on~$f$ and strong self-concordance of~$F$, we see from \cref{lem:sum of lsc is lsc} that~$F_t$ is strongly $\alpha$-self-concordant on $D$ for all $t \geq 0$.
  % Furthermore, $(\nabla^2 F)_p$ is positive definite because~$D$ is bounded.
  We shall prove by induction on~$\ell$ that for every~$\ell \geq 0$, we have $p_\ell \in D$ and
  \begin{equation*}
    \lambda_{F_{t_\ell}, \alpha}(p_{\ell-1}) \leq \lambda^{(1)}, \quad \lambda_{F_{t_\ell}, \alpha}(p_\ell) \leq \lambda^{(2)}.
  \end{equation*}
  Let us first check that~$\lambda_{F_{t_0}, \alpha}(p_{-1}) = \lambda_{F_{t_0}, \alpha}(p) \leq \lambda^{(1)}$.
  For every~$u \neq 0$, we have
  \begin{align*}
    \abs{d(F_{t_0})_p(u)}
    % & = \abs{t_0 df_p(u) + dF_p(u)} \\
    & \leq t_0 \abs{df_p(u)} + \abs{dF_p(u)} \\
    & = (\sqrt{\alpha} \lambda^{(1)} - \lambda_F(p)) \frac{\abs{df_p(u)}}{\norm{df_p}_{F,p}^*}  + \abs{dF_p(u)} \\
    & \leq (\sqrt{\alpha} \lambda^{(1)} - \lambda_F(p)) \norm{u}_{F,p} + \norm{dF_p}_{F,p}^* \norm{u}_{F,p} \\
    & = \sqrt{\alpha} \lambda^{(1)} \norm{u}_{F,p} \\
    & \leq \sqrt{\alpha} \lambda^{(1)} \norm{u}_{F_{t_0},p},
  \end{align*}
  hence~$\norm{d(F_{t_0})_p}_{F_{t_0},p}^* \leq \sqrt{\alpha} \lambda^{(1)}$, which is equivalent to~$\lambda_{F_{t_0},\alpha}(p) \leq \lambda^{(1)}$.
  Next, if~$\lambda_{F_{t_\ell}, \alpha}(p_{\ell-1}) \leq \lambda^{(1)}$ for some~$\ell\geq0$, then by applying \cref{thm:newton decrement after newton step}, we find that the Newton iterate~$p_\ell$ is in~$D$ satisfies
  \begin{equation*}
    \lambda_{F_{t_\ell}, \alpha}(p_\ell) \leq \left( \frac{\lambda^{(1)}}{1 - \lambda^{(1)}} \right)^2 \leq \lambda^{(2)}.
  \end{equation*}
  Lastly, it remains to verify that if~$\lambda_{F_{t_\ell}, \alpha}(p_\ell) \leq \lambda^{(2)}$ for some~$\ell\geq0$, then~$\lambda_{F_{t_{\ell+1}},\alpha}(p_\ell) \leq \lambda^{(1)}$.
  The~$t_\ell$ are chosen exactly so that
  \begin{align*}
    % \timemetric_{\alpha,\lambda^{(1)}}(t_\ell, t_{\ell+1}) & =
    \left( 1 + \frac{\sqrt{\theta}}{\lambda^{(1)} \sqrt{\alpha}} \right) \abs*{\log \frac {t_\ell} {t_{\ell+1}}}
  = \left( 1 + \frac{\sqrt{\theta}}{\lambda^{(1)} \sqrt{\alpha}} \right) \left( \frac{\lambda^{(1)} - \lambda^{(2)}}{\lambda^{(1)} + \sqrt{\theta / \alpha}} \right)
  = 1 - \frac{\lambda^{(2)}}{\lambda^{(1)}}.
  \end{align*}
  We conclude that~$\lambda_{F_{t_{\ell+1}},\alpha}(p) \leq \lambda^{(1)}$ by \cref{prop:sc family time switch}.
  Lastly, the bound on~$f(p_l) - \inf_{q \in D} f(q)$ follows from \cref{prop:function gap at time t}, where we use that~$\lambda^{(2)} < \frac13$ and~$\rho(\frac13) \approx 0.072 \leq 2$.
\end{proof}
We end with a simple but useful lemma to upper bound the quantity~$\norm{df_p}_{F,p}^*$.
\begin{lem}
  \label{lem:local norm of derivative bound}
  Let~$p \in D$, and~$f, F\colon D \to \RR$ be such that~$f$ is convex and~$F$ is strongly $1$-self-concordant on~$D$.
  Then
  \begin{equation*}
    \norm{df_p}_{F,p}^*
  \leq \sup_{q \in D} f(q) - f(p)
  \leq \sup_{q \in D} f(q) - \inf_{q \in D} f(q).
  \end{equation*}
\end{lem}
\begin{proof}
  By \cref{cor:dikin}, the Dikin ellipsoid~$B := B_{F,p}^\circ(1)$ of radius~$1$ is contained in~$D$.
  Then the convexity of~$f$ gives
  \begin{align*}
    \norm{df_p}_{F,p}^*
    = \sup_{\substack{u \in T_p M \\ \norm{u}_{F,p} < 1}} \abs{df_p(u)}
    = \sup_{\substack{u \in T_p M \\ \norm{u}_{F,p} < 1}} df_p(u)
    \leq \sup_{\substack{u \in T_p M \\ \norm{u}_{F,p} < 1}} f(\Exp_p(u)) - f(p)
    = \sup_{q \in B} f(q) - f(p),
    % \leq \sup_{q \in D} f(q) - f(p)
  \end{align*}
  which is at most~$\sup_{q \in D} f(q) - f(p)$ as~$B \subseteq D$.
\end{proof}

%=============================================================================
\section{The squared distance function}\label{sec:distsq}
%=============================================================================
In this section we discuss self-concordance of the squared distance function to a point.
In \cref{subsec:distsq hadamard} we recall some useful formulas that apply for arbitrary Hadamard manifolds.
In \cref{subsec:pdn geometry} we focus on the space~$\PD(n)$ of positive-definite complex $n \times n$ matrices and prove that the distance squared to any point is self-concordant.
This relies on explicit computations of higher covariant derivatives.
Next, in \cref{subsec:model spaces} we use these same formulas to deduce stronger self-concordance estimates in the case of hyperbolic space~$\HH^n$, and use these to construct a barrier for the distance function rather than its square;
all this generalizes readily to the model spaces of arbitrary constant negative curvature.
The results of this section are applied in \cref{sec:applications}.

%-----------------------------------------------------------------------------
\subsection{Hadamard manifolds}
\label{subsec:distsq hadamard}
%-----------------------------------------------------------------------------
Let~$M$ be a Hadamard manifold, i.e., a simply-connected geodesically-complete Riemannian manifold with non-positive sectional curvature (cf.~\cref{subsec:geodesics and hadamard}).
Fix~$p_0 \in M$ and consider the function that computes the \emph{squared distance} to the point~$p_0$, that is,
\[ f\colon M \to \RR, \quad f(p) = d(p, p_0)^2. \]
Then it is known that~$f$ is~$2$-strongly convex (which follows from variational principles for the energy of a curve, cf.~\cite[Thm.~10.22]{lee-riemannian-manifolds}).
In fact, this is a defining property of the more general class of~$\CAT(0)$-spaces~\cite{bridson-haefliger-nonpositive-curvature}.
It will also be useful to consider the distance to~$p_0$,
\[ g\colon M \to \RR, \quad g(p) = d(p, p_0), \]
which is still convex.
The following lemma summarizes well-known properties of these functions.

\begin{lem}
  \label{lem:distsq gradient and hess lower bound}
  Let~$M$ be a Hadamard manifold, let~$p_0 \in M$, and define~$f, g\colon M \to \RR$ by~$f(p) = d(p, p_0)^2$ and~$g(p) = d(p, p_0)$.
  Then~$f$ is 2-strongly convex and $g$ is convex.
  For every~$p \neq p_0$, $g$ is smooth at~$p$, and the differentials and Hessians satisfy
  \begin{align}
    \label{eq:dist grad formulas}
    df_p &= 2 g(p) dg_p = - 2 \braket{\Exp_p^{-1}(p_0), \cdot}_p, \\
    \label{eq:distsq hess lower bound}
    \nabla^2 f &= 2 g \nabla^2 g + 2 \, dg \otimes dg \succeq 2 \, dg \otimes dg = \frac{df \otimes df}{2 f}.
  \end{align}
\end{lem}
\begin{proof}
  The strong convexity of~$f$ and convexity of~$g$ hold on any $\CAT(0)$-space~\cite[Cor.~II.2.5]{bridson-haefliger-nonpositive-curvature}.
  Whenever $p \neq p_0$, $f(p) \neq 0$ and hence $g = \sqrt{f}$ is smooth at $p$.
  By the chain rule, $df = 2 g \, dg$.
  To compute these, note that $g$ is~$1$-Lipschitz by the triangle inequality, so~$\abs{dg_p(u)} \leq \norm{u}_p$ for all~$u \in T_p M$.
  But since the geodesic from~$p$ in the direction~$\Exp_p^{-1}(p_0)$ has constant speed and reaches~$p_0$ at time~$1$, it follows that
  \begin{equation*}
    dg_p(\Exp_p^{-1}(p_0)) = -g(p).
  \end{equation*}
  As~$\norm{\Exp_p^{-1}(p_0)}_p = g(p)$, this means that the Cauchy--Schwarz inequality applied to
  \begin{equation*}
    g(p) = \abs{dg_p(\Exp_p^{-1}(p_0))} = \abs{\braket{(\grad g)_p, \Exp_p^{-1}(p_0)}} \leq \norm{(\grad g)_p}_p \norm{\Exp_p^{-1}(p_0)}_p \leq \norm{\Exp_p^{-1}(p_0)}_p
  \end{equation*}
  holds with equality, hence~$(\grad g)_p = -g(p)^{-1} \Exp_p^{-1}(p_0)$ and $dg_p = - g(p)^{-1} \braket{\Exp_p^{-1}(p_0), \cdot}$, and~$df_p = - 2 \braket{\Exp_p^{-1}(p_0), \cdot}_p$ follows.
  We finally derive the formulas for the Hessians.
  Applying the product rule to~$df = 2 g \, dg$ yields
  \begin{equation*}
    (\nabla^2 f)_p = 2 g(p) (\nabla^2 g)_p + 2 \, dg_p \otimes dg_p,
  \end{equation*}
  The lower bound in \cref{eq:distsq hess lower bound} follows since $(\nabla^2 g)_p \succeq 0$, as a consequence of the convexity of~$g$.
\end{proof}

\begin{cor}
  \label{cor:squared distance newton decrement}
  The Newton decrement of~$f(p) = d(p, p_0)^2$ is given by $\lambda_{f}(p) = \sqrt{2} \, d(p, p_0)$.
\end{cor}
\begin{proof}
  Recall the variational characterization of the Newton decrement in \cref{eq:newton decrement variational}:
  \begin{align*}
    \lambda_f(p)
    = \min \braces*{ \lambda \geq 0 : df_p \ot df_p \preceq \lambda^2 \, (\nabla^2f)_p }.
  \end{align*}
  Thus, $\lambda_f \leq \sqrt{2} f$ by \cref{eq:distsq hess lower bound}.
  As~$g$ is linear in the direction~$\Exp_p^{-1}(p_0)$, its Hessian vanishes in this direction and so we in fact have equality, by the first equality in \cref{eq:distsq hess lower bound}.
\end{proof}

We use \cref{lem:distsq gradient and hess lower bound} to prove the following result, which is used later to prove \cref{thm:hypn dist epigraph barrier}.

\begin{lem}
  \label{lem:psi concave}
  Let $\Psi\colon M \times \RR \times \RR_{>0} \to \RR$ be the function defined by
  \begin{equation*}
    \Psi(p, R, S) = R - S^{-1} d(p, p_0)^2.
  \end{equation*}
  Then $\Psi$ is concave, with Hessian given by
  \begin{equation*}
    \nabla^2 \Psi = - \frac{ 2 \left(S^{-1} g \, dS - dg \right)^{\otimes 2} + \left( \nabla^2 f - 2 dg \otimes dg \right) } S \preceq 0,
  \end{equation*}
  where $f, g$ are as in \cref{lem:distsq gradient and hess lower bound}, $dS$ is the differential of the projection $(p,R,S) \mapsto S$, and we write~$dg$ for the differential of $(p,R,S) \mapsto g(p)$ by a slight abuse of notation.
  Moreover, for~$u = (u_p, u_R, u_S)$ and~$w = (w_p, w_R, w_S)$ tangent vectors at~$(p,R,S)$, one has
  \begin{equation*}
    \nabla^3 \Psi(w, u, u) = -2 \frac{u_S}{S} \nabla^2 \Psi(w, u) - \frac{w_S}{S} \nabla^2 \Psi(u, u) - \frac{1}{S} \nabla^3 f(w_p, u_p, u_p).
  \end{equation*}
\end{lem}
\begin{proof}
  Clearly,
  \begin{equation*}
    d\Psi = dR + S^{-2} f \, dS - S^{-1} \, df.
  \end{equation*}
  Since~$\nabla dR \equiv 0 \equiv \nabla dS$, this yields
  \begin{align}\label{eq:hess Psi in proof}
    \nabla^2 \Psi = - 2 S^{-3} f \, dS \otimes dS + S^{-2} \, df \otimes dS + S^{-2} \, dS \otimes df - S^{-1} \nabla^2 f.
  \end{align}
  We now use \cref{eq:dist grad formulas,eq:distsq hess lower bound} to rewrite the above as
  \begin{align*}
    \nabla^2 \Psi & = - 2 S^{-3} g^2 \, dS \otimes dS + 2 S^{-2} g \, dg \otimes dS + 2 S^{-2} g \, dS \otimes dg - S^{-1} (2 g \nabla^2 g + 2 \, dg \otimes dg) \\
                  & = - 2 S^{-1} (S^{-1} g \, dS - dg)^{\otimes 2} - 2 S^{-1} g \nabla^2 g.
  \end{align*}
  Taking one more derivative in \cref{eq:hess Psi in proof}, we obtain
  \begin{align*}
    \nabla^3 \Psi(w, u, u)
  & = 6 S^{-4} f \, dS(w) \, dS(u)^2 - 2 S^{-3} df(w) \, dS(u)^2 - 4 S^{-3} dS(w) \, df(u) \, dS(u) \\
  & + 2 S^{-2} \nabla^2 f(w, u) \, dS(u) + S^{-2} dS(w) \nabla^2 f(u, u)  - S^{-1} \nabla^3 f(w, u, u) \\
  % & = -2 S^{-1} dS(u) \, \parens*{ - 2 S^{-3} f \, dS(w) dS(u) + S^{-2} \, df(w) dS(u) + S^{-2} \, dS(w) df(u) - S^{-1} \nabla^2 f(w,u) } \\
  % &- S^{-1} dS(w) \, \parens*{ - 2 S^{-3} f \, dS(u) dS(u) + S^{-2} \, df(u) dS(u) + S^{-2} \, dS(u) df(u) - S^{-1} \nabla^2 f(u,u) } \\
  % &- S^{-1} \nabla^3 f(w, u, u) \\
  & = -2 S^{-1} dS(u) \, \nabla^2 \Psi(w, u) - S^{-1} dS(w) \, \nabla^2 \Psi(u, u) - S^{-1} \nabla^3 f(w, u, u). \qedhere
  \end{align*}
\end{proof}

\begin{cor}
  \label{cor:most important barrier term convex}
  Let $D = \{ (p, R, S) \in M \times \RR_{>0} \times \RR_{>0} : RS - f(p) > 0 \}$.
  Then the function~$F\colon D \to \RR$ defined by $F(p, R, S) = -\log(R - S^{-1} d(p, p_0)^2)$ is convex.
\end{cor}
% \begin{proof}
%   Logarithms of concave functions are concave.
% \end{proof}

%-----------------------------------------------------------------------------
\subsection{Positive definite matrices}
\label{subsec:pdn geometry}
%-----------------------------------------------------------------------------
In this subsection, we specialize to the space $\PD(n) = \PD(n, \CC)$ of positive definite Hermitian $n \times n$ matrices, which is a Hadamard manifold when endowed with a well-known Riemannian metric.
We collect a number of well-known results from the literature and then derive explicit formulas for the higher derivatives of the squared distance on this space by using techniques from matrix analysis.
The main result of this section is \cref{thm:distsq pdn self-concordant expanded}, where we show that the squared distance is self-concordant on~$\PD(n)$.
As explained in the introduction, this implies that the squared distance is self-concordant on arbitrary Hadamard symmetric spaces.

We will often use notation of the form $h(P)$ where $h\colon \RR_{>0} \to \RR$ is some scalar-valued function, which refers to the Hermitian matrix obtained by expanding~$P$ in an eigenbasis and applying~$h$ to its eigenvalues.
Examples include but are not limited to expressions of the form $P^t$ with $t \in \RR$, $P + \lambda = P + \lambda I$ where $\lambda \in \RR$, $\log(P)$, et cetera.

We think of $\PD(n)$ as an open submanifold of the $n \times n$ Hermitian matrices $\Herm(n) \subseteq \CC^{n \times n}$, so that we can identify $T_P\PD(n) \cong \Herm(n)$ at any~$P\in\PD(n)$.
Concretely, $X \in \Herm(n)$ corresponds to the tangent vector of the curve $t \mapsto P + Xt = P^{1/2} (I + t P^{-1/2} X P^{-1/2}) P^{1/2}$ at~$t=0$.
These curves would be geodesics if we equipped $\PD(n)$ with the Euclidean metric inherited from~$\Herm(n)$.
Instead, we introduce the following Riemannian metric on $\PD(n)$:
\begin{equation}
  \label{eq:pdn riemannian metric}
  \braket{X, Y}_P := \Tr\left[(P^{-1/2} X P^{-1/2}) (P^{-1/2} Y P^{-1/2})\right] = \Tr \left[ P^{-1} X P^{-1} Y\right]
\end{equation}
for $X, Y \in T_P \PD(n)$.
This is real-valued as the Hilbert-Schmidt inner product of two Hermitian matrices.
% Note that this is a real-valued inner product, because $P^{-1/2} X P^{-1/2}$ is Hermitian and similar for $Y$.
Interstingly, $\braket{\cdot, \cdot}_P$ is also the \emph{Euclidean} Hessian of the function $P \mapsto -\log \det(P)$, which is a Euclidean self-concordant barrier for $\PD(n)$.

It is immediate from the definition that for every $P \in \PD(n)$, the bijection $Q \mapsto P^{1/2} Q P^{1/2}$ is a Riemannian isometry of $\PD(n)$, meaning it preserves inner products between tangent vectors.
% This is in the sense that its derivative preserves inner products between tangent vectors: its derivative is given by $X \mapsto P^{1/2} X P^{1/2}$ for $X \in T_Q \PD(n)$, and one easily verifies
% \begin{equation*}
%   \braket{P^{1/2} X P^{1/2}, P^{1/2} Y P^{1/2}}_{P^{1/2} Q P^{1/2}} = \braket{X, Y}_Q.
% \end{equation*}
% Isometries that are also bijective automatically preserve the distance between two points.
% It follows that the path-length metric $d$ induced by the Riemannian metric satisfies
Then it also preserves the distance between any two points:
for any~$P, Q, Q' \in \PD(n)$, we have
\begin{equation*}
  d(Q, Q') = d(P^{1/2} Q P^{1/2}, P^{1/2} Q' P^{1/2}).
\end{equation*}
Therefore, if one is interested in properties of squared distance $f(P) = d(P, P_0)^2$, one may choose~$P_0 = I$ without loss of generality.
This will be convenient for our purposes.

We now give explicit formulas for the geodesics on $\PD(n)$.
For any~$P \in \PD(n)$, the exponential map at~$P$ reads
\begin{equation}\label{eq:pdn exponential map}
  \Exp_P(X) = P^{1/2} e^{P^{-1/2} X P^{-1/2}} P^{1/2}
\end{equation}
and hence the geodesics through~$P$ take the form
\begin{align*}
  P(t) = \Exp_P(tX) = P^{1/2} e^{t P^{-1/2} X P^{-1/2}} P^{1/2}.
\end{align*}
In particular, the geodesics through $P = I$ are of the form $\Exp_I(tX) = e^{tX}$.
From the description of the exponential map above it follows that $\Exp_P\colon T_P \PD(n) \to \PD(n)$ is a smooth bijection for all $P$, with smooth inverse given by
\begin{equation*}
  \Exp_P^{-1}(Q) = P^{1/2} \log(P^{-1/2} Q P^{-1/2}) P^{1/2}.
\end{equation*}
By the Hopf--Rinow theorem, there exists a length-minimizing geodesic, which is unique by the bijectivity of the exponential map; hence the distance induced by the Riemannian metric is
\begin{align*}
  d(P,Q)
  = \normHS{\log(P^{-1/2} Q P^{-1/2})}
  = \normHS{\log(Q^{-1/2} P Q^{-1/2})}.
\end{align*}
where~$\normHS{\cdot}$ denotes the Hilbert--Schmidt (Frobenius) norm, because~$d(P, Q) = \norm{\Exp_P^{-1}(Q)}_P$.

The geodesics on~$\PD(n)$ can be naturally described using the operator geometric mean, which is defined for $P, Q = \PD(n)$ and $t \in [0,1]$ to be
\begin{equation*}
  P \#_t Q := P^{1/2} (P^{-1/2} Q P^{-1/2})^t P^{1/2}.
\end{equation*}
The above formula for the geodesics through $P$ shows that this is equal to $\Exp_P(t \Exp_P^{-1}(Q))$, and so it is the ``time-$t$''-geodesic-midpoint between $P$ and $Q$.

One can also explicitly describe the parallel transport along geodesics.
For~$P, Q \in \PD(n)$, the parallel transport of~$X \in T_P \PD(n)$ along the unique geodesic from $P$ to $Q$ is given by%
\footnote{One way of proving~\cref{eq:pdn parallel transport} is as follows~\cite[Lem.~IV.6.2]{sakai-riemannian-geometry}: for every~$P \in \PD(n)$, the \emph{geodesic inversion} map~$s_P\colon \PD(n) \to \PD(n)$ given by~$s_P(Q) = \Exp_P(-\Exp_P^{-1}(Q)) = P Q^{-1} P$ is an isometry (more generally, the maps~$Q \mapsto Q^{-1}$ and~$Q \mapsto A Q A^*$ are isometries for every~$A \in \GL(n, \CC)$).
Let~$P_0, P_1 \in \PD(n)$, and let~$\gamma\colon \RR \to M$ be the unique geodesic such that~$\gamma(0) = P_0$ and~$\gamma(1) = P_1$.
Then~$s_{P_0}(\gamma(t)) = \gamma(-t)$ and~$s_{P_1}(\gamma(t)) = \gamma(1-t)$.
If~$X_t$ is a parallel vector field along~$\gamma$, then so is~$d(s_{P_0}) (X_{-t})$, as~$s_{P_0}$ is an isometry; but~$d(s_{P_0})_{P_0} = -I_{T_{P_0} \PD(n)}$, and so~$d(s_{P_0}) (X_{-t}) = - X_t$ by the uniqueness of parallel vector fields.
Similarly, $d(s_{\gamma(1/2)}) (X_{1/2 - t}) = - X_{1/2 + t}$, and so $d(s_{\gamma(1/2)} \circ s_{P_0}) (X_0) = X_1 = \transport{P_0}{P_1}(X_0)$.
Expanding the definition of~$s_{\gamma(1/2)} \circ s_{P_0}$ (also called a~\emph{transvection}), it is easy to see that its derivative is exactly the right-hand side in~\cref{eq:pdn parallel transport}.%
% Explicitly, one has
% \begin{equation}
%   \label{eq:geodesic inversion composition}
%   (s_{\gamma(1/2)} \circ s_{P_0})(P) = P_0^{1/2} (P_0^{-1/2} Q_0 P_0^{-1/2})^{1/2} P_0^{-1/2} P P_0^{-1/2} (P_0^{-1/2} Q_0 P_0^{-1/2})^{1/2} P_0^{1/2},
% \end{equation}
% which is just a linear map in~$P$, so its derivative is given by the same expression, replacing~$P$ by~$X \in T_P \PD(n)$.
% \Cref{eq:geodesic inversion differentiates to parallel transport,eq:geodesic inversion composition} together prove the identity~\cref{eq:pdn parallel transport}.
}
\begin{align}
  \label{eq:pdn parallel transport}
  \transport{P}{Q}(X) = P^{1/2} (P^{-1/2} Q P^{-1/2})^{1/2} P^{-1/2} X P^{-1/2} (P^{-1/2} Q P^{-1/2})^{1/2} P^{1/2}.
\end{align}
This may be conveniently restated as
\begin{align}
  \label{eq:pdn parallel transport exp formulation}
  \transport{P}{\Exp_P(tY)}(X) = P^{1/2} (e^{\frac t2 P^{-1/2} Y P^{-1/2}}) P^{-1/2} X P^{-1/2} (e^{\frac t2 P^{-1/2} Y P^{-1/2}}) P^{1/2}
\end{align}
which for the geodesics emanating from the identity specializes to
\begin{align*}
  \transport{I}{e^{tZ}}(X) = e^{\frac t2 Z} X e^{\frac t2 Z},
\end{align*}
i.e.,
\begin{align*}
  \transport{I}{Q}(X) = Q^{1/2} X Q^{1/2}.
\end{align*}

Now consider a function $f\colon\PD(n)\to\RR$.
It follows from the previous considerations and the discussion in \cref{subsec:geodesics and hadamard} that the third derivative at $I \in \PD(n)$ can be computed as follows for $X,Z\in T_I\PD(n)$:
\begin{align*}
  (\nabla^3 f)_I(Z, X, X)
  & = \partial_{t=0} (\nabla^2 f)_I(\transport{I}{\exp_I(tZ)}(X), \transport{I}{\exp_I(tZ)}(X)) \\
  & = \partial_{t=0} (\nabla^2 f)_I(e^{\frac t2 Z} X e^{\frac t2 Z},e^{\frac t2 Z} X e^{\frac t2 Z}).
\end{align*}

Although we will not need it explicitly, one can also use the above to determine the covariant derivative of a general vector field.
More precisely, the covariant derivative $\nabla_X Y$, where $X \in T_P\PD(n)$ and $Y(t)$ is a vector field defined along the curve $P(t) = \Exp_P(tX)$, is given by
\begin{align*}
  \nabla_X Y = \partial_{t=0} \transport{P(t)}{P} (Y(t)).
\end{align*}
For $P=I$, we have 
\begin{align*}
  \nabla_X Y
  = \partial_{t=0} \transport{e^{tX}}{I}(Y(t))
  = \partial_{t=0} e^{-\frac t2 X} Y(t) e^{-\frac t2 X}
  = \dot Y(0) -\frac 12 \{X, Y(0)\}
\end{align*}
where we write~$\{X, Y\} = X Y + Y X$ for the anticommutator of~$X$ and~$Y$.

Lastly, we have an explicit expression for the Riemann curvature tensor on~$\PD(n)$. The fact that the curvature tensor is of this form follows from~\cite[Thm.~IV.4.2]{helgason1979differential}, and the prefactor of~$\frac{1}{4}$ can be deduced from the fact that~$\SPD(2,\CC)$ is a model space for constant curvature~$-\frac{1}{2}$ (the prefactor appears because we work directly with positive-definite matrices, rather than the quotient~$\GL_n(\CC) / \mathrm{U}(n)$).
Alternatively, one may consult the self-contained explicit proof available in~\cite{dolcettiDifferentialPropertiesGL2014}:
\begin{lem}
  \label{lem:pdn curvature tensor}
  The Riemann curvature $(1,3)$-tensor at~$P \in \PD(n)$ is given by
  \begin{equation*}
    R(X, Y)Z = - \frac{1}{4} [[P^{-1/2} X P^{-1/2}, P^{-1/2} Y P^{-1/2}], P^{-1/2} Z P^{-1/2}]
  \end{equation*}
  for every~$X,Y,Z \in T_P \PD(n)$.
  In particular, the curvature tensor is parallel along any geodesic.
\end{lem}
This last property may be more succinctly stated as follows: if one thinks of~$R$ as a $(0,4)$-tensor, then $\nabla R \equiv 0$.
Therefore~$\PD(n)$ is a \emph{locally symmetric space}, see~\cite[Thm.~10.19]{lee-riemannian-manifolds}, and because it is simply connected, it is also a globally symmetric space.
A simple computation using the above lemma shows that~$\PD(n)$ has sectional curvatures bounded by an~$n$-independent constant with our normalization of the metric:
\begin{lem}
  \label{lem:pdn bounded sectional curvatures}
  The space~$\PD(n)$ has all sectional curvatures in~$[-\frac{1}{2}, 0]$.
\end{lem}
\begin{proof}
  Let~$X, Y \in T_I \PD(n) = \Herm(n)$ have~$\norm{X}_I = \norm{Y}_I = 1$ and~$\braket{X, Y}_I = \Tr[XY] = 0$.
  Assume without loss of generality that~$Y$ is diagonal.
  Then
  \begin{equation*}
    \braket{R(X,Y)Y, X} = - \frac{1}{4} \sum_{i,j = 1}^n \abs{X_{ij}}^2 (Y_{jj} - Y_{ii})^2.
  \end{equation*}
  This is clearly at most~$0$, and
  \begin{equation*}
    \sum_{i,j = 1}^n \abs{X_{ij}}^2 (Y_{jj} - Y_{ii})^2 \leq 2 \sum_{i,j = 1, \, i \neq j}^n \abs{X_{ij}}^2 (Y_{jj}^2 + Y_{ii}^2) \leq 2 \sum_{i,j = 1}^n \abs{X_{ij}}^2 \norm{Y}_I^2 = 2 \norm{X}_I^2 \norm{Y}_I^2 = 2,
  \end{equation*}
  so~$K(X, Y) \geq - \frac{1}{2}$.
\end{proof}

We now turn to the task of computing higher derivatives of the squared distance on $\PD(n)$.
Recall from \cref{subsec:pdn geometry} that the distance between $P, Q \in \PD(n)$ is given by $d(P, Q)^2 = \normHS{\log(P^{-1/2} Q P^{-1/2})}^2$.
To differentiate this, we use the following integral expression for the operator logarithm: for~$Q \in \PD(n)$, one has
\begin{equation}
  \label{eq:matrix logarithm integral expression}
  \log(Q) = %(Q - I) \int_0^1 \frac{1}{(1 - t) I +  t Q} \, dt =
  \int_0^\infty \parens*{ \frac{1}{I + \lambda} - \frac{1}{Q + \lambda} } \, d\lambda,
\end{equation}
where $Q + \lambda$ is shorthand for $Q + \lambda I$, and $\frac{1}{Q+\lambda} = (Q+\lambda)^{-1}$.
% To see this, we need only verify that for every~$q > 0$, one has
% \begin{equation*}
%   \log(q) = \int_0^\infty \frac{1}{1 + \lambda} - \frac{1}{q + \lambda} \, d\lambda.
% \end{equation*}
% First observe that the integral is finite:
% \begin{equation*}
%   \abs*{\int_0^\infty \frac{1}{1 + \lambda} - \frac{1}{q+\lambda}} = \abs*{\int_0^\infty \frac{q-1}{(q + \lambda)(1 + \lambda)}} \leq \int_0^\infty \abs*{\frac{q - 1}{(q + \lambda)(1 + \lambda)}} \, d\lambda \leq \abs{q-1} \int_{\min(q, 1)}^\infty \frac{1}{\lambda^2} \, d\lambda < \infty.
% \end{equation*}
% Furthermore, the integral is clearly~$0$ when~$q = 0$, and the dominated convergence theorem shows that its derivative with respect to~$q$ is
% \begin{align*}
%   \partial_q \parens*{\int_0^\infty \frac{1}{1 + \lambda} - \frac{1}{q + \lambda} \,d\lambda} = \int_0^\infty \frac{1}{(q + \lambda)^2} \, d\lambda = \frac{1}{q},
% \end{align*}
% so the integral must be~$\log(q)$.
The advantage of this expression is that it is an integral of \emph{rational} functions of $Q$, which is straightforward to differentiate using the Leibniz integral rule and the following rule for differentiating matrix inverses: if~$t \mapsto Q_t \in \PD(n)$ is a smooth curve defined on an open interval containing~$0$, then
\begin{equation}
  \label{eq:matrix inverse derivative}
  \partial_{t=0} (Q_t^{-1}) = - Q_0^{-1} (\partial_{t=0} Q_t) Q_0^{-1},
\end{equation}
as can be seen from differentiating the identity~$Q_t Q_t^{-1} = I$.

We now use this integral representation to compute derivatives of the squared distance.
For convenience, we consider only the squared distance to the identity $I \in \PD(n)$, but this is without loss of generality; to compute the derivatives of $d(\cdot, P)^2$ for $P \in \PD(n)$, one may use the fact that $Q \mapsto P^{1/2} Q P^{1/2}$ is an isometry sending $I$ to $P$.
First, we record the formula for the first derivative.
\begin{prop}
  \label{prop:pdn distsq first derivative}
  Let $f(Q) = d(Q, I)^2 = \normHS{\log(Q)}^2$.
  Then for $U \in T_Q \PD(n)$,
  \begin{equation*}
    df_Q(U) = 2 \Tr[Q^{-1} \log(Q) U] = 2 \braket{Q^{1/2} \log(Q) Q^{1/2}, U}_Q,
  \end{equation*}
  where $\braket{\cdot, \cdot}_Q$ is the Riemannian metric in $\PD(n)$ defined in \cref{eq:pdn riemannian metric}.
\end{prop}
\begin{proof}
  Let $Q_t = \Exp_Q(tU)$ be the geodesic through $Q$ in the direction $U$.
  Then by \cref{eq:pdn exponential map}, we have
  \begin{align*}
    Q_t = Q^{1/2} e^{t Q^{-1/2} U Q^{-1/2}} Q^{1/2},
  \end{align*}
  and so
  \begin{align*}
    \partial_{t=0} f(Q_t) & = \partial_{t=0} \normHS{\log(Q_t)}^2 = 2 \Tr[\log(Q) \cdot \partial_{t=0} \log(Q_t)].
  \end{align*}
  To evaluate $\partial_{t=0} \log(Q_t)$, we use \cref{eq:matrix logarithm integral expression} and \cref{eq:matrix inverse derivative} to obtain
  \begin{align*}
    \partial_{t=0} \log(Q_t) & = \partial_{t=0} \int_0^\infty \parens*{ \frac{1}{I + \lambda} - \frac{1}{Q_t+\lambda} } \, d\lambda
    % \\
                             %& = - \int_0^\infty \partial_{t=0} \frac{1}{Q_t+\lambda} \, d\lambda \\
                             % & = \int_0^\infty \frac{1}{Q+\lambda} (\partial_{t=0} Q_t) \frac{1}{Q + \lambda} \, d\lambda \\
                             = \int_0^\infty \frac{1}{Q + \lambda} U \frac{1}{Q + \lambda} \, d\lambda.
  \end{align*}
  Therefore
  \begin{align*}
    \partial_{t=0} f(Q_t) & = 2 \Tr \left[ \log(Q) \cdot \int_0^\infty \frac{1}{Q + \lambda} U \frac{1}{Q + \lambda} \, d\lambda \right] = 2 \Tr \left[ Q^{-1} \log(Q) \cdot U \right],
  \end{align*}
  where we used cyclicity of the trace and $\int_0^\infty \frac{1}{(q + \lambda)^2} \, d\lambda = q^{-1}$.
\end{proof}
\begin{rem}
  In the above proof, one may also use the curve~$t \mapsto Q + tU$ instead of the geodesic, because they agree in first order: it holds that~$\partial_{t=0} (Q + tU) = U = \partial_{t=0} \Exp_Q(tU)$, and hence first derivatives of functions are not affected.
  However, for the second derivative, $(\nabla^2 f)_P(U, U) = \partial_{t=0}^2 f(\Exp_Q(tU))$ and~$\partial_{t=0}^2 f(Q + t U)$ are generally distinct; a simple example is given by the function~$f(P) = \Tr[P]$, differentiating at~$Q = I$.
\end{rem}
\begin{rem}
  \label{rem:distsq pdn gradient}
  One may observe that
  \begin{equation*}
    -Q^{1/2} \log(Q) Q^{1/2} = \Exp_Q^{-1}(I)
  \end{equation*}
  so that~$df_Q(U) = - 2 \braket{\Exp_Q^{-1}(I), U}_Q$, which also follows from \cref{lem:distsq gradient and hess lower bound}.
\end{rem}

In the next theorem, we compute the higher covariant derivatives of the squared distance.
We write~$\{A, B\} := AB + BA$ for the anticommutator of two matrices.
\begin{thm}
  \label{thm:derivative expressions}
  Let $f(Q) = d(Q, I)^2$, and $U, W \in T_Q \PD(n)$.
  Set $\tilde U = Q^{-1/2} U Q^{-1/2}$ and $\tilde W = Q^{-1/2} W Q^{-1/2}$.
  Then the second derivative of $f$ satisfies
  \begin{align*}
    (\nabla^2 f)_Q(U, U) % & = \Tr \left[ \{ D\log(Q)[\tilde U], Q \} \tilde U  \\
                         & = \int_0^\infty d\lambda \Tr \left[ \frac{1}{Q+\lambda} U \frac{1}{Q+\lambda} \{ Q^{-1}, U \} \right]
                         = \int_0^\infty d\lambda \Tr \left[ \frac{1}{Q+\lambda} \tilde U \frac{1}{Q+\lambda} \{ Q, \tilde U \} \right],
  \end{align*}
  and the third derivative is given by
  \begin{align*}
    & (\nabla^3 f)_Q(W, U, U) \\
    & = \int_0^\infty d\lambda \Tr \left[ \frac{1}{Q+\lambda} \tilde U \frac{1}{Q+\lambda} (\tilde U \tilde W Q + Q \tilde W \tilde U) - \frac{1}{Q + \lambda} (\tilde U \frac{Q}{Q+\lambda} \tilde W + \tilde W \frac{Q}{Q+\lambda} \tilde U) \frac{1}{Q+\lambda} \{ \tilde U, Q \} \right].
  \end{align*}
\end{thm}
\begin{proof}
  For the second derivative, we use the identity $(\nabla^2 f_Q)(U, U) = \partial_{t=0}^2 f(Q_t)$ where $Q_t = \Exp_Q(tU)$. From \cref{prop:pdn distsq first derivative} it follows that
  \begin{align*}
    \partial_t f(Q_t) & = 2 \Tr \left[ Q_t^{-1} \log(Q_t) (\partial_t Q_t) \right].
  \end{align*}
  As $Q_t = \Exp_Q(tU) = Q^{1/2} e^{t Q^{-1/2} U Q^{-1/2}} Q^{1/2}$, we have
  \begin{align*}
    \partial_t Q_t & =  U Q^{-1/2} e^{t Q^{-1/2} U Q^{-1/2}} Q^{1/2}, \quad \partial_{t=0}^2 Q_t = U Q^{-1} U,
  \end{align*}
  which together with \cref{eq:matrix inverse derivative} leads to
  \begin{align*}
    \frac{1}{2} \partial_{t=0}^2 f(Q_t) & = \Tr \left[ (- Q^{-1} U Q^{-1} \log(Q) U) + Q^{-1} (\partial_{t=0} \log(Q_t)) U + Q^{-1} \log(Q) (\partial_{t=0}^2 Q_t) \right] \\
                                        & = \Tr \left[ Q^{-1} \int_0^\infty \frac{1}{Q + \lambda} U \frac{1}{Q+\lambda} \, d\lambda \, U \right] \\
                                        & = \int_0^\infty \Tr \left[ \frac{1}{Q + \lambda} U \frac{1}{Q + \lambda} U Q^{-1} \right] d\lambda.
  \end{align*}
  To replace the last $U Q^{-1}$ by $\frac{1}{2} \{ U, Q^{-1} \}$, note that
  \begin{equation*}
    \Tr[(Q+\lambda)^{-1} U (Q+\lambda)^{-1} U Q^{-1}] = \Tr[(Q+\lambda)^{-1} U Q^{-1} (Q+\lambda)^{-1} U] = \Tr[(Q+\lambda)^{-1} U (Q+\lambda)^{-1} Q^{-1} U]
  \end{equation*}
  where we first used cyclicity and next that $Q^{-1}$ and $(Q+\lambda)^{-1}$ commute.
  Using the definition $\tilde U = Q^{-1/2} U Q^{-1/2}$ yields the statement in the lemma.

  We now turn to the third derivative.
  Let $U, W \in T_Q \PD(n)$, set $Q_t = \Exp_Q(tW)$ and let $U_t = \transport{Q}{Q_t}(U)$, explicitly given in \cref{eq:pdn parallel transport exp formulation}:
  \begin{equation*}
    U_t = \transport{Q}{Q_t}(U) = Q^{1/2} (e^{\frac{t}{2} Q^{-1/2} W Q^{-1/2}}) Q^{-1/2} U Q^{-1/2} (e^{\frac{t}{2} Q^{-1/2} W Q^{-1/2}}) Q^{1/2}.
  \end{equation*}
  Then
  \begin{equation*}
    (\nabla^3 f)_Q(W, U, U) = \partial_{t=0} (\nabla^2 f)_{Q_t}(U_t, U_t).
  \end{equation*}
  The two basic derivatives that we need are
  \begin{align*}
    \partial_{t=0} U_t = \frac{1}{2} (W Q^{-1} U + U Q^{-1} W), \quad \partial_{t=0} Q_t = W.
  \end{align*}
  This yields, again using \cref{eq:matrix inverse derivative},
  \begin{align*}
    \partial_{t=0} (\nabla^2 f)_{Q_t}(U_t, U_t) & = \partial_{t=0} \int_0^\infty \Tr \left[ \frac{1}{Q_t + \lambda} U_t \frac{1}{Q_t + \lambda} \{ Q_t^{-1}, U_t \}  \right] \, d\lambda \\
                                               & = \int_0^\infty \Tr \left[ - \frac{1}{Q + \lambda} W \frac{1}{Q + \lambda} U \frac{1}{Q + \lambda} \{ Q^{-1}, U \}  \right] \\
                                               & \quad + \frac{1}{2} \Tr \left[ \frac{1}{Q + \lambda} (W Q^{-1} U + U Q^{-1} W) \frac{1}{Q + \lambda} \{ Q^{-1}, U \}  \right] \\
                                               & \quad + \Tr \left[ - \frac{1}{Q + \lambda} U \frac{1}{Q + \lambda} W \frac{1}{Q + \lambda} \{ Q^{-1}, U \}  \right] \\
                                               & \quad + \Tr \left[ \frac{1}{Q + \lambda} U\frac{1}{Q + \lambda} \{ - Q^{-1} W Q^{-1}, U \}  \right] \\
                                               & \quad + \frac{1}{2} \Tr \left[ \frac{1}{Q + \lambda} U\frac{1}{Q + \lambda} \{ Q^{-1}, W Q^{-1} U + U Q^{-1} W \}  \right] \, d\lambda \\
                                               % & = \int_0^\infty \Tr \left[ \frac{1}{Q + \lambda} (W Q^{-1} U + U Q^{-1} W) \frac{1}{Q + \lambda} \{ Q^{-1}, U \} \right] \\
                                               % & + \Tr \left[ \frac{1}{Q + \lambda} U\frac{1}{Q + \lambda} \{ - Q^{-1} W Q^{-1}, U \}  \right] \\
                                               % & - \Tr \left[ \frac{1}{Q + \lambda} (W \frac{1}{Q + \lambda} U + U \frac{1}{Q+\lambda} W) \frac{1}{Q + \lambda} \{ Q^{-1}, U \}  \right] \, d\lambda \\
                                               % & = \int_0^\infty \Tr \left[ \frac{1}{Q + \lambda} (W Q^{-1} U + U Q^{-1} W) \frac{1}{Q + \lambda} \{ Q^{-1}, U \} \right] \\
                                               % & + \Tr \left[ \frac{1}{Q + \lambda} U\frac{1}{Q + \lambda} \{ - Q^{-1} W Q^{-1}, U \}  \right] \\
                                               % & - \Tr \left[ \frac{1}{Q + \lambda} (W \frac{1}{Q + \lambda} U + U \frac{1}{Q+\lambda} W) \frac{1}{Q + \lambda} \{ Q^{-1}, U \}  \right] \, d\lambda \\
                                               & = \int_0^\infty \Tr \Big[ \frac{1}{Q + \lambda} U \frac{1}{Q + \lambda} W Q^{-1} U Q^{-1} + \frac{1}{Q + \lambda} U Q^{-1} W \frac{1}{Q + \lambda} U Q^{-1} \\
                                               & \qquad\qquad - \frac{1}{Q + \lambda} (W \frac{1}{Q + \lambda} U + U \frac{1}{Q+\lambda} W) \frac{1}{Q + \lambda} \{ Q^{-1}, U \}  \Big] \, d\lambda.
  \end{align*}
  Substituting $W = Q^{1/2} \tilde W Q^{1/2}$ and $U = Q^{1/2} \tilde U Q^{1/2}$ yields the desired expression.
\end{proof}

We now explicitly compute the integral expressions in \cref{thm:derivative expressions} in terms of the entries of the matrices $\tilde U$ and $\tilde W$.
We assume without loss of generality that $Q = \diag(q_1, \dotsc, q_n)$ by considering the expression in an eigenbasis of $Q$.
Furthermore, we shall assume that all $q_i$ are distinct; expressions at general $Q$ may be obtained by taking limits, but the inequalities we will derive automatically hold for all $Q$ by continuity.
Let us start with the second derivative.
Take~$U \in \Herm(n)$.
Then for~$\tilde U = Q^{-1/2} U Q^{-1/2}$ we have
\begin{align}
  (\nabla^2 f)_Q(U, U) & = \int_0^\infty d\lambda \Tr \left[ \frac{1}{Q + \lambda} \tilde U \frac{1}{Q + \lambda} \{ \tilde U, Q \}  \right] \nonumber \\
                                 & = \sum_{k,l} \int_0^\infty d\lambda \frac{1}{q_k + \lambda} \tilde U_{kl} \frac{1}{q_l + \lambda} \tilde U_{lk} (q_k + q_l) \nonumber \\
                                 & = 2 \sum_k \abs{\tilde U_{kk}}^2 + \sum_{k \neq l} \abs{\tilde U_{kl}}^2 \frac{(q_k + q_l) \log(q_k / q_l)}{q_k - q_l}. \label{eq:pdn second deriv expression}
\end{align}
where we evaluated the integral using the identities
\begin{equation}
  \label{eq:pdn second deriv expression integral helper}
  \int_0^\infty \frac{1}{(x+\lambda)^2} \,  d\lambda = \frac{1}{x}, \quad
  \int_0^\infty \frac{1}{(x+\lambda) (y+\lambda)} \, d\lambda
  = \frac{\log(x/y)}{x-y}
\end{equation}
for distinct~$x, y > 0$.
We now evaluate the third derivative in a similar manner. The only new difficulty is in performing the integration with respect to~$\lambda$, for which we record the following lemma.
\begin{lem}
  \label{lem:third derivative integral helper lemma}
  For distinct $x, y, z > 0$, one has
  \begin{align*}
    \int_0^\infty \frac{1}{(x + \lambda)(y + \lambda)(z + \lambda)} d\lambda
    = \frac{z (\log(x) - \log(y)) + y (\log(z) - \log(x)) + x (\log(y) - \log(z))}{(x-y)(y-z)(x-z)}.
  \end{align*}
  % and
  % \begin{align*}
  %   \int_0^\infty \frac{\lambda}{(x + \lambda)(y + \lambda)(z + \lambda)} d\lambda
  %   = \frac{x\log(x) (y-z) + y \log(y) (z-x) + z \log(z) (x-y)}{(x-y)(y-z)(x-z)}.
  % \end{align*}
\end{lem}
\begin{proof}
  One can deduce from a partial fraction decomposition that
  \begin{equation*}
    \frac{(x-y)(y-z)(x-z)}{(x+\lambda)(y + \lambda)(z + \lambda)} = \frac{y-z}{x + \lambda} + \frac{z-x}{y+\lambda} + \frac{x-y}{z+\lambda},
  \end{equation*}
  and the latter integrates to
  \begin{align*}
    -\int_0^\infty \frac{y-z}{x + \lambda} + \frac{z-x}{y+\lambda} + \frac{x-y}{z+\lambda} \, d\lambda & = \int_0^\infty (y-z) \left( \frac{1}{1 + \lambda} - \frac{1}{x + \lambda} \right) d\lambda  \\
    & + \int_0^\infty (z-x) \left( \frac{1}{1 + \lambda} - \frac{1}{y + \lambda} \right) d\lambda  \\
    & + \int_0^\infty (x-y) \left( \frac{1}{1 + \lambda} - \frac{1}{z + \lambda} \right) d\lambda \\
    & = (y-z) \log(x) + (z-x) \log(y) + (x-y) \log(z). \qedhere
  \end{align*}
  % For the second integral one can use a similar argument with the decomposition
  % \begin{equation*}
  %   -\frac{(x-y)(y-z)(x-z) \lambda}{(x+\lambda)(y + \lambda)(z + \lambda)} = \frac{x(y-z)}{x + \lambda} + \frac{y(z-x)}{y+\lambda} + \frac{z(x-y)}{z+\lambda}. \qedhere
  % \end{equation*}
\end{proof}
For convenience we will use the following notation.
Define~$H\colon \R_{>0}^2 \to \RR$ by
\begin{equation}
  \label{eq:H defn}
  H(x, y) = \frac{(x+y) \log(x/y)}{x-y},
\end{equation}
if~$x, y > 0$ are distinct, and
\begin{align}\label{eq:H special case}
  H(x, x) = 2.
\end{align}
Next, we define~$T\colon \R_{>0}^3 \to \RR$ by
\begin{equation}
  \label{eq:T defn}
  T(x, y, z) = \frac{x + y}{x - y} \left( \frac{x + z}{x - z} \log(x / z) - \frac{y + z}{y - z} \log(y / z)  \right),
\end{equation}
for distinct~$x, y, z > 0$.
Then~$T$ extends to a continuous function on~$\R_{>0}^3$, such that
\begin{equation}
  \label{eq:T special cases}
  \begin{aligned}
    T(x, x, z) & = \frac{2 x^2 - 2 z^2 - 4 x z \log(x/z)}{(x - z)^2}, \\
    T(x, y, x) & = \frac{2 x^2 - 2 y^2 - (x + y)^2 \log(x/y)}{(x - y)^2}, \\
    % & = \frac{x + y}{x - y} \left( 2 - \frac{y + x}{y - x} \log(y/x) \right), \\
    T(x, x, x) & = 0.
  \end{aligned}
\end{equation}
Furthermore, $T(x, y, z)$ is symmetric in~$x$ and~$y$, for every~$c > 0$ satisfies~$T(cx, cy, cz) = T(x, y, z)$, and~$T(x^{-1}, y^{-1}, z^{-1}) = - T(x,y,z)$.
Then we have the following proposition.

\begin{prop}
  \label{prop:pdn explicit derivative expressions}
  Let $f(Q) = d(Q, I)^2$ and $U, W \in T_Q \PD(n)$.
  Then for $Q = \diag(q_1, \dotsc, q_n)$, and $\tilde U = Q^{-1/2} U Q^{-1/2}$, $\tilde W = Q^{-1/2} W Q^{-1/2}$, one has
  \begin{align}
    \label{eq:pdn explicit hess}
    (\nabla^2 f)_Q(U, U) & = \sum_{k, l = 1}^n \abs{\tilde U_{kl}}^2 H(q_k, q_l),
                         % = \sum_{k, l = 1}^n \abs{\tilde U_{kl}}^2 \frac{(q_k + q_l) \log(q_k / q_l)}{q_k - q_l},
\\
    \nonumber
    (\nabla^3 f)_Q(W, U, U) & = \sum_{k, l, m=1}^n \tilde W_{kl} \tilde U_{lm} \tilde U_{mk} T(q_k, q_l, q_m)
    % & = \sum_{k, l, m=1}^n \tilde W_{kl} \tilde U_{lm} \tilde U_{mk} \left(\frac{q_k + q_l}{q_k - q_l} \left(\frac{q_k + q_m}{q_k - q_m} \log (q_k / q_m) - \frac{q_l + q_m}{q_l - q_m} \log (q_l / q_m)\right)\right).
  \end{align}
  where~$H\colon \R_{>0}^2 \to \RR$ and~$T\colon \R_{>0}^3 \to \RR$ are defined in \cref{eq:H defn,eq:H special case,eq:T defn,eq:T special cases}, and the subscripts refer to the respective matrix entries.
\end{prop}
\begin{proof}
  The formula for the Hessian of~$f$ was already derived in \cref{eq:pdn second deriv expression}.
  For the third derivative, one can evaluate the trace in~\cref{thm:derivative expressions} as
  \begin{align*}
    \Tr \left[ \tilde W Q (Q+\lambda)^{-1} \tilde U (Q+\lambda)^{-1} \tilde U \right] & = \sum_{k,l,m} \tilde W_{kl} \frac{q_l}{q_l + \lambda} \tilde U_{lm} \frac{1}{q_m + \lambda} \tilde U_{mk}, \\
    \Tr \left[ \tilde W \tilde U (Q+\lambda)^{-1} \tilde U (Q+\lambda)^{-1} Q \right] & = \sum_{k,l,m} \tilde W_{kl} \tilde U_{lm} \frac{1}{q_m + \lambda} \tilde U_{mk} \frac{q_k}{q_k + \lambda}, \\
    \Tr \left[ \tilde W (Q + \lambda)^{-1} \{ \tilde U, Q \} (Q+\lambda)^{-1} \tilde U (Q+\lambda)^{-1} Q \right] & = \sum_{k,l,m} \tilde W_{kl} \tilde U_{lm} \frac{q_l + q_m}{(q_l + \lambda)(q_m + \lambda)} \tilde U_{mk} \frac{q_k}{q_k + \lambda}, \\
    \Tr \left[ \tilde W Q (Q + \lambda)^{-1} \tilde U (Q + \lambda)^{-1} \{ \tilde U, Q \} (Q + \lambda)^{-1} \right] & = \sum_{k,l,m} \tilde W_{kl} \frac{q_l}{q_l + \lambda} \tilde U_{lm} \frac{q_m + q_k}{(q_m + \lambda)(q_k + \lambda)} \tilde U_{mk},
  \end{align*}
  so that the third derivative satisfies
  \begin{align*}
    & (\nabla^3 f)_Q(W, U, U) \\
    & = \int_0^\infty d\lambda \, \sum_{k,l,m} \tilde W_{kl} \tilde U_{lm} \tilde U_{mk} \left( \frac{q_k}{(q_k + \lambda)(q_m + \lambda)} \left( 1 - \frac{q_l + q_m}{q_l + \lambda} \right) + \frac{q_l}{(q_l + \lambda)(q_m + \lambda)} \left( 1 - \frac{q_k + q_m}{q_k + \lambda} \right) \right).
  \end{align*}
  Using \cref{eq:pdn second deriv expression integral helper} and \cref{lem:third derivative integral helper lemma}, this integrates to (interpreting expressions as limits whenever not all~$q_k, q_l, q_m$ are distinct)
  \begin{align*}
    & \sum_{k,l,m} \tilde W_{kl} \tilde U_{lm} \tilde U_{mk} \left( \frac{q_k \log(q_k / q_m)}{q_k - q_m} + \frac{q_l \log(q_l / q_m)}{q_l - q_m} \right) \\
    & - \sum_{k,l,m} \tilde W_{kl} \tilde U_{lm} \tilde U_{mk} \frac{(q_k (q_l + q_m) + q_l (q_k + q_m)) (q_m \log(q_k / q_l) + q_l \log(q_m / q_k) + q_k \log(q_l / q_m))}{(q_k - q_l)(q_l - q_m)(q_k - q_m)} \\
    & = \sum_{k,l,m} \tilde W_{kl} \tilde U_{lm} \tilde U_{mk} \frac{q_k + q_l}{q_k - q_l} \left( \frac{(q_k + q_m)(q_l - q_m) \log(q_k / q_m) - (q_l + q_m)(q_k - q_m) \log(q_l / q_m)}{(q_l - q_m)(q_k - q_m)} \right) \\
    & = \sum_{k,l,m} \tilde W_{kl} \tilde U_{lm} \tilde U_{mk} \frac{q_k + q_l}{q_k - q_l} \left(\frac{q_k + q_m}{q_k - q_m} \log(q_k / q_m) - \frac{q_l + q_m}{q_l - q_m} \log(q_l / q_m)\right) \\
    & = \sum_{k,l,m} \tilde W_{kl} \tilde U_{lm} \tilde U_{mk} T(q_k, q_l, q_m),
  \end{align*}
  which is exactly the desired expression for the third derivative.
\end{proof}
We note here that \cref{prop:pdn explicit derivative expressions} can be used to verify that the squared distance is~$2$-strongly convex, which is a general property of Hadamard manifolds as mentioned before.
Indeed,~$\norm{U}_Q = \normHS{\tilde U}$ by definition of the Riemannian metric, so one has to show that $(\nabla^2 f)_Q(U, U) \geq 2 \normHS{\tilde U}^2$.
In view of \cref{eq:pdn explicit hess}, it suffices to prove that $H(x,y)\geq2$.
This follows directly from the logarithmic-arithmetic mean inequality:
for every~$x, y > 0$, one has
\begin{equation}\label{eq:log vs arith}
  % 0 <
  \frac{x - y}{\log(x) - \log(y)} \leq \frac{x + y}{2},
\end{equation}
where the quantity~$(x-y) / (\log(x) - \log(y))$ is known as the \emph{logarithmic mean} of~$x$ and~$y$ (it is defined as~$x$ when $x=y$).
It is known to be inbetween the geometric and arithmetic mean of~$x$ and~$y$~\cite{carlsonLogarithmicMean1972}.
A short proof of \cref{eq:log vs arith} is as follows.
Assume without loss of generality that~$x < y$; then the lower bound of the Hermite--Hadamard inequality applied to the function~$z \mapsto 1/z$ yields
\begin{equation*}
  \frac{\log(y) - \log(x)}{y - x} = \frac{1}{y - x} \int_x^y \frac{1}{z} \, dz \geq \left( \frac{x+y}{2} \right)^{-1}.
\end{equation*}
One can also reverse this strategy: $\PD(n)$ is a Hadamard manifold, hence the squared distance is~$2$-strongly convex, which in turn implies the logarithmic-arithmetic mean inequality.
It would be interesting to understand whether there is a more direct relation between the logarithmic-arithmetic mean inequality and the $2$-strong-convexity of the squared distance, for instance via midpoint-strong-convexity considerations.

We now study the coefficients appearing in \cref{prop:pdn explicit derivative expressions} to show that the squared distance is self-concordant on~$\PD(n)$.
Let~$a = \log(q_k / q_m)$ and~$b = \log(q_l / q_m)$.
Then
\begin{equation*}
  T(q_k, q_l, q_m) = \coth ( (a-b)/2 )  \left( a \coth ( a/2 ) - b \coth(b/2) \right),
\end{equation*}
whereas the square root of the product of the coefficients of~$\abs{\tilde W_{kl}}^2$, $\abs{\tilde U_{lm}}^2$, and $\abs{\tilde U_{mk}}^2$ in~$\nabla^2 f$ is
\begin{equation*}
  \sqrt{H(q_k, q_l) H(q_l, q_m) H(q_k, q_m)} = \sqrt{a b (a-b) \coth(a/2) \coth(b/2) \coth((a-b)/2)}.
\end{equation*}
\begin{lem}
  \label{lem:funky inequality}
  The constant~$C = \sqrt{2}$ is such that for all~$a, b \in \RR$, one has
  \begin{equation*}
    \abs*{\coth ((a-b)/2)  \left( a \coth ( a/2 ) - b \coth(b/2) \right)} \leq C \sqrt{a b (a-b) \coth(a/2) \coth(b/2) \coth((a-b)/2)}.
  \end{equation*}
  As a consequence, for all~$x, y, z > 0$, we have
  \begin{equation}
    \label{eq:funky inequality}
    \abs{T(x,y,z)} \leq C \sqrt{H(x,y) H(y,z) H(x,z)}.
  \end{equation}
\end{lem}
\begin{rem}
  \label{rem:conjectured inequality improvement}
  We conjecture, based on numerical evidence, that the optimal constant in the above inequality is~$C = 1/\sqrt{2}$.
  Let~$A(x, y) = (x+y)/2$ and~$G(x, y) = \sqrt{xy}$ be the arithmetic and geometric mean, respectively.
  The inequality for~$C = 1/\sqrt{2}$ is equivalent to the following ``reverse arithmetic-geometric mean inequality'': for all~$a, b \in \RR$,
  \begin{equation*}
    \frac{A(a^2 \coth(a)^2, b^2 \coth(b)^2)}{G(a^2 \coth(a)^2, b^2 \coth(b)^2)} \leq 1 + \frac{(a-b) \tanh(a-b)}{2}.
  \end{equation*}
\end{rem}
\begin{proof}[Proof of \cref{lem:funky inequality}]
  Consider~$h(x) = x \coth(x/2)$.
  Then~$h$ is~$1$-Lipschitz: its derivative is given by
  \begin{align*}
    \partial_x h(x) % & = \frac{\cosh(x/2)}{\sinh(x/2)} - \frac{x}{2 \sinh(x/2)^2} \\
          % & = \frac{\cosh(x/2) \sinh(x/2)}{\sinh(x/2)^2} - \frac{x}{\cosh(x) - 1} \\
          % & = \frac{\sinh(x)}{\cosh(x) - 1} - \frac{x}{\cosh(x) - 1} \\
          & = \frac{\sinh(x) - x}{\cosh(x) - 1}.
  \end{align*}
  It is clear that~$\abs{\sinh(x) - x} \leq \cosh(x) - 1$: for~$x \geq 0$, the difference is~$\cosh(x) - 1 - (\sinh(x) - x) = x + e^{-x} - 1$, which is convex and has zero derivative at~$x=0$, where it evaluates to~$0$.
  For~$x \leq 0$, the difference is~$\cosh(x) - 1 + \sinh(x) - x = e^x - x - 1 \geq 0$.

  We rewrite the left- and right-hand sides of the inequality:
  \begin{align*}
    \coth ((a-b)/2)  \left( a \coth ( a/2 ) - b \coth(b/2) \right) = \frac{h(a-b) (h(a) - h(b))}{a-b}
  \end{align*}
  and
  \begin{align*}
    \sqrt{a b (a-b) \coth(a/2) \coth(b/2) \coth((a-b)/2)} = \sqrt{h(a) h(b) h(a-b)}.
  \end{align*}
  Therefore it suffices to prove that
  \begin{align*}
    \abs*{\frac{h(a) - h(b)}{a - b}} \leq C \sqrt{\frac{h(a) h(b)}{h(a-b)}}.
  \end{align*}
  Because~$h$ is~$1$-Lipschitz, the left-hand side is at most~$1$.

  We now claim that the following lower- and upper bounds on~$h$ hold: $h(x) \geq 1 + \frac{\abs{x}}{2}$, and~$h(x) \leq 2 + \abs{x}$.
  The upper bound follows from~$h$ being~$1$-Lipschitz and~$h(0) = 2$.
  For the lower bound, we restrict to~$x \geq 0$, in which case it suffices to prove~$x \cosh(x/2) \geq (1 + x/2) \sinh(x/2)$.
  This is simple: we have~$x \cosh(x/2) \geq 2 \sinh(x/2)$ (by a power series comparison for~$x \cosh(x)$ and~$\sinh(x)$), and~$x \cosh(x/2) \geq x \sinh(x/2)$ since~$\cosh(x/2) \geq \sinh(x/2)$.
  Therefore~$x \cosh(x/2)$ is greater than their average.

  We now finish up the argument: we have
  \begin{align*}
    \frac{h(a) h(b)}{h(a-b)} \geq \frac{1 + \frac{\abs{a} + \abs{b}}{2} + \frac{\abs{ab}}{4}}{2 + \abs{a} + \abs{b}} \geq \frac{1}{2},
  \end{align*}
  so we conclude that
  \begin{align*}
    C \sqrt{\frac{h(a) h(b)}{h(a-b)}} \geq \frac{C}{\sqrt{2}} \geq 1 \geq \frac{h(a) - h(b)}{a - b}.
  \end{align*}
  holds for~$C = \sqrt{2}$.
\end{proof}
This directly implies that the squared distance is self-concordant (with an~$n$-independent constant), hence also proving \cref{thm:distsq pdn self-concordant}.
\begin{thm}
  \label{thm:distsq pdn self-concordant expanded}
  Let~$C \geq 0$ be such that the inequality in \cref{lem:funky inequality} holds.
  Then the function~$f\colon \PD(n) \to \RR$ defined by~$f(Q) = d(Q, I)^2$ satisfies for~$Q \in \PD(n)$ and~$U, W \in T_Q \PD(n)$ the inequality
  \begin{align*}
    \abs*{(\nabla^3 f)_Q(W, U, U)} & \leq C \sqrt{(\nabla^2 f)_Q(W, W)} \, (\nabla^2 f)_Q(U, U)
  \end{align*}
  In particular, from the choice~$C = \sqrt{2}$ it follows that~$f$ is~$2$-self-concordant.
\end{thm}
\begin{proof}
  By \cref{eq:funky inequality} and consecutive applications of Cauchy--Schwarz, we have
  \begin{align*}
    & \abs*{(\nabla^3 f)_Q(W, U, U)} \\
    & \leq \sum_{k,l,m} \abs{\tilde W_{kl} \tilde U_{lm} \tilde U_{mk}} \abs{T(q_k, q_l, q_m)} \\
    & \leq C \sum_{k,l,m} \abs{\tilde W_{kl} \tilde U_{lm} \tilde U_{mk}} \sqrt{H(q_k, q_l) H(q_l, q_m) H(q_k, q_m)} \\
    & \leq C \sqrt{\sum_{k,l} \abs{\tilde W_{kl}}^2 H(q_k, q_l)} \sqrt{\sum_{k,l} \parens*{\sum_m \abs{\tilde U_{lm} \tilde U_{mk}} \sqrt{H(q_l, q_m) H(q_k, q_m)}}^2} \\
    & \leq C \sqrt{\sum_{k,l} \abs{\tilde W_{kl}}^2 H(q_k, q_l)} \sqrt{\sum_{k,l} \parens*{\sum_m \abs{\tilde U_{lm}}^2 H(q_l, q_m)} \parens*{\sum_m \abs{\tilde U_{mk}}^2 H(q_k, q_m)}} \\
    & = C \sqrt{\sum_{k,l} \abs{\tilde W_{kl}}^2 H(q_k, q_l)} \sqrt{\parens*{\sum_{l,m} \abs{\tilde U_{lm}}^2 H(q_l, q_m)}^2} \\
    & = C \sqrt{(\nabla^2 f)_Q(W, W)} (\nabla^2 f)_Q(U, U). \qedhere
  \end{align*}
\end{proof}
One can use this to construct a strongly self-concordant function on the open epigraph of the squared distance using \cref{thm:compatible function epigraph barrier construction}, hence also proving \cref{prop:hadamard distsq epigraph barrier}.
By imposing an additional upper bound on the value of the squared distance one can use this to construct a barrier for the epigraph, albeit with a distance-dependent barrier parameter; see \cref{sec:applications} for similar constructions.

\subsection{Constant negative curvature}
\label{subsec:model spaces}
In this subsection, we prove that the squared distance on~$n$-dimensional hyperbolic space~$\HH^n$ is self-concordant with a larger self-concordance parameter, and other refinements of the self-concordance estimate.
We use this to construct a barrier for the epigraph of the (squared) distance in \cref{thm:hypn dist epigraph barrier}, which is useful for our applications in \cref{sec:applications}.
Instead of dealing just with~$\HH^n$, we consider, more generally, the \emph{model spaces}~$M_{-\kappa}^n$ with constant sectional curvature~$-\kappa < 0$ (we recall that $\HH^n$ is~$M_{-1}^n$).
The main result of this subsection is the following.
\begin{thm}\label{thm:SC_hyperbolic}
 Let~$n \geq 2$,~$\kappa > 0$, set~$M = M_{-\kappa}^n$, let $p_0 \in M$, and consider~$f,g\colon M \to \RR$ defined by~$f(p) = d(p, p_0)^2$ and~$g(p) = d(p, p_0)$.
  One has the following estimates:
	\begin{enumerate}
    \item \label{item:SC_hyperbolic sc}  $\displaystyle \abs{(\nabla^3 f)_p(w, u, u)} \leq \sqrt{\frac{\kappa}{2}} \sqrt{(\nabla^2 f)_p(w,w)} (\nabla^2 f)_p(u,u)$,
  so $f$ is $\frac{8}{\kappa}$-self-concordant, and this constant cannot be improved.
		\item \label{item:SC_hyperbolic scag}  $\displaystyle \abs{(\nabla^3 f)_p(u, u, u)} \leq \sqrt{\frac{8 \kappa}{27}} ((\nabla^2 f)_p(u,u))^{3/2}$, so $f$ is $\frac{27}{2\kappa}$-self-concordant along geodesics, and this constant cannot be improved.
		\item \label{item:SC_hyperbolic compatibility} $\begin{aligned}[t]
        \abs{(\nabla^3 f)_p(w, u, u)} & \leq 2 \zeta \sqrt{\kappa} \abs{dg_p(w)} ((\nabla^2 f)_p(u,u)- 2 dg_p(u)^2) \\
                                      & \quad + 2 \sqrt{\kappa} \abs{dg_p(u)} \sqrt{(\nabla^2 f)_p(u,u) - 2 dg_p(u)^2} \sqrt{(\nabla^2 f)_p(w,w) - 2 dg_p(w)^2} \\
                                      & \leq 2 \zeta \sqrt{\kappa} \norm{w}_p (\nabla^2 f)_p(u,u) + 2 \sqrt{\kappa} \norm{u}_p \sqrt{(\nabla^2 f)_p(u,u)} \sqrt{(\nabla^2 f)_p(w,w)},
      \end{aligned}$\\ where $\WeirdConstant = \sup_{x \in \RR} \abs{\sinh(x)^{-1} - x^{-1}} \leq \frac{1}{2}$.
\end{enumerate}
\end{thm}

By \cref{lem:rescaling curvature rescales distsq self-concordance,lem:basic} it suffices to prove the above estimates for~$M = M_{-1}^n$ and then to appropriately rescale the estimate when the curvature changes.
The estimate in~\ref{item:SC_hyperbolic compatibility} is a refinement of self-concordance for~$f$ (albeit with different constants), because~$2 \norm{W}_Q^2 \leq \norm{W}_{f,Q}^2$ by the~$2$-strong-convexity of~$f$ (and in the presence of curvature, these norms can differ by a factor that scales with the distance to the base point and the curvature).
The estimate also implies that, in the terminology of \cref{subsec:compatibility}, the squared distance is compatible with~\emph{every} strongly convex function, which is relevant for computing geometric means on~$M_{-\kappa}^n$ as discussed in \cref{subsec:riemannian barycenter}.
The presence of the ``correction terms''~$-2 dg_p(u)^2$ and similar for~$w$ will also be useful for proving \cref{thm:hypn dist epigraph barrier}, which we use later for the purpose of computing geometric medians.

Before starting with the proof of~\cref{thm:SC_hyperbolic}, we provide estimates on some single-variable functions which we use.
\begin{lem}
  \label{lem:sinhx x recip and tanh x x recip diff ineq}
  \begin{enumerate}
    \item\label{item:Phix bound} Define $\Phi\colon \RR \to \RR$ by
      \begin{equation}\label{eqn:Phi}
        \Phi(x) := \partial_x (x \coth(x)) = \coth(x) + x - x \coth(x)^2, \quad x \neq 0,
      \end{equation}
      and~$\Phi(0) = 0$.
      Then~$\Phi$ is smooth, and for~$x \in \RR_{\geq 0}$, it holds that
      \begin{equation}
        0 \leq \Phi(x) \leq \min (x,1),
      \end{equation}
      and $\lim_{x \to \infty} \Phi(x) = 1$.
    \item\label{item:zeta bound} It holds that
      \begin{equation}\label{eq:recip sinh x diff}
        \zeta := \sup_{x \in \RR_{\geq 0}} \frac{\Phi(x)}{2 x \coth(x)}
        % = \frac{\coth x  + x - x \coth^2 x}{2 x \coth x}
        % = \frac{1}{2x} + \frac{1}{2 \coth x} - \frac{\coth x}{2}
        % = \frac{1}{2x} + \frac{\sinh x}{2 \cosh x} - \frac{\cosh x}{2 \sinh x}
        % = \frac{1}{2x} + \frac{\sinh^2 x - \cosh^2 x}{2 \sinh x \cosh x}
        % = \frac{1}{2x} - \frac{-1}{\sinh(2x)}
        = \sup_{x \in \RR} \abs*{\frac{1}{\sinh(x)} - \frac{1}{x}} < \frac{1}{2}.
      \end{equation}
  \end{enumerate}
\end{lem}
We note here that numerical evaluation suggests the value of $\WeirdConstant$ is approximately $0.23536$, which is slightly smaller than $\frac{1}{3 \sqrt{2}} \approx 0.23570$.
\begin{proof}
  We first prove \ref{item:Phix bound}.
  By $\sinh(x) = x + (1/3!)x^3+\cdots$ and $\cosh(x) = 1 + (1/2!)x^2 +\cdots$, and by the identities $\cosh(x)^2 - \sinh(x)^2 = 1$, $2 \cosh(x) \sinh(x) = \sinh(2x)$, and $2\sinh(x)^2 = \cosh(2x) - 1$, it holds that
  \begin{equation*}
    \Phi(x)
    = \frac{\cosh(x)}{\sinh(x)} + x (1 - \coth(x)^2)
    % = \frac{\cosh x}{\sinh x} - \frac{x}{\sinh^2 x}
    % = \frac{\sinh x \cosh x)}{\sinh^2 x} - \frac{x}{\sinh^2 x}
    =  \frac{\sinh(x) \cosh(x) -x}{\sinh(x)^2} = \frac{\sinh(2x)  - 2x}{\cosh(2x) - 1} =  \frac{(2x)^3/3! + \dotsb}{(2x)^2/2! + \dotsb}.
  \end{equation*}
  From this, we deduce that~$\Phi(x) \geq 0$ for~$x \geq 0$, and
  \[
    \lim_{x \to 0} \Phi(x) = 0 = \Phi(0).
  \]
  Therefore~$\Phi$ is continuous at~$0$.
  The above argument shows that~$\Phi$ is a ratio of the analytic functions~$\sinh(2x) - 2x$ and~$\cosh(2x) - 1$, and the continuity at~$0$ shows that~$\Phi$ has no singularity at~$0$, which is the only zero of~$\cosh(2x) - 1$; hence~$\Phi$ must in fact be smooth on~$\RR$.

  We now show that~$\Phi(x) \leq \min(x,1)$ for~$x \geq 0$.
  We have
  \[
    \lim_{x \to 0} x \coth(x) = \lim_{x \to 0} \frac{x (1 + x^2/2! + \cdots)}{x + x^3/3! + \cdots} = 1.
  \]
  By $\partial_x (x \coth x) = \Phi(x) \geq 0$ for~$x \geq 0$, we have
  \[
    x \coth(x) \geq 1.
  \]
  This implies that~$\Phi$ is nondecreasing, since
  \[
    \partial_x \Phi(x) = \frac{2 (x \coth(x) -1)}{\sinh(x)^2} \geq 0.
  \]
  Thus we have
  \[
    \sup_{x \in [0,\infty)} \Phi(x) = \lim_{x \to \infty} \Phi(x) =\lim_{x \to \infty} \coth x - x/\sinh^2 x = 1.
  \]
  Lastly, $\Phi(x) \leq x$ follows from
  \[
    x - \Phi(x) = \coth(x) \, (x \coth(x) - 1) \geq 0.
  \]

  We now prove \ref{item:zeta bound}. Observe that~$\lim_{x \to 0} \sinh(x)^{-1} - x^{-1} = 0$
  % proof:
  % \begin{align*}
  %   & \lim_{x \to 0} \left(\frac{1}{\sinh(x)} - \frac{1}{x} \right)
  %   = \lim_{x \to 0} \left(\frac{x - \sinh(x)}{x \sinh(x)} \right) \\
  %   & = \lim_{x \to 0} \left(\frac{1 - \cosh(x)}{\sinh(x) + x \cosh(x) } \right)
  %   = \lim_{x \to 0} \left(\frac{- \sinh(x)}{2 \cosh(x) + x \sinh(x) } \right) = 0
  % \end{align*}
  by two applications of L'H\^opital's rule, so~$\sinh(x)^{-1} - x^{-1}$ has a continuous extension to all of~$\RR$.
  A similar argument shows that $\coth(x) - x^{-1}$ can be continuously extended to~$x = 0$ with value~$0$.
  For both inequalities it suffices to treat the case $x > 0$.
  The inequality~$\abs{\sinh(x)^{-1} - x^{-1}} \leq \frac{1}{2}$ is equivalent to
  \begin{equation*}
    \abs*{x - \sinh(x)} = \sinh(x) - x \leq \frac{x \sinh(x)}{2}.
  \end{equation*}
  We have equality for $x = 0$, and
  \begin{equation*}
    \partial_x (\sinh(x) - x) = \cosh(x) - 1, \quad \partial_x \sinh(x) = \sinh(x) + x \cosh(x)
  \end{equation*}
  agree for $x = 0$ as well.
  Differentiating once more yields
  \begin{equation*}
     \partial_x^2 (\sinh(x) - x) = \sinh(x), \quad \partial_x^2 (x \sinh(x)) = 2 \cosh(x) + x \sinh(x).
  \end{equation*}
  Clearly, $\frac{1}{2} (2 \cosh(x) + x \sinh(x)) \geq \cosh(x) \geq \sinh(x)$, and so we have proven $\WeirdConstant \leq \frac{1}{2}$.
  % The inequality holds numerically with a factor $1/3$ instead of $1/2$.
\end{proof}

Although there are several models of $M_{-\kappa}^n$ in which explicit computations can be performed (such as~$\SPD(2, \CC)$, which is~$M_{-1/2}^3$), for proving Theorem~\ref{thm:SC_hyperbolic}, we take a  ``model-free" approach based on Jacobi fields.
For a geodesic~$\gamma\colon [0,l] \to M$, a \emph{Jacobi field} along $\gamma$ is a vector field $X = (X(t))_{t \in [0,l]}$ along~$\gamma$, where $X(t) \in T_{\gamma(t)} M$ satisfies the Jacobi equation:\footnote{The meaning of~$\nabla_{\dot\gamma(t)}$ here is slightly different from its previous meaning: instead of acting on tensor fields on an open subset of~$M$, it acts on tensor fields along the curve~$\gamma$. The two notions they agree whenever~$X(t)$ is locally the restriction of a vector field on~$M$, see~\cite[Ch.~4]{lee-riemannian-manifolds} for more information.}
\begin{equation}\label{eqn:Jacobi}
	\nabla_{\dot{\gamma}(t)}\nabla_{\dot{\gamma}(t)} X(t) + R(X(t), \dot{\gamma}(t))\dot{\gamma}(t) = 0, \quad t \in [0,l].
\end{equation}
This is a linear differential equation.
Therefore, the solution $X(t)$ is uniquely determined by the initial values $X(0), \nabla_{\dot{\gamma}(0)}X(0)$, or by its boundary values $X(0), X(l)$.
Jacobi fields are relevant to the task of differentiating the squared distance because they arise variation fields of geodesics: the distance~$d(p_0, p)$ is the minimal length of a geodesic between~$p_0$ and~$p$, and varying~$p$ leads to a \emph{family} of geodesics.
More precisely, one has the following classical result:
\begin{lem}[{see \cite[p.35, 36]{sakai-riemannian-geometry}}]
  \label{lem:Jacobi}
	Let $\alpha\colon [0,l] \times (- \epsilon, \epsilon) \to M$ be a smooth map such that the curve $t \mapsto \alpha(t,s)$ is a geodesic for each $s \in (-\epsilon,\epsilon)$.
	Then $d\alpha(t,0) (\frac{\partial}{\partial s})$ is a Jacobi field along geodesic $t \mapsto \alpha(t,0)$.
\end{lem}
It can also be shown that every Jacobi field (along a geodesic on a compact interval) arises in this way~\cite[Prop.~10.4]{lee-riemannian-manifolds}, but we will not need this fact.
The derivative and the Hessian of $p \mapsto f(p) = d(p,p_0)^2$ can be determined using Jacobi fields as follows.
\begin{lem}[{see \cite[p.108--110]{sakai-riemannian-geometry}}]\label{lem:Dd^2Hd^2}
  Let~$p, p_0 \in M$ be distinct points, let $\gamma\colon [0,l] \to M$ be the unique unit-speed geodesic with $\gamma(0) = p_0$, $\gamma(l) = p$, and $l := g(p) = d(p,p_0)$.
  For $u \in T_p M$, it holds that:
  \begin{enumerate}
    \item\label{item:Dd^2Hd^2 dg}  $dg_p(u) = \langle \dot \gamma(l),u\rangle_p$,
    \item\label{item:Dd^2Hd^2 df} $df_p(u) = 2 l \langle \dot \gamma(l),u\rangle_p$, and
    \item\label{item:Dd^2Hd^2 Hf} $(\nabla^2 f)_p (u,u) = 2 l \, \langle \nabla_{\dot{\gamma}(l)}X(l),u \rangle_p$, where $X$ is the Jacobi field along $\gamma$ under the boundary condition
      \[
        X(0) = 0, \quad X(l) = u.
      \]
  \end{enumerate}
\end{lem}
Note that~\ref{item:Dd^2Hd^2 dg} and~\ref{item:Dd^2Hd^2 df} are reformulations of~\cref{eq:dist grad formulas}, and in light of~\cref{eq:distsq hess lower bound}, \ref{item:Dd^2Hd^2 Hf} is essentially a claim about~$(\nabla^2 g)_p$.

We shall use the following fact about spaces of constant curvature~$-\kappa$~\cite[Lem.~II.3.3]{sakai-riemannian-geometry}: their Riemann curvature tensor~$R$ satisfies
\begin{equation}\label{eqn:R(XY)Z_kappa}
  R(X,Y)Z = - \kappa (\langle Y,Z \rangle X - \langle X,Z\rangle Y),
\end{equation}
where we recall that~$\braket{\cdot, \cdot}$ is the Riemannian metric.
This allows one to explicitly write down the solutions of the Jacobi equation, as given in the following lemma.
While this, and explicit expressions for the Hessian of the (squared) distance are well-known (see e.g.~\cite[p.~136, p.~154]{sakai-riemannian-geometry} or~\cite[Prop.~10.12, Prop.~11.3]{lee-riemannian-manifolds}), we provide a proof for completeness.
\begin{lem}\label{lem:hyperbolic jacobi field equation}
  Let~$p, p_0 \in M = \HH^n$ with~$p \neq p_0$, and let~$\gamma\colon [0,l] \to M$ be the unit-speed geodesic from~$p_0$ to~$p$ with~$l := g(p) = d(p, p_0)$.
  Let~$u \in T_p M$ and decompose~$u = u^\top + u^\perp$ such that~$u^\top = \braket{u, \dot\gamma(l)}_{p} \dot\gamma(l)$ is the part of~$u$ parallel to~$\dot\gamma(l)$, and~$u^\perp$ orthogonal to~$\dot\gamma$, i.e., $\braket{u^\perp, \dot\gamma(l)}_p = 0$.
  Then the unique Jacobi field~$X(t)$ along~$\gamma$ with~$X(0) = 0$ and~$X(l) = u$ satisfies
  \begin{align*}
    X(t) & = \frac{t}{l} \tau_{\gamma,t-l} u^\top + \frac{\sinh(t)}{\sinh(l)} \, \tau_{\gamma,t-l}u^\perp,
  \end{align*}
  where~$\tau_{\gamma,t-l}\colon T_{\gamma(l)} M \to T_{\gamma(t)} M$ is the parallel transport along~$\gamma$.
\end{lem}
\begin{proof}
  It is clear that~$X(l) = u$ and~$X(0) = 0$.
  Therefore it remains to check that~$X$ is a Jacobi field: we have
  \begin{equation*}
    \nabla_{\dot\gamma(t)} X(t) = \frac{1}{l} \tau_{\gamma,t-l} u^\top + \frac{\cosh(t)}{\sinh(l)} \tau_{\gamma, t-l} u^\perp
  \end{equation*}
  % Some additional justification: \tau_{\gamma,t-l} u^\top is parallel along \gamma, so it's covariant derivative along \gamma vanishes, and one only has contributions from the product rule.
  and
  \begin{equation*}
    \nabla_{\dot\gamma(t)} \nabla_{\dot\gamma(t)} X(t) = \frac{\sinh(t)}{\sinh(l)} \tau_{\gamma, t-l} u^\perp.
  \end{equation*}
  From~\cref{eqn:R(XY)Z_kappa} it follows that
  \begin{equation*}
    R(X(t), \dot\gamma(t)) \dot\gamma(t)
    = - [X(t) - \braket{X(t), \dot\gamma(t)}_{\gamma(t)} \dot\gamma(t)].
  \end{equation*}
  Therefore
  \begin{align*}
    \nabla_{\dot\gamma(t)} \nabla_{\dot\gamma(t)} X(t) + R(X(t), \dot\gamma(t)) \dot\gamma(t)
    & = \frac{\sinh(t)}{\sinh(l)} \tau_{\gamma, t-l} u^\perp - X(t) + \braket{X(t), \dot\gamma(t)}_{\gamma(t)} \dot\gamma(t) \\
    & = - \frac{t}{l} \tau_{\gamma,t-l} u^\top + \braket{X(t), \dot\gamma(t)}_{\gamma(t)} \dot\gamma(t) \\
    & = - \frac{t}{l} \braket{u, \dot\gamma(l)}_{p} \dot\gamma(t) + \braket{X(t), \dot\gamma(t)}_{\gamma(t)} \dot\gamma(t) \\
    & = 0,
  \end{align*}
  where the penultimate equality follows from~$u^\top = \braket{u, \dot\gamma(l)}_{\gamma(l)} \dot\gamma(l)$ and~$\tau_{\gamma,t-l} \dot\gamma(l) = \dot\gamma(t)$, and the last equality follows from~$\tau_{\gamma,t-l}$ being an isometry and~$\braket{u,\dot\gamma(l)} = \braket{u^\top,\dot\gamma(l)}$.
\end{proof}
Using this description of the Jacobi fields leads to the following description of the Hessian, and the third covariant derivative of the squared distance.
\begin{prop}\label{prop:nablaHf_hyperbolic}
  Let $p, p_0 \in M = \HH^n$ with $p \neq p_0$ and let $\gamma\colon [0,l] \to M$ be the unique geodesic from $p_0$ to $p$ with $l := g(p) =d(p,p_0)$.
  Then~$f(p) = d(p, p_0)^2$ satisfies
	\begin{align}
			(\nabla^2 f)_p (u,u) & = 2 (l \coth l) \left(\langle u,u \rangle_p - \langle u, \dot{\gamma}(l) \rangle_p^2 \right)  +  \langle u, \dot{\gamma}(l)\rangle_p^2,  \label{eqn:Hd^2(uu)}\\
			(\nabla^3 f)_p (w,u,u) & = 2 \Phi(l) \langle w, \dot \gamma (l) \rangle_p \left( \langle u, u \rangle_p - \langle u, \dot\gamma(l)\rangle_p^2 \right)  \nonumber \\
		& + 4  \left( l - \Phi(l) \right)
		\langle u, \dot\gamma(l)\rangle_p \left( \langle w, \dot \gamma (l) \rangle_p \langle u, \dot\gamma(l)\rangle_p - \langle u, w \rangle_p \right).
	\end{align}
\end{prop}
\begin{proof}
  By~\cref{lem:hyperbolic jacobi field equation}, the Jacobi field~$X(t)$ along~$\gamma$ with~$X(0) = 0$ and~$X(l) = u$ satisfies
  \begin{equation*}
    X(t) = \frac{t}{l} \tau_{\gamma,t-l} u^\top + \frac{\sinh(t)}{\sinh(l)} \tau_{\gamma,t-l} u^\perp
  \end{equation*}
  where~$u = u^\top + u^\perp$ is a decomposition with~$u^\top = \braket{u, \dot\gamma(l)}_{p} \dot\gamma(l)$ parallel and~$u^\perp = u - u^\top$ orthogonal to~$\dot\gamma(l)$, respectively.
  Therefore
  \begin{equation}
    \label{eqn:nabla_X(l)}
    \nabla_{\dot\gamma(l)} X(l) = \frac{1}{l} u^\top + \frac{\cosh(l)}{\sinh(l)} u^\perp = \frac{1}{l} \braket{u, \dot\gamma(l)}_{p} \dot\gamma(l) + \frac{\cosh(l)}{\sinh(l)} (u - u^\top)
  \end{equation}
  Now apply~\cref{lem:Dd^2Hd^2}\ref{item:Dd^2Hd^2 Hf} to obtain \cref{eqn:Hd^2(uu)}.

	Consider the geodesic $s \mapsto c(s) := \Exp_p(s w)$.
	Let $\gamma_{s}\colon [0, l] \to M$ be the geodesic from $p$ to $c(s)$ (not necessarily parametrized by the arc-length).
	For $s \in (-\epsilon, \epsilon)$, let $l_s := d(c(s),p_0)$ and $u_s := \tau_{c,s}u$.
  Applying~\cref{eqn:Hd^2(uu)} to the reparametrized geodesic $t \mapsto \gamma_s((l /l_s)t)$ $(t \in [0,l_s])$, we obtain
	\begin{equation}\label{eqn:u_s}
    (\nabla^2 f)_{c(s)}(u_s,u_s) =
    2(l_s \coth(l_s)) \langle u_s,u_s \rangle +  2 \left(1 - l_s \coth(l_s)\right) (l/l_s)^2 \langle u_s, \dot\gamma_s(l)\rangle^2.
	\end{equation}
  By \cref{eq:deriv via transport}, the covariant derivative $(\nabla^3 f)_p(w, u,u)$ is obtained by computing the~$s$-derivative of \cref{eqn:u_s} at~$s = 0$.
  We use that
  \[
    \partial_{s = 0} l_s  = \langle \dot \gamma (l), w \rangle,
    \quad
    \partial_{s = 0} \langle u_s, u_s\rangle = 0,
    \quad
    \partial_{s = 0} \langle u_s, \dot \gamma_s(l) \rangle = \langle u, \left. \nabla_{\dot c(s)} \dot \gamma_s(l) \right|_{s = 0} \rangle,
  \]
  where the first equality follows from~\cref{lem:Dd^2Hd^2}\ref{item:Dd^2Hd^2 dg}, and the other two follow from $X \langle Y,Z\rangle = \langle \nabla_X Y,Z\rangle + \langle Y, \nabla_X Z\rangle$ and $\nabla_{\dot c(s)} u_s = 0$.
  Hence we have
  \begin{align}
    (\nabla^3 f)_p(w,u,u) & = 2 \Phi(l) \langle \dot \gamma (l), w \rangle \langle u, u \rangle + 2 \left( - \Phi(l) - 2/l + 2 \coth l \right)	\langle \dot \gamma (l), w \rangle \langle u, \dot\gamma(l)\rangle^2 \nonumber \\
                          & \quad + 4 \left(1 - l \coth l \right) \langle u, \dot\gamma(l) \rangle \langle u, \left. \nabla_{\dot c(s)} \dot \gamma_s(l) \right|_{s = 0}  \rangle \nonumber \\
                          \begin{split}
                          & = 2 \Phi(l) \langle \dot \gamma (l), w \rangle (\langle u, u \rangle -  \langle u, \dot \gamma (l) \rangle^2) \\
                          &\quad +  4 \left(1 - l \coth l \right) \langle u, \dot\gamma(l) \rangle \, [ \langle u, \left. \nabla_{\dot c(s)} \dot \gamma_s(l) \right|_{s = 0}\rangle  - \langle w, \dot \gamma (l) \rangle \langle u, \dot\gamma(l)\rangle/l  ].
                          \end{split}
                          \label{eqn:nabla_vf(uu)}
  \end{align}
	To compute $\nabla_{\dot c(s)} \dot \gamma_s(l) |_{s = 0}$, consider the (smooth) map $\alpha\colon [0,l] \times (-\epsilon, \epsilon) = M$ given by $(t,s) \mapsto \gamma_s(t)$.
  Let $\frac{\partial \alpha}{\partial s}(t,s) := d\alpha_{(t,s)}(\frac{\partial}{\partial t})$ and $\frac{\partial \alpha}{\partial t}(t,s) := d\alpha_{(t,s)}(\frac{\partial}{\partial s})$.
  Then $\left. \nabla_{\dot c(s)} \dot \gamma_s(l) \right|_{s = 0} = \nabla_{\frac{\partial \alpha}{\partial s}} \frac{\partial \alpha}{\partial t}(l,0) =\nabla_{\frac{\partial \alpha}{\partial t}} \frac{\partial \alpha}{\partial s} (l,0)$, since $\nabla_{\frac{\partial \alpha}{\partial s}} \frac{\partial \alpha}{\partial t} = \nabla_{\frac{\partial \alpha}{\partial t}} \frac{\partial \alpha}{\partial s} + [\frac{\partial \alpha}{\partial s}, \frac{\partial \alpha}{\partial t}]$ and $[\frac{\partial \alpha}{\partial s}, \frac{\partial \alpha}{\partial t}] = d\alpha ([\frac{\partial}{\partial s}, \frac{\partial}{\partial t}]) = 0$; see \cite[Lem.~II.2.2]{sakai-riemannian-geometry} or~\cite[Lem.~6.2]{lee-riemannian-manifolds}.
  By~\cref{lem:Jacobi}, $Y(t) := \frac{\partial \alpha}{\partial s} (t,0)$ is a Jacobi field along the geodesic $\gamma$, and satisfies $Y(0) = 0$ and $Y(l) = w$.
  Therefore~\cref{eqn:nabla_X(l)} yields
  \begin{equation*}
    \left. \nabla_{\dot c(s)} \dot \gamma_s(l) \right|_{s = 0}
      = \nabla_{\dot\gamma(l)} Y(l)
      = \frac{1}{l} \dot\gamma(l)\langle w, \dot\gamma(l)  \rangle
      + \coth(l) (w - \dot \gamma(l)\langle w, \dot\gamma(l)\rangle).\label{eqn:nabla_dot_cs}
  \end{equation*}
  By substituting this into~\cref{eqn:nabla_vf(uu)}, we obtain~\cref{eqn:proof_nabla_vHd^2(uu)}.
\end{proof}

We are now ready to prove Theorem~\ref{thm:SC_hyperbolic}.
\begin{proof}[Proof of Theorem~\ref{thm:SC_hyperbolic}]
  We first restrict to the case~$\kappa = -1$.
  We are going to bound
  \[
    \sigma_p(u,w) := \frac{\abs{(\nabla^3 f)_p(w, u,u)}}{\sqrt{(\nabla^2f)_p (w,w)} (\nabla^2 f)_p(u,u)}, \quad u,v \in T_pM \setminus \{0\}.
  \]
  From $d(p,p_0) = l$, it holds that $\norm{\dot \gamma(l)} = 1$.
  We can also assume that $\norm{u}_p = \norm{w}_p = 1$.
  Therefore, $u,v, \dot\gamma(l)$ can be assumed to be unit vectors in~$\RR^3$, and represented in the spherical coordinate system as
  $\dot \gamma(l) = (0, 0, 1)$,
  $u = (\sin \theta, 0,\cos \theta)$,
  $w = (\sin \varphi \cos \alpha, \sin \varphi \sin \alpha,\cos \varphi)$
  for $\theta, \varphi \in [0, \pi]$ and $\alpha \in [0,2\pi]$.
  By~\cref{prop:nablaHf_hyperbolic}, we have
  \begin{align}
    (\nabla^2 f)_p(w,w) & = 2 \cos^2 \varphi + 2 l \coth l  \sin^2 \varphi,\label{eqn:proof_Hd^2(vv)}\\ 
    (\nabla^2 f)_p (u,u) &= 2 \cos^2 \theta +  2 l \coth l \sin^2 \theta, \label{eqn:proof_Hd^2(uu)}\\
    (\nabla^3 f)_p(w, u,u) % & = 2 \Phi(l) \langle w, \dot \gamma (l) \rangle_p \left( \langle u, u \rangle_p - \langle u, \dot\gamma(l)\rangle_p^2 \right)  \nonumber \\
                   % & + 4  \left( l - \Phi(l) \right) \langle u, \dot\gamma(l)\rangle_p \left( \langle w, \dot \gamma (l) \rangle_p \langle u, \dot\gamma(l)\rangle_p - \langle u, w \rangle_p \right) \nonumber \\
                   % & = 2 \Phi(l) \cos \varphi \sin^2 \theta  \nonumber \\
                   % & + 4 \left( l - \Phi(l) \right) \cos \theta \left( \cos \theta \cos \varphi - \sin \theta \sin \varphi \cos \alpha - \cos \theta \cos \varphi \right) \nonumber \\
                     & = 2 \Phi(l) \cos \varphi \sin^2 \theta + 4 \left( l - \Phi(l) \right) \cos \theta \sin \varphi \sin \theta (- \cos \alpha ). \label{eqn:proof_nabla_vHd^2(uu)}
  \end{align}
  %If $l = 0$, then $\nabla^2 f(u,u) = \nabla^2 f(v,v) = 1$ and $|\nabla^3 f(v,u,u)| \leq 1$; the claim is true.
  %
  %Suppose $l \neq 0$.

  By~\cref{lem:sinhx x recip and tanh x x recip diff ineq}\ref{item:Phix bound} the quantities $\Phi(l)$, $l - \Phi(l)$, $\sin(\theta)$, and $\sin(\varphi)$ in \cref{eqn:proof_nabla_vHd^2(uu)} are all non-negative.
  %	For fixed $\varphi, \theta$, the maximum of the absolute value is attained by $\alpha = 0$ or $\pi$ according to the sign of $\cos \varphi \cos \theta$.
	Thus
	\begin{equation}
    \abs{(\nabla^3 f)_p(w,u,u)}
    \leq 2 \Phi(l) \abs{\cos(\varphi)} \sin(\theta)^2
    + 4 \left( l - \Phi(l) \right) \sin(\varphi) \sin(\theta) \abs{\cos(\theta)}.
	\end{equation}
	For $C := l \coth(l) \geq 1$, observe that
  \begin{eqnarray*}
  &&  \max_{\phi \in [0,\pi]} \frac{\abs{\cos \phi}}
  {\sqrt{\cos^2 \phi + C \sin^2 \phi}} = 1,\quad
  \max_{\phi \in [0,\pi]} \frac{\sin \phi}
  {\sqrt{\cos^2 \phi + C \sin^2 \phi}} = \frac{1}{\sqrt{C}}, \\
  && \max_{\theta \in [0,\pi]} \frac{\sin \theta \abs{\cos \theta}}
  {\cos^2 \theta + C \sin^2 \theta}  =
  \max_{\theta \in [0,\pi]}\frac{\abs{\tan \theta}}{1+ C \tan^2 \theta}
  =  \max_{z \in [0,\infty)} \frac{z}{1+Cz^2}=
  \frac{1}{2\sqrt{C}}.
  \end{eqnarray*}
  %The first three can be seen from $\cos^2 \phi + C \sin^2 \phi = 1+ (C-1) \sin^2 \phi \in [1,C]$. 
  %The last one can be seen from  $\max_{\theta \in [0,\pi]}|\tan \theta|/(1+ C \tan^2 \theta) 
  %= \max_{z \in [0,\infty)} z /(1+Cz^2)$, where the maximum is attained at $z =1/\sqrt{C}$.
  Therefore
	\begin{align*}
    \sigma_p(u,w) & \leq \max_{\varphi, \theta \in [0,\pi]} \frac{2 \Phi(l) \abs{\cos \varphi} \sin^2 \theta + 4(l - \Phi(l))\sin \varphi \sin \theta \abs{\cos \theta}}{\sqrt{2\cos^{2} \varphi + 2 C \sin^2\varphi} \left(2 \cos^2 \theta + 2 C \sin^2 \theta \right)} \\ 
                  &\leq \frac{\Phi(l)}{\sqrt{2} C} + \frac{ l - \Phi(l)}{\sqrt{2} C} = \frac{\tanh(l)}{\sqrt{2}} \leq \frac{1}{\sqrt{2}}.
	\end{align*}
  This shows the $8$-self-concordance of $f$ on~$M_{-1}^n$.

  We now show that this estimate is tight.
  Choose $\varphi=\pi/2$, $\tan^2 \theta = 1/C$, and $\alpha\in \{0,\pi\}$.
  From \cref{eqn:proof_nabla_vHd^2(uu)} we have
  \[
    \sigma_p(u,w)
    = \frac{2(l - \Phi(l)) \abs{\cos(\theta)} \sin(\theta)}{\sqrt{2C} (\cos(\theta)^2 + C \sin(\theta)^2)}
    = \frac{l - \Phi(l)}{\sqrt{2} C}
    = \frac{(l - \Phi(l)) \tanh(l)}{\sqrt{2} l}
    = \frac{l \coth(l) - 1}{\sqrt{2} l}.
  \]
  For $l \to \infty$, it holds that $\sigma_p(u,w) \to 1/\sqrt{2}$, and so the estimate~$\sigma_p(u, w) \leq 1/\sqrt{2}$ is tight.
  This completes the proof of~\ref{item:SC_hyperbolic sc}.
  Note the choice of~$\alpha$ guarantees that we are essentially working with~$u, w, \dot\gamma(l) \in \RR^2$, so the argument is still valid for~$n = 2$.

  For~\ref{item:SC_hyperbolic scag}, we consider the case of $u = w$; 
  then $\varphi = \theta$ and $\alpha = 0$.
  From \cref{eqn:proof_nabla_vHd^2(uu)}, we have
  \begin{equation}\label{eqn:proof_nabla_u_Hd^2(uu)}
    (\nabla^3f)_p (u,u,u) = 2 ( - 2 l+ 3 \Phi(l) )
    \cos \theta \sin^2 \theta.
  \end{equation}
  Then we have
  \begin{align*}
    \sup_{u \in T_pM} \sigma_p(u,u) & = \max_{\theta \in [0,\pi/2]} \frac{ \abs{2 l - 3 \Phi(l)} \tan^2 \theta}{\sqrt{2}(1 + C \tan^2 \theta )^{3/2}}
    = \max_{z \in [0,\infty)} \frac{\abs{ 2l - 3 \Phi(l) } z}{\sqrt{2}(1 + C z )^{3/2}} \\
                                    & =  \sqrt{\frac{2}{27}}\frac{\abs{ 2l - 3 (\coth l + l - l\coth^2 l) }}{C}
                                    = \sqrt{\frac{2}{27}} \abs{ - 3/l - \tanh l  + 3 \coth l },
  \end{align*}
  where the maximum of $z/(1+ C z)^{3/2}$  is attained at $z = 2/C = 2 (\tanh l)/l$.
  The supremum of the last quantity is attained at $l  \to \infty$, and equals $\sqrt{2/27}$.
  This implies~\ref{item:SC_hyperbolic scag}, i.e., that $f$ is $27/2$-self-concordant along geodesics, and that this bound is tight.

  Finally we show~\ref{item:SC_hyperbolic compatibility}.
  Again, we may assume $\norm{w}_p = \norm{u}_p = 1$, and we use the above spherical coordinates.
  By \cref{eqn:proof_Hd^2(uu),eqn:proof_Hd^2(vv)} and~\cref{lem:Dd^2Hd^2}\ref{item:Dd^2Hd^2 dg}, we have
  \begin{equation*}
    \abs{\sin \theta} = \sqrt{\frac{(\nabla^2 f)_p(u,u) - 2 dg_p(u)^2}{2l \coth l}},
    \quad
    \abs{\sin \varphi} = \sqrt{\frac{(\nabla^2 f)_p(w,w) - 2 dg_p(w)^2}{2l \coth l}}.
  \end{equation*}
  By substituting these into~\cref{eqn:proof_nabla_vHd^2(uu)} and using $dg_p(u) = \cos \theta$ and~$dg_p(w) = \cos \varphi$ we obtain
  \begin{align*}
    (\nabla^3 f)_p(w,u,u) &\leq \frac{\Phi(l)}{l \coth l} |dg_p(w)| ((\nabla^2 f)_p(u,u) - 2 dg_p(u)^2) \\
                          & + \frac{2(l-\Phi(l))}{l\coth l} |dg_p (u)| \sqrt{(\nabla^2 f)_p(w,w) - 2 dg_p(w)^2} \sqrt{(\nabla^2 f)_p(u,u) - 2 dg_p(u)^2} \\
                          & \leq 2 \zeta ((\nabla^2 f)_p(u,u) - 2 dg_p(u)^2)  \\
                          &    + 2 \sqrt{(\nabla^2 f)_p(w,w) - 2 dg_p(w)^2} \sqrt{(\nabla^2 f)_p(u,u) - 2 dg_p(u)^2},
  \end{align*}
  where we used~\cref{lem:sinhx x recip and tanh x x recip diff ineq} for the second inequality.
  This implies~\ref{item:SC_hyperbolic compatibility} for $\kappa = 1$.

  Finally, the statements for~$M_{-\kappa}^n$ follow from~\cref{lem:rescaling curvature rescales distsq self-concordance,lem:basic}.
  Note for part~\ref{item:SC_hyperbolic compatibility} that rescaling the Riemannian metric on~$M_{-1}^n$ by a factor~$1/\kappa$ yields sectional curvature~$\kappa$, and rescales the distance~$g$ by a factor~$1/\sqrt{\kappa}$, so to compensate one must use the prefactors~$2 \WeirdConstant \sqrt{\kappa}$ and~$2 \sqrt{\kappa}$.
\end{proof}

We now use \cref{thm:SC_hyperbolic} to prove the following theorem, which for~$\kappa = 1$ yields \cref{thm:hypn dist epigraph barrier intro}:
\begin{thm}
  \label{thm:hypn dist epigraph barrier}
  Let~$\kappa > 0$, $M = M_{-\kappa}^n$, $p_0 \in M$, and define~$f\colon M \to \RR$ by~$f(p) = d(p, p_0)^2$.
  Define an open convex set~$D \subseteq M \times \RR_{>0} \times \RR_{>0}$ by
  \begin{equation*}
    D = \{ (p, R, S) \in M \times \RR_{>0} \times \RR_{>0} : RS - f(p) > 0 \},
  \end{equation*}
  and define a function $F\colon D \to \RR$ by
  \begin{equation*}
    F(p, R, S) = - \log(RS - f(p)) + \kappa \, f(p)
  \end{equation*}
  Then~$F$ is convex and strongly $\frac12$-self-concordant.
  Furthermore, $\lambda_{F,\frac12}(p, R, S)^2 \leq 4 + 4 \kappa f(p)$.
\end{thm}
\begin{proof}
  Recall from \cref{cor:most important barrier term convex} that~$F$ is convex.
  Let $u = (u_p, u_R, u_S)$ and $w = (w_p, w_R, w_S)$ be tangent vectors at $(p, R, S) \in D$.
  Throughout the rest of this proof, we suppress the base point~$(p, R, S)$ for derivatives.
  Set
  \begin{equation*}
    \Psi(p,R,S) = R - S^{-1} d(p, p_0)^2.
  \end{equation*}
  Instead of immediately taking~$F$ as stated, we leave the prefactor of~$f$ as a quantity~$\xi > 0$ to be chosen later.
  The derivative of~$F = -\log \Psi - \log S + \xi \, f$ is given by
  \begin{equation*}
    dF(u) = - \frac{1}{\Psi} d\Psi(u) - \frac{u_S}{S} + \xi \, df(u_p).
  \end{equation*}
  Define~$A_u = d\Psi(u) / \Psi$, $B_u = \sqrt{- \nabla^2 \Psi(u, u) / \Psi}$, $C_u = S^{-1} u_S$ and~$D_u = \sqrt{\xi \, \nabla^2 f(u_p, u_p)}$.
  We recall from \cref{lem:psi concave} that~$\Psi$ is concave, so that~$B_u$ is well-defined.
  The Hessian of~$F$ is then given by
  \begin{equation}
    \label{eq:spd2 dist barrier hessian}
    \nabla^2 F(u, u) = \underbrace{\frac{1}{\Psi^2} (d\Psi(u))^2}_{= A_u^2} \underbrace{{}- \frac{1}{\Psi} \nabla^2 \Psi(u, u)}_{= B_u^2} + \underbrace{\frac{1}{S^2} u_S^2}_{= C_u^2} + \underbrace{\xi \, \nabla^2 f(u_p, u_p)}_{= D_u^2}.
  \end{equation}
  For convenience we also write~$B_{uw} = - \nabla^2 \Psi(u, w)$.
  The third derivative of~$F$ is given by
  \begin{align}
    \nabla^3 F(w, u, u) = & - 2 \frac{1}{\Psi^3} (d\Psi(w)) \, (d\Psi(u))^2 + 2 \frac{1}{\Psi^2} (d\Psi(u)) \, (\nabla^2 \Psi(w, u)) + \frac{1}{\Psi^2} d\Psi(w) \, (\nabla^2 \Psi(u,u)) \nonumber \\
                          & - \frac{1}{\Psi} \nabla^3 \Psi(w, u, u) - 2 \frac{1}{S^3} w_S u_S^2 + \xi \, \nabla^3 f(w_p, u_p, u_p) \nonumber \\
    = & - 2 A_w A_u^2 - 2 A_u B_{uw} - A_w B_u^2 - 2 C_w C_u^2 - \frac{1}{\Psi} \nabla^3 \Psi(w, u, u) + \xi \, \nabla^3 f(w_p, u_p, u_p). \label{eq:dist epigraph barrier third deriv clean}
  \end{align}
  It is easy to see that the first four terms in \cref{eq:dist epigraph barrier third deriv clean} are bounded by a constant multiple of~$\sqrt{\nabla^2 F(w,w)}\nabla^2 F(u,u)$, and similar for the last term (by $\alpha$-self-concordance of~$f$).
  The term~$\nabla^3 \Psi(w,u,u) / \Psi$ requires more effort.
  Recall from \cref{lem:psi concave}, if $g = d(p, p_0) = \sqrt{f}$, then
  \begin{equation*}
    \nabla^2 \Psi = - S^{-1} (2 (S^{-1} g \, dS - dg)^{\otimes 2} + (\nabla^2 f - 2 \, dg \otimes dg))),
  \end{equation*}
  and the third derivative satisfies
  \begin{align*}
    \nabla^3 \Psi(w, u, u) & = - 2 S^{-1} u_S \nabla^2 \Psi(w,u) - S^{-1} w_S \nabla^2 \Psi(u,u) - S^{-1} \nabla^3 f(w_p, u_p, u_p).
  \end{align*}
  Therefore
  \begin{align*}
    \nabla^3 F(w, u, u) = % & - 2 \frac{1}{\Psi^3} (d\Psi(w)) \, (d\Psi(u))^2 + 2 \frac{1}{\Psi^2} (d\Psi(u)) \, (\nabla^2 \Psi(w, u)) + \frac{1}{\Psi^2} d\Psi(w) \, (\nabla^2 \Psi(u,u)) \\
                          % & + \frac{1}{\Psi} (2 \frac{u_S}{S} \nabla^2 \Psi(w, u) + \frac{w_S}{S} \nabla^2 \Psi(u, u) + \frac{1}{S} \nabla^3 f(w_p, u_p, u_p)) \\
                          % & - 2 \frac{1}{S^3} w_S u_S^2 + \xi \, \nabla^3 f(w_p, u_p, u_p) \\
                          % & = - 2 \frac{1}{\Psi^3} (d\Psi(w)) \, (d\Psi(u))^2 + 2 \frac{1}{\Psi} (\frac{d\Psi(u)}{\Psi} + \frac{u_S}{S}) \, (\nabla^2 \Psi(w, u)) + \frac{1}{\Psi} (\frac{d\Psi(w)}{\Psi} + \frac{w_S}{S}) \, (\nabla^2 \Psi(u,u)) \\
                          % & + \frac{1}{\Psi \, S} \nabla^3 f(w_p, u_p, u_p)) - 2 \frac{1}{S^3} w_S u_S^2 + \xi \, \nabla^3 f(w_p, u_p, u_p) \\
                          & = -2 A_w A_u^2 - 2 B_{wu} (A_u + C_u) - B_u^2 (A_w + C_w) - 2 C_w C_u^2 + (\frac{1}{\Psi S} + \xi) \nabla^3 f \\
                          & = - 2 A_w (A_u^2 - \frac12 B_u^2) - 2 B_{wu} (A_u + C_u) - 2 C_w (\frac{1}{2} B_u^2 + C_u^2) + (\frac{1}{\Psi S} + \xi) \nabla^3 f.
  \end{align*}
  We now use the bound from \cref{thm:SC_hyperbolic}\ref{item:SC_hyperbolic compatibility} and the~$2$-strong-convexity of~$f$:
  \begin{align*}
    \abs*{\frac{\nabla^3 f(w_p, u_p, u_p)}{S \Psi}} & \leq \frac{\abs{dg(w_p)} \cdot \abs{\nabla^2 f(u_p,u_p) - 2 (dg (u_p))^2}}{S \Psi} \cdot C_1 \\
    & + \frac{\abs{dg(u_p)} \cdot \sqrt{\nabla^2 f(u_p,u_p) - 2 (dg(u_p))^2} \cdot \sqrt{\nabla^2 f(w_p,w_p) - 2 (dg(w_p))^2}}{S \Psi} \cdot C_2 \\
    & \leq \frac{1}{\sqrt{2}} \sqrt{\nabla^2 f(w_p, w_p)} B_u^2 \cdot C_1 + \frac{1}{\sqrt{2}}\sqrt{\nabla^2 f(u_p, u_p)} B_u B_w \cdot C_2
  \end{align*}
  where~$C_1 = 2 \zeta \sqrt{\kappa}$ and~$C_2 = 2 \sqrt{\kappa}$, and~$\zeta \leq \frac12$ is defined in \cref{lem:sinhx x recip and tanh x x recip diff ineq}.
  Furthermore, $f$ is~$\alpha$-self-concordant with~$\alpha = 8 / \kappa$ (cf.~\cref{thm:SC_hyperbolic}\ref{item:SC_hyperbolic sc}).
  The triangle inequality gives
  \begin{align*}
    & \abs{\nabla^3 F(w, u, u)} \\
    & \leq 2 \abs{A_w (A_u^2 - \frac12 B_u^2)} + 2 \abs{B_{wu}} \abs{A_u + C_u} + 2 \abs{C_w (\frac{1}{2} B_u^2 + C_u^2)} \\
    & + \sqrt{\nabla^2 f(w_p, w_p)} (\frac{C_1}{\sqrt{2}} B_u^2 + \frac{2 \xi}{\sqrt{\alpha}} \nabla^2 f(u_p, u_p)) + \frac{C_2}{\sqrt{2}} \abs{B_w} \abs{B_u} \sqrt{\nabla^2 f(u_p, u_p)} \\
    & = 2 \abs{A_w (A_u^2 - \frac12 B_u^2)} + 2 \abs{B_{wu}} \abs{A_u + C_u} + 2 \abs{C_w (\frac{1}{2} B_u^2 + C_u^2)} \\
    & + D_w \abs*{\frac{C_1}{\sqrt{2 \xi}} B_u^2 + \frac{2}{\sqrt{\alpha \xi}} D_u^2} + \abs{B_w B_u} \frac{C_2}{\sqrt{2 \xi}} D_u \\
    & \leq 2 \abs{A_w (A_u^2 - \frac12 B_u^2)} + 2 \abs{B_w} \abs{B_u} (\abs{A_u + C_u} + \frac{C_2}{2 \sqrt{2 \xi}} D_u) + 2 \abs{C_w (\frac{1}{2} B_u^2 + C_u^2)} \\
    & + D_w \abs*{\frac{C_1}{\sqrt{2 \xi}} B_u^2 + \frac{2}{\sqrt{\alpha \xi}} D_u^2}\\
    % & \leq 2 \abs{A_w (A_u^2 - \frac12 B_u^2)} + 2 \abs{B_w} \abs{B_u} (\abs{A_u + \frac{u_S}{S}} + \frac{C_2}{2 \sqrt{2}} \sqrt{\nabla^2 f(u_p, u_p)}) + 2 \abs{\frac{w_S}{S} (\frac{1}{2} B_u^2 + \frac{u_S^2}{S^2})} \\
    % & + 2 \sqrt{C \nabla^2 f(w_p, w_p)} \abs*{\frac{C_1}{2 \sqrt{2 C}} B_u^2 + \frac{\sqrt{C}}{\sqrt{\alpha}} \nabla^2 f(u_p, u_p)} \\
    & \leq 2 \sqrt{A_w^2 + B_w^2 + C_w^2 + D_w^2} \sqrt{L},
  \end{align*}
  where we applied~$\abs{B_{uw}} \leq \abs{B_u} \abs{B_w}$ to get the penultimate inequality, Cauchy--Schwarz to get the last inequality, and~$L$ is defined as
  \begin{align*}
    L & = (A_u^2 - \frac12 B_u^2)^2 + \abs{B_u}^2 (\abs{A_u + C_u} + \frac{C_2}{2 \sqrt{2 \xi}} D_u)^2 + (\frac{1}{2} B_u^2 + C_u^2)^2 + \abs*{\frac{C_1}{2 \sqrt{2 \xi}} B_u^2 + \frac{1}{\sqrt{\alpha \xi}} D_u^2}^2.
  \end{align*}
  We now show that~$L \leq 2 (\nabla^2 F(u, u))^2$ for the choice~$\xi = \kappa$.
  First, we use that~$C_1 = 2 \zeta \sqrt{\kappa}$, $C_2 = 2 \sqrt{\kappa}$ and~$\alpha = 8 / \kappa$.
  Therefore~$L$ is
  \begin{align*}
    L & = (A_u^2 - \frac12 B_u^2)^2 + \abs{B_u}^2 (\abs{A_u + C_u} + \sqrt{\frac{\kappa}{2 \xi}} D_u)^2 + (\frac{1}{2} B_u^2 + C_u^2)^2 + \abs*{\frac{\zeta \sqrt{\kappa}}{\sqrt{2 \xi}} B_u^2 + \sqrt{\frac{\kappa}{8 \xi}} D_u^2}^2 \\
      & = A_u^4 - A_u^2 B_u^2 + \frac14 B_u^4 + \abs{B_u}^2 (\abs{A_u + C_u}^2 + \sqrt{\frac{2 \kappa}{\xi}} \, \abs{A_u + C_u} D_u +  \frac{\kappa}{2 \xi} D_u^2) \\
      & + \frac14 B_u^4 + B_u^2 C_u^2 + C_u^4 + \frac{\zeta^2 \kappa}{2 \xi} B_u^4 + \frac{\zeta \kappa}{2 \xi} B_u^2 D_u^2 + \frac{\kappa}{8 \xi} D_u^4 \\
      & = A_u^4 + B_u^4 \parens*{\frac12 + \frac{\zeta^2 \kappa}{2 \xi}} + C_u^4 + \frac{\kappa}{8 \xi} D_u^4 \\
      & + 2 B_u^2 \abs{A_u} \abs{C_u} + 2 B_u^2 C_u^2 + \sqrt{\frac{2 \kappa}{\xi}} B_u^2 \abs{A_u + C_u} D_u + \frac{\kappa}{2 \xi} \parens*{1 + \zeta} B_u^2 D_u^2.
  \end{align*}
  As~$\zeta \leq \frac12$, we have~$\zeta^2 \leq \frac14$.
  Therefore the choice~$\xi = \kappa$ ensures that
  \begin{align*}
    L & \leq A_u^4 + \frac58 B_u^4 + C_u^4 + \frac{1}{8} D_u^4 \\
      & + 2 B_u^2 \abs{A_u} \abs{C_u} + 2 B_u^2 C_u^2 + \sqrt{2} B_u^2 \abs{A_u + C_u} D_u + \frac34 B_u^2 D_u^2 \\
      & \leq A_u^4 + \frac{5}{8} B_u^4 + C_u^4 + \frac{1}{8} D_u^4 \\
      & + \frac12 B_u^4 + 2 A_u^2 C_u^2 + \frac{\sqrt{2}}{2} B_u^2 (A_u^2 + C_u^2 + 2 D_u^2) + \frac34 B_u^2 D_u^2 \\
      & \leq \frac{9}{8} (\nabla^2 F(u, u))^2 \leq 2 (\nabla^2 F(u, u))^2
  \end{align*}
  as~$\nabla^2 F(u,u) = A_u^2 + B_u^2 + C_u^2 + D_u^2$.
  To conclude, we have shown that
  \begin{equation*}
    \abs{\nabla^3 F(w, u, u)} \leq 2 \sqrt{2} \sqrt{\nabla^2 F(w, w)} \, \nabla^2 F(u, u)
  \end{equation*}
  and~$F$ is~$\frac12$-self-concordant.%
  \footnote{Bounding~$L$ by~$(\nabla^2 F(u, u))^2$ would lead to~$1$-self-concordance of~$F$, but it is not clear whether there is a choice of~$\xi > 0$ such that~$F$ is~$1$-self-concordant and its Newton decrement is not too adversely affected.}

  We now verify the bound on the Newton decrement.
  For~$u = (u_p, u_R, u_S) \in T_{(p,R,S)} D$ such that~$\nabla^2 F(u, u) \neq 0$ and~$u_p \neq 0$, we have
  \begin{equation*}
    \abs{dF(u)} = \abs{- A_u - C_u + \xi df(u_p)} \leq \sqrt{A_u^2 + C_u^2 + D_u^2} \sqrt{1 + 1 + \frac{\xi^2 \abs{df(u_p)}^2}{D_u^2}},
  \end{equation*}
  and
  \begin{equation*}
    \frac{\xi^2 \abs{df(u_p)}^2}{D_u^2} = \frac{\xi^2 \abs{df(u_p)}^2}{\xi \nabla^2 f(u_p, u_p)} \leq \xi \, \lambda_f(p)^2 = 2 \xi \, f(p)
  \end{equation*}
  by \cref{cor:squared distance newton decrement}.
  Since we chose~$\xi = \kappa$, this shows that $\lambda_{F,1/2}(p,R,S)^2 \leq 2 (2 + 2 \kappa f(p))$.
\end{proof}

%-----------------------------------------------------------------------------
\section{Applications}\label{sec:applications}
%-----------------------------------------------------------------------------
In this section, we discuss applications of our interior-point method framework.
In \cref{subsec:kempf-ness} we show that the framework can be used to solve non-commutative optimization and scaling problems.
In \cref{subsec:MEB,subsec:geometric median,subsec:riemannian barycenter}, we use the previously constructed barriers for the epigraph of the squared distance on Hadamard symmetric spaces and the epigraph of the distance on the model spaces for constant negative sectional curvature to the natural geometric problems of computing minimum enclosing balls, geometric medians, and Riemannian barycenters.
To achieve the above, we build on the results of \cref{sec:barriers compatibility path-following,sec:distsq}.

%-----------------------------------------------------------------------------
\subsection{Non-commutative optimization and scaling problems}\label{subsec:kempf-ness}
%-----------------------------------------------------------------------------
In this subsection we show that the problem of minimizing log-norm or Kempf--Ness functions, as discussed in~\cref{subsubsec:application kempf ness}, can be solved using our interior-point methods.
This leads to also naturally leads to algorithms for \emph{scaling problems}.

We briefly recap the general setup for the norm minimization problem and refer to~\cite{burgisserTheoryNoncommutativeOptimization2021} for more detail.
Throughout this section we let~$G \subseteq \GL(n, \CC)$ be a connected algebraic Lie group such that~$g^* \in G$ for every~$g \in G$.
(In mathematics, such groups are known as reductive and they are particularly well-behaved~\cite{wallachGeometricInvariantTheory2017}.)
We also fix~$\pi\colon G \to \GL(V)$ to be a finite-dimensional rational complex representation of~$G$.
Let~$K = G \cap \U(n)$, which is a maximal compact subgroup of~$G$, and assume that~$V$ is endowed with a~$K$-invariant inner product~$\braket{\cdot|\cdot}$.%
\footnote{Following Dirac notation, we will also write $\braket{v|A|w} := \braket{v|Aw}$ for vectors $v,w\in v$ and operators~$A$ on~$V$.}
For a non-zero vector~$0\neq v \in V$, the goal is to minimize~$\norm{\pi(g) v}^2 = \braket{v|\pi(g)^*\pi(g)|v} = \braket{v|\pi(g^* g)|v}$ over~$g \in G$, where we used that~$\pi(g)^* = \pi(g^*)$.%
\footnote{Because~$K$ acts unitarily and the Lie algebra representation~$\Pi = d\pi_I$ is complex linear, one has~$\Pi(X^*) = \Pi(X)^*$ for~$X \in \Lie(G)$.
% It sends~$\Lie(K)$ to~$\Lie(\U(V))$ and is complex linear, hence maps~$i\Lie(K)$ to~$i\Lie(\U(V)) = \Herm(V)$.
By the Cartan decomposition every~$g \in G$ is a product~$g = k \exp(H)$ with~$k \in K$ and~$H \in i\Lie(K)$, so~$\pi(g)^* = (\pi(k) \exp(\Pi(H)))^* = \exp(\Pi(H)) \pi(k)^{-1} = \exp(\Pi(H)) \pi(k^{-1}) = \pi(g^*)$ (cf.~\cite{burgisserTheoryNoncommutativeOptimization2021,hiraiConvexAnalysisHadamard2022}).}
Therefore, this is equivalent to minimizing~$\braket{v|\pi(p)|v}$ over~$p \in M = \{ g^* g : g \in G \} = G \cap \PD(n) \subseteq \PD(n)$.
The log-norm or Kempf--Ness function computes the logarithm of this quantity:%
\footnote{Alternatively, because of the~$K$-invariance, $g \mapsto \norm{\pi(g) v}^2$ descends to a map on the quotient~$K \backslash G$. This space is naturally isometric to~$M$ via the map $Kg \mapsto g^* g$: for~$G = \GL(n, \CC)$ one can prove this using the polar decomposition, which generalizes to the Cartan decomposition for reductive~$G$. As such, this is the same as \cref{defn:kempf ness}.}
\begin{defn}[Kempf--Ness function]
  \label{defn:kempf ness}
  Let~$M = \{ g^* g : g \in G \} = G \cap \PD(n) \subseteq \PD(n)$.
  For~$0\neq v \in V$, the \emph{Kempf--Ness function}~$\phi_v$ is defined by
  \begin{equation}
    \label{eq:kempf ness defn}
    \phi_v\colon M \to \RR, \quad \phi_v(p) = \log \braket{v|\pi(p)|v}\!.
  \end{equation}
  For the special case where~$G = \GL(n, \CC)$, $V = \CC^n$ and~$\pi$ is the identity map, we write
  \begin{equation}
    \label{eq:fv defn}
    f_v\colon \PD(n) \to \RR, \quad f_v(P) = \log\braket{v|P|v}.
  \end{equation}
\end{defn}
We note that~$M$ is a convex subset of~$\PD(n)$~\cite[Thm.~10.58,~Lem.~10.59]{bridson-haefliger-nonpositive-curvature}, so the geodesics in~$M$ are precisely the geodesics in~$\PD(n)$ which lie completely in~$G$.
Thus the tangent space~$T_I M$ consists of those Hermitian matrices~$H \in \Herm(n) = T_I \PD(n)$ which also are in~$\Lie(G) := T_I G$, the Lie algebra of~$G$.
For~$G = \GL(n, \CC)$, we simply have that $T_I M = \Herm(n)$.

Because~$K$ acts unitarily, $\pi$ restricts to a map~$M \to \PD(V)$, and one can verify that it sends geodesics to geodesics (i.e., it is \emph{geodesically affine}).
At the identity, we have the explicit description
\begin{equation*}
  \pi \left( \Exp_I(tH) \right) = \Exp_I(t \Pi(H))
\end{equation*}
for~$H \in T_I M$ and~$\Pi\colon \Lie(G) \to \End(V) = \Lie(\GL(V))$ is given by the derivative of~$\pi$, i.e., $\Pi = d \pi_I$.
The linear map~$\Pi$ is also known as the Lie algebra homomorphism induced by~$\pi$.
Therefore, the Kempf--Ness function is the composition of the geodesically affine map~$M \to \PD(V)$, $p \mapsto \pi(p)$, and the map~$\PD(V) \to \RR$ given by~$P \mapsto\log \braket{v|P|v}$, i.e., the Kempf--Ness function for the definining representation of~$\GL(V)$.
To establish bounds on the derivatives of~$\phi_v$, it therefore suffices to prove bounds on the derivatives of~$f_v$, and to translate the results via~$\Pi$.

Below, we prove the well-known fact that the Kempf--Ness functions are convex on~$M$ (see, e.g.,~\cite{burgisserTheoryNoncommutativeOptimization2021}).
As explained above it suffices to prove this for the special case where~$G = \GL(n, \CC)$ and~$V = \CC^n$, with~$\pi\colon G \to \GL(V)$ given by the identity map.
\begin{prop}
  \label{prop:kempf ness convex}
  For~$0 \neq v \in \CC^n$, the Hessian of the function~$f_v\colon \PD(n) \to \RR$ defined in~\cref{eq:fv defn} satisfies for every $P \in \PD(n)$ and~$U \in T_P \PD(n)$ the identity
  \begin{equation*}
    (\nabla^2 f_v)_P(U, U) = \frac{\bra{\tilde v} (\tilde U - \frac{\braket{\tilde v|\tilde U|\tilde v}}{\braket{\tilde v|\tilde v}} I)^2 \ket{\tilde v}}{\braket{\tilde v|\tilde v}},
  \end{equation*}
  where we use the notation~$\tilde v = P^{1/2} v$ and~$\tilde U = P^{-1/2} U P^{-1/2}$.
  As a consequence, all Kempf--Ness functions are convex.
\end{prop}
\begin{proof}
  We compute the Hessian of~$f := f_v$.
  First off, we have
  \begin{align*}
    \partial_t f(\Exp_P(tU)) = \partial_t \log \braket{v|\Exp_P(tU)|v} = \frac{\braket{\tilde v|\tilde U e^{t \tilde U}|\tilde v}}{\braket{\tilde v|e^{t\tilde U}|\tilde v}}.
  \end{align*}
  The second derivative is given by
  \begin{align*}
    \partial_{t=0}^2 f(\Exp_P(tU)) & = \frac{\braket{\tilde v|\tilde U^2|\tilde v} \braket{\tilde v|\tilde v} - \braket{\tilde v|\tilde U|\tilde v}^2}{\braket{\tilde v|\tilde v}^2} = \frac{\bra{\tilde v} (\tilde U - \frac{\braket{\tilde v|\tilde U|\tilde v}}{\braket{\tilde v|\tilde v}} I)^2 \ket{\tilde v}}{\braket{\tilde v|\tilde v}},
  \end{align*}
  hence is non-negative.
\end{proof}
The expression for the first- and second derivatives can be understood in terms of the expectation and variance of corresponding random variables, as pointed out in~\cite{burgisserTheoryNoncommutativeOptimization2021}.%
\footnote{Similarly, the higher derivatives \emph{along geodesics} can be related to higher cumulants, see~\cite[Rem.~3.16]{burgisserTheoryNoncommutativeOptimization2021}.}
This will be useful for bounding the third derivative.
Define a linear map~$\Phi_v\colon \CC^{n \times n} \to \CC$ by
\begin{equation}
  \label{eq:Phi definition}
  \Phi_v(A) = \frac{\braket{v|A|v}}{\braket{v|v}}.
\end{equation}
Then~$\Phi_v$ is what is known as a \emph{completely positive} and \emph{unital} map.%
\footnote{This means that~$\Phi_v(I) = 1$, and the complete positivity refers to the fact that for every~$n' \geq 1$, the map~$\Phi_v \otimes I_{\CC^{n' \times n'}}\colon \CC^{n \times n} \otimes \CC^{n' \times n'} \to \CC^{n' \times n'}$ sends positive-semidefinite operators to positive-semidefinite operators.}
Such a map is to be interpreted as taking the expectation with respect to a random variable, where the random variable is now specified by a complex matrix.
One can define the \emph{covariance} between two matrices~$A, B \in \CC^{n \times n}$ as
\begin{equation}
  \label{eq:covariance definition}
  \cov_v(A,B) = \Phi_v(A^*B) - \Phi_v(A)^*\Phi_v(B).
\end{equation}
The variance of~$A$ is defined accordingly as~$\var_v(A) = \cov_v(A, A)$.
With this notation, we can more succinctly write
\begin{equation*}
  (\nabla^2 f_v)_P(U, U) = \var_{\tilde v}(\tilde U),
\end{equation*}
where~$\tilde v = P^{1/2} v$ and~$\tilde U = P^{-1/2} U P^{-1/2}$ as before.
Then the third derivative can be computed as follows.

\begin{prop}
  \label{prop:kempf ness third derivative}
  Let~$0 \neq v \in \CC^n$ and let~$f_v\colon \PD(n) \to \RR$ be as defined in~\cref{eq:fv defn}.
  Then for every~$U, W \in T_I \PD(n) = \Herm(n)$, its third derivative satisfies
  \begin{align*}
    & (\nabla^3 f_v)_I(W, U, U) \\
    & = \frac{1}{2} \frac{\braket{v|\{W, U^2\}|v}}{\braket{v|v}} - \frac{\braket{v|U^2|v} \braket{v|W|v}}{\braket{v|v}^2} - \frac{\braket{v|U|v} \braket{v|\{W,U\}|v}}{\braket{v|v}^2} + 2 \frac{\braket{v|U|v}^2 \braket{v|W|v}}{\braket{v|v}^3} \\
    & = \Re \left( \cov(W, U^2 - 2 \Phi(U) U) \right).
  \end{align*}
\end{prop}
\begin{proof}
  To compute the third derivative of~$f := f_v$ at~$I \in \PD(n)$, note that
  \begin{align*}
    & \partial_{t=0} (\nabla^2 f)_{\Exp_{I}(tW)}(\tau_{I \to \Exp_I(tW)} U, \tau_{I \to \Exp_I(tW)} U) \\
    & = \partial_{t=0} \parens*{\frac{\bra{v}e^{\frac{t}{2} W} U^2 e^{\frac{t}{2} W}\ket{v}}{\braket{v|e^{tW}|v}} - \frac{\braket{v|e^{\frac{t}{2} W} U e^{\frac{t}{2} W}|v}^2}{\braket{v|e^{tW}|v}^2}} \\
    & = \frac{1}{2} \frac{\braket{v|\{W, U^2\}|v}}{\braket{v|v}} - \frac{\braket{v|U^2|v} \braket{v|W|v}}{\braket{v|v}^2} - \frac{\braket{v|U|v} \braket{v|\{W,U\}|v}}{\braket{v|v}^2} + 2 \frac{\braket{v|U|v}^2 \braket{v|W|v}}{\braket{v|v}^3}.
  \end{align*}
  Using the map~$\Phi$ and the associated covariance defined in \cref{eq:Phi definition,eq:covariance definition} , we may rewrite the above more succinctly as
  \begin{align*}
    (\nabla^3 f)_I(W, U, U) & = \frac{1}{2} \Phi(\{W, U^2\}) - \Phi(U^2) \Phi(W) - \Phi(U) \Phi(\{W, U\}) + 2 \Phi(U)^2 \Phi(W) \\
                            & = \frac{1}{2} (\cov(W, U^2) + \cov(U^2, W)) - \Phi(U) (\cov(U, W) + \cov(W, U)) \\
                            & = \Re \left( \cov(W, U^2 - 2 \Phi(U) U) \right). \qedhere
  \end{align*}
\end{proof}
\begin{rem}
  The Kempf--Ness functions are not necessarily self-concordant, even along geodesics.
  To see this, consider~$v = \frac1{\sqrt{2}} (e_1 - e_2)$ and for~$z \in \RR$ the matrix~$U_z \in \Herm(2)$ given by
  \begin{equation*}
    U_z = \begin{bmatrix}
      1 & z \\
      z & 0
    \end{bmatrix}.
  \end{equation*}
  Then
  \begin{equation*}
    (\nabla^2 f_v)_I(U_z, U_z) = \frac14, \quad (\nabla^3 f_v)_I(U_z, U_z, U_z) = \frac{z}{2},
  \end{equation*}
  so~$\abs{(\nabla^3 f_v)_I(U_z, U_z, U_z)}$ can be arbitrarily large compared to $(\nabla^2 f_v)_I(U_z, U_z)^{3/2}$.
\end{rem}
Although self-concordance does not hold, we do have the following bound on its third derivative, which implies that it is compatible (in the sense of \cref{subsec:compatibility}) with any strongly convex function.
This generalizes~\cite[Prop.~3.15]{burgisserTheoryNoncommutativeOptimization2021} beyond the case~$W = U$.
\begin{thm}
  \label{thm:kempf ness compatibility bound}
  Let~$0 \neq v \in \CC^n$ and let~$f_v\colon \PD(n) \to \RR$ be as defined in~\cref{eq:fv defn}.
  For every~$P \in \PD(n)$ and~$U, W \in T_P \PD(n) = \Herm(n)$, one has the estimate
  \begin{align*}
    \abs*{(\nabla^3 f_v)_P(W, U, U)} & \leq 4 \norm{\tilde U}_\infty \sqrt{(\nabla^2 f_v)_P(W,W)} \sqrt{(\nabla^2 f_v)_P(U,U)} \\
                                   & \leq 4 \norm{U}_P \sqrt{(\nabla^2 f_v)_P(W,W)} \sqrt{(\nabla^2 f_v)_P(U,U)} \\
                                   & = 4 \norm{U}_P \norm{W}_{f_v,P} \norm{U}_{f_v,P}.
  \end{align*}
  where~$\tilde U = P^{-1/2} U P^{-1/2}$, and~$\norm{\cdot}_\infty$ is the spectral norm.
\end{thm}
\begin{proof}
  We prove the statement for~$P = I$, and set~$f : = f_v$.
  Writing~$\var(A) = \cov(A,A)$, an operator version of the Cauchy--Schwarz inequality~\cite{bhatiaMoreOperatorVersions}  yields
  \begin{align*}
    \abs*{(\nabla^3 f)_I(W, U, U)}^2 \leq \abs*{\cov(W, U^2 - 2 \Phi(U) U)}^2 \leq \var(W) \var(U^2 - 2 \Phi(U) U).
  \end{align*}
  Using that for every~$A, B \in \CC^{n \times n}$,
  \begin{align*}
    \var(A + B) = \var(A) + \var(B) + \cov(A, B) + \cov(B, A) \leq 2 (\var(A) + \var(B)),
  \end{align*}
  one can deduce for Hermitian~$A$ that
  \begin{align*}
    \var(U^2 - 2 \Phi(U) U) & \leq 2 \var(U (U - \Phi(U))) + 2 \var(\Phi(U) U) \\
                            & \leq 2 \norm{U^2}_\infty \var(U - \Phi(U)) + 2 \Phi(U)^2 \var(U) \\
                            & \leq 4 \norm{U}_\infty^2 \var(U)
  \end{align*}
  where the second inequality follows from
  \begin{align*}
    \var(U (U - \Phi(U))) & \leq \Phi( (U - \Phi(U)) U U (U - \Phi(U))) \\
                          & = \frac{\braket{v|(U - \Phi(U)) U U (U - \Phi(U))|v}}{\braket{v|v}} \\
                          & \leq 2 \norm{U^2}_\infty \frac{\braket{v|(U - \Phi(U))^2|v}}{\braket{v|v}} \\
                          & = 2 \norm{U^2}_\infty \var(U - \Phi(U)).
  \end{align*}
  The theorem now follows from the observation that~$(\nabla^2 f)_I(U, U) = \var(U - \Phi(U)) = \var(U)$.
\end{proof}
\begin{cor}
  For~$0 \neq v \in V$, the Kempf--Ness function~$\phi_v$ defined in~\cref{eq:kempf ness defn} satisfies for all~$p \in M$ and~$u, w \in T_p M$ the inequality
  \begin{equation*}
    \abs*{(\nabla^3 \phi_v)_p(w, u, u)} \leq 4 \norm{d\pi_p(u)}_{\pi(p)} \sqrt{(\nabla^2 \phi_v)_p(w, w)} \sqrt{(\nabla^2 \phi_v)_p(u, u)}.
  \end{equation*}
\end{cor}
The quantity~$\norm{d\pi_p(u)}_{\pi(p)}$ can be understood by observing that
\begin{equation*}
  d\pi_p(u) = \partial_{t=0} \pi(\Exp_p(tu)) = \partial_{t=0} \pi \left( p^{1/2} e^{t p^{-1/2} u p^{-1/2}} p^{1/2} \right) = \pi(p^{1/2}) \Pi(p^{-1/2} u p^{-1/2}) \pi(p^{1/2}).
\end{equation*}
Therefore~$\norm{d\pi_p(u)}_{\pi(p)} = \norm{\Pi(p^{-1/2} u p^{-1/2})}_I$.
For convenience, we write~$N(\pi) = \norm{\Pi}$ for the operator norm of~$\Pi\colon \Lie(G) \to \End(V)$.
This quantity is known as the \emph{weight norm} of~$\pi$ in~\cite{burgisserTheoryNoncommutativeOptimization2021}, as it is determined as the largest norm of any highest weight appearing in the decomposition of the representation~$\pi$ into irreducible components.
Then the above computation shows that the operator norm of~$d_p\pi$ with respect to~$\norm{\cdot}_p$ and~$\norm{\cdot}_{\pi(p)}$ is exactly the weight norm~$N(\pi)$.
\begin{cor}
  \label{cor:kempf ness compatibility weight norm}
  Let~$N(\pi) = \norm{\Pi}$ the weight norm of~$\pi$.
  Then for~$0 \neq v \in V$, the Kempf--Ness function~$\phi_v$ defined in~\cref{eq:kempf ness defn} satisfies for all~$p \in M$ and~$u, w \in T_p M$ the inequality
  \begin{equation*}
    \abs*{(\nabla^3 \phi_v)_p(w, u, u)} \leq 4 N(\pi) \, \norm{u}_p \sqrt{(\nabla^2 \phi_v)_p(w, w)} \sqrt{(\nabla^2 \phi_v)_p(u, u)}.
  \end{equation*}
\end{cor}
We now apply the above to obtain an algorithmic result for optimizing Kempf--Ness functions over balls of fixed radius.
Recall from \cref{thm:distsq pdn self-concordant expanded} that~$h(p) = \frac12 d(p, I)^2$ is~$1$-self-concordant on~$\PD(n)$.
Therefore, the same holds on~$M$.
It directly follows from \cref{thm:compatible function epigraph barrier construction} that one can construct a strongly self-concordant function on its open epigraph, as~$h$ is~$(0,1)$-compatible with itself:
\begin{prop}
  Let~$h\colon M \to \RR$ be defined by~$h(p) = \frac12 d(p, I)^2$.
  Let~$S_0 > 0$ and consider~$D = \{ p \in M : h(p) < S_0 \}$.
  Then the function~$F\colon D \to \RR$ defined by
  \begin{equation*}
    F(p) = -\log(S_0 - h(p)) + h(p)
  \end{equation*}
  is a self-concordant barrier for~$D$ with barrier parameter~$\theta = 1 + S_0$.
\end{prop}
The claim that it has barrier parameter at most~$1 + S_0$ follows from~$\lambda_F(p)^2 \leq 1 + \lambda_h(p)^2$ and \cref{cor:squared distance newton decrement}.
Since~$h$ is~$1$-strongly convex, we see that the Kempf--Ness function~$\phi_v$ is~$(0, 2 N(\pi))$-compatible with~$F$ in the sense of \cref{defn:compatibility}.
Therefore by \cref{prop:compatible functions self concordance}, for every~$t \geq 0$, the function~$F_t := t \phi_v + F$ is~$\alpha$-self-concordant, where~$\alpha$ is given by
\begin{equation}
  \label{eq:kempf ness self-concordance of family}
  \alpha = \begin{cases}
    \frac{4 N(\pi)^2 - 1}{4 N(\pi)^4} & \text{if } 2 N(\pi)^2 > 1,  \\
    1 & \text{otherwise.}
  \end{cases}
\end{equation}
Lastly, we can exactly give the analytic center of~$F$: one easily verifies that it is given by~$p = I$.
We obtain the following algorithmic result.
\begin{thm}
  \label{thm:kempf ness complexity}
  For~$0 \neq v \in V$, let~$\phi_v\colon M \to \RR$ be the Kempf--Ness function defined in~\cref{eq:kempf ness defn}.
  Let~$\alpha \geq 0$ be as in \cref{eq:kempf ness self-concordance of family}.
  Then for every~$S_0 > 0$, using
  \begin{equation*}
    % \left\lceil \frac{\lambda^{(1)} + \sqrt{(1 + S_0) / \alpha}}{\lambda^{(1)} - \lambda^{(2)}} \, \log(\frac{2 (1 + S_0 + \alpha)}{\sqrt{\alpha} \lambda^{(1)} \eps})\right\rceil
    \parens*{\frac95 + \frac{36}{5} \sqrt{\frac{1 + S_0}{\alpha}}} \log \left( \frac{8 (1 + S_0 + \alpha)}{\sqrt{\alpha} \eps} \right)
  \end{equation*}
  iterations of the path-following method, one can compute a point~$P_\eps \in M$ such that
  \begin{equation*}
    \phi_v(p_\eps) - \inf_{p \in D} \phi_v(p) \leq \eps.
  \end{equation*}
\end{thm}
\begin{proof}
  Set~$\lambda^{(1)} = \frac14$ and~$\lambda^{(2)} = \frac19$.
  Let~$p_i$ be the sequence of points defined in \cref{thm:main stage} with these choices of~$\lambda^{(i)}$.
  These satisfy
  \begin{equation*}
    \phi_v(p_i) - \inf_{p \in D} \phi_v(p) \leq \frac{2 (1 + S_0 + \alpha) \norm{d\phi_v}_{F,p}^*}{\sqrt{\alpha} \lambda^{(1)}} \exp \left( -i \frac{\lambda^{(1)} - \lambda^{(2)}}{\lambda^{(1)} + \sqrt{(1 + S_0)/\alpha}} \right)
  \end{equation*}
  where we used that the barrier parameter~$\theta$ of~$F$ is~$1 + S_0$.
  The norm~$\norm{d(\phi_v)_p}_{F,p}^*$ is at most~$N(\pi)$, because~$F$ is strongly~$1$-convex and~$d(\phi_v)_p$ is~$N(\pi)$-Lipschitz:~$f_v$ is easily checked to be~$1$-Lipschitz, and~$\pi$ is~$N(\pi)$-Lipschitz.
  Therefore we just need to ensure that
  \begin{equation*}
    i \frac{\lambda^{(1)} - \lambda^{(2)}}{\lambda^{(1)} + \sqrt{(1 + S_0) / \alpha}} \geq \log\left( \frac{2 (1 + S_0 + \alpha)}{\sqrt{\alpha} \lambda^{(1)} \eps} \right),
  \end{equation*}
  which amounts to
  \begin{equation*}
    i \geq \frac{\lambda^{(1)} + \sqrt{(1 + S_0) / \alpha}}{\lambda^{(1)} - \lambda^{(2)}} \log \left( \frac{2 (1 + S_0 + \alpha)}{\sqrt{\alpha} \lambda^{(1)} \eps} \right) = \parens*{\frac95 + \frac{36}{5} \sqrt{\frac{1 + S_0}{\alpha}}} \log \left( \frac{8 (1 + S_0 + \alpha)}{\sqrt{\alpha} \eps} \right). \qedhere
  \end{equation*}
\end{proof}
\begin{cor}
  For~$0 \neq v \in V$, let~$\phi_v$ be the Kempf--Ness function defined in~\cref{eq:kempf ness defn}.
  Then for every~$\eps > 0$ and~$R_0 > 0$, an~$\eps$-approximate minimizer of~$\phi_v$ over a ball of radius~$R_0$ around~$I \in M \subseteq \PD(n)$ can be found using
  \begin{equation*}
    \bigO\parens*{(1 + R_0) (1 + N(\pi)) \log \left( \frac{R_0 \, N(\pi)}{\eps} \right)}
  \end{equation*}
  iterations of the path-following method.
\end{cor}
We shall not explicitly relate the norm minimization problem to the scaling problem here, but note that approximate minimizers of the Kempf--Ness function necessarily have small gradient (hence their \emph{moment map} image is close to zero), and determining whether the gradient can become arbitrarily close to zero is the \emph{null-cone problem}, to which the general scaling problem can be reduced. See~\cite{burgisserTheoryNoncommutativeOptimization2021} for details.

We briefly comment on the geometric meaning of the Kempf--Ness functions.
For the purpose of optimization, it is natural to consider whether there exists an analogue of (non-constant) linear functions on~$\RR^n$.
This is generally not the case; in fact, if~$M$ is a complete Riemannian manifold with a non-constant smooth function~$h \colon M \to \RR$ such that~$\nabla^2 h = 0$, then~$M$ is isometric to a product~$M' \times \RR$, such that after this identification, $h$ is some multiple of the projection onto the second coordinate~\cite{innamiSplittingTheoremsRiemannian1982}.\footnote{For Hadamard~$M$, this may be deduced as follows: $\nabla^2 h = 0$ implies that~$\norm{dh}$ is a constant function on~$M$. Since~$h$ is non-constant, $\norm{dh}$ is nonzero. The gradient flow of~$h$ is by isometries, without fixed points. If~$z\colon M \to M$ denotes the map given by following the gradient flow for time~$1$, then~$d(z(p), p)$ is also constant as a function of~$p \in M$, and the subgroup of the isometries of~$M$ generated by~$z$ acts properly by semi-simple isometries on~$M$, in the sense of~\cite[Def.~I.8.2, Def.~II.6.1]{bridson-haefliger-nonpositive-curvature}. Hence by~\cite[Thm.~7.1]{bridson-haefliger-nonpositive-curvature}, $M$ splits as a product~$M' \times \RR$.}
There does exist another useful generalization, namely the class of Busemann functions; see~\cite[II.8]{bridson-haefliger-nonpositive-curvature} for general background.
These may be defined on any Hadamard manifold~$M$ (and also more generally) as follows~\cite{hiraiConvexAnalysisHadamard2022}: for a (not necessarily unit-speed) geodesic~$\gamma\colon \RR \to M$ with~$\dot\gamma\neq0$, define~$b_\gamma\colon M \to \RR$ by
\begin{equation}
  \label{eq:busemann function defn}
  b_\gamma(p) := \norm{\dot\gamma(0)} \parens*{\lim_{t \to \infty} d(p, \gamma(t / \norm{\dot\gamma(0)})) - t}.
\end{equation}
This limit is well-defined and the resulting function turns out to be convex, and in the specific case of~$M = \RR^n$, reduces to an arbitrary (suitably normalized) affine function.
For~$M = \PD(n)$, whenever~$\gamma$ converges to a \emph{rational} point at infinity, the Busemann function is a multiple of the Kempf--Ness function associated with a highest weight vector for an irreducible representation of~$\GL(n)$.
This follows, e.g., by comparing~\cite[Lem.~2.34]{hiraiConvexAnalysisHadamard2022} and~\cite[Thm.~5.7]{franksMinimalLengthOrbit2022}.
The Kempf--Ness functions for~$v \in \CC^n$ with~$\norm{v} = 1$, considered as a vector in the defining representation of~$\GL(n)$, correspond to those~$\gamma$ for which~$\dot\gamma(0)$ is~$- vv^*$, which may also be deduced from e.g.~\cite[Prop.~10.69]{bridson-haefliger-nonpositive-curvature}.

\subsection{The minimum enclosing ball problem}
\label{subsec:MEB}
In the remainder of this section we show to apply the results of~\cref{sec:barriers compatibility path-following,sec:distsq} to various geometric problems, all of which involve the distance function or its square.

We first study the \emph{minimum~enclosing~ball problem (MEB)} on a manifold $M$:
given~$m \geq 3$ distinct points $p_1, p_2, \ldots, p_m$ in $M$, find the smallest ball containing all of them.
More formally, finding the MEB amounts to solving the following nonsmooth optimization problem:
\begin{equation}
  \label{eq:MEB}
  \text{minimize } R \text{ s.t. } (p,R) \in M \times \RR, \ d(p, p_i) \leq R \ (i=1,2,\ldots,m).
\end{equation}
In the case of Euclidean space $M = \RR^n$, MEB is a well-studied problem in computational geometry, and can be formulated as a second-order cone program to which an interior-point method is applicable; see e.g. \cite{KMY03}.

Nielsen and Hadjeres~\cite{NH15} addressed this problem for a hyperbolic space $M$.
We shall assume that~$M$ is a complete convex submanifold of~$\PD(n)$, but we note that similar results may be obtained for products of (rescalings of) these spaces, hence for all Hadamard symmetric spaces as explained in~\cref{subsubsec:examples of sc functions}.
To apply our framework, we reformulate~\cref{eq:MEB} as a convex optimization problem over the following bounded domain.
\begin{lem}\label{lem:D for MEB}
  Set $S_0 =  \max_{i \neq j} d(p_i, p_j)^2$.
  Let $D \subseteq M \times \RR$ be defined by
  \begin{equation}
    D = \{ (p,S) \in M \times \RR \mid  d(p, p_i)^2 < S < 2S_0 \quad (i=1,2,\ldots,m)\}.
  \end{equation}
  Then $D$ is convex, open, bounded and non-empty, as $(p_j, \frac{3}{2}S_0) \in D$ for every~$j = 1, \dotsc, m$.
\end{lem}
\begin{proof}
  Since $D$ is the intersection of open epigraphs of squared distance functions and an open halfspace defined by $S < 2S_0$, it is open and convex.
  The boundedness of~$D$ is clear, as is the containment~$(p_j, \frac{3}{2}S_0) \in D$ for every~$j = 1, \dotsc, m$.
\end{proof}
Clearly, the optimal radius of a MEB is at most~$R_0 := \max_{i \neq j} d(p_i, p_j) = \sqrt{S_0}$.
% One can even pick~$R_0 = \min_i \max_j$, but this helps at most by a factor two, and may make bounds later less clear.
It is also at least half of that:
\begin{lem}
  \label{lem:MEB optimal value}
  Let~$R_*$ be the optimum of~\cref{eq:MEB} and~$R_0 = \max_{i \neq j} d(p_i, p_j)$.
  Then~$2 R_* \geq R_0$.
\end{lem}
\begin{proof}
  For every~$p \in M$, we have
  \begin{equation*}
    d(p_i, p_j) \leq d(p_i, p) + d(p, p_j) \leq 2 \max_k d(p_k, p).
  \end{equation*}
  Minimizing the right-hand side with respect to~$p \in M$ yields~$d(p_i, p_j) \leq 2 R_*$ for every~$i, j$; maximizing over~$i \neq j$ gives the desired bound.
\end{proof}
Replacing the objective function~$R$ by $R^2 = S$, finding the MEB is equivalent to solving
\begin{equation}\label{eqn:MEB'}
  \text{minimize } S \text{ s.t. } (p,S) \in D.
\end{equation}
As a natural application of our results, we obtain a self-concordant barrier for $D$.
\begin{prop}
  Let $D$ be as in~\cref{lem:D for MEB}.
  Define $G\colon D \to \RR$ by
  \[
    G(p,S) = - \log (2S_0 -S) + \sum_{i=1}^m \left(- \log (S- d(p,p_i)^2) + \frac{1}{2}d(p,p_i)^2 \right).
  \]
  Then $G$ is a self-concordant barrier for $D$, with barrier parameter $\theta = 1 + m (1 + 2 S_0)$.
\end{prop}
\begin{proof}
  Let $F_i(p,S) := - \log (S- d(p,p_i)^2) + \frac{1}{2}d(p,p_i)^2$.
  By~\cref{cor:distsq sym space self-concordant} and~\cref{prop:hadamard distsq epigraph barrier}, $F_i$ is is $1$-self-concordant.
  Furthermore, it satisfies $\lambda_{F_i}(p,S)^2 \leq 1 + d(p,p_i)^2 \leq 1 + 2 S_0$.
  As~$-\log(2 S_0 - S)$ is $1$-self-concordant, so is~$G$.
  The Newton decrement of~$G$ then satisfies~$\lambda_G(p, S)^2 \leq 1 + m (1 + 2 S_0)$.
  Hence $G$ is a self-concordant barrier with the claimed parameter.
\end{proof}

To initialize the path-following method, we use the damped Newton method from~\cref{thm:damped}.
To estimate its iteration complexity, we need a lower bound on $G$.
\begin{lem}
  For every $(p,S) \in D$, we have
  \[
    G(p,S) \geq - (1+m) \log (2S_0).
  \]
\end{lem}
\begin{proof}
  Since $x \mapsto - \log(x)$ is decreasing, $d(p,p_i)^2 \geq 0$ and~$S > 0$, we have $G(p,S) \geq - \log (2S_0) - m \log (2S_0) = - (1+m) \log (2S_0)$.
\end{proof}
The main result of this subsection is then the following.
\begin{thm}
  \label{thm:meb complexity}
  Let $p_1,p_2,\ldots,p_m \in M$, and let $R_*$ denote the radius of the minimum enclosing ball for these points.
  Set $R_0 = \max_{i\neq j} d(p_i,p_j)$.
  For $\eps > 0$, with~$\bigO(mR_0^2)$ iterations of a damped Newton method and
  \[
    \bigO\parens*{\sqrt{1+m(R_0^2+1)} \log\parens*{\frac{m (R_0^2 + 1)}{\eps}}}
  \]
  iterations of the path following method, one can find $(p_\eps,R_\eps) \in M \times \RR$ such that $R_\eps \leq R_* + \eps$, and the ball with center~$p_\eps$ and radius $R_\eps$ includes $p_1,p_2,\ldots,p_m$.
\end{thm}
\begin{proof}
  Set~$\lambda^{(1)} = \frac14$, $\lambda^{(2)} = \frac19$.
  The damped Newton method of \cref{thm:damped} with starting point~$(p_j, \frac{3}{2} S_0)$ yields a point~$(q, S)$ with~$\lambda_G(q, S) \leq \frac{1}{2} \lambda^{(1)}$ within the order of
  \begin{align*}
      & \frac{G(p_j, \frac32 S_0) - \inf_{(p,S) \in D} G(p,S)}{\frac12 \lambda^{(1)}} \\
      & \leq \frac{- \log (S_0/2) +  \sum_{i=1}^m (- \log ((3/2)S_0 - d(p_j,p_i)^2) + (1/2)d(p_j,p_i)^2) + (1+m) \log (2S_0)}{\frac{1}{2} \lambda^{(1)}} \\
      & \leq \frac{- \log (S_0/2) - m \log (S_0/2) + (m/2) S_0 + (1+m) \log (2S_0)}{\frac{1}{2} \lambda^{(1)}} \\
      & = \frac{(1+m) \log 4 + (m/2) S_0 }{\frac{1}{2} \lambda^{(1)}}
  \end{align*}
  iterations. Consider the path-following method in \cref{thm:main stage} from the initial point $(q,S)$, with objective~$s\colon D \to \RR$ defined by $(p,S) \mapsto S$.
  Since this is a linear map,~$t \, s + G$ is~$1$-self-concordant for all~$t > 0$.
  The starting time $t_0$ is given by
  \begin{equation*}
    t_0 = \frac{\lambda^{(1)} - \lambda_G(q,S)}{\norm{ds_{(q,S)}}_{G,(q,S)}^*} \geq \frac{\lambda^{(1)} - \lambda_G(q,S)}{2S_0},
  \end{equation*}
  where $\norm{ds_{(q,S)}}_{G,(q,S)}^*$ is bounded by $2S_0$ by \cref{lem:local norm of derivative bound}.
  Thus the path-following method yields a sequence of points $(q_l,S_l)$ such that
  \[
    S_l - R_*^2 \
    \leq \frac{8S_0 (\theta + 1)}{\lambda^{(1)}} \exp\parens*{- l \, \frac{\lambda^{(1)} - \lambda^{(2)}}{\lambda^{(1)} + \sqrt{\theta}}},
  \]
  where $\theta = 1+ m(1 + 2 S_0)$ is the barrier parameter of $G$ and we used $\lambda(q,S) \leq \lambda^{(1)} / 2$.
  For $\eps' > 0$, after
  \begin{equation*}
    l \geq \frac{\frac14 + \sqrt{\theta}}{\frac14 - \frac19} \log\parens*{\frac{32(\theta + 1)}{\eps'}}
  \end{equation*}
  iterations, we have
  \[
    S_l - R_*^2 \leq \eps' S_0.
  \]
  For $R_l = \sqrt{S_l}$, it holds that
  \[
    R_l - R_* \leq \eps' S_0/(R_l + R_*)  \leq  \eps' S_0/ 2R_*  \leq \eps' S_0/R_0,
  \]
  where the last inequality follows from~\cref{lem:MEB optimal value}.
  % R_* = max_i d(p_*, p_i)
  % Now d(p_i, p_j) <= d(p_i, p_*) + d(p_*, p_j) <= 2 R_*
  Therefore, choosing $\eps' = \eps R_0/S_0 = \eps/R_0$ yields the desired estimate.
\end{proof}

\subsection{The geometric median on model spaces}
\label{subsec:geometric median}
In this subsection we show how to apply the methods from \cref{sec:barriers compatibility path-following} to compute geometric medians on the model spaces~$M_{-\kappa}^n$ for constant sectional curvature~$-\kappa$, where~$\kappa > 0$.
For now, we shall work with general~$M$; later, we restrict to the model spaces because it is there that we have a barrier for the epigraph of the distance function (cf.~\cref{thm:hypn dist epigraph barrier}).
Recall from the introduction that the geometric median problem is as follows: given~$m \geq 3$ points~$p_1, \dotsc, p_m \in M$, not all contained in a single geodesic, find~$p_0 \in M$ such that
\begin{equation}
  \label{eq:geometric median problem}
  p_0 \in \argmin_{p \in M} \medianobj(p) := \sum_{i=1}^m d(p, p_i).
\end{equation}
This is a convex optimization objective, as the distance to a point is convex by \cref{lem:distsq gradient and hess lower bound}.
Let us first construct define a suitable domain to optimize over.
\begin{lem}
  \label{lem:geometric median domain}
  Set~$R_0 = \max_{i \neq j} d(p_i, p_j)$.
  Let~$D \subseteq M \times \RR$ be defined by
  \begin{align*}
    D = \{ (p, R) \in M \times \RR^m: R_i^2 > d(p, p_i)^2, \, 2 R_0 > R_i > 0 \}.
  \end{align*}
  Then~$D$ is convex, open, and non-empty: for every~$j \in [m]$, we have~$(p_j, \frac{3}{2} R_0 \, \vec{1}) \in D$, where~$\vec{1} \in \RR^m$ is the all-ones vector.
\end{lem}
\begin{proof}
  The convexity of~$D$ follows from the convexity of the distance function, see \cref{lem:distsq gradient and hess lower bound}.
  The fact that~$D$ is open is obvious.
  Lastly, the given points are in~$D$ because
  \begin{equation*}
    d(p_j, p_i) \leq R_0 < \frac{3}{2} R_0. \qedhere
  \end{equation*}
\end{proof}
\begin{lem}
  \label{lem:geometric median domain appropriate}
  Define~$c\colon M \times \RR^m \to \RR$ by~$c(p, R) = \sum_{i=1}^m R_i$, and let~$\medianobj\colon M \to \RR$ be as in \cref{eq:geometric median problem}.
  Then
  \begin{equation*}
    \inf_{(p, R) \in D} c(p, R) = \inf_{p \in M} \medianobj(p)
  \end{equation*}
\end{lem}
The proof relies on the fact that the geometric median of~$p_1, \dotsc, p_m$ is contained in the convex hull of these points, for which we essentially follow the argument given in~\cite[Prop.~2.4]{yangRiemannianMedianIts2010}, where this fact is proven for more general distributions (rather than just discrete distributions).
\begin{proof}
  First, we observe that for fixed~$(p, R) \in D$,
  \begin{equation*}
    \inf_{R' : (p, R') \in D} c(p, R') = \sum_{i=1}^m d(p, p_i) = \medianobj(p).
  \end{equation*}
  Thus it suffices to prove that if~$p_0 \in \argmin_{p \in M} \medianobj(p)$, then there exists some~$R \in \RR^m$ such that~$(p_0, R) \in D$.
  We claim that any such~$p_0$ is in the convex hull of the~$p_j$.
  From this claim one immediately deduces that~$(p_0, R) \in D$ for~$R = \frac32 R_0 \vec{1}$, since by \cref{lem:geometric median domain}, $D$ is convex and $(p_j, \frac32 R_0 \vec{1}) \in D$ for every~$j \in [m]$.

  We now establish the claim by proving its contrapositive.
  Suppose~$p$ is not in the convex hull~$C$ of the points~$p_1, \dotsc, p_m$, and let~$q$ be the projection of~$p$ onto~$C$, which is automatically distinct from~$p$.
  We use the notion of Alexandrov angle, which for three points~$a,b,c \in M$ with~$a \neq b,c$ is defined as the unique~$\angle_a(b,c) \in [0,\pi]$ such that
  \begin{equation*}
    \cos \angle_a(b, c) = \frac{\braket{\Exp_{a}^{-1}(b), \Exp_{a}^{-1}(c)}_a}{d(a, b) \, d(a, c)}.
  \end{equation*}
  Suppose first that~$q = p_j$ for some~$j \in [m]$.
  Then~$\angle_p(q, p_j) = 0$.
  On the other hand, if~$q \neq p_j$, by~\cite[Prop.~II.2.4]{bridson-haefliger-nonpositive-curvature}, we have~$\angle_{q}(p, p_j) \geq \pi/2$.
  On a Hadamard manifold, the angles of a triangle add to at most~$\pi$, hence~$\angle_p(q, p_j) \leq \pi/2$.
  Since we have~$m \geq 3$, there must exist at least two~$j$ such that~$q \neq p_j$.
  Furthermore, for at least one such~$j$, the inequality must be strict: if the inequality is not strict then we must have~$\angle_{p_j}(q, p) = 0$, so $p, q, p_j$ all lie on a single geodesic.
  Since~$p$ is distinct from all~$p_j$, $s$ is differentiable at~$p$, and it follows from \cref{lem:distsq gradient and hess lower bound} that
  \begin{equation*}
    \grad(\medianobj)_p = - \sum_{j=1}^m \frac{\Exp_{p}^{-1}(p_j)}{d(p, p_j)}.
  \end{equation*}
  Since we have shown that $\angle_p(q, p_j) \leq \pi/2$ for every~$j \in [m]$, with strict inequality for at least one~$j$, we have
  \begin{equation*}
    \braket{\grad(s)_p, \Exp_p^{-1}(q)}_p = - d(p,q) \sum_{j=1}^m \cos \angle_p(q, p_j) < 0
  \end{equation*}
  because~$d(p, q) \neq 0$.
  In particular, $\grad(s)_p \neq 0$ and~$p$ is not a minimizer of~$s$.
\end{proof}

We now construct a barrier for the domain~$D$.
From here onwards, we assume that~$M = M_{-\kappa}^n$ with~$\kappa > 0$.
\begin{prop}
  \label{prop:median domain barrier}
  Let~$D$ be as in \cref{lem:geometric median domain}.
  Define~$G \colon D \to \RR$ by
  \begin{align*}
    G(p, R) % & = \sum_{i=1}^m (\Psi(R_i) + 2 F_i(p, R_i, R_i)) \\
            & = \sum_{i=1}^m \left(-\log(2 R_0 - R_i) - 2 \log(R_i^2 - d(p, p_i)^2) + 2 \kappa \, d(p, p_i)^2\right).
  \end{align*}
  Then~$G$ is a self-concordant barrier for~$D$, with barrier parameter~$\theta = 5 m + 16 m \kappa R_0^2$.
\end{prop}
\begin{proof}
  Let~$\Psi(r) = -\log(2 R_0 - r)$ and recall from \cref{thm:hypn dist epigraph barrier} that~$F_i(p, R, S) = -\log(RS - d(p, p_i)^2) + \kappa \, d(p, p_i)^2$ is strongly~$\frac12$-self-concordant.
  Using \cref{lem:basic} and the strong $1$-self-concordance of~$-\log(2 R_0 - R)$, we deduce that~$G$ is strongly $1$-self-concordant.
  Then for every~$(p, R) \in D$, we have
  \begin{equation*}
    d((p, R), (p_1, R_0))^2 = d(p, p_1)^2 + \abs{R - R_0}^2 \leq 2 R^2 - 2 R_0 R + R_0^2 \leq 8 R_0^2 - 4 R_0^2 + R_0^2 = 5 R_0^2.
  \end{equation*}
  where~$d$ on the left-hand side refers to the distance on~$M \times \RR$.
  Furthermore, for every~$(p, R) \in D$, the bound on~$\lambda_{F_i, 1/2}(p, R, S) = \lambda_{2 F_i, 1}(p, R, S)$ from \cref{thm:hypn dist epigraph barrier} implies that
  \begin{equation*}
    \lambda_G(p, R)^2 \leq \sum_{i=1}^m \lambda_\Psi(R_i)^2 + \lambda_{2 F_i, 1}^2(p, R, R) \leq m + \sum_{i=1}^m (4 + 4 \kappa \, d(p, p_i)^2)
    \leq 5 m + 16 m \kappa R_0^2.
  \end{equation*}
  Therefore~$G$ is a self-concordant barrier with barrier parameter~$\theta = 5 m + 16 m \kappa R_0^2$.
\end{proof}

We now consider how to initialize the path-following method for the objective
\begin{equation*}
  c(p, R) = \sum_{i=1}^m R_i,
\end{equation*}
which is such that~$t c + G$ is~$1$-self-concordant for every~$t \geq 0$, because~$c$ is linear.
To apply \cref{thm:main stage}, we need to find a point~$(q, S) \in D$ such that~$\lambda_G(q,S) < \lambda^{(1)}$.\footnote{For \emph{fixed}~$q$, it is easy to determine the optimal~$S$, by explicitly solving the first-order optimality conditions.}
We can do this using the damped Newton method from \cref{thm:damped}.
To bound the number of iterations, we must bound the potential gap of~$G$.
\begin{lem}
  \label{lem:geometric median barrier potential gap}
  For every~$(p, R) \in D$, we have
  \begin{equation*}
    G(p, R) \geq - m \log(32 R_0^5).
  \end{equation*}
\end{lem}
\begin{proof}
  The function~$x \mapsto -\log(x)$ is decreasing.
  Because~$R_i > 0$ for every~$i \in [m]$, % 2 R_0 - R_i \leq 2 R_0
  we have~$-\log(2R_0 - R_i) \geq -\log(2R_0)$.
  Similarly, because~$R_i < 2 R_0$ and~$d(p, p_i) \geq 0$ for every~$i \in [m]$, each~$-\log(R_i^2 - d(p, p_i)^2)$ term is at least~$-\log(4R_0^2)$.
  % R_i^2 - d(p, p_i)^2 \leq 4 R_0^2
  Hence~$G(p, R) \geq -m \log(2R_0) - 2 m \log(4 R_0^2) = -m \log(32 R_0^5)$, concluding the proof.
\end{proof}
We now prove the main result of this subsection.
\begin{thm}
  \label{thm:geometric median complexity}
  Let~$p_1, \dotsc, p_m \in M_{-\kappa}^n$ with~$\kappa > 0$ be~$m \geq 3$ points, not all on one geodesic, and set $R_0 = \max_{i \neq j} d(p_i, p_j)$.
  Define~$\medianobj(p) = \sum_{j=1}^m d(p, p_j)$, and let~$\eps > 0$.
  Then with~$\bigO((m + 1) \kappa R_0^2)$ iterations of a damped Newton method and
  \begin{equation*}
    \bigO\parens*{\sqrt{m (\kappa R_0^2 + 1)} \log \left( \frac{m R_0 (\kappa R_0^2 + 1)}{\eps} \right)}
  \end{equation*}
  iterations of the path following method, one can find~$p_\eps \in M_{-\kappa}^n$ such that
  \begin{equation*}
    \medianobj(p_\eps) - \inf_{q \in M} \medianobj(q) \leq \eps.
  \end{equation*}
\end{thm}
\begin{proof}
  Set~$\lambda^{(1)} = \frac14$, $\lambda^{(2)} = \frac19$.
  The damped Newton method of \cref{thm:damped} with starting point~$(p_j, \frac{3}{2} R_0 \vec{1})$ yields a point~$(q, S)$ with~$\lambda_G(q, S) \leq \frac{1}{2} \lambda^{(1)}$ within the order of
  \begin{align*}
  & \frac{G(p_j, \frac32 R_0 \vec{1}) - \inf_{(p,R) \in D} G(p,R)}{\frac12 \lambda^{(1)}} \\
  & \leq \frac{G(p_j,\frac{3}{2} R_0 \vec{1}) + m \log(32 R_0^5)}{\frac{1}{2} \lambda^{(1)}} \\
  & = \frac{-m \log(\frac{R_0}{2}) - 2 \sum_{i=1}^m \log(\frac{9}{4} R_0^2 - d(p_j, p_i)^2) + m \log(32 R_0^5) + 2 \kappa \sum_{i=1}^m d(p_j, p_i)^2}{\frac{1}{2} \lambda^{(1)}} \\
  & \leq \frac{-m \log(\frac{R_0}{2}) - 2 \sum_{i=1}^m \log(\frac{5}{4} R_0^2) + m \log(32 R_0^5) + 8 \kappa m R_0^2}{\frac{1}{2} \lambda^{(1)}} \\
  & = \frac{m \log(\frac{1024}{25}) + 8 m \kappa R_0^2}{\frac{1}{2} \lambda^{(1)}}
  \end{align*}
  iterations.
  A suitable choice of starting time is for the path-following method from \cref{thm:main stage} is then
  \begin{equation*}
    t_0 = \frac{\lambda^{(1)} - \lambda_G(q,S)}{\norm{dc_{(q,S)}}_{G,(q,S)}^*}.
  \end{equation*}
  It remains to be shown that this is not too small.
  We give an upper bound on~$\norm{dc_{(q,S)}}_{G,(q,S)}^*$.
  The domain~$D$ is constructed so that~$c(p,R) \leq 2 m R_0$ for every~$(p,R) \in D$, and~$c(q,S) \geq 0$.
  It follows by \cref{lem:local norm of derivative bound} that
  \begin{equation*}
    \norm{dc_{(q,S)}}_{G,(q,S)}^* \leq 2 m R_0,
  \end{equation*}
  and so~$t_0 \geq (\lambda^{(1)} - \lambda_G(q,S)) / (2 m R_0)$.
  Therefore, initializing the algorithm from \cref{thm:main stage} with initial point~$(q, S)$ and the above~$t_0$ yields a sequence of points~$(q_l, S_l)$ such that
  \begin{equation*}
    c(q_l, S_l) - \inf_{(p, R) \in D} c(p, R) \leq \frac{4 m R_0 (\theta + 1)}{\lambda^{(1)}} \exp\parens*{- l \, \frac{\lambda^{(1)} - \lambda^{(2)}}{\lambda^{(1)} + \sqrt{\theta}}}
  \end{equation*}
  where~$\theta$ is the barrier parameter of~$G$, and we used that~$\lambda^{(1)} - \lambda_G(q,S)$ is at least~$\frac12 \lambda^{(1)}$.
  Rewriting the above and using \cref{lem:geometric median domain appropriate} shows that
  \begin{equation*}
    \medianobj(q_l) - \inf_{q \in M} \medianobj(q) \leq c(q_l, S_l) - \inf_{(q,R) \in D} c(q,R) \leq \eps
  \end{equation*}
  whenever
  \begin{equation*}
    l \geq \frac{\frac14 + \sqrt{\theta}}{\frac14 - \frac19} \log\parens*{\frac{4 m R_0 (\theta + 1)}{\eps}}.
  \end{equation*}
  The theorem now follows from filling in~$\theta = 5 m + 16 m \kappa R_0^2$.
\end{proof}

\subsection{The Riemannian barycenter}
\label{subsec:riemannian barycenter}
We end this section by briefly commenting on the problem of finding the Riemannian barycenter, first introduced by Cartan, and sometimes also called the Fr\'echet or Karcher mean, see e.g.~\cite{afsariRiemannianCenterMass2011} for some historical context on this topic.
It is defined as follows: given points~$p_1, \dotsc, p_m \in M$, find~$p_0 \in M$
\begin{equation*}
  p_0 \in \argmin_{p \in M} f(p) := \sum_{i=1}^m d(p, p_i)^2.
\end{equation*}
The point~$p_0$ is known as the barycenter of~$p_1, \dotsc, p_m$, and is unique on Hadamard manifolds by strong convexity of~$f$.
It is trivial to find~$p_0$ when~$M = \RR^n$ is Euclidean space, as it is given by~$p_0 = \frac{1}{m} \sum_{i=1}^m p_i$.
Furthermore, the solution is unique on any Hadamard manifold, as the squared distance is~$2$-strongly convex, and hence~$f$ is~$2m$-strongly convex.
Even for hyperbolic space it is not clear whether one can give a closed-form solution to the above problem.
However, if~$M$ has sectional curvatures in~$[-\kappa, 0]$, then $f$ is~$\bigO(m \sqrt{\kappa} R / \tanh(R \sqrt{\kappa}))$-smooth at~$p$ with~$R = \max_j d(p, p_j)$, which follows from standard variational arguments~\cite[Prop.~10.12,~Thm.~10.22]{lee-riemannian-manifolds}, hence the function~$f$ is well-conditioned.
Therefore a standard gradient descent method gives an algorithm which converges relatively quickly; one can find an~$\eps$-approximate minimizer of~$f$ in~$\bigO(\sqrt{\kappa} R_0 \log([f(p) - \inf_q f(q)]/\eps) / \tanh(\sqrt{\kappa} R_0))$ iterations, where~$R_0$ is some a priori bound on size of the domain one restricts to, and~$p$ is the starting point.
This can be deduced from a simple adaptation of the Euclidean argument in~\cite[Thm.~3.8]{bansalPotentialFunctionProofsGradient2019}).
We note that one could also apply more sophisticated first-order methods such as accelerated gradient descent to this problem, see~\cite{ahnNesterovEstimateSequence2020a}.

It is natural to determine what complexity our interior-point methods give for this problem.
In the setting of~$M = M_{-\kappa}^n$, we can (up to logarithmic factors) recover the above iteration complexity.
We restrict the above optimization problem to a ball of radius~$R_1 = \max_{j \neq 1} d(p_1, p_j)$ around the point~$p_1$, and use the barrier~$F(p) = -\log(R_1^2 - d(p, p_1)^2) + \kappa d(p, p_1)^2$, which has barrier parameter~$1 + \bigO(\kappa R_1^2)$.
Then, observe that by \cref{thm:SC_hyperbolic}\ref{item:SC_hyperbolic compatibility} and \cref{lem:compatibility conic}, $f$ is~$(\sqrt{2} \zeta \sqrt{\kappa}, \sqrt{2 \kappa})$-compatible with any squared distance function, as each of the~$d(p, p_i)^2$'s is.
As a consequence,~$f$ is~$(\sqrt{2} \zeta, \sqrt{2})$-compatible with~$F$, and~$tf + F$ is~$\bigO(1)$-self-concordant for every~$t \geq 1$ by \cref{prop:compatible functions self concordance}.
The path-following method, initialized with starting point~$p_1$ (which is the analytic center of~$F$), then yields an~$\eps$-approximate minimizer of~$f$ within~$\bigO\parens{(1 + \sqrt{\kappa} R_1) \log(m \kappa R_1/\eps)}$ iterations.
While this specific choice of barrier may seem odd, it has the advantage that we know its analytic center to be~$p_1$, so it is easy to initialize the path-following method.
This shows again that it is useful to have a general path-following method capable of dealing with compatible objectives, rather than just linear ones: if one included a barrier term for the epigraph of every~$d(p, p_i)^2$, then it would both be harder to find the analytic center (for initialization), and the barrier parameter would scale with~$m$.
We note that a similar approach works on~$\PD(n)$ if one suitably generalized~\cref{thm:SC_hyperbolic}\ref{item:SC_hyperbolic compatibility}.

%=============================================================================
\section{Outlook}
\label{sec:outlook}
%=============================================================================
In this work, we extend the basic theory of interior-point methods to manifolds, and show that the developed framework is capable of capturing interesting geodesically convex optimization problems.
In particular, we define a suitable version of self-concordance on Riemannian manifolds, and show that it gives the same guarantees for Newton's method as in the Euclidean setting.
This is used to analyze a path-following method for the optimization of compatible objectives over domains for which one has a self-concordant barrier.
We exhibit non-trivial examples of self-concordant functions, namely squared distance functions on~$\PD(n)$, and more generally symmetric spaces with non-positive curvature, and construct related self-concordant barriers.
The framework is able to capture the optimization of Kempf--Ness functions, a problem which has connections to many areas of mathematics and computer science, leading to algorithms with state-of-the-art complexity guarantees.
It also applies to computing the geometric median on hyperbolic space, for which we give an algorithm capable of finding high-precision solutions.
This demonstrates the power of the framework, and we believe that it encompasses many more problems.
Our work suggests several directions for further investigation:
\begin{itemize}
\item It is natural to search for self-concordant barriers for the aforementioned applications which have better barrier parameters.
Alternatively, is it possible to prove lower bounds that show that the constructions given in our work are essentially optimal?

\item In Euclidean convex optimization, there are universal constructions of self-concordant barriers, cf.~\cite{nesterov-nemirovskii-ipm,hildebrand2014canonical,fox2015schwarz,bubeck2019entropic,chewi2021entropic}.
Can one find such a construction for manifolds?
We describe a concrete proposal.
Let~$D \subseteq M$ be a compact convex subset of a Hadamard manifold~$M$, with non-empty interior.
Denote by~$CM^\infty$ the cone over the boundary at infinity of~$M$~\cite{hiraiConvexAnalysisHadamard2022}.
Its elements can be identified with the geodesic rays~$\gamma$ emanating from a fixed base point and hence determine Busemann functions~$b_\gamma$ as in \cref{eq:busemann function defn}.
% On $\PD(n)$, they are related to the Kempf--Ness functions, cf.\ \cref{subsec:kempf-ness}.
Define~$F^*\colon CM^\infty \to \RR$ by $F^*(\gamma) = \log \int_D \exp(- b_\gamma(q)) \, d \mathrm{vol}(q)$.
Then the inverse Legendre--Fenchel conjugate~$F\colon D \to \RR$ of~$F^*$, given by $F(p) = \sup_{\gamma \in CM^\infty} - b_\gamma(p) - F^*(\gamma)$, is a natural candidate for a barrier for~$D$.
Indeed, for Euclidean space $M=\R^n$ it reduces precisely to the \emph{entropic barrier} of Bubeck and Eldan~\cite{bubeck2019entropic}.

\item From the perspective of interior-point methods, we currently only treat the \emph{main stage}, which minimizes an objective given a starting point that is well-centered with respect to the barrier~$F$.
Can one give a general procedure for finding such a starting point from an arbitrary feasible point~$p\in D$?
In the Euclidean setting, this is achieved by applying the path-following method with the linear objective~$f := -\braket{\grad(F)_p, \cdot}$ \emph{in reverse}, starting at~$t=1$.
This is sensible as~$p$ is exactly a minimizer of~$F_t = tf + F$ at~$t=1$.
Busemann functions generalize linear functions to Hadamard manifolds, hence is natural to instead use~$f=b_\gamma$ with~$\gamma$ the geodesic ray starting at~$p\in M$ with direction~$\grad(F)_p$.
When~$f$ is compatible with~$F$ (as we show in \cref{subsec:kempf-ness} for specific~$f$ and~$F$), then one can use the same time steps as for the main stage, and switch to the main stage as soon as~$\lambda_{F,\alpha} \leq \frac13$.
One method for lower bounding the~$t$ for which this happens is as follows: if~$F$ is~$\mu$-strongly convex and~$f$ is~$\nu$-smooth, then~$\lambda_{F,\alpha}(q)$ is at most~$\lambda_{F_t,\alpha}(q) \sqrt{1 + t \nu/\mu} + t \norm{df_q}_{F,q,\alpha}^* / \alpha$, and~$\norm{df_q}_{F,q,\alpha}^*$ can be bounded (for instance) using Lipschitzness of~$f$ and strong convexity of~$F$.
% \begin{align*}
%   \lambda_{F}(q) & = \sup_u \frac{\abs{dF_q(u)}}{\sqrt{\nabla^2 F_q(u, u)}} \\
%                  & \leq \sup_u \frac{\abs{d(F_t)_q(u)}}{\sqrt{\nabla^2 F_q(u, u)}} + t \sup_u \frac{\abs{df_q(u)}}{\sqrt{\nabla^2 F_q(u, u)}} \\
%                  & = \sup_u \frac{\abs{(F_t)_q(u)}}{\sqrt{\nabla^2 (F_t)_q(u,u)}} \cdot \frac{\sqrt{\nabla^2 (F_t)_q(u,u)}}{\sqrt{\nabla^2 F_q(u,u)}} + t \norm{df_q}_{F,q}^* \\
%                  & \leq \sup_u \frac{\abs{(F_t)_q(u)}}{\sqrt{\nabla^2 (F_t)_q(u,u)}} \cdot \sqrt{1 + \frac{t \, \nu}{\mu}} + t \norm{df_q}_{F,q}^* \\
%                  & = \lambda_{F_t}(q) \cdot \sqrt{1 + \frac{t \, \nu}{\mu}} + t \norm{df_q}_{F,q,\alpha}^* / \sqrt{\alpha}
% \end{align*}
% Note that the importance of the first term on the right hand side can be significantly weakened by performing more Newton steps while keeping the time parameter the same.
We leave a more careful analysis of this idea to future work.
We note that in the Euclidean setting, the complexity is often bounded in terms of the \emph{asymmetry} of domain~$D$ with respect to the point~$p$, see~\cite[Eq.~(3.2.24)]{nesterov-nemirovskii-ipm} for details, but such a bound does not seem to generalize to the Riemannian setting.

\item It would be interesting to understand whether there is a suitable notion of primal-dual methods in the Riemannian setting, or a notion of duality which interacts well with self-concordance.
While there exists a version of Legendre--Fenchel duality for Hadamard manifolds~$M$, where the dual space is~$CM^\infty$, the cone over the boundary at infinity of~$M$ discussed above, the conjugate of a convex function need not be convex~\cite{hiraiConvexAnalysisHadamard2022}.
Other proposals such as~\cite{bergmannFenchelDualityTheory2021} require a stronger notion of convexity.

\end{itemize}

\iffocs\else
%=============================================================================
\section*{Acknowledgements}
% \addcontentsline{toc}{section}{Acknowledgements}
%=============================================================================
We thank Peter B\"urgisser and Cole Franks for delightful discussions.
HN and MW acknowledge grant OCENW.KLEIN.267 by the Dutch Research Council (NWO).
MW acknowledges support by the European Union (ERC, SYMOPTIC, 101040907), by the Deutsche For\-schungs\-ge\-mein\-schaft (DFG, German Research Foundation) under Germany's Excellence Strategy - EXC\ 2092\ CASA - 390781972, and by the BMBF through project QuBRA.
HH acknowledges support by JST PRESTO Grant Number JPMJPR192A, Japan.
Views and opinions expressed are those of the author(s) only and do not necessarily reflect those of the European Union or the European Research Council Executive Agency.
Neither the European Union nor the granting authority can be held responsible for them.
\fi

\bibliographystyle{alphaurl}
\bibliography{barriers}

\end{document}